\documentclass[10pt]{article}
\usepackage{amsmath}
\usepackage{amsthm}
\usepackage{amssymb}
\usepackage{xypic}
\usepackage{tikz}
\usepackage[margin=1in]{geometry}

\makeatletter
\renewcommand*\l@section[2]{%
   \ifnum \c@tocdepth >\m@ne
     \addpenalty{-\@highpenalty}%
     \vskip 1.0em \@plus\p@
     \setlength\@tempdima{1.5em}%
     \begingroup
       \parindent \z@ \rightskip \@pnumwidth
       \parfillskip -\@pnumwidth
       \leavevmode \bfseries
       \advance\leftskip\@tempdima
       \hskip -\leftskip
       #1\nobreak\leaders\hbox{\normalfont$\m@th
         \mkern \@dotsep mu\hbox{.}\mkern \@dotsep
         mu$}\hfill \nobreak\hb@xt@\@pnumwidth{\hss #2}\par
       \penalty\@highpenalty
     \endgroup
   \fi}
\makeatother

\newtheorem{theorem}{Theorem}[section]
\newtheorem{cor}[theorem]{Corollary}
\newtheorem{prop}[theorem]{Proposition}
\newtheorem{lemma}[theorem]{Lemma}

\theoremstyle{definition}
\newtheorem{definition}[theorem]{Definition}

\theoremstyle{definition}

\theoremstyle{remark}

\theoremstyle{definition}
\newtheorem{example}[theorem]{Example}

\theoremstyle{remark}
\newtheorem{remark}[theorem]{Remark}

\theoremstyle{remark}

\theoremstyle{definition}
\newtheorem{property}[theorem]{Property}

\theoremstyle{definition}

\renewcommand{\span}{\textrm{span}\hspace{2pt}}
\newcommand{\R}{\mathbb{R}}
\newcommand{\C}{\mathbb{C}}
\newcommand{\Z}{\mathbb{Z}}

\renewcommand{\hom}{\text{Hom}}
\newcommand{\fg}{\frak{g}}
\newcommand{\pt}{\displaystyle\frac{\partial}{\partial t}}

\begin{document}

\begin{titlepage}

\begin{tikzpicture}[remember picture,overlay]
\node [yshift=-2in] at (current page.north)
[text width = \textwidth, anchor=base, text centered]
{
\large A CLASSIFICATION OF TORIC, FOLDED-SYMPLECTIC MANIFOLDS
};
\node [yshift=-3.6in] at (current page.north)
[text width=.1\textwidth, anchor=base, text centered]
{
\large BY
};
\node [yshift=-4in] at (current page.north)
[text width = 0.7\textwidth, anchor=base, text centered]
{
\large DANIEL HOCKENSMITH
};
\node [yshift=-5.72in] at (current page.north)
[text width = 0.7\textwidth, anchor=base, text centered]
{
\large DISSERTATION
};
\node [yshift=-6in] at (current page.north)
[text width = 0.7\textwidth, anchor=base, text centered]
{
\normalsize Submitted in partial fulfillment of the requirements \\ for the degree of Doctor of Philosophy in Mathematics \\
in the Graduate College of the \\
University of Illinois at Urbana-Champaign, 2015
};

\node [yshift=-7.875in] at (current page.north)
[text width = 0.7\textwidth, anchor=base, text centered]
{
\normalsize Urbana, Illinois
};

\node [yshift=-8.375in] at (current page.north)
[text width = 0.7\textwidth, anchor=base]
{
\normalsize Doctoral Committee: \\ \mbox{ } \\
\hspace{.5in} Professor Ely Kerman, Chair \\
\hspace{.5in} Professor Eugene Lerman, Director of Research \\
\hspace{.5in} Professor Susan Tolman \\
\hspace{.5in} Doctor Jordan Watts
};
\end{tikzpicture}

\end{titlepage}

\pagebreak

\pagenumbering{roman}
\setcounter{page}{2}

\begin{center}
\Large{\textbf{ABSTRACT}}
\end{center}

Given a $G$-toric, folded-symplectic manifold with co-orientable folding hypersurface, we show that its orbit space is naturally a manifold with corners $W$ equipped with a smooth map $\psi: W \to \frak{g}^*$, where $\frak{g}^*$ is the dual of the Lie algebra of the torus, $G$.  The map $\psi$ has fold singularities at points in the image of the folding hypersurface under the quotient map and it is a unimodular local embedding away from these points.  Thus, to every $G$-toric, folded-symplectic manifold we can associate its orbit space data $\psi:W \to \fg^*$, a unimodular map with folds.  We fix a unimodular map with folds $\psi:W \to \fg^*$ and show that isomorphism classes of $G$-toric, folded-symplectic manifolds whose orbit space data is $\psi:W \to \fg^*$ are in bijection with $H^2(W; \mathbb{Z}_G\times \R)$, where $\mathbb{Z}_G= \ker(\exp:\frak{g} \to G)$ is the integral lattice of $G$.  Thus, there is a pair of characteristic classes associated to every $G$-toric, folded-symplectic manifold.  This result generalizes a classical theorem of Delzant as well as the classification of toric, origami manifolds, due to Cannas da Silva, Guillemin, and Pires, in the case where the folding hypersurface is co-orientable.

We spend a significant amount of time discussing the fundamentals of equivariant and non-equivariant folded-symplectic geometry.  In particular, we characterize folded-symplectic forms in terms of their induced map from the sheaf of vector fields into a distinguished sheaf of one-forms, we relate the existence of an orientation on the folding hypersurface of a fold-form to the intrinsic derivative of the contraction mapping from the tangent bundle to the cotangent bundle, and we show that $G$-toric, folded-symplectic manifolds are stratified by $K$-toric, folded-symplectic submanifolds, where $K$ varies over the subtori of $G$ and the action is principal on each stratum.  We show how these structures give rise to the rigid orbit space structure of a toric, folded-symplectic manifold used in the classification.  We also give a robust description of folded-symplectic reduction, which we use to construct local models of toric, folded-symplectic manifolds.

\pagebreak

\begin{center}
\Large{\textbf{ACKNOWLEDGEMENTS}}
\end{center}

This thesis is a culmination of my own work as well as the patience, encouragement, and thoughts of those around me.  I cannot thank Ely Kerman enough for introducing me to the world of differential geometry and Morse theory.  Furthermore, he has always been willing to lend his support, for which I am grateful.  Many thanks go to my colleagues who have been an inspiration through their superb intellect, wit, and curiosity.  In particular, thanks to Brian Collier for an unending willingness to discuss geometry and for offering the question: what is a Morse section of a line bundle?  Thanks to Daan Michiels for offering his perplexing question after a seminar on Morse theory: what are the other functions?  I have found some amusement in the answer that they are simply different types of Morse functions.  Thank you to Seth Wolbert and Sean Shahkarami for many discussions about differential geometry and for listening to some of my vague ramblings.

Many thanks go to the department of mathematics for fostering an environment of success, integrity, and interest.  My needs were always met, my questions always answered, and there was never an end to the wide variety of seminars available: it was truly a collegiate environment.  Thank you to the symplectic and differential geometers at UIUC.  In particular, thanks to Sue Tolman for serving on my committee, thanks to Jordan Watts for his diligence and insight, and thanks to Rui Fernandes for offering helpful comments regarding folded-symplectic reduction.  A special thanks to Jordan for organizing and maintaining the symplectic seminar.

Finally, I would like to thank two individuals to whom I am most grateful.  First, thank you to my adviser, Eugene Lerman.  I will be forever grateful to him for sharing his perspective and knowledge in regard to both symplectic geometry and mathematics in general.  His patience throughout my career is to be lauded.  Lastly, thank you to my wife, Anna Jean Wirth, for her ceaseless support, fantastic mind, and her ability to make the dreariest day beautiful.

\vfill

\let\thefootnote\relax\footnotetext{The author acknowledges support from National Science Foundation grant DMS 08-38434 ``EMSW21-MCTP: Research Experience for Graduate Students.''}

\pagebreak

\tableofcontents

\pagebreak

\cleardoublepage
\setcounter{page}{1}
\pagenumbering{arabic}

\section{Introduction}

A folded-symplectic manifold is a $2m$-dimensional manifold $M$ equipped with a closed $2$-form $\sigma$ so that $\sigma^m$ vanishes transversally and $\sigma$ restricted to its degenerate hypersurface $Z$ is maximally non-degenerate.  We call the hypersurface, $Z$, the \emph{folding hypersurface}.  Folded-symplectic forms arise naturally: the connected sum of two symplectic manifolds possesses a folded-symplectic structure \cite{CGW}, every four manifold possessses a folded-symplectic structure \cite{CdS}, and the pullback of a symplectic form by a map with fold singularities is a folded-symplectic form.  A toric, folded-symplectic manifold is a (connected) folded-symplectic manifold $(M^{2m},\sigma)$ equipped with an effective, Hamiltonian action of a torus of dimension $m$.  These completely integrable systems are generalizations of toric, symplectic manifolds, where we have simply inserted hypersurfaces on which the $2$-form governing the dynamics is allowed to degenerate.

The study of toric folded-symplectic manifolds is a union between two seemingly opposed mathematical viewpoints: toric symplectic geometry, where degeneracies are few and far between, and singularity theory, where one allows for smooth functions to degenerate in a controlled manner.  From the latter point of view, one could say folded-symplectic geometry began in 1969 with Martinet's study of generic singularities of differential forms \cite{M}.  Indeed, on a general four-manifold, the author shows that a two form has a singularity of type $\Sigma_{2,0}$ if it can be written in coordinates as $xdx\wedge dy + dz\wedge dt$, which is the canonical example of a folded-symplectic form.  Incidentally, one observes that $xdx\wedge dy + dz\wedge dt$ is the pullback of the standard symplectic structure $\omega_{\R^4}$ on $\R^4$ by the fold map $\psi(x,y,z,t)=(\frac{x^2}{2},y,z,t)$, which validates the name \emph{folded-symplectic}.

From the symplectic point of view, Cannas da Silva, Guillemin, and Woodward began a series of studies in folded-symplectic geometry in 1999-2000 focused on relating the existence of \emph{spin}-c structures to the existence of folded-symplectic forms (on connected sums of symplectic manifolds).  In \cite{CGW}, they describe an unfolding operation in which one may decompose a folded-symplectic manifold into disjoint symplectic pieces, provided the intersection $\ker(\sigma)\cap TZ$ defines a $1$-dimensional foliation of the folding hypersurface by circles.  This assumption that the null foliation induces a circle fibration permeates much of the literature on folded-symplectic manifolds.  In \cite{CGP}, Cannas da Silva, Guillemin, and Pires study \emph{toric origami manifolds}, which are compact toric folded-symplectic manifolds where the null foliation on the fold is generated by a locally free circle action, so that the folding hypersurface $Z$ fibrates over a compact base $B$.  They show that there is a one-to-one correspondence between toric origami manifolds and \emph{origami templates}, which are collections of unimodular polytopes in the dual of the Lie algebra of the torus, $\fg^*$.  This is a generalization of Delzant's classical theorem \cite{De} which states that there is a one-to-one correspondence between compact toric symplectic manifolds and unimodular polytopes in $\fg^*$.

The primary goal of this thesis is to extend these classification results to non-compact, toric, folded-symplectic manifolds where there is no assumption on the null-foliation on the fold.  Some work has already been done in this area: in \cite{LeeC} Lee gives a sufficient condition for the existence of isomorphisms between toric, folded-symplectic four manifolds and in \cite{KL} Karshon-Lerman classify all \emph{symplectic} toric manifolds.  In our classification, we make a mild assumption that the folding hypersurface is co-orientable.  As we show in chapter $4$, every folding hypersurface $Z$ with an Hamiltonian action of a Lie group $G$ may be equivariantly embedded into a folded-symplectic, Hamiltonian $G$-manifold $(M,\sigma)$ as a co-orientable folding hypersurface (q.v. corollary \ref{cor:preserves}), so this assumption isn't unfounded.  Furthermore, one can show that if the folding hypersurface is \emph{not} co-orientable, then the ambient folded-symplectic manifold $M$ cannot be orientable, in which case one need only pass to the orientable double cover $\tilde{M}$ of $M$ to be back in a situation where the folding hypersurface is co-orientable.  The primary reason we make this assumption about co-orientability is to avoid cases where the action of the torus $G$ on the folding hypersurface is \emph{not} effective.  We follow the approach of \cite{KL} in our classification of non-compact toric symplectic manifolds.  We first prove:

\begin{theorem}
Let $(M,\sigma,\mu:M\to \fg^*)$ be a toric, folded-symplectic manifold with moment map $\mu:M \to \fg^*$, where $\fg$ is the Lie algebra of the torus $G$ acting on $M$.  Assume the fold $Z\subseteq M$ is co-orientable.  Then,
\begin{enumerate}
\item $M/G$ is naturally a manifold with corners and
\item the moment map $\mu$ descends to a smooth map $\bar{\mu}:M/G \to \fg^*$, which is a unimodular map with folds (q.v. definition \ref{def:umf}).
\end{enumerate}
\end{theorem}
We then fix a manifold with corners $W$ and a unimodular map with folds $\psi:W \to \fg^*$ and we define a category $\mathcal{M}_{\psi}(W)$ (q.v. definition \ref{def:empsi} or below).

\begin{definition}
Let $W$ be a manifold with corners and let $\psi:W\to \fg^*$ be a unimodular map with folds, where $\fg$ is the Lie algebra of a torus $G$.  We define $\mathcal{M}_{\psi}(W)$ to be the category whose objects are triples:
\begin{displaymath}
(M,\sigma, \pi:M \to W)
\end{displaymath}
where $\pi$ is a quotient map and $(M,\sigma, \psi \circ \pi)$ is a toric, folded-symplectic manifold with co-orientable folding hypersurface, where the torus is $G$, with moment map $\psi \circ \pi$.  We refer to an object as a \emph{toric, folded-symplectic manifold over $\psi$}.  A morphism between two objects $(M_i,\sigma_i,\pi_i:M \to W)$, $i=1,2$, is an equivariant diffeomorphism $\phi:M_1 \to M_2$ that induces a commutative diagram:

\begin{displaymath}
\xymatrix{
M_1 \ar[rr]^{\phi} \ar[dr]^{\pi_1} & & M_2 \ar[dl]_{\pi_2} \\
                                   &W \ar[r]^{\psi}& \fg^*
}
\end{displaymath}
and satisfies $\phi^*\sigma_2=\sigma_1$.  That is, $\phi$ is an equivariant folded-symplectomorphism that preserves moment maps.  By definition, every morphism is invertible, hence $\mathcal{M}_{\psi}(W)$ is a groupoid.
\end{definition}

We seek to classify objects in $\mathcal{M}_{\psi}(W)$ up to isomorphism and we prove the following classifcation result.

\begin{theorem}
Isomorphism classes of objects in $\mathcal{M}_{\psi}(W)$ are in bijection with $H^2(W; \mathbb{Z}_G \times \R)$, where $\mathbb{Z}_G= \ker(\exp:\fg \to G)$ is the integral lattice of the torus $G$ that acts on objects of $\mathcal{M}_{\psi}(W)$.
\end{theorem}

The proof of these two results constitute the bulk of the material presented in chapters $5-8$.  The secondary goal of this thesis, comprising the remaining $3$ chapters, is to provide a foundational framework for the study of folded-symplectic forms.  Given the fact that folded-symplectic forms have recently found their way into fields such as four manifolds \cite{CdS} and Higgs bundles \cite{H}, it is the author's belief that a rigorous study of some of the structures associated to a fold form will be useful.

The organization of the thesis is as follows.  In chapter $2$ we begin with an introduction to jet bundles on manifolds with corners.  Jet bundles are typically only studied on manifolds (without corners).  Michor defines jet bundles for manifolds with corners in \cite{Mi2}, though they are defined using restrictions of jet bundles over $\R^n$ to quadrants.  We offer an equivalent approach to the construction of the first jet bundle which avoids restrictions and excessive choices of coordinates.  We then discuss the various structures associated with the first jet bundle, such as jet fields and connections, and move on to define fold maps for manifolds with corners.  There are two reasons we rigorously develop the notion of a fold singularity for maps between manifolds with corners.  First, the definition of a fold singularity of a map between manifolds with corners doesn't appear in the literature. It is tempting to take the definition of a fold singularity (e.g. in \cite{Ho} or \cite{GG}) and simply replace \emph{manifold} by \emph{manifold with corners}.  However, it is not difficult to construct examples of maps with only fold singularities under this definition whose locus of degeneracies is neither a manifold or a manifold with corners.  The issue is that a fold singularity is defined via a transversal intersection (q.v. definition \ref{def:folds}) and if one uses the traditional notion of transversality in the category of manifolds with corners, then inverse images of submanifolds with corners may fail to be submanifolds with corners (q.v. \cite{Mi2}).  Thus, the second reason we take time to develop the theory is because we must adopt an adequate notion of transversality and show that this notion leads to a theory of fold singularities in the category of manifolds with corners that is consistent with the theory in the category of manifolds.  We conclude the chapter with an exercise in the theory of first jet bundles where we give a generalization of Morse functions to Morse sections of $1$-dimensional fiber bundles.

We begin our study of folded-symplectic manifolds in chapter $3$, giving the definition along with some examples and then constructing a folded version of the cotangent bundle, which may be seen as an attempt to lengthen the list of naturally occurring sources of folded-symplectic forms.  This construction dualizes the construction of the $b$-cotangent bundle found in \cite{GMP} (example $9$).  We then discuss Moser's argument for folded-symplectic manifolds in detail, ultimately arriving at a characterization of folded-symplectic forms in terms of their induced maps on a distinguished pair of sheaves.  We conclude by discussing a non-equivariant normal form for a neighborhood of a co-orientable folding hypersurface, which was proven in the case where the manifold is compact and oriented in \cite{CGW} (the equivariant version may be found in chapter $4$).  To this end, we show that the fold inherits a canonical orientation from the fold form $\sigma$.  This fact is briefly discussed in \cite{M} using a choice of orientation near each point in the fold.  We give an alternate approach to show that the orientation arises from the intrinsic derivative of the contraction mapping $C_{\sigma}:TM \to T^*M$, $C_{\sigma}(X)=i_X\sigma$, from tangent bundle to cotangent bundle.

In chapters $4$ and $5$ we discuss Hamiltonian actions of Lie groups on folded-symplectic manifolds.  Most of the material in chapter $4$ is standard material in symplectic geometry that we will need to prove that the orbit space of a toric, folded-symplectic manifold (with co-orientable folding hypersurface) is a manifold with corners.  The two new results we prove are the following.  First, there is a well-defined symplectic normal bundle to orbit-type strata in an Hamiltonian folded-symplectic manifold.

\begin{prop}
Let $G$ be a compact, connected Lie group and let $(M,\sigma)$ be a folded-symplectic Hamiltonian $G$-manifold with moment map $\mu:M\to \fg^*$, where $\fg$ is the Lie algebra of $G$.  Suppose the folding hypersurface $Z$ is co-orientable.  Let $H\le G$ be a subgroup and suppose $M_H$ is nonempty.  Then there exists a vector bundle $\widetilde{(TM_H)}^{\sigma}\to M_H$ with the following properties:

\begin{enumerate}
\item $\widetilde{(TM_H)}^{\sigma}$ is a subbundle of $TM\big\vert_{M_H}$.
\item The restriction $\widetilde{(TM_H)}^{\sigma}\big\vert_{M\setminus Z}$ to the symplectic portion of $M$ is the vector bundle $T(M_H \setminus Z)^{\sigma}\to (M_H\setminus Z)$.
\item $TM\big\vert_{M_H}$ splits $H$-equivariantly as $TM\big\vert_{M_H} = TM_H \oplus \widetilde{(TM_H)}^{\sigma}$.
\item $\widetilde{(TM_H)}^{\sigma}$ equipped with the restriction of $\sigma$ is a symplectic vector bundle over $M_H$.
\item $\widetilde{(TM_H)}^{\sigma}\big\vert_{Z_H}$ is a subbundle of $TZ_H$.
\end{enumerate}
In other words, the symplectic normal bundle to $M_H\setminus Z$ extends across the fold $Z$ to give us a symplectic normal bundle to $M_H$ and, at points of the intersection $Z_H=M_H\cap Z$, it is tangent to the fold.
\end{prop}
\noindent Second, we have the penultimate structure theorem for toric, folded-symplectic manifolds.

\begin{theorem}\label{thm:intro}
Let $(M,\sigma,\mu:M\to \fg^*)$ be a toric, folded-symplectic manifold with co-orientable folding hypersurface.  Then,
\begin{enumerate}
\item The orbit type strata $M_H$ are transverse to the folding hypersurface and each $(M_H,i_{M_H}^*\sigma, \mu\vert_{M_H}:M_H \to \frak{h}^o)$ is a toric, folded-symplectic manifold, hence $M$ is stratified by toric, folded-symplectic manifolds.
\item The orbit space $M/G$ is a manifold with corners and the boundary strata of $M/G$ are given by the images of the orbit-type strata $M_H/G$.
\item The moment map descends to $\psi:M/G\to \fg^*$, a unimodular map with folds.  Furthermore, since each $(M_H,i_{M_H}^*\sigma)$ is a toric, folded-symplectic manifold, the restriction of $\psi$ to $M_H/G$ is a map with fold singularities if we view it as a map into $\frak{h}^o$.  Hence $\psi:M/G \to \fg^*$ is a unimodular map with folds that restricts to maps with fold singularities on the boundary strata.
\item The null-foliation on $Z$ may be recovered from $\psi$, along with its orientation induced by $\sigma$ using the intrinsic derivative of $\psi$ and the map $(\cdot)_Z: \pi^*\operatorname{Im}(d\psi)^o \to \ker(\sigma)\cap TZ$ (q.v. lemma \ref{lem:nullfoliation}).
\item The remainder of the bundle $\ker(\sigma)$ can be constructed by choosing lifts of elements of $\ker(d\psi)$.
\item The representation of $H$ on the fibers of $(\widetilde{TM_H})^{\sigma}$ at $Z$ may be read from the orbital moment map.
\item The local structure of the folding hypersurface is determined by the image of $\psi(Z/G)$ (q.v. corollary \ref{cor:foldnorm}).
\end{enumerate}
Thus, the fold, the null foliation, the orientation, the kernel bundle, and the symplectic slice representation may all be recovered from the orbital moment map.  And, one may recover all symplectic invariants away from the fold just by reading the weights of the symplectic slice representation from the orbital moment map.
\end{theorem}
\noindent In the end of chapter $5$, we introduce the two categories $\mathcal{M}_{\psi}(W)$ and $\mathcal{B}_{\psi}(W)$ of toric folded-symplectic manifolds and bundles, respectively.  We have already seen $\mathcal{M}_{\psi}(W)$ in this introduction.  The category of toric, folded-symplectic bundles consists of principal $G$ bundles $\pi:P \to W$ equipped with a folded-symplectic structure so that the action of $G$ is Hamiltonian with moment map $\psi\circ \pi$.  We then argue that all objects of $\mathcal{M}_{\psi}(W)$ are locally isomorphic (q.v. lemma \ref{lem:locunique}).

In chapter $6$, we develop our last tool before we begin the process of classifying toric, folded-symplectic manifolds: folded-symplectic reduction.  We precisely describe the conditions under which the reduced space is symplectic and the conditions under which the reduced space is folded-symplectic with nonempty fold.  In particular, we show that one need only understand how the zero level set of the moment map intersects the folding hypersurface in order to accurately predict when the reduced space is symplectic or not.  We conclude the chapter by adapting the minimal coupling procedure of Sternberg \cite{St} to symplectic vector bundles over folded-symplectic manifolds and describing how this relates to the orbit-type strata of toric, folded-symplectic manifolds.  In particular, the structure theorem (q.v theorem \ref{thm:intro}) implies that the orbit type strata are themselves folded-symplectic with well-defined symplectic normal bundles, hence we may apply the minimal coupling procedure to equip a neighborhood of the zero section of the normal bundle with a folded-symplectic structure.

In chapters $7$ and $8$ we classify toric, folded-symplectic manifolds over a fixed unimodular map with folds up to isomorphism.  We begin by classifying objects in the category $\mathcal{B}_{\psi}(W)$ of toric, folded-symplectic bundles (with corners) over $\psi:W \to \fg^*$ (q.v. definition \ref{def:bpsi}), following the approach of \cite{KL}.  We prove:

\pagebreak

\begin{theorem}
Let $\psi:W \to \fg^*$ be a unimodular map with folds, where $\fg$ is the Lie algebra of a torus $G$.  Let $\mathcal{B}_{\psi}(W)$ be the category of toric, folded-symplectic bundles over $\psi$.  Then there is a bijection $b:\pi_0(\mathcal{B}_{\psi}) \to H^2(W; \mathbb{Z}_G \times \R)$.  That is, isomorphism classes of toric, folded-symplectic bundles over $\psi$ are parameterized by cohomology classes in $H^2(W,\mathbb{Z}_G\times \R)$.
\end{theorem}

\noindent We then develop a functor $c:\mathcal{B}{\psi}(W) \to \mathcal{M}_{\psi}(W)$ by cutting away the corners to introduce stabilizers.  This comprises the first three sections of chapter $8$.  To finish the classification, we show that $c$ is an equivalence of categories, hence it induces a bijection on isomorphism classes of objects.

As a final note to the reader, we will often use transversality for manifolds with corners (q.v. definition \ref{def:A1:trans}), which we denote as $\pitchfork_s$ for \emph{strong transversality}.  As we have already stated, the traditional notion of transversality does not suffice in the category of manifolds with corners.  If $f:W \to N$ is a map between manifolds with corners, $S\subseteq N$ a submanifold, and $f\pitchfork S$ in the sense of manifolds, then $f^{-1}(S)$ may simply be a topological space or a differential space (q.v. \cite{Mi2}).  To give it the structure of a submanifold with corners, we need to require that the restriction of $f$ to each boundary stratum of $W$ is transverse to $S$ (q.v. \cite{CD}).  This condition is clearly quite limiting, which is why we choose to call this type of transversality \emph{strong transversality}.

\pagebreak
\section{First Order Jet Bundles over Manifolds with Corners and Fold Maps}
The purpose of this section is two-fold.  First, we will need jets on manifolds with corners, so we would like to introduce the basic theory of the first-order jet bundle $J^1(E)$ of a fiber bundle $\pi:E \to M$ over a manifold with corners.  This task is somewhat cumbersome if we use the usual approach to jet bundles, which involves plastering the total space $J^1(E)$ with derivative coordinate charts as in \cite{GG, Mi2, Sa} and then either realizing $J^1(E)$ as a restriction $J^1(\tilde{E})\big\vert_E$, where $\tilde{E}$ is a manifold without corners containing $E$ or defining it locally as such a restriction and gluing the pieces together (as in \cite{Mi2}).  We present an alternate approach which does not rely on excessive coordinate computations and avoids the problem of embedding $E$ into a manifold $\tilde{E}$ without corners.  The approach relies on the understanding of $J^1(M\times F)$, the first jet bundle of the trivial fiber bundle $M\times F \to F$.  In short, if one knows what $J^1(M\times F)$ should be as a manifold with corners, then $J^1(E)$ is many copies of $J^1(U\times F)$, $U\subseteq M$ open, glued together which gives $J^1(E)$ the structure of a smooth manifold with corners.  This is proposition \ref{prop:jettopology}.  The advantage of our approach is that it is very clean while the obvious disadvantage is that it does not generalize to higher order jet bundles.  We prove that every first order jet bundle $J^1(E)$ is isomorphic to $\hom(\pi^*TM,V)$, where $V\to E$ is the vertical bundle, and such isomorphisms may be specified by a choice of connection $\chi$ on $E$.  Thus, a first order jet bundle has a very recognizable form as a $\hom$ bundle, albeit perhaps non-canonically.  This is the content of proposition \ref{prop:C2}.  We then use our understanding of the first-order jet bundle to generalize Morse functions to fiber bundles with fiber-dimension $1$, which represents original work.

The main purpose of this section is to introduce the notion of an equidimensional map with fold singularities on manifolds with corners, or simply a map with fold singularities between manifolds (with corners) of the same dimension.  Maps with fold singularities have found their place in much of mathematical literature, but there does not appear to be a definition of a fold singularity for maps between manifolds with corners.  The definition of a map with fold singularities using the intrinsic derivative, given in the appendix of \cite{Ho}, would seem to be adequate for generalizing fold maps to manifolds with corners, but then pathological examples arise where fold maps $f:M \to N$ become homeomorphisms or the folding hypersurface has connected components of varying dimensions.  Indeed, consider the following set of examples.

\begin{example}
Let $f:\R^2 \to \R^2$ be the map given by $f(x,y)=(x,y^2)$, which has fold singularities along the $x$-axis.  We consider three scenarios:

\begin{itemize}
\item Let $W=\{(x,y)\vert \mbox{ } y\ge 0\}$.  Then, using the definition of a fold singularity found in \cite{Ho}, $f\big\vert_W:W \to \R^2$ has fold singularities along $y=0$, but no folding is accomplished by $f\big\vert_W$.  In fact, it is a homeomorphism of $W$ onto itself.  This example may not raise too many objections, so consider the next example.

\item Let $W=\{(x,y) \vert \mbox{}  y\ge \vert x \vert\}$.  Then $f\big\vert_W: W \to \R^2$ has a fold singularity at $(0,0)$ (in the sense of \cite{Ho}).  If we are to mimic the behaviour of fold maps in the category of manifolds, the locus of degeneracies should be a codimension $1$ submanifold with corners, but it is a codimension $2$ submanifold with corners here.  Again, this may not be unreasonable, so let us consider another example.

\item Let $W\subset \R^2$ be any manifold with corners that lies inside the upper half plane ($y\ge 0$) whose intersection with the $x$-axis has a component of dimension $0$ and a component of dimension $1$.  Then $f\big\vert_W:W \to \R^2$ has fold singularities and its locus of degeneracies is a disjoint union of manifolds of varying dimensions.  This is highly undesirable.
\end{itemize}
\end{example}

\noindent As we have stated in the introduction, the primary issue is transversality: if $f:M\to N$ is a map of manifolds with corners and $S$ is a submanifold with corners, then $f\pitchfork S$ in the traditional sense of manifolds is not enough to guarantee that $f^{-1}(S)$ is a submanifold with corners.  One needs to require that $f$ restricted to each stratum of $M$ is transverse to $S$.

Our approach to defining fold maps is to use the theory of first order jet bundles on manifolds with corners that we develop in section 2.1 and adapt Guillemin and Golubitsky's definition of a submersion with fold singularities (definition 4.1 in \cite{GG}) to our needs.  In particular, we generalize the definition of a submersion with folds to arbitrary fiber bundles, we present a definition of an equidimensional map with fold singularities between manifolds with corners, we develop an extremely useful computational tool for understanding equidimensional fold maps, and we provide a normal form for fold maps, all of which is original work since it hasn't been done in the case of manifolds with corners.  The most important results that the reader ought to carry on to the subsequent chapters are the computational corollary \label{cor:folds4-0}:

\begin{cor}
Let $f:M \to N$ be a smooth map between two $m$-dimensional manifolds with corners.  Then $f$ is a map with fold singularities if and only if
\begin{enumerate}
\item The induced map $(df)^m: M \to \hom(\Lambda^m(TM), \Lambda^m(TN))$ is transverse (in the sense of manifolds with corners) to the zero section $\mathcal{O}$ of $\hom(\Lambda^m(TM),\Lambda^m(TN)) \to M \times N$ (i.e. locally the determinant of $df$ vanishes transversally), and
\item $\ker(df) \pitchfork ((df)^m)^{-1}(\mathcal{O})$.
\end{enumerate}
\end{cor}
\noindent which we will use without reserve to show that certain maps are fold maps, and the following factorization corollary to proposition \ref{prop:folds5},
\begin{cor}\label{cor:folds5}
Let $f:M^n \to N^n$ be an equidimensional map with fold singularities, suppose $f$ is strata-preserving, and suppose for all $p\in Z$, the folding hypersurface, $\ker(df_p)$ is tangent to the stratum of $M$ containing $p$ (i.e. $\ker(df) \to Z$ is stratified).  Then, for each $p\in Z$ there exist a neighborhood $U\subset M$, a strata-preserving fold map $\psi:U\to U$, and a strata-preserving open embedding $\phi:U \to N$ so that $f\vert_U=\phi\circ \psi$. Hence, locally, every map $f$ satisfying the conditions of the proposition factors as a diffeomorphism composed with a strata-preserving fold map.
\end{cor}
\noindent which we will use when we construct a cutting procedure for toric, folded-symplectic bundles (q.v. chapters $7$ and $8$).

\subsection{First Order Jet Bundles}
Our constructions are motivated by remark 2 on page 41 of \cite{GG} and part 7 of theorem 21.5 in \cite{Mi1}, which state that $J^1(M\times N)$ is canonically isomorphic to a $\hom$ bundle, $\hom(TM,TN)$, which we define below (q.v. \ref{def:hombdle}).  In particular, since we know what the first jet bundle of a product \emph{should} be, we essentially define the first jet bundle of a product to be $\hom(TM,TN)$ and argue that the first jet bundle of a general fiber bundle is a collection of these bundles glued together.  While the presentation of the first jet bundle found here is somewhat nonstandard, most of these results are \emph{not} new.  We will alert the reader to known results along the way.
\subsubsection{First Order Jets}
Let $\pi:E \to M$ be a smooth fiber bundle with typical fiber $F$ over a smooth $m$-dimensional manifold with corners, $M$, where $F$ may have corners.  We wish to define the first jet bundle $J^1(E)$ and give it a smooth atlas.

\begin{definition}\label{def:locsec}
A \emph{local section} of $E$ near $p\in M$ is a neighborhood $U$ of $p$ and a section $\phi:U \to E\vert_U \subset E$.  A local section near $p$ will be denoted as a pair $(U, \phi)$.  The space of local sections near $p$ will be denoted $S_p=\{(U,\phi)\vert \mbox{ } p\in U\}$.
\end{definition}

\begin{definition}\label{def:1equiv}
Two local sections $(U,\phi)$, $(V,\psi)$ are $1$-equivalent at $p\in M$ if:
\begin{enumerate}
\item $p\in U\cap V$,
\item $\phi(p)=\psi(p)$, and
\item $d\phi_p = d\psi_p$.
\end{enumerate}
\end{definition}
\noindent In coordinates, $(U,\phi)$ and $(V,\psi)$ are $1$-equivalent at $p$ if they agree at $p$ and their first partials are equal at $p$.  The following lemma is an immediate consequence of the definition.

\begin{lemma}\label{lem:equivrelation}
$1$-equivalence at $p\in M$ defines an equivalence relation $\sim$ on the set of local sections near $p$, $S_p=\{(U,\phi)\vert \mbox { } p\in U\}$.
\end{lemma}

\begin{definition}\label{def:1jetfiber}
We define the \emph{space of $1$-jets at $p\in M$} to be $(J^1_pE) = S_p/\sim$.  If $(U,\phi)$ is a local section near $p$ then we denote its equivalence class in $S_p/\sim$ by $j^1_p\phi$.
\end{definition}

\begin{definition}\label{def:firstjet}
We define \emph{the first order jet space} $J^1(E)$ (as a set) to be the set $J^1(E) = \sqcup_{p\in M}J^1_pE$.  It is equipped with a projection map to $M$, $\pi_1:J^1(E) \to M$, given by $\pi_1(j^1_p\phi)=p$ and a projection map to $E$, $\pi_{1,0}:J^1(E) \to E$, given by $\pi_{1,0}(j^1_p\phi)=\phi(p)$.  If $U\subset M$ is an open set, we define $J^1(E)\vert_U = \sqcup_{p\in U}J^1_pE = \pi_1^{-1}(U)$.
\end{definition}

Note that $\pi_{1,0}$ is well-defined since all elements of the equivalence class $j^1\phi$ have the same value $\phi(p)$ at $p$.  We can assign a topology and smooth structure to $J^1(E)$ as follows.  First, we need the following definition:

\begin{definition}\label{def:hombdle}
Let $M,F$ be two manifolds with corners.  Let $pr_1:M\times F \to M$ and $pr_2:M\times F \to F$ be the projections.  We define:
\begin{displaymath}
\hom(TM,TF):=\hom(pr_1^*TM, pr_2^*TF)
\end{displaymath}
which is a bundle over $M\times F$ with fiber $\hom(T_mM,T_fF)$.
\end{definition}

Now, let $U\subset M$ be an open set so that $E\vert_U$ is trivializable.  That is, there exists an isomorphism $\Phi_U$ of fiber bundles:

\begin{equation}\label{eq:triv}
\xymatrixcolsep{5pc}\xymatrix{
E\vert_U \ar[r]^{\Phi_U} \ar[dr]^{\pi} & U\times F \ar[d]^{pr_1}\\
                         & U
}
\end{equation}

The map $\Phi_U$ of equation \ref{eq:triv} defines a map of sets

\begin{displaymath}
\begin{array}{l l}
\tilde{\Phi}_U & : J^1(E)\vert_U \to \hom(TU,TF) \\
                & \mbox{\hspace{6mm}} j^1_p\phi \mbox{\hspace{2mm}}\mapsto (p,\Phi_U(\phi(p)),dpr_2(d\Phi_U(d\phi_p))) \\
\end{array}
\end{displaymath}
giving us a commutative diagram:

\begin{displaymath}
\xymatrixcolsep{5pc}\xymatrix{
J^1(E)\vert_U \ar[r]^{\tilde{\Phi}_U} \ar[d]^{\pi_{1,0}} & \hom(TU,TF) \ar[d]^{pr_U\times pr_F} \\
E\vert_U \ar[r]^{\Phi_U} \ar[dr]^{\pi}                   & U\times F \ar[d]^{pr_1} \\
                                                         & U                       \\
}
\end{displaymath}

\begin{lemma}\label{lem:jettopology}
The map $\tilde{\Phi}_U$ is a well-defined bijection with an inverse given by:
\begin{displaymath}
\tilde{\Phi}_U^{-1}(p,f,A) = j^1_p\phi
\end{displaymath}
where $j^1_p\phi$ is the equivalence class of local sections $(U,\phi)$ at $p$ whose value at $p$ is $\Phi_U^{-1}(p,f)$ and differential is $d\phi_p= (d\Phi_U)_{(p,f)}^{-1}(id_{T_pU}\oplus A)$.
\end{lemma}

\begin{remark}
Before we begin, let us explain the notation $id_{T_pU \oplus A}$.  That is, let us review how we may recover the differential of a section $\phi$ of $U\times F$ given its differential $dpr_2(d\phi):TU \to TF$ in the vertical direction.  If we are given a section $\phi:U \to U\times F$ of the trivial fiber bundle $U\times F$, then it may be written as $\phi(p)=(p,g(p))$ for some smooth map $g:U\to F$.  Then $d\phi_p = id_{T_pU} \oplus dg_p$, using the canonical splitting $T(U\times F) = pr_1^*TU \oplus pr_2^*TF$, where $pr_i$ is projection onto the $i^{th}$ factor of $U\times F$.  Note that $dg_p\in \hom(T_pU, T_{g(p)}F)$. Thus, the differential of any section of $U\times F$ has the form $id_{T_pU} \oplus A$ for some $A\in \hom(T_pU, T_fF)$.

\vspace{2mm}

Conversely, if we are given an element $A\in \hom(T_pU,T_fF)$, then there is a map $g:V \to F$ defined on a neighborhood $V$ of $p$ satisfying $g(p)=f$ and $dg_p= A$.  Then $id_{T_pU} \oplus dg_p: T_pU \to T_pU \oplus T_fF$ is the differential of the local section $\phi(u)=(u,g(u))$.  Hence, we may append the identity map $id_{T_pU}$ to any element $A\in \hom(T_pU,T_fF)$ to transform it into the differential of a section.
\end{remark}

\begin{proof}[Proof of lemma \ref{lem:jettopology}] \mbox{ } \newline
\begin{enumerate}
\item $\tilde{\Phi}_U$ is well-defined since $(U_1,\phi_1)\sim (U_2,\phi_2)$ at $p$ if and only if $\phi_1(p)=\phi_2(p)$ and $(d\phi_1)_p = (d\phi_2)_p$, hence $dpr_2(d\Phi_U((d\phi_1)_p))=dpr_2(d\Phi_U((d\phi_2)_p))$.
\item We now show that the so-called inverse map, $\tilde{\Phi}_U^{-1}$ is a well-defined map whose image is in $J^1(E)$.  Namely, we must show that if we are given $(p,f,A)$, where $A\in \hom(T_pU, T_fF)$, then there exists a local section $(V,\phi)$ near $p$ whose value is $\Phi_U^{-1}(p,f)$ and whose differential is $d\phi_p=(d\Phi_U)_{(p,f)}^{-1}(id_{TU}\oplus A)$.
    \vspace{5mm}
    There exist neighborhoods $V\subset U$ of $p$ and $W\subset F$ of $f$ and a map $\psi:V\to W$ so that:
    \begin{enumerate}
    \item $\psi(p)=f$ and
    \item $d\psi_p = A$.
    \end{enumerate}
    Then $\psi$ defines a local section of $U\times F$ near $p\in U$: $\phi(x) = (x, \psi(x))$ and $(V, \Phi_U^{-1}\circ \phi)$ is a local section of $E$ near $p$.  The value of $\Phi_U^{-1}\circ \phi$ at $p$ is $\Phi_U^{-1}(p,f)$.  The differential of $\Phi_U^{-1}\circ \psi$ at $p$ is:
    \begin{displaymath}
    d(\Phi_U^{-1}\circ \phi)_p = d(\Phi_U)_{(p,f)}^{-1}(id_{T_pU}\oplus A)
    \end{displaymath}
    hence our definition of $\tilde{\Phi}\vert_U^{-1}(p,f,A)$ gives us a well-defined element of $J^1_pE$.

\item To see that the two maps are inverses, we compute:
\begin{displaymath}
\tilde{\Phi}_U \circ \tilde{\Phi}_U^{-1}(p,f,A) = \tilde{\Phi}_U(j^1_p\phi) = (\Phi_U(\phi(p)),dpr_2(d\Phi_U(d\phi_p)))=(p,f,dpr_2(id_{T_pU}\oplus A)) = (p,f,A)
\end{displaymath}
where we have used that $\tilde{\Phi}_U^{-1}(p,f,A)=j^1_p\phi$, where $d\phi_p = d\Phi_U^{-1}(id_{T_pU}\oplus A)$.  We also have,
\begin{displaymath}
\tilde{\Phi}\vert_U^{-1} \circ \tilde{\Phi}_U(j^1_p\phi) = \tilde{\Phi}\vert_U^{-1}(\Phi_U(\phi(p)),dpr_2(d\Phi_U(d\phi_p)))=j^1\eta
\end{displaymath}
where $j^1\eta$ is the equivalence class of local sections $(U',\eta)$ near $p$ whose value at $p$ is $\Phi_U^{-1}(\Phi_U(\phi(p)))=\phi(p)$ and whose differential at $p$ is:

\begin{displaymath}
(d\Phi\vert_U)_{(p,f)}^{-1}(id_{T_pU}\oplus dpr_2(d\Phi_U(d\phi_p))) = d\phi_p
\end{displaymath}
Therefore, $j^1\eta = j^1\phi$.
\end{enumerate}
\end{proof}
\noindent Since $\hom(TU,TF)$ is a topological space and $\tilde{\Phi}_U$ is a bijection, we can pull back the topology on $\hom(TU,TF)$ to $J^1(E)\vert_U$ so that $\tilde{\Phi}_U$ is a homeomorphism.  This gives us a global topology on $J^1(E)$ making the maps $\pi_{1,0}$ and $\pi_1$ continuous.

\begin{definition}\label{def:jettopology}
Let $\pi:E\to M$ be a fiber bundle and $J^1(E)$ the first order jet space.  The topology $\mathcal{T}$ on $J^1(E)$ is the finest topology so that for each open set $U\subset M$ where $E$ is trivializable, the induced map $\tilde{\Phi}_U$ of lemma \ref{lem:jettopology} is a homeomorphism.  We refer to $J^1(E)$ with the topology $\mathcal{T}$ as the \emph{first order topological jet space}.
\end{definition}

\begin{definition}\label{def:chart}
Let $\pi:E \to M$ be a fiber bundle over a manifold with corners $M$ with typical fiber $F$.  Let $U\subset M$ be an open subset so that $E\vert_U$ is trivializable and let $\tilde{\Phi}_U:J^1(E)\vert_U \to \hom(TU,TF)$ be the induced map of lemma \ref{lem:jettopology}.  We will say that the pair $(U,\tilde{\Phi}_U)$ is a \emph{chart} on $J^1(E)$ even though $U\subset M$.  This is because the domain of $\tilde{\Phi}_U$ is $\pi_1^{-1}(U)$, hence specifying $U$ uniquely specifies the domain of $\tilde{\Phi}_U$.  Therefore, to save on notation, we will use $(U,\tilde{\Phi}_U)$ instead of $(\pi_1^{-1}(U), \tilde{\Phi}_U)$ when referring to charts on $J^1(E)$.  We say a collection of charts $\{(U,\tilde{\Phi}_U)\}$ is a $C^{\infty}$-\emph{atlas} if the transition maps $\tilde{\Phi}_V \circ \tilde{\Phi}_U^{-1}$ are diffeomorphisms of manifolds with corners.
\end{definition}

\noindent Now, we have a collection of charts $\{(U,\tilde{\Phi}_U)\vert \mbox{ } E\vert_U \simeq U \times F\}$ on $J^1(E)$.  Our next task is to show that this collection forms a smooth atlas.

\begin{lemma}\label{lem:jettopology1}
Let $\pi:E \to M$ be a fiber bundle and let $J^1(E)$ be the first order topological jet space.  Let $(U,\tilde{\Phi}_U)$, $(V,\tilde{\Phi}_V)$ be two charts on $J^1(E)$ so that $U\cap V \ne \emptyset$.  Then $\tilde{\Phi}_V\circ \tilde{\Phi}_U^{-1}: \hom(T(U\cap V), TF) \to \hom(T(U\cap V), TF)$ is a diffeomorphism of manifolds with corners.
\end{lemma}

\begin{proof} \mbox{ } \newline
We compute:
\begin{displaymath}
\tilde{\Phi}_V (\tilde{\Phi}_U(p,f,A) = (\Phi_V (\Phi_U^{-1}(p,f)), dpr_2(d\Phi_V\circ d\Phi_U^{-1}(id_{TU}\oplus A))) = ((\Phi_V\circ \Phi_U^{-1})(p,f),d(pr_2 \circ \Phi_V \circ \Phi_U^{-1})(id_{TU}\oplus A))
\end{displaymath}
Since $\Phi_V\circ \Phi_U^{-1}$ is a diffeomorphism of manifolds with corners, the result follows.
\end{proof}

\begin{remark}
Recall that if we have a fiber bundle $\pi:E \to M$, then we can define the vertical subbundle $V$ of $TE$ to be the bundle whose fiber at $e$ is $V_e=\ker(d\pi_e)$.
\end{remark}

\begin{prop}\label{prop:jettopology}
Let $\pi:E \to M$ be a smooth fiber bundle and let $J^1(E)$ be the first order topological jet space with charts $\mathcal{A}=\{(U,\tilde{\Phi}_U) \vert \mbox{ } E\vert_U \simeq U \times F\}$.  Then $\mathcal{A}$ is a $C^{\infty}$ atlas and $J^1(E)$ is a smooth $C^{\infty}$ manifold with corners.  Furthermore,
\begin{enumerate}
\item The projection map $\pi_{1,0}:J^1(E) \to E$ gives $J^1(E)$ the structure of a smooth fiber bundle with fiber at $e$ isomorphic to $\hom(T_{\pi(e)}U, V_e)$, where $V_e$ is the fiber of the vertical bundle $V \to E$ of $E$.
\item The projection map $\pi_1:J^1(E) \to M$ gives $J^1(E)$ the structure of a smooth fiber bundle with fiber at $p$ isomorphic to $\hom(T_pU, TF)$.
\item The transition maps $\tilde{\Phi}_V \circ \tilde{\Phi}_U^{-1}: \hom(T(U\cap V), TF) \to \hom(T(U\cap V), TF)$ are \emph{affine} maps of vector bundles, hence $J^1(E)$ is generally not a vector bundle over $E$ or $M$.
\end{enumerate}
\end{prop}

\begin{remark}
Proposition 4.1.7, lemma 4.1.9, and corollary 4.1.10 of \cite{Sa} imply the first two results.
\end{remark}

\begin{remark}
The purpose of explicitly stating part $3$ of proposition \ref{prop:jettopology} is to make it clear that $\pi_1$ and $\pi_{1,0}$ are not necessarily vector bundles, despite their role as a generalization of the tangent and cotangent bundles.  However, they are \emph{close} to being vector bundles in the sense that the transition maps are affine maps when restricted to the fibers.  This result is encapsulated in proposition 4.6.3 in \cite{Sa}, which states that $J^1(E)$ is an affine bundle modeled on $\pi^*TM\otimes V$ with $V\to E$ being the vertical bundle.
\end{remark}

\begin{proof}[Proof of proposition \ref{prop:jettopology}] \mbox{ }

\begin{enumerate}
\item The charts $(U, \tilde{\Phi}_U)$ identify $J^1(E)\vert_U$ with $\hom(TU,TF)$ and $\pi_{1,0}$ with \newline $pr_U\times pr_F:\hom(TU,TF)\to U\times F$ and the transition maps are smooth maps of \emph{fiber} bundles.  That is, for any two charts $(U,\tilde{\Phi}_U)$, $(V,\tilde{\Phi}_V)$ the diagram:

    \begin{displaymath}
    \xymatrixcolsep{5pc}\xymatrix{
    \hom(T(U\cap V),TF) \ar[dr]^{pr_U\times pr_F} \ar[r]^{\tilde{\Phi}_V \circ \tilde{\Phi}_U^{-1}} & \hom(T(U\cap V),TF) \ar[d]^{pr_U \times pr_F} \\
                                                                                            & U \times F
    }
    \end{displaymath}
    commutes.  Therefore, $\pi_{1,0} : J^1(E) \to E$ is a smooth fiber bundle.  We can identify the typical fiber by identifying the fiber of $pr_U\times pr_F: \hom(TU,TF) \to U\times F$, which is $\hom(T_uU, T_fF) \simeq \hom(T_uU,V_e)$, where $\pi(e)=u$.

\item The charts $(U,\tilde{\Phi}_U)$ identify $\pi_1$ with the projection $pr_U: \hom(TU,TF) \to U$, which is a fiber bundle with typical fiber $\hom(T_uU, TF)$, hence $\pi_1\vert_U:J^1(E)\vert_U \to U$ has the structure of a fiber bundle with fiber $\hom(T_uU,TF)$.  The transition maps $\tilde{\Phi}_V \circ \tilde{\Phi}_U^{-1}$ make the diagram:

    \begin{displaymath}
    \xymatrixcolsep{5pc}\xymatrix{
    \hom(T(U\cap V), TF) \ar[dr]^{pr_U} \ar[r]^{\tilde{\Phi}_V \circ \tilde{\Phi}_U^{-1}} & \hom(T(U\cap V),TF) \ar[d]^{pr_U} \\
                                                                                  & U\cap V
    }
    \end{displaymath}
    commute, hence they are isomorphisms of fiber bundles.  Thus, $\pi_1:J^1(E) \to M$ is a smooth fiber bundle over $M$.

\item We check that $\tilde{\Phi}_V \circ \tilde{\Phi}_U^{-1}: \hom(T_u(U\cap V), T_fF) \to \hom(T_u(U\cap V), T_{f'}F)$ is affine (linear up to the addition of a constant term), where $f'=pr_2(\Phi_V \circ \Phi_U^{-1})(u,f)$.  Let $A\in \hom(T_u(U\cap V),T_fF)$.  Then $\tilde{\Phi}_V \circ \tilde{\Phi}_U^{-1}(A)$ is the composition:

    \begin{displaymath}
    A \to id_{T_uU} \oplus A \to dpr_2(d(\Phi_V \circ \Phi_U^{-1})(id_{T_uU} \oplus A))
    \end{displaymath}
    where the last arrow is a linear map and the first arrow is an affine map.  Therefore, the composition is an affine map.
\end{enumerate}
\end{proof}

\noindent Maps between jet bundles are simply maps of fiber bundles.  However, there is a distinguished class of maps that are induced by maps of fiber bundles.  They are studied in section 4.2 of \cite{Sa}.

\begin{definition}\label{def:jetmap}
Let $\pi:E \to M$, $\pi':E' \to M$ be two fiber bundles over a manifold with corners $M$.  Let $\gamma:E \to E'$ be a map of fiber bundles.  Then the induced map $\tilde{\gamma}:J^1(E) \to J^1(E')$ is defined as:

\begin{displaymath}
\tilde{\gamma}(j^1_p\phi) = j^1_p(\gamma(\phi))
\end{displaymath}
where $j^1_p\phi$ is the $1$-equivalence class of sections at $p$ with representative $\phi$.
\end{definition}

\begin{lemma}\label{lem:jetmap}
Let $\pi:E \to M$, $\pi':E'\to M$ be two fiber bundles over a manifold with corners $M$.  Let $\gamma:E\to E'$ be a map of fiber bundles and let $\tilde{\gamma}:J^1(E)\to J^1(E')$ be the induced map.  Then,
\begin{enumerate}
\item $\tilde{\gamma}$ is a smooth map of fiber bundles over $M$.
\item There is a commutative diagram:

\begin{displaymath}
\xymatrix{
J^1(E) \ar[r]^{\tilde{\gamma}} \ar[d]^{\pi_{1,0}} & J^1(E') \ar[d]^{\pi'_{1,0}} \\
E \ar[r]^{\gamma}                                 & E '
}
\end{displaymath}

\item $\tilde{\gamma}$ is a diffeomorphism if and only if $\gamma$ is a diffeomorphism.
\end{enumerate}
\end{lemma}

\begin{proof} \mbox{} \newline
\begin{enumerate}
\item We use a chart $(U,\tilde{\Phi}_U)$ on $J^1(E)\vert_U\simeq \hom(TU,TF)$ and a chart $(U,\tilde{\Phi}'_U)$ on $J^1(E')\vert_U \simeq \hom(TU,TF') $ to see that the map $\tilde{\gamma}$ in coordinates is:

    \begin{displaymath}
    \tilde{\Phi}'_U \circ \tilde{\gamma} \circ \tilde{\Phi}_U^{-1} (p,f,A) = (p,pr_2(\Phi'_U(\gamma(\Phi_U^{-1}(p,f)))), dpr_2(d(\Phi_U' \circ \gamma \circ \Phi_U^{-1})(id_{T_pU} \oplus A)))
    \end{displaymath}
    which is smooth since $\Phi_U'$, $\Phi_U^{-1}$, $\gamma$, and $pr_2:U\times F' \to F'$ are smooth.  It is a map of fiber bundles because:

    \begin{displaymath}
    \pi'(\tilde{\gamma}(j^1_p\phi))= \pi'(j^1_p(\gamma(\phi))) = p = \pi(j^1_p\phi)
    \end{displaymath}
    hence $\pi'\circ \tilde{\gamma} = \pi$.

\item We need to verify that $\pi'_{1,0}\circ \tilde{\gamma} = \gamma\circ \pi_{1,0}$.  We have that for any $s\in J^1(E)$, $s=j^1_p\phi$ and:

\begin{displaymath}
\pi'_{1,0}(\tilde{\gamma}(j^1_p\phi)) = \pi'_{1,0}(j^1_p(\gamma(\phi))) = \gamma(\phi(p)) = \gamma(\pi_{1,0}(j^1_p\phi))
\end{displaymath}
which proves that $\pi'_{1,0}(\tilde{\gamma}(s))=\gamma(\pi_{1,0}(s)).$

\item If $\gamma$ is a diffeomorphism, then it has a smooth inverse $\gamma^{-1}$, which is also a map of fiber bundles.  Hence $\gamma^{-1}$ induces a map of jet bundles $\tilde{\gamma^{-1}}:J^1(E') \to J^(E)$ and it is straightforward to see that $\tilde{\gamma}^{-1} = \tilde{\gamma^{-1}}$.  On the other hand, if $\tilde{\gamma}$ is a diffeomorphism then we can construct the inverse, $\gamma^{-1}$, of $\gamma$ as follows:

    \begin{enumerate}
    \item For each $e' \in E'$, let $(U_{e'},\phi_{e'})$ be a local section near $p=\pi'(e')$ so that $\phi_{e'}(u)=e'$ for all $u\in U_{e'}$.  That is, we choose a local section that is constant.  Similarly, for $e\in E$ we can define a local section $\phi_e$ near $\pi(e)$ satisfying $\phi(u)=e$ for all $u$ in a neighborhood of $p=\pi(e)$.  Then we may define:

        \begin{displaymath}
        \gamma^{-1}(e') = \pi_{1,0}(\tilde{\gamma}^{-1}(j^1_p\phi_{e'})
        \end{displaymath}
        and, using the definition of $\phi_e$, we have
        \begin{displaymath}
        \begin{array}{l c l}
        \gamma^{-1}(\gamma(e)) &= & \pi_{1,0}(\tilde{\gamma}^{-1}(j^1_p\phi_{\gamma(e)})) \\
                               &= & \pi_{1,0}(\tilde{\gamma}^{-1}(\tilde{\gamma}(j^1_p\phi_e))) \\
                               & = & \pi_{1,0}(j^1_p\phi_e) \\
                               & = & e
        \end{array}
        \end{displaymath}
        Where the second line follows since $\phi_{\gamma(e)}$ and $\gamma(\phi_e)$ are both local, constant sections near $p$ with value $\gamma(e)$ and $\tilde{\gamma}$ is one-to-one, hence $j^1_p\phi$ is the unique $1$-jet mapping to $j^1_p\phi_{\gamma(e)}$.  From part $2$, we also have that $\gamma \circ \pi_{1,0} \circ \tilde{\gamma}^{-1} = \pi_{1,0}'$, hence:

        \begin{displaymath}
        \begin{array}{lcl}
        \gamma(\gamma^{-1}(e')) & = & \gamma(\pi_{1,0}(\tilde{\gamma}^{-1}(j^1_p\phi_{e'}))) \\
                                & = & \pi_{1,0}'(j^1_p\phi_{e'}) \\
                                & = & e'
        \end{array}
        \end{displaymath}
        Therefore, $\gamma^{-1}$ is a set-theoretic inverse of $\gamma$.
    \item To see that $\gamma^{-1}$ is smooth, use charts $(U,\tilde{\Phi}_U)$, $(U,\tilde{\Phi}_U')$ on $J^1(E)$ and $J^1(E')$ respectively.  We have the diagram:

        \begin{displaymath}
        \xymatrixcolsep{5pc}\xymatrix{
        \hom(TU,TF) \ar[d]_{pr_U\times pr_F} \ar[r]^{\tilde{\Phi}_U'\circ \tilde{\gamma}\circ\tilde{\Phi}_U^{-1}} & \hom(TU,TF')\ar[d]^{pr_U\times pr_F} \\
                                   U \times F  \ar[r]^{\Phi_U' \circ \gamma \circ \Phi_U^{-1}}            &     U \times F' \ar@/^/[u]^{\mathcal{O}}
        }
        \end{displaymath}
        where the map $\mathcal{O}$ is the embedding of $U\times F'$ as the zero section.  Then $\gamma^{-1}$ in these coordinates is the composition:
        \begin{displaymath}
        \Phi_U \circ \gamma^{-1} \circ (\Phi_U')^{-1} = (pr_U \times pr_F) \circ (\tilde{\Phi}_U \circ \tilde{\gamma}^{-1} \circ (\tilde{\Phi}_U')^{-1}) \circ \mathcal{O}
        \end{displaymath}
        This is true because we may define the map $e \to j^1_{\pi(e)}\phi_e$, where $\phi_e$ is a locally constant section with value $e$ at $\pi(e)$, on $E\vert_U\simeq U\times F$.  In coordinates near $\pi(e)$, this map is exactly the zero section $\mathcal{O}$: $\phi_e$ is constant, hence all of its derivatives near $e$ vanish.
    \end{enumerate}
\end{enumerate}
\end{proof}

\begin{remark}
At the end of the proof of lemma \ref{lem:jetmap}, we defined a map $e \to j^1_{\pi(e)}\phi_e$ which sent $e$ to the equivalence class of a locally constant section defined near $\pi(e)$ with value $e$.  The reader may wonder why this map is only locally defined.  The concise answer is that the transition maps of $E$ may not send locally constant sections to locally constant sections.  To be more specific, let us suppose that it extends to a well defined section $\chi:E \to J^1(E)$.  Using a chart $(U,\tilde{\Phi}_U)$, one may readily see that
\begin{displaymath}
\tilde{\Phi}_U\circ \chi \circ \Phi_U^{-1}:U\times F \to \hom(TU,TF)
\end{displaymath}
\emph{must} be the zero section $\mathcal{O}$ since $\chi(e)$ is the equivalence class of a locally constant section whose derivatives vanish.  However, if we choose another chart $(V,\tilde{\Phi}_V)$ and consider:
\begin{displaymath}
\tau=(\tilde{\Phi}_V\circ \tilde{\Phi}_U^{-1}) \circ \mathcal{O} \circ (\Phi_V \circ \Phi_U^{-1})
\end{displaymath}
then there is no guarantee that this is the zero section of $\hom(T(U\cap V),TF)$ since the maps $\tilde{\Phi}_V\circ \tilde{\Phi}_U^{-1}$ are affine maps when restricted to the fibers (q.v. proposition \ref{prop:jettopology}, statement 3).  In diagrammatic form:

\begin{displaymath}
\xymatrixcolsep{5pc}\xymatrix{
\hom(T(U\cap V),TF) \ar[d]^{pr_U \times pr_F}  & \ar[l]_{\tilde{\Phi}_V^\circ \tilde{\Phi}_U^{-1}}^{\text{affine}} \hom(T(U\cap V), TF) \ar[d]^{pr_U \times pr_F} \\
(U\cap V) \times F  \ar@/^/[u]^{\tau}          & \ar[l]^{\Phi_V\circ \Phi_U^{-1}} (U\cap V)\times F \ar@/^/[u]^{\mathcal{O}}
}
\end{displaymath}
where $\tau$ is the pushforward of the zero section $\mathcal{O}$ using the top and bottom arrows.  Since the top arrow is an affine transformation, $\tau$ may not be the zero section.
\end{remark}

\begin{example}\label{ex:products}
Trivial fiber bundles are the focus of \cite{GG}, hence all of the facts (in the case of manifolds) we are about to present may be found there.  Let $M$ and $N$ be smooth manifolds with corners and let $M\times N \to M$ be the trivial fiber bundle over $M$ with fiber $N$.  We have essentially \emph{defined} $J^1(M\times N)$ to be $\hom(TM,TN)\to M\times N$ using the canonical map $\Phi:J^1(M\times N) \to \hom(TM,TN)$ given by:

\begin{displaymath}
\Phi(j_p^1\phi) = (p,\phi(p), dpr_2(d\phi_p))
\end{displaymath}
where $pr_2:M\times N \to N$ is the projection.  By lemma \ref{lem:jettopology}, this map is a bijection and therefore induces a smooth structure on $J^1(E)$ by pulling back the smooth structure on $\hom(TM,TN)$, which identifies $J^1(E)$ with $\hom(TM,TN)$.

\vspace{3mm}

Two interesting cases of jets of trivial bundles arise when the base space is $\R$ or the fiber is $\R$.

\begin{enumerate}
\item If $M=\R$, then the fiber bundle is $\R \times N \to \R$.  Let $pr_2:\R \times N \to N$ be the projection onto $N$.  Then any section $\phi(t)=(t,\gamma(t))$ gives rise to a curve $\gamma(t) = pr_2(\phi(t))$ and any curve gives rise to a section.  The first jet bundle $J^1(\R \times N)$ keeps track of the derivatives of these curves, hence it is no surprise that we obtain:

    \begin{displaymath}
    J^1(\R \times N) = \hom(T\R, TN) \simeq TN \times \R
    \end{displaymath}
    The extra factor of $\R$ arises because we are allowed to differentiate our curves $\gamma(t)$ at any $t\in \R$, hence the arbitrary convention of differentiating at $0$ and taking $\gamma'(0)$ as a tangent vector is removed.

\item If the fiber $N$ satisfies $N=\R$, then the bundle is $M\times \R$ and we have:

    \begin{displaymath}
    J^1(M\times \R) = \hom(TM,T\R) \simeq \hom(TM,\R)\times \R = T^*M \times \R
    \end{displaymath}
    The extra factor of $\R$ arises due to the fact that we must keep track of the image of the functions $f:M\to \R$ that we differentiate at $p \in M$ to obtain covectors $df_p\in T_pM$.  That is, the convention of forming the cotangent space $T_p^*M$ by differentiating functions $f:M\to \R$ at $p$ with value $f(p)=0$ is removed.
\end{enumerate}

\end{example}

\begin{example}\label{ex:vectorbundles}
This example is proposition 4.1.12 in \cite{Sa}.  Let $\pi:E \to M$ be a vector bundle.  Then there exists a canonical vector bundle structure on the bundle $\pi_1:J^1(E) \to M$.  We define:

\begin{itemize}
\item $j^1_p\phi + j^1_p\psi := j^1_p(\phi + \psi)$ and
\item $cj^1_p\phi := j^1_p(c\phi)$.
\end{itemize}
These operations are smooth since we are simply differentiating the addition and scalar multiplication operations on $E$.  We will see that vector bundle structures on $\pi_{1,0}:J^1(E) \to M$ are defined choices of connection on $E$, hence the vector bundle structures on $\pi_{1,0}$ are not canonically defined unless there is a canonical choice of connection.
\end{example}

\begin{example}\label{ex:mobius}
This is example 4.1.18 in \cite{Sa}, but we add a few more details.  Let $\R\times[0,1]/(t,0)\sim(-t,1)$ be the M\"{o}bius bundle over $S^1$.  It is a vector bundle, hence example \ref{ex:vectorbundles} reveals that $\pi_1:J^1(E) \to S^1$ has the structure of a vector bundle.  We claim that $J^1(E)\simeq E \oplus E$ as vector bundles over $S^1$.  We prove this fact by showing the two bundles have the same transition maps over the same cover of $S^1$.

\vspace{3mm}

For simplicity's sake, embed $S^1$ in $\R^2$ as the unit circle and cover it with the open sets $U_0=S^1\setminus \{(1,0)\}$ and $U_1=S^1\setminus \{(-1,0)\}$.  Let $U_{10}=U_1\cap U_0$ be the intersection.  The transition map for $E$ is:

\begin{displaymath}
\xymatrixrowsep{.5pc}\xymatrix{
\Phi_{10}:U_{10}\times \R \ar[r] &  U_{10}\times \R \\
           (p,t)  \ar[r]         &      (p,-t)
}
\end{displaymath}
meaning the induced transition map on $J^1(E)$ is given by:

\begin{displaymath}
\xymatrixrowsep{.5pc}\xymatrix{
\tilde{\Phi}_{10}:\hom(TU_{10},T\R) \ar[r]  & \hom(TU_{10},T\R)  \\
                   \tilde{\Phi}_{10}(p,t,A) \ar[r] & (p,-t,-A)
                   }
\end{displaymath}
Since $\hom(TU_{10},T\R)$ is a trivial $2$-plane bundle over $U_{10}$, we see that $J^1(E)$ and $E\oplus E$ have the same defining cover and transition map, meaning they are isomorphic.  Here, we use the vector field $\frac{\partial}{\partial \theta}$ on $S^1$ to trivialize $TS^1$.
\end{example}

\begin{example}\label{ex:principal}
Assume $G$ is a compact Lie group and $\pi:P\to M$ is a principal $G$-bundle.  Then $J^1(P)$ is equipped with an action of $G$ defined by:

\begin{displaymath}
g\cdot (j^1_p\phi) := j^1_p(g\cdot \phi)
\end{displaymath}
where $g\in G$.  This action is free:

\begin{displaymath}
g\cdot j_p^1\phi = j_p^1\phi  \mbox{ } \to \mbox{ } g\cdot\phi(p)=\phi(p) \mbox{ } \to \mbox{ } g=e
\end{displaymath}
since the action of $G$ on $P$ is free.  The action is proper since $G$ is compact, hence $J^1(P)/G$ is a manifold with corners.  The map $\pi_1:J^1(P) \to M$ is $G$-invariant and descends to a smooth map on the orbit space:

\begin{displaymath}
\bar{\pi}_1: J^1(P)/G \to M
\end{displaymath}
We will see that sections of $\bar{\pi}_1$ correspond to principal connections on $P$.
\end{example}

\subsubsection{Sections of $J^1(E)$ and Prolongations}
\noindent Given a local section $(U,\phi)$ of a fiber bundle $\pi:E \to M$, we would like to describe its prolongation $(U,j^1\phi)$ to a local section of $J^1(E)$ as in \cite{Sa}.  The prolongation will give us information about the derivative of the section $\phi$, hence first order singularities of $\phi$ will be encoded in the $1$-jet prolongation $j^1\phi$.

\begin{remark}
Here, a singularity is not a point where $\phi$ ceases to be smooth: a singular point will be a point where the differential of $\phi$ drops rank or corank in some way, which we cover in section 2.1.4.
\end{remark}

\begin{definition}\label{def:prolongation}
Let $(U,\phi)$ be a local section of a fiber bundle $\pi:E \to M$.  We define its \emph{prolongation} or $1$-\emph{jet}, $(U,j^1\phi)$, to a local section of $\pi_1:J^1(E)\to M$ to be the map:

\begin{displaymath}
(j^1\phi)(p) = j^1_p\phi
\end{displaymath}
where $j^1_p\phi$ is the equivalence class of sections at $p$ whose value is $\phi(p)$ and derivative is $d\phi_p$.
\end{definition}

\begin{lemma}\label{lem:prolongation}
Let $(U,\phi)$ be a local section of $\pi:E\to M$.  Then $(U,j^1\phi)$ is a smooth local section of $\pi_1:J^1(E)\to M$.
\end{lemma}

\begin{proof}\mbox{} \newline
Using a trivializiation $(V,\Phi_V)$ of $E$, we get a chart $(V,\tilde{\Phi}_V)$ and obtain the diagram:

\begin{displaymath}
\xymatrixcolsep{5pc}\xymatrix{
J^1(E)\vert_{U\cap V} \ar[d]^{\pi_{1,0}} \ar[r]^{\tilde{\Phi}_V} & \hom(T(U\cap V),TF) \ar[d]_{pr_U\times pr_F} \\
E \ar[d]^\pi \ar[r]^{\Phi_V}                             & (U\cap V)\times F \ar[d]^{pr_1} \\
U\cap V \ar@<-.5ex>[r]^{id} \ar@/^/[u]^{\phi} \ar@/^2pc/[uu]^{j^1\phi}                                         & U \cap V \ar@/^/[u]^{\Phi_V(\phi)} \ar@/_3pc/[uu]_{d(\Phi_V(\phi))}
}
\end{displaymath}
and we see that $j^1\phi$ is identified with the local section of $\hom(T(U\cap V),TF)$ given by $d(\Phi_V\circ \phi)$.  Since $\phi$ and $\Phi_V$ are smooth, $j^1\phi$ is smooth.
\end{proof}

\noindent We may also define smooth sections of $\pi_{1,0}:J^1(E) \to E$.

\begin{definition}\label{def:jetfields}
Let $\pi:E \to M$ be a fiber bundle over a manifold with corners and let $J^1(E)$ be its first jet bundle with projection map $\pi_{1,0}:J^1(E) \to E$.  A \emph{jet field} $\chi$ is a (global) section $\chi:E \to J^1(E)$ of $\pi_{1,0}$.
\end{definition}

\subsubsection{Connections, Jet Fields, and Vector Bundle Structures on $J^1(E)$}
Let $\pi:E\to M$ be a fiber bundle over a manifold with corners and let $J^1(E)$ be its first jet bundle.  We would like to study the relationships between Ehresmann connections on $E$, jet fields, and vector bundle structures on $J^1(E)$.

\begin{remark}\label{rem:connections}
Recall, if $\pi:E \to M$ is a fiber bundle over a manifold with corners $M$, then a (Ehresmann) connection $\chi$ is defined to be a smooth splitting of the tangent bundle $TE$ into $TE \simeq H \oplus V$, where $V=\ker(d\pi)$ is the \emph{vertical bundle} of $E$ and $H$ is a choice of \emph{horizontal bundle} transverse to $V$.  This choice gives us two projections:

\begin{itemize}
\item $p_H:TE \to H$ and
\item $p_V:TE \to V$.
\end{itemize}
where $p_H\big\vert_H = id_H$, $p_V\big\vert_V=id_V$, $p_H\big\vert_V=0$, and $p_V\big\vert_H=0$.  Conversely, given a projection operator $p_V:TE \to V$ satisfying $p_V\big\vert_V=id_V$, we may define $H=\ker(p_V)$ and $p_H:= id_{TE}-p_V$.  Similarly, if we are given an operator $p_H:TE \to TE$ satisfying $p_H^2=p_H$, $\ker(p_H)=V$, and $p_H\big\vert_{p_H(TE)} = id$, then we may define $H:= p_H(TE)$ as the horizontal subbundle and we obtain a vertical projection $p_V:= id_{TE}-p_V$.  Thus, a choice of an Ehresmann connection is equivalent to a choice of projection operator $p_V$ onto the vertical bundle or a choice of a projection operator $p_H$ onto a horizontal bundle.
\end{remark}

\begin{remark}
The following proposition is proposition 4.6.3 in \cite{Sa}.  Our proof is modeled closely on the one found in \cite{Sa}, but we avoid coordinates.
\end{remark}

\begin{prop}\label{prop:C1}
Let $\pi:E \to M$ be a fiber bundle over a manifold with corners and let $J^1(E)$ be its first jet bundle.  Let $\Gamma(J^1(E))$ be the space of jet fields, which are sections of $\pi_{1,0}:J^1(E) \to E$.  Let $\operatorname{Con}(E)$ denote the space of Ehresmann connections on $E$, which we will view as the space of projection operators $p_H:TE \to TE$ satisfying:

\begin{enumerate}
\item $p_H^2 = p_H$,
\item $\ker(p_H)=V$, and
\item $p_H\big\vert_{p_H(TE)} = id$.
\end{enumerate}
Then there is a bijection:

\begin{displaymath}
\Sigma:\Gamma(J^1(E)) \to \operatorname{Con(E)}
\end{displaymath}
hence a jet field on $E$ uniquely specifies a connection on $E$.
\end{prop}

\begin{proof} \mbox{ } \newline
Using the notation of proposition \ref{prop:C1}, we let $\Gamma(J^1(E))$ be the space of jet fields (sections of $\pi_{1,0}:J^1(E) \to E$) and let $\operatorname{Con}(E)$ denote the space of Ehresmann connections on $E$.  We will construct the map
\begin{displaymath}
\Sigma: \Gamma(J^1(E) \to \operatorname{Con}(E)
\end{displaymath}
show it is injective, and then argue that injectivity implies surjectivity using a gluing argument.

\begin{enumerate}
\item The map $\Sigma$ is relatively easy to define.  Let $\chi$ be a jet field on $E$.  Then for each $e\in E$, $\chi(e)=j^1_p\phi$ for some local section $\phi$ with $\phi(p)=e$.  This defines a map $(p_H)_e:T_eE \to TE$ given by $(p_H)_e:=d\phi_p \circ d\pi_e$ with $(p_H)_e^2=(p_H)_e$ and $\ker(p_H)_e=V_e$, hence $H_e=\operatorname{Image}(p_H)_e$ is a horizontal subspace of $TE$.  Now, $(p_H)_e$ is well-defined since it only depends on the derivative $d\phi_p$, which is the same for all elements of $j^1_p\phi$.  Performing the construction for each $e\in E$ gives us a projection operator $p_H:TE \to TE$.  It is smooth since $\chi$ is smooth (the derivatives $d\phi_p$ vary smoothly) and its image is a horizontal subbundle $H$ of $TE$.  Thus, $p_H$ defines an Ehresmann connection on $E$ which we'll denote $\Sigma(\chi)$.
\item Suppose $\Sigma(\chi_1)=\Sigma(\chi_2)$.  Then the horizontal subbundles $H_1,H_2$ agree, i.e. $H_1=H_2$, and the projection operators $p_H^i:TE \to H$, $i=1,2$ agree.  In order to show that $\chi_1(e)=\chi_2(e)$, we need only check that the derivatives of the representative sections agree.  For each point $e\in E$ we have $d\phi^1_p \circ d\pi_e = d\phi^2_p\circ d\pi_e$, where $\phi^i_p$ is a representative of $\chi_i(e)$.  Since $\pi$ is a submersion, this implies that $d\phi^1_p=d\phi^2_p$, hence $\chi_1(e)=\chi_2(e)$.
\item Now, let $p_V:TE \to V$ be a connection on $E$ with $p_H:TE \to H$ given by $p_H=id_{TE} - p_V$.  We will construct a jet field $\chi$ on $E$ so that $\Sigma(\chi)$ is $p_H:TE \to H$.  For each $p\in M$, let $U_p$ be a neighborhood so that $E\vert_{U_p}$ is trivializable, hence we may assume $E=U\times F$ with projection $p_V:TU\times TF \to pr_2^*TF$, where $pr_2:U\times F \to F$ is the projection onto the second factor.  We will construct the requisite jet field on each $U_p$ and show that these jet fields agree on overlaps $U_{p_1}\cap U_{p_2}$, $p_1,p_2\in M$, hence they glue together to give a global jet field.  Let $pr_1:U\times F \to U$ be the projection on the first factor.  Since $dpr_1(p_V)=0$, we have:
    \begin{displaymath}
    dpr_1(p_H)=dpr_1(id_{TU\times TF} - p_V) = dpr_1
    \end{displaymath}
    That is, we have:
    \begin{displaymath}
    p_H = dpr_1(p_H) \oplus dpr_2(p_H) = dpr_1 \oplus dpr_2(p_H)
    \end{displaymath}
    The term $dpr_2(p_H)$ is a $TF$-valued $1$-form on $U\times F$ that vanishes on vertical vectors, meaning it vanishes on $pr_2^*TF$.  We have an injective map $\circ dpr_1$:

    \begin{displaymath}
    \xymatrixcolsep{4pc}\xymatrix{
    \hom(TU,TF) \ar[r]^{\circ dpr_1} \ar[d] & \hom(TU\times TF, TF) \ar[d] \\
     U\times F   \ar[r]^{id}      &  U \times F
    }
    \end{displaymath}
      given by pre-composition with $dpr_1$.  The image consists of all maps for which the kernel is the vertical bundle $pr_2^*TF \to U\times F$.  Because $dpr_2(p_H)$ vanishes on vertical vectors, it is in the image of $\circ dpr_1$ and we have that $dpr_2(p_H)=\beta\circ dpr_1$ where $\beta$ is a unique (by injectivity of $\circ dpr_1$) smooth section of $\hom(TU,TF) \to U\times F$.  Note that $\hom(TU,TF)\simeq J^1(U\times F)$ so that $\beta$ corresponds to a jet field on $U\times F$.  At a point $(u,f)\in U\times F$, the jet corresponding to $\beta_{(u,f)}$ is the equivalence class:
    \begin{displaymath}
    j^1_u\phi = [(\phi(u)=(u,f),d\phi_u=(id_{T_uU} \oplus \beta))]
    \end{displaymath}
    Then the projection operator we obtain at $(u,f)$, using the recipe of part $1$ of the proof, is:
    \begin{displaymath}
    d\phi_u\circ (dpr_1)_{(u,f)} = id_{T_uU}\circ dpr_1 \oplus \beta \circ dpr_1 = dpr_1 \oplus dpr_2(p_H) = p_H
    \end{displaymath}
    which means $\beta$ gives us the correct connection on $U\times F$.

    \vspace{5mm}
    Now, by part $2$ of the proof, $\beta$ is the unique jet field on $E\vert_U$ corresponding to the connection $p_V:T(E\vert_U) \to V$.  Thus, we may cover $E$ by subsets of the form $E\vert_U$, where $E\vert_U$ is trivializable, and construct a local jet field $\beta_U$ on each such set corresponding to $p_V$ restricted to $E\vert_U$.  On overlaps $E\vert_{U_1} \cap E\vert_{U_2}=E\vert_{U_1\cap U_2}$, $\beta_{U_1}$ and $\beta_{U_2}$ both correspond to the unique jet field that map to the connection $p_V$, hence they agree.  Thus, the $\beta_{U_i}'s$ define a global jet field $\beta$ so that $\Sigma(\beta)=\chi$.  We therefore have that $\Sigma$ is a surjection.
\item $\Sigma$ is injective by part $2$ and surjective by part $3$, so we have proven that $\Sigma$ is a bijection.
\end{enumerate}
\end{proof}

\begin{prop}\label{prop:C2}
Let $\pi:E \to M$ be a fiber bundle over a manifold with corners, let $\tau_E:TE\to E$ be the projection, let $\tau_E:V\to E$ be the vertical bundle of $E$, and let $J^1(E)$ be the first jet bundle of $E$.  Then,

\begin{enumerate}
\item A choice of connection $\chi$ on $E$ specifies an isomorphism of fiber bundles, $F_{\chi}$, over $E$ given by:
\begin{displaymath}
\xymatrixcolsep{4pc}\xymatrixrowsep{1pc}\xymatrix{
F_{\chi}: J^1(E) \ar[r] & \hom(\pi^*TM, V) \\
j^1_p\phi  \ar[r]         & p_V(d\phi_p)
}
\end{displaymath}
where $p_V:TE \to V$ is the vertical projection afforded by $\chi$.
\item The isomorphism $F_{\chi}$ induces a vector bundle structure on $J^1(E)$ so that $F_{\chi}$ is an isomorphism of vector bundles.  Furthermore, all such structures are isomorphic.
\item If $U_1,U_2$ are two neighborhoods with trivializations $\Phi_i:E\vert_{U_i} \to U_i\times F$ and induced charts \newline $\tilde{\Phi}_{U_i}:J^1(E)\vert_{U_i} \to \hom(TU_i, TF)$, then the transition maps
    \begin{displaymath}
    \tilde{\Phi}_{12}=\tilde{\Phi}_{U_2}\circ \tilde{\Phi}_{U_1}^{-1}
    \end{displaymath}
    are linear with respect to the vector bundle structure induced by $F_{\chi}$.
\end{enumerate}
\end{prop}

\begin{remark}
The third claim about linearity of the transition maps is proven in \cite{Sa} by first showing $J^1(E)$ is an affine bundle modeled on $\pi^*TM\otimes V$ (theorem 4.1.11) and then proving lemma 2.4.8, which states that a section of an affine bundle modeled on a vector bundle induces a vector bundle structure on the affine bundle.  Since connections \emph{are} sections of $J^1(E)$, this is one way to see that connections induce vector bundle structures on $J^1(E)$.  In proposition \ref{prop:C2}, the vector bundle structure is made explicit when we construct the isomorphism with $\hom(\pi^*TM,V)$.
\end{remark}

\begin{proof} \mbox{ } \newline
\begin{enumerate}
\item Let $\chi$ be a connection on $\pi:E\to M$ and let $p_V:TE \to V$ be the corresponding projection onto the vertical bundle.  We define the map $F_{\chi}$ to be:

    \begin{displaymath}
    F_{\chi}(j^1_p\phi) = p_V(d\phi_p)
    \end{displaymath}
    which is well defined since it only depends on $\phi(p)$ and $d\phi_p$.  It is a map of fiber bundles since $\pi_{1,0}(j^1_p\phi)=\phi(p) = pr_E(p_V(d\phi_p))$, where $pr_E:\hom(\pi^*E, V) \to E$ is the projection.  We can see smoothness as follows.  A choice of trivialization $\Phi_U:E\vert_U \to U\times F$ gives us an identification $J^1(E)\simeq \hom(TU,TF)$, $\pi^*TM \simeq pr_1^*TU$, $V=pr_2^*TF$, and $\hom(\pi^*TM,V) = \hom(TU,TF)$, where $pr_1:U\times F \to U$, $pr_2:U\times F \to F$ are projection onto the first and second factors, respectively.  The map $F_{\chi}$ is then the composition:

    \begin{displaymath}
    \xymatrixrowsep{.5pc}\xymatrixcolsep{5pc}\xymatrix{
    \hom(TU,TF) \ar[r] & \hom(TU,TU\times TF) \ar[r] & \hom(TU,TF) \\
    (u,f,A) \ar[r]     & (u,f,id_{T_uU} \oplus A)   \ar[r] & (u,f, p_V(id_{T_uU} \oplus A))
    }
    \end{displaymath}
    where $(u,f) \in U\times F$.  Each operation is smooth, hence $F_{\chi}$ is locally a composition of smooth operations.  Furthermore, we can see that $F_{\chi}$ is a fibrewise isomorphism using this local picture.  The map is injective since, for any $A_1,A_2 \in \hom(T_uU,T_fF)$:

    \begin{displaymath}
    p_V(id_{T_uU}\oplus A_1) = p_v(id_{T_uU}\oplus A_2) \iff p_V(0\oplus (A_1-A_2)) = 0 \iff A_1-A_2 = 0 \iff A_1=A_2
    \end{displaymath}
    since $p_V\vert_V = id_V$ and $(0\oplus A_1-A_2)$ maps into the vertical bundle of $U\times F$.  It is an affine map of vector bundles since $A \to id_{T_uU} \oplus A$ is affine and projection onto $V$ is linear.  Since it is an injective, affine map of vector of the same rank, it is an isomorphism of $\hom(TU,TF)$ with itself.

    This shows that $F_{\chi}$ is locally an isomorphism of fiber bundles.  Since it is globally a map of fiber bundles, it is therefore a global isomorphism of fiber bundles.

\item It is a general fact that if $\phi:E_1 \to E_2$ is an isomorphism of fiber bundles and $E_2$ is a vector bundle, then one can use $\phi$ to pull back the vector bundle structure on $E_2$ to $E_1$.  In our case, it will be useful to give explicit formulas for addition and scalar multiplication. To do so, we'll need a formula for the inverse of $F_{\chi}$.  Let $\tau_E:H\to E$ be the horizontal bundle and consider the pullback diagram:

    \begin{displaymath}
    \xymatrixcolsep{4pc}\xymatrix{
    H \ar@/_2pc/[ddr]^{\tau_E} \ar@/^2pc/[drr]^{d\pi} \ar@{-->}[dr]^{\tilde{d\pi}} &                          & \\
                                                                             & \pi^*TM \ar[d]^{\pi^*\tau_M} \ar[r]^{\bar{\pi}} & TM \ar[d]^{\tau_M} \\
                                                                             & E \ar[r]^{\pi}                                & M
    }
    \end{displaymath}
    The universal property of pullbacks guarantees the existence of a smooth map $\tilde{d\pi}:H \to \pi^*TM$.  Since $d\pi:H \to TM$ is a fibrewise isomorphism, we must have that $\tilde{d\pi}$ is an isomorphism of vector bundles over $E$.  We can now write:

    \begin{displaymath}
    F_{\chi}^{-1}(e, A)= [(\pi(e),e, \tilde{d\pi}_e^{-1} + A)]
    \end{displaymath}
    where $[(\pi(e),e, \tilde{d\pi}_e^{-1} + A)]$ is the equivalence class of local sections at $\pi(e)$ with value $e$ and derivative $\tilde{d\pi}_e^{-1} + A$.  Note that all we have done is added the horizontal portion of the derivative of a section, $\tilde{d\pi}^{-1}$, to the vertical portion, $A$, to reconstruct the whole derivative of the local section.  To see this explicitly, note that if $p_H:TE \to H$ is the horizontal projection then for any local section $\phi$ of $E$ at $p\in M$ we have:

    \begin{displaymath}
    \tilde{d\pi}_{\phi(p)} \circ p_H \circ d\phi_p = id_{T_pM} \iff p_H\circ d\phi_p = \tilde{d\pi}_{\phi(p)}^{-1}
    \end{displaymath}
    hence the horizontal portion $p_H\circ d\phi$ is indeed $\tilde{d\pi}^{-1}$.  We may use this fact to write out the vector bundle operations on $J^1(E)$.

    \vspace{2mm}
     We will use $+_J$ and $(c)_J$ to denote addition and scalar multiplication on the jet bundle $J^1(E)$.
    \begin{displaymath}
    \begin{array}{lcl}
    j^1_p\phi_1 +_J j^1_p\phi_2 & := & F_{\chi}^{-1}(F_{\chi}(j^1_p\phi_1) + F_{\chi}(j^1_p\phi_2)) \\
                                &  = & F_{\chi}^{-1}(p_V((d\phi_1)_p + (d\phi_2)_p)) = [(p,\phi_1(p), \tilde{d\pi}^{-1} + p_V((d\phi_1)_p + (d\phi_2)_p)] \\

    (c)_J(j^1_p\phi)           &  := & F_{\chi}^{-1}(cF_{\chi}(j^1_p\phi)) = [(p,\phi(p), \tilde{d\pi}^{-1} + cp_V(d\phi_p))]
    \end{array}
    \end{displaymath}
    Note that we are defining addition in the fiber of $J^1(E)$ over $e$, hence $\phi_1(p)=\phi_2(p)$ in the definition of $+_J$.  Now, it is an easy consequence of the definitions of $F_{\chi}$ and the vector space structure that any two choices of connection, $\chi_1,\chi_2$ give isomorphic vector bundle structures on $J^1(E)$.  Note that we have:

    \begin{displaymath}
    \xymatrixcolsep{5pc}\xymatrix{
    J^1(E) \ar[r]^{F_{\chi_1}}   & \hom(\pi^*TM, V) & \ar[l]_{F_{\chi_2}} J^1(E)
    }
    \end{displaymath}
    where each map is an isomorphism of vector bundles with respect to the induced vector bundle structures.  Hence the composite $F_{\chi_2}^{-1} \circ F_{\chi_1}:J^1(E) \to J^1(E)$ is an isomorphism of vector bundles.

\item We have two trivializations of $E$ which give us a commutative diagram:

\begin{displaymath}
\xymatrixcolsep{4pc}\xymatrix{
\hom(T(U_1 \cap U_2), TF) \ar[r]^-{\tilde{\Phi}_{U_1}^{-1}} \ar[d]& J^1(E) \ar[d]^{\pi_{1,0}} \ar[r]^-{\tilde{\Phi}_{U_2}} & \hom(T(U_1 \cap U_2), TF) \ar[d] \\
(U_1\cap U_2) \times F \ar[r]^{\Phi_{U_1}^{-1}} \ar[dr]^{pr_1}  & E                   \ar[r]^{\Phi_{U_2}}              & (U_1 \cap U_2)\times F \ar[dl]^{pr_1}    \\
                                                                & U_1\cap U_2  &                                                                               
}
\end{displaymath}
and we wish to check that the composite $\tilde{\Phi}_{12}:= \tilde{\Phi}_{U_2} \circ \tilde{\Phi}_{U_1}^{-1}$ is linear with respect to the induced vector bundle structures on $\hom(T(U_1\cap U_2), TF)$.  Let us give formulas for these structures.  We have a collection of maps:

\begin{displaymath}
\begin{array}{ll}
pr_1:(U_1\cap U_2)\times F \to (U_1 \cap U_2) & \text{Projection onto the first factor.} \\
pr_2:(U_1\cap U_2) \times F \to F             & \text{Projection onto the second factor.} \\
p_V^i:T(U_1\cap U_2) \times TF \to pr_2^*TF   & \text{Projection onto the vertical bundle in $U_i$.} \\
id_{TU} \oplus 0: TU \rightarrow TU\times TF  & \text{The section of $\hom(TU,TU\times TF)$ given by}\\
                                              & (id_{TU}\oplus 0)_{(u,f)}(v)= (v,0)\in T_uU\times T_fF
\end{array}
\end{displaymath}
We have the relations:
\begin{enumerate}
\item $p_V^i(0\oplus dpr_2)=dpr_2$, which follows since $p_V^i$ restricted to the vertical bundle is the identity.
\item $dpr_2(0 \oplus p_V^i) = dpr_2 + p_V^i(dpr_1 \oplus 0)$, which follows from the first identity since $p_V^i=p_V^i(dpr_1\oplus dpr_2)$.
\item Letting $\Phi_{12}=\Phi_{U_2}\circ \Phi_{U_1}^{-1}$, we have that $d\Phi_{12}(0 \oplus p_V^1) = (0\oplus p_V^2)(d\Phi_{12})$.  That is, the transition maps are connection preserving.  This is because the connections on $U_i\times F$ are induced by the global connection on $E$.
\item We also have $(dpr_1\oplus 0)\circ d\Phi_{12} = dpr_1 \oplus 0$ since $\Phi_{12}$ is a map of fiber bundles, hence
\begin{displaymath}
(dpr_1\oplus 0)(d\Phi_{12}(id_{T_uU} \oplus A)) = (dpr_1\oplus 0)(id_{T_uU}\oplus A) = id_{T_uU} \oplus 0
\end{displaymath}
\end{enumerate}
We also have nice formulas for addition, $+_J$, and scalar multiplication, $(c)_J$.  For $A_1,A_2\in \hom(T_uU,T_fF)$ we compute:

\begin{displaymath}
\begin{array}{lcl}
A_1 +_J A_2 & = & dpr_2(\tilde{d\pi}^{-1} + p_V^i((id_{T_uU} \oplus A_1) + (id_{T_uU} \oplus A_2)))\\
            & = & dpr_2(\tilde{d\pi}^{-1} + p_V^i(id_{T_uU} \oplus 0) + p_V^i(id\oplus (A_1 + A_2))) \\
            & = & dpr_2(id_{T_uU} \oplus p_V^i(id_{T_uU}\oplus(A_1 + A_2))) \\
            & = & p_V^i(id_{T_uU}\oplus(A_1 + A_2)) \\
            &   & \text{where the last line follows since $dpr_2(p_V)=p_V$ at $(u,f)$.} \\
            & \mbox{ } & \\
(c)_J(A)  & = & dpr_2(\tilde{d\pi}^{-1} + cp_V^i(id_{T_uU} \oplus A)) \\
          & = & dpr_2(id + p_V^i(c(id_{T_uU}\oplus A) - id_{T_uU}\oplus 0))\\
          & = & p_V^i(c(id_{T_uU}\oplus A) - id_{T_uU}\oplus 0)
\end{array}
\end{displaymath}
Recall that $\tilde{\Phi}_{12}(A) = dpr_2(d\Phi_{12}(id_{T_uU}\oplus A))$.  We now show the transition maps are linear.
\begin{displaymath}
\begin{array}{lcl}
\tilde{\Phi}_{12}(A_1+_J A_2) &=& dpr_2(d\Phi_{12}(id_{T_uU}\oplus(A_1+_J A_2))) \\
                               &=& dpr_2(d\Phi_{12}(id_{T_uU}\oplus(p_V^i(id_{T_uU}\oplus(A_1+A_2))))) \\
                               &=& dpr_2(d\Phi_{12}(id_{T_uU}\oplus 0) + dpr_2(d\Phi_{12}(0\oplus p_V^i(id_{T_uU}\oplus(A_1+A_2)))) \\
                               &=& dpr_2(d\Phi_{12}(id_{T_uU}\oplus 0) + dpr_2((0\oplus p_V^2)d\Phi_{12}(id_{T_uU}\oplus(A_1+A_2))) \\
                               &=& dpr_2(d\Phi_{12}(id_{T_uU}\oplus 0) + dpr_2(d\Phi_{12}(id_{T_uU}\oplus(A_1+A_2))) +\\
                               & & p_V^2(dpr_1\oplus 0)d\Phi(id_{T_uU}\oplus A_1 + A_2) \\
                               &=& dpr_2(d\Phi_{12}(id_{T_uU}\oplus 0) + dpr_2(d\Phi_{12}(id_{T_uU}\oplus(A_1+A_2))) + p_V^2(id_{T_uU}\oplus 0)\\
                               & & \text{where we have appealed to two of the previously listed identities.} \\
                               &\mbox{} & \\
\tilde{\Phi}_{12}(A_1)+_J \tilde{\Phi}_{12}(A_2) &=& p_V^2(id_{T_uU}\oplus(\tilde{\Phi}_{12}(A_1)+ \tilde{\Phi}_{12}(A_2))) \\
                                                 &=& p_V^2(id_{T_uU}\oplus 0) + p_V^2(0\oplus dpr_2(d\Phi_{12}(id_{T_uU}\oplus A_1 + id_{T_uU}\oplus A_2)))\\
                                                 &=& p_V^2(id_{T_uU}\oplus 0) + dpr_2(d\Phi_{12}(id\oplus 0)) + dpr_2(d\Phi_{12}(id\oplus(A_1+A_2)))
\end{array}
\end{displaymath}
Evidently, $\tilde{\Phi}_{12}(A_1+_J A_2)=\tilde{\Phi}_{12}(A_1)+_J \tilde{\Phi}_{12}(A_2)$.  For scalar multiplication, we have:

\begin{displaymath}
\begin{array}{lcl}
\tilde{\Phi}_{12}((c)_J A) & = & dpr_2(d\Phi_{12}(id_{T_uU} \oplus p_V^1(c(id_{T_uU}\oplus A) -id_{T_uU}\oplus 0))) \\
                           & = & dpr_2(d\Phi_{12}(id_{T_uU}\oplus 0))+dpr_2(d\Phi_{12}(0\oplus p_V^1)(c(id_{T_uU}\oplus A) - id_{T_uU}\oplus 0)) \\
                           & = & dpr_2(d\Phi_{12}(id_{T_uU}\oplus 0)) +dpr_2(0\oplus p_V^2)(d\Phi_{12}(c(id_{T_uU}\oplus A) -id_{T_uU}\oplus 0)) \\
                           & = & dpr_2(d\Phi_{12}(id_{T_uU}\oplus 0)) +dpr_2(d\Phi_{12}(c(id_{T_uU}\oplus A) -id_{T_uU}\oplus 0)) + \\
                           &   & p_V^2(dpr_1 \oplus 0)(c(id_{T_uU}\oplus A) -id_{T_uU}\oplus 0) \\
                           & = & cdpr_2(d\Phi_{12}(id \oplus A)) + (c-1)p_V^2(id_{T_uU}\oplus 0) \\
                           & \mbox {} & \\
(c)_J\tilde{\Phi}_{12}(A)  & = & p_V^2(c(id_{T_uU}\oplus \tilde{\Phi}_{12}(A)) - id_{T_uU} \oplus 0) \\
                           & = & (c-1)p_V^2(id_{T_uU}\oplus 0) + cp_V^2(0 \oplus \tilde{\Phi}_{12}(A)) \\
                           & = & (c-1)p_V^2(id_{T_uU}\oplus 0) + cp_V^2(0 \oplus dpr_2(d\Phi_{12}(id\oplus A)))\\
                           & = & (c-1)p_V^2(id_{T_uU}\oplus 0) + cdpr_2(d\Phi_{12}(id\oplus A))
\end{array}
\end{displaymath}
hence $\tilde{\Phi}_{12}((c)_J A)=(c)_J\tilde{\Phi}_{12}(A)$.
\end{enumerate}
\end{proof}

\begin{prop}\label{prop:C3}
Let $\pi:E \to M$ be a fiber bundle over a manifold with corners, $M$, equipped with a connection $\chi$.  Let $F_{\chi}:J^1(E) \to \hom(\pi^*TM,V)$ be the induced isomorphism of fiber bundles, which equips $J^1(E)$ with the structure of a vector bundle.  Let $\Sigma : \Gamma(\pi_{1,0}:J^1(E) \to E) \to \operatorname{Con}(E)$ be the bijection between jet fields and connections from proposition \ref{prop:C1}.  Let $\mathcal{O}\in \Gamma(\hom(\pi^*TM, V))$ be the zero section.  Then:

\begin{displaymath}
\Sigma^{-1}(\chi)=F_{\chi}^{-1}(\mathcal{O})
\end{displaymath}
Thus, the zero section of $J^1(E)$ corresponds to the connection $\chi$.
\end{prop}

\begin{proof} \mbox{} \newline
We need only check that $\Sigma(F_{\chi}^{-1}(\mathcal{O}))$ corresponds to the horizontal projection operator $p_H:TE \to H$ induced by $\chi$.  We have:
\begin{displaymath}
j^1_p\phi \in F_{\chi}^{-1}(\mathcal{O}) \mbox{ } \iff p_V(d\phi_p)=0 \mbox{ } \iff (id_{TE})_{\phi(p)} = d\phi_p\circ d\pi_{\phi(p)} + (p_V)_{\phi(p)}  \end{displaymath}
hence $(p_H)_{\phi_(p)} = (id_{TE})_{\phi_(p)} - (p_V)_{\phi_(p)} =  d\phi_p\circ d\pi_{\phi(p)}$ and we see that the horizontal projections associated with $\chi$ and $\Sigma(F_{\chi}^{-1}(\mathcal{O}))$ are the same, hence they are the same connection and, since $\Sigma$ is invertible, we have:

\begin{displaymath}
\Sigma(F_{\chi}^{-1}(\mathcal{O}))=\chi \iff \Sigma^{-1}(\chi)=F_{\chi}^{-1}(\mathcal{O})
\end{displaymath}
To clarify the last statement of the proposition, note that with the vector bundle structure on $J^1(E)$ induced by $F_{\chi}$, the zero section is $F_{\chi}^{-1}(\mathcal{O})$.
\end{proof}

\begin{example}\label{ex:mobius2}
Let $\pi:E \to S^1$ be the M\"{o}bius bundle, which we view as the space $E=\R\times [0,1]/(t,0)\sim(-t,1)$.  We have seen that $\pi_1:J^1(E)\to S^1$ is endowed with a canonical vector bundle structure arising from addition and scalar multiplication on $E$.  By propositions \ref{prop:C1} and \ref{prop:C2}, we may specify a vector bundle structure on $\pi_{1,0}:J^1(E) \to E$ by specifying a jet field (connection) $\chi:E \to J^1(E)$.

\vspace{3mm}

Let $\chi:E \to J^1(E)$ be the jet field defined by $\chi(e)=j^1_p\phi$, where $\phi(p)=e$ and $\phi$ is constant in a neighborhood of $p$.  The transition map for $E$ corresponds to multiplication by $\pm 1$, hence if $\phi$ is locally constant in one trivialization, the transition map sends it to a locally constant section in another trivialization.  Therefore, $\chi$ defines a jet field on $E$.  Note that the collection of tangent spaces defined by $\chi$, i.e. the connection defined by $\chi$, have integral submanifolds:

\begin{itemize}
\item The zero section, which we view as an embedded copy of $S^1$ in $E$, and
\item the connected double covers of $S^1$ given by the orbits of the lifted action on $S^1$.
\end{itemize}

In the model $\R \times [0,1]/\sim$, we may equivalently describe these submanifolds as:

\begin{itemize}
\item The zero section $\{[(0,s)]\}$ and
\item the strips $\{[(t,s)]\} \cup \{[(-t,s)]\}$, where $t\in \R$ is fixed and nonzero.
\end{itemize}

\end{example}

\begin{example}\label{ex:vectorbundle2}
We have seen that if $\pi:E \to M$ is a vector bundle, then $\pi_1:J^1(E) \to M$ is endowed with a canonical vector bundle structure induced by addition and scalar multiplication on $E$.  A linear connection on $E$ is defined to be a choice of horizontal subspace $H_e \subset T_eE$ at each $e\in E$ that is invariant under (the derivatives of) the addition and scalar multiplication operations of $E$.  We can encode this in the language of jet fields by defining a linear connection on $E$ to be a jet field that induces a map of vector bundles:

\begin{displaymath}
\xymatrix{
\chi: E \ar[r] \ar[d]^\pi &  J^1(E) \ar[d]^{\pi_1} \\
M \ar[r]                    &   M}
\end{displaymath}
\end{example}

\begin{example}\label{ex:principal2}
As in example \ref{ex:principal}, let $G$ be a compact Lie group.  Let $\pi:P\to M$ be a principal $G$ bundle and recall from example \ref{ex:principal} that $J^1(P)$ is equipped with a free and proper action of $G$.  A principal connection on $P$ is a choice of $G$-invariant splitting $H \oplus V$ of $TP$. We can encode this in the language of jet fields by defining a principal connection to be a $G$-equivariant section $\chi:P \to J^1(P)$.

\vspace{3mm}

To see that this is the correct definition, let $p\in P$ and let $H_p$ be the horizontal distribution at $p$.  We may represent $H_p$ as the image of the differential of some section $\phi$ at $\pi(p)$.  That is, $H_p = \operatorname{Im}(d\phi_m)$, where $m=\pi(p)$.  Let $\tau_g:P \to P$ be the action of $g\in G$.  $G$-invariance then implies that:

\begin{displaymath}
H_{g\cdot p} = d\tau_g(\operatorname{Im}(d\phi_m)) = \operatorname{Im}(d(g\cdot \phi)_m) = d\tau_g(H_p).
\end{displaymath}
Thus $\chi(g\cdot p) = j^1_m(g\cdot \phi) = g\cdot (j^1_m\phi)$.  Now, $G$-equivariance of $\chi$ implies that $\chi$ descends to a section:

\begin{displaymath}
\bar{\chi}:P/G \to J^1(P)/G
\end{displaymath}
of $\bar{\pi}_1:J^1(P)/G \to M$.  Conversely, any section $\bar{\chi}:M \to J^1(P)/G$ lifts to a $G$-equivariant section $\chi:P \to J^1(P)/G$ as follows:

\begin{itemize}
\item $\bar{\chi}(m) = [j^1_m\phi]$ is an equivalence class of jets where $j^1_m\phi$ and $j^1_m\psi$ are in this equivalence class if and only if $j^1_m\phi = j^1_m(g\cdot \psi)$ for some $g\in G$.  Consequently, if $\psi(m)=\phi(m)$, then $g=e$ since the action of $G$ is free on $P$ and $d\phi_m = d\psi_m$.
\item The transitivity of the action of $G$ on $P$ implies that for each point $p\in \pi^{-1}(m)$ in the fiber over $m$, there exists an element $j^1_m\psi \in \bar{\chi}(m)$ so that $\psi(m)=p$.  By the preceding bullet, this element is unique.
\item Thus, we may define the lift $\chi(p)$ to be the unique element $j^1_m\phi \in \bar{\chi}(\pi(p))$ satisfying $\phi(m)=p$.  To see that $\chi$ is a \emph{smooth} lift, it's enough to check the claim locally, hence we may assume $P=M\times G$.  We leave this computation to the reader.
\end{itemize}
\end{example}

\subsubsection{1-Jets of Corank r}

In this section, we closely follow the constructions of \cite{GG}, but we adapt everything to arbitrary fiber bundles.  Given a local section $(U,\phi)$ of a fiber bundle $\pi:E \to M$, we have seen that $j^1\phi:U \to J^1(E)$, the $1$-jet of $\phi$, is a smooth local section of $\pi_1:J^1(E)\to M$.  We would like to garner as much information as possible about $\phi$ from the properties of the $1$-jet, $j^1\phi$, of $\phi$.  One way to do this is to study the rank or corank of the $1$-jet, which will give us information about the subset of points where the derivative of $\phi$ fails to be injective or surjective, depending whether we study rank or corank, respectively.

\begin{definition}\label{def:rank}
Let $\pi:E \to M$ be a fiber bundle over a manifold with corners $M$ equipped with a connection $\chi$.  That is, we have a projection $p_V:TE \to V$ to the vertical bundle of $E$ induced by the connection $\chi$.  We define $\operatorname{rank}:J^1(E) \to \Z$ and $\operatorname{corank}:J^1(E) \to \Z$ to be:

\begin{displaymath}
\begin{array}{l c l}
\operatorname{rank}(j^1_p\phi) &=& \operatorname{rank}(p_V(d\phi_p))\\
\operatorname{corank}(j^1_p\phi) &=& \operatorname{corank}(p_V(d\phi_p))
\end{array}
\end{displaymath}
\end{definition}
It is immediate from the definition that rank and corank are nonnegative.  In light of proposition \ref{prop:C2}, we have defined the rank of a $1$-jet $j^1_p\phi$, given $\chi$, to be the rank of the corresponding element $F_{\chi}(j^1_p\phi) \in \hom(\pi^*TM,V)_{\phi(p)}$.  We are interested in a particular subset of $J^1(E)$ defined as follows:

\begin{definition}\label{def:Sr}
Let $\pi:E \to M$ be a fiber bundle over a manifold with corners equipped with a connection $\chi$.  We define $S_r \subset J^1(E)$ to be the set:

\begin{displaymath}
S_r := \{j^1_p\phi \vert \mbox{ } \operatorname{corank}(j^1_p\phi)=r\}
\end{displaymath}
\end{definition}

Our next goal is to prove the following proposition:

\begin{prop}\label{prop:Sr}
Let $\pi:E \to M$ be a fiber bundle over a manifold with corners equipped with a connection $\chi$.  Then $S_r \subset J^1(E)$ is a sub-fiber bundle of $\pi_{1,0}:J^1(E) \to E$ with typical fiber given by $L^r_{m,n}$, where $m=\dim(M)$, $n=\dim(E)-\dim(M)$, and:

\begin{displaymath}
L^r_{m,n}:= \{ A \in \hom(\R^m, \R^n) \vert \mbox{ } \operatorname{corank}(A)=r\}
\end{displaymath}
Furthermore, $S_r$ has codimension $(m-q+r)(n-q+r)$, where $q=\min\{m,n\}$.
\end{prop}

The proof of the proposition is straightforward once we have the following lemma, which is lemma 5.2 and proposition 5.3 of \cite{GG} combined.

\begin{lemma}\label{lem:Lrmn}
Let $L^r_{m,n} = \{A \in \hom(\R^m, \R^n) \vert \mbox{ } \operatorname{corank}(A)=r\}$.  Then $L^r_{m,n}$ is a smooth submanifold of $\hom(\R^m,\R^n)$ of codimension $(m-q+r)(n-q+r)$, where $q=\min\{m,n\}$.
\end{lemma}

\begin{proof} \mbox{ } \newline
Let $S$ be an $m\times n$ matrix where $S=\begin{pmatrix}A & B \\ C & D \\ \end{pmatrix}$, where $A$ is a $q-r \times q-r$ invertible matrix.  Define

\begin{displaymath}
T=\begin{pmatrix}
I_{q-r} & 0 \\
-CA^{-1} & I_r \\
\end{pmatrix}
\end{displaymath}
and note that $T$ is invertible.  Hence $\operatorname{rank}(S)=\operatorname{rank}(TS) = \operatorname{rank}(\begin{pmatrix} A & B \\ 0 & D-CA^{-1}B\\ \end{pmatrix})$.  The latter matrix has rank $q-r$ if and only if $D-CA^{-1}B=0$.  We can use this fact to see $L^r_{m,n}$ is a submanifold as follows.

Let $S \in L^r_{m,n}$ and choose bases of $\R^n$, $\R^m$ so that $S=\begin{pmatrix}A & B \\ C & D \\ \end{pmatrix}$ where $A$ is a $q-r\times q-r$ invertible matrix.  Let $U$ be the open neighborhood of $S$ consisting of all matrices of the form $S'=\begin{pmatrix}A' & B' \\ C' & D' \\ \end{pmatrix}$, with $A'$ a $q-r\times q-r$ invertible matrix.

\vspace{2mm}

Define $f:U\to \hom(\R^{n-q+r}, \R^{m-q+r})$ by $f(S')=D'-C'(A')^{-1}B'$.  We can see that $f$ is a submersion as follows.  Fix $A'$, $B'$, $C'$, and $D'$, let $D\in T_{f(S')}\hom(\R^{n-q+r}, \R^{m-q+r})$, and define $\gamma(t)= \begin{pmatrix}A' & B' \\ C' & D' + tD  \\ \end{pmatrix}$.  Then
\begin{displaymath}
\frac{\partial}{\partial t} f(\gamma(t)) = D
\end{displaymath}
which shows that $df_{S'}$ is surjective.  Consequently, $f^{-1}(0)$ is a smooth submanifold of $U$.  Since $\operatorname{rank}(S')=q-r \iff D'-C'(A')^{-1}B=0$, we have that $f^{-1}(0) = L^r_{m,n}\cap U$, which shows $L^r_{m,n}$ is a smooth submanifold of $\hom(\R^n, \R^m)$.  Since $\{0\}$ has codimension $(n-q+r)(m-q+r)$ in $\hom(\R^{n-q+r}, \R^{m-q+r})$, $L^r_{m,n}$ has codimension $(n-q+r)(m-q+r)$.
\end{proof}

\begin{proof}[proof of proposition \ref{prop:Sr}] \mbox{ } \newline
We have a fiber bundle $\pi:E \to M$ and a connection $\chi$.  Let $\tau_E:V\to E$ denote the vertical bundle of $E$ with projection $p_V:TE \to V$.  By proposition \ref{prop:C2}, $\chi$ defines an isomorphism of fiber bundles:

\begin{displaymath}
\xymatrixrowsep{.5pc}\xymatrix{
F_{\chi}:J^1(E) \ar[r] &  \hom(\pi^*TM, V)\\
 j^1_p\phi      \ar[r]   &   p_V(d\phi_p)
}
\end{displaymath}
Note that the typical fiber of $\hom(\pi^*TM,V)$ is $\hom(\R^m,\R^n)$, where $m=\dim(M)$ and $n=\dim(E)-\dim(M)$ is the dimension of the fibers of $E$.  By lemma \ref{lem:Lrmn}, $L^r_{m,n}$ is a submanifold of $\hom(\R^m,\R^n)$.  Since it is invariant under the structure group $GL(\R,n) \times GL(\R,m)$ of $\hom(\pi^*TM, V)$, it defines a smooth sub-fiber bundle $L_r\subset \hom(\pi^*TM, V)$ with typical fiber isomorphic to $L^r_{m,n}$.  The fibers of $L_r$ are exactly the elements with corank $r$ and

\begin{displaymath}
\operatorname{corank}(j^1_p\phi) = r \mbox{ } \iff \mbox{ } \operatorname{corank}(p_V(d\phi)p=r \mbox{ } \iff \mbox{ } \operatorname{corank}F_{\chi}(j^1_p\phi) = r
\end{displaymath}
by the definitions of corank and $F_{\chi}$, hence $F_{\chi}\vert_{S_r}:S_r \to L_r$ is a bijection.  Therefore,

\begin{displaymath}
S_r = F_{\chi}^{-1}(L_r)
\end{displaymath}
Since $L_r$ is a sub-fiber bundle and $F_{\chi}^{-1}$ is an isomorphism of fiber bundles, $S_r$ is a sub-fiber bundle of $J^1(E)$.  Since the codimension of $L_r$ is $(m-q+r)(n-q+r)$, where $q=\min\{m,n\}$, we have that the codimension of $S_r$ is $(m-q+r)(n-q+r)$.
\end{proof}

Since rank and corank depend on the connection, they do not behave well with general maps of fiber bundles.  However, these properties are preserved by connection-preserving diffeomorphisms of fiber bundles.

\begin{lemma}\label{lem:rank}
Let $\pi:E\to M$ and $\pi':E\to M$ be two fiber bundles over a manifold with corners $M$ with connections $\chi$, $\chi'$, respectively.  That is, we have projection maps $p_V:TE \to V$ and $p_{V'}:TE' \to V'$ to the vertical bundles of $E$ and $E'$, respectively.  Let $\tau:E\to E'$ be a diffeomorphism of fiber bundles so that $\tau^*\chi' = \chi$.  Then $\operatorname{rank}\circ \tilde{\tau} = \operatorname{rank}:J^1(E) \to \Z$ and $\operatorname{corank}\circ \tilde{\tau}=\operatorname{corank}:J^1(E)\to \Z$.  That is, the notion of rank and corank is preserved by connection-preserving diffeomorphisms of fiber bundles.
\end{lemma}

\begin{proof} \mbox{} \newline
$\tau^*\chi' = \chi$ is equivalent to the projection operators satisfying:

\begin{displaymath}
p_V = d\tau^{-1} \circ p_{V'} \circ d\tau \iff d\tau \circ p_V = p_{V'} \circ d\tau
\end{displaymath}
Hence, for any element $j^1_p\phi \in J^1(E)$ we have:

\begin{displaymath}
\begin{array}{lcl}
\operatorname{rank}(\tilde{\tau}(j^1_p\phi)) & = & \operatorname{rank}(j^1_p(\tau \circ \phi)) \\
                                             & = & \operatorname{rank}(p_{V'}(d\tau(d\phi_p))) \\
                                             & = & \operatorname{rank}(d\tau(p_V(d\phi_p))) \\
                                             & = & \operatorname{rank}(p_V(d\phi_p)) \\
                                             & = & \operatorname{rank}(j^1_p\phi)
\end{array}
\end{displaymath}
where the second-to-last line follows since $d\tau\vert_V:V \to V'$ is an isomorphism, hence it doesn't change the rank.  Replacing rank by corank and performing the same computations, one sees that:

\begin{displaymath}
\operatorname{corank}(\tilde{\tau}(j^1_p\phi)) = \operatorname{corank}(j^1_p\phi)
\end{displaymath}
\end{proof}

\subsubsection{The Intrinsic Derivative}
We now discuss the intrinsic derivative of maps $f:M \to N$ between manifolds with corners and of maps $\rho:E \to F$ between vector bundles $E$, $F$ over a fixed manifold with corners, $M$.  We will need this notion when we study orientations on folding hypersurfaces in folded-symplectic manifolds.  The construction presented here is the same as the construction for manifolds (without corners) presented in \cite{GG}, p. 150.  The notion of an intrinsic derivative is due to Porteous \cite{P} and an interpretation in coordinates can be found in appendix C of \cite{Ho}.

Before we describe the intrinsic derivative of a map between manifolds, we first consider maps between vector bundles over the same base space.  Suppose we are given two vector bundles $E$ and $F$ over a manifold with corners $M$ and suppose $\operatorname{ranke}(E)\ge \operatorname{rank}(F)$.  If we have a smooth map $\rho: E\to F$ of vector bundles, then we may view $\rho$ as a smooth section:

\begin{displaymath}
\rho: M \hookrightarrow \hom(E,F)
\end{displaymath}
where we are abusing notation by referring to both the map and the section as $\rho$.  From now on, let us think of $\rho$ as a section of $\hom(E,F)\to M$.  Fix a point $p\in M$, let $K_p$ be the kernel $\ker(\rho_p)$, and let $L_p$ be the cokernel $\operatorname{coker}(\rho_p) := F_p/\rho_p(E_p)$.  Similar to the bundle $S_r$, we define a subfiber bundle of $\hom(E,F)$ whose fiber consists of the elements of corank $r$.

\begin{definition}\label{def:LrEF}
Let $E$ and $F$ be two vector bundles over a manifold with corners $M$.  We define the subset of $\hom(E,F)$ of elements of corank $r$:

\begin{displaymath}
L^r(E,F):= \{ A \in \hom(E_p,F_p) \mbox{ } \vert \mbox { } \operatorname{corank}(A)=r, \mbox{ } p\in M\}
\end{displaymath}
which is a subfiber bundle of $\hom(E,F)$ with typical fiber $L^r_{mn}$ by lemma \ref{lem:Lrmn}.
\end{definition}
If $r=\dim(L_p)$ is the corank of $\rho$ at the point $p\in M$, then we have that the section $\rho \in \Gamma(\hom(E,F))$ intersects $L^r(E,F)$ at $p$.  If we differentiate the section, we obtain a sequence of arrows:

\begin{displaymath}
\xymatrix{
T_pM \ar[r]^-{d\rho_p} & T_{\rho_p}(\hom(E_p,F_p)) \ar[r]^-q & \nu(L^r(E,F))_{\rho_p}:= T_{\rho_p}(\hom(E_p,F_p))/T_{\rho_p}L^r(E,F)
}
\end{displaymath}
where the map $q:T_{\rho_p}(\hom(E_p,F_p)) \to \nu(L^r(E_p,F_p))_{\rho_p}$ is just the projection.  Since $T_{\rho_p}\hom(E_p,F_p)$ is canonically identified with $\hom(E_p,F_p)$, there is a canonical map:

\begin{equation}\label{eq:canmap}
RP:T_{\rho_p}\hom(E_p,F_p) \to \hom(K_p,L_p)
\end{equation}
given by restricting a linear map in $T_{\rho_p}\hom(E_p,F_p)$ to the kernel of $\rho_p$, $K_p$, and projecting onto the cokernel of $\rho_p$, $L_p$.  As shown in lemma 3.2 in \cite{GG} (p.150), the kernel of this map is exactly $T_{\rho_p}L^r(E,F)_p$, hence it descends to a canonical isomorphism:

\begin{equation}\label{eq:canmap1}
RP: \nu(L^r(E,F))_{\rho_p} \to \hom(K_p,L_p)
\end{equation}
hence we may add another arrow to the above sequence to obtain a map:

\begin{displaymath}
\xymatrix{
T_pM \ar[r]^-{d\rho_p} & T_{\rho_p}(\hom(E_p,F_p)) \ar[r]^-q & \nu(L^r(E,F))_{\rho_p}\ar[r]^-{RP} & \hom(K_p,L_p).
}
\end{displaymath}
We define the intrinsic derivative $(D\rho)_p$ of $\rho$ at $p$ to be the composition of the above arrows, hence it is a map:

\begin{displaymath}
(D\rho)_p:T_pM \to \hom(K_p,L_p)
\end{displaymath}
Now, let us return to smooth maps between manifolds with corners $f:M \to N$ and let us assume that the dimensions of source and target spaces satisfy $\dim(M) \ge \dim(N)$.  Our discussion will be similar to the case of maps between vector bundles, but we'll need to be more careful since we won't have vector bundles over the same base space: $TM \to M$ and $TN \to N$ are vector bundles over different spaces, in general.

Differentiating $f$ gives us a map $df:TM \to TN$, which is a section of the fiber bundle $\hom(TM,TN) \to M$ and this is \emph{not} a vector bundle unless $N$ is a vector space: its fiber is $T_pM \times TN$.  Let $p\in M$ and let $r=\operatorname{corank}(df_p)$ be the corank of the differential at $p$.  Let $K_p = \ker(df_p)\subseteq T_pM$ be the kernel and let $L_p = T_{f(p)}N / df_p(T_pM)$ be the cokernel.  Note that, while the kernel is a subspace of the domain of $df_p$, the cokernel is not canonically identified with a subspace of $T_{f(p)}N$.  Using the above discussion of vector bundles with $E$ as $TM\to M\times N$, $F$ as $TN \to M \times N$, $\hom(E,F)$ as $\hom(TM,TN) \to M\times N$, and $L^r(E,F)$ as $L^r(TM,TN) \to M\times N$, we obtain a sequence of maps:

\begin{displaymath}
\xymatrix{
T_pM \ar[r]^-{(d(df))_p} & T_{df_p}(\hom(TM, TN)) \ar[r]^-{q} & \nu(L^r(TM,TN))_{\rho_p}\ar[r]^-{RP} & \hom(K_p,L_p)
}
\end{displaymath}
The composition give us a linear map:

\begin{displaymath}
F_p:T_pM \to \hom(\ker(df_p),\operatorname{coker}(df_p))
\end{displaymath}
and restriction to $\ker(df_p)$ gives us a quadratic map, which we call the intrinsic second derivative of $f$ at $p$, or just \emph{the intrinsic derivative of f at $p$}:

\begin{displaymath}
(Df)_p:\ker(df_p)\otimes \ker(df_p) \to \operatorname{coker}(df_p)
\end{displaymath}
given by $(Df)_p(\eta \otimes v) = F_p(\eta)(v)$.  The following examples are more of a guide for computing the intrinsic derivative in the case of maps between vector bundles and the case of maps between manifolds.

\begin{example}\label{ex:intrVB}
Let $E$ and $F$ be two vector bundles over a fixed manifold with corners $M$.  Let $\rho\in \Gamma(\hom(E,F))$ be a section of $\hom(E,F)$ and let $p\in M$.  Let's see how we can compute the intrinsic derivative at $p$.  Choose a basis $\{e_1,\dots,e_n\}$ for $E_p$ so that $e_i\in \ker(\rho_p)$ for $1\le i \le j$, where $j\le n$.  Choose a local frame of $F$ near $p$, $\{f_1,\dots,f_k\}$.

\begin{itemize}
\item Let $e_i \in \ker(\rho_p)$ and let $\tilde{e_i}$ be an extension to a local section of $E$ near $p$.
\item Then $\rho(\tilde{e}_i)= \sum_{l=1}^k a_lf_l$, where $a_l \in C^{\infty}(M)$ is smooth for each $l$.
\item Then, for each $X\in T_pM$, we can consider $\sum_{l=1}^k d(a_l)_p(X)f_l \in F_p$.
\item We can then send $\sum_{l=1}^kd(a_l)_p(X)f_l$ to its image in $F_p/\rho_p(E_p)$ and we claim that this gives us $D\rho_p(X)(e_i)$.
\item Indeed, if $\tilde{e}_i'$ is any other extension, then the difference $\tilde{e}_i-\tilde{e}_i'$ vanishes at $p$.  If $\{v_1,\dots,v_n\}$ is a local frame for $E$ near $p$, then we can write:
    \begin{displaymath}
    \tilde{e}_i-\tilde{e}_i'= \sum_{l=1}^n g_lv_l
    \end{displaymath}
    where the $g_l's$ vanish at $p$.  Then,
    \begin{displaymath}
    \rho_p(\tilde{e}_i-\tilde{e}_i')=\sum_{l=1}^n g_l\rho_p(v_l) = \sum_{l=1}^{n}\sum_{r=1}^k g_la_{lr}f_r
    \end{displaymath}
    where the $a_{lr}'s$ are smooth.  Then $d(g_la_{lr})_p = a_{lr}(p)d(g_l)_p$, since $g_l(p)=0$.  Thus, for any $X\in T_pM$,

    \begin{displaymath}
    \sum_{l=1}^{n}\sum_{r=1}^k d(g_l)_p(X)a_{lr}(p)f_r(p) = \sum_{l=1}^n d(g_l)_p(X)\rho_p(v_l)=\rho_p(\sum_{l=1}^n d(g_l)_p(X))
    \end{displaymath}
    hence if we apply the above recipe to any two extensions of $e_i$, the results differ by an element in the image of $\rho_p$, meaning we get the same element of the cokernel when we project.
\item If we choose a different local frame of $F$ near $p$, then the two frames are related by a local automorphism of $F$, which will induce an automorphism of $\operatorname{coker}(\rho_p)$ sending one coordinate representation of $D\rho_p(e_i)$ to the other, hence the map we get from $\ker(\rho_p)$ to $\operatorname{coker}(\rho_p)$ is coordinate independent.
\item Lastly, the reason that this is the intrinsic derivative is that we may extend the $e_i$'s to a local frame of $E$ and then $\rho$ becomes a matrix: all we are doing is differentiating the coefficients in this matrix, restricting to the kernel of $\rho_p$, and projecting to the cokernel of $\rho_p$, which is exactly the recipe of the intrinsic derivative.
\end{itemize}
\end{example}

\begin{example}\label{ex:intrFunc}
Given a smooth map $f:M \to N$ between manifolds with corners, we will show how one may compute the intrinsic derivative in coordinates.  In \cite{Ho}, H\"{o}rmander uses Taylor expansions in order to show the existence of the quadratic map defined above.  We offer an alternate approach in coordinates, but the reader is invited to peruse either.

Fix a point $p\in M$ and choose coordinates around $p$, $f(p)$ so that we may assume $f$ is a smooth map $f:\R^m \to \R^n$.  Technically, it is a smooth map defined on a quadrant since $M$ and $N$ have corners, but smoothness implies it extends to an open subset of $\R^m$ and $\R^n$ so we just extend it for the sake of simplicity.  In general, almost everything we are about to say only makes sense in coordinates, but the intrinsic derivative is a local construction so this is acceptable.  Now,

\begin{itemize}
\item choose a vector $v\in \ker(df_p)$,
\item extend it to a local vector field $\tilde{v}$ near $p$,
\item define the map $g:\R^m \to \R^n$ given by $g(x)=df_x(\tilde{v}(x))$, which has values in the fiber $\R^n$ of $T(\R^n)$,
\item and differentiate $g$ using $v$: $dg_p(v) = v(df(\tilde{v}))$.
\end{itemize}
Because $g$ takes values in the fiber, $\R^n$, of $T(\R^n)$, we may view it as a tangent vector to $\R^n$ at $f(p)$.  Any two extensions $\tilde{v}_1$ and $\tilde{v}_2$ of $v$ agree at $p$, hence the difference $\tilde{v}_1 - \tilde{v}_2$ vanishes at $p$.  Thus, if one differentiates $(df(\tilde{v}_1-\tilde{v}_2))$ in any direction at $p$, the result may be interpreted (in coordinates) as an element in the image of $df_p$.  Indeed, if $p$ is the origin in $\R^n$ and $\displaystyle X=\sum_{i=1}^n a_i\frac{\partial}{\partial x_i}$ is a vector field vanishing at the origin, then,

\begin{displaymath}
\frac{\partial}{\partial x_j}\big\vert_0 df(\sum_{i=1}^n a_i \frac{\partial}{\partial x_i}) = \sum_{i=1}^n\frac{\partial a_i}{\partial x_j} df_0(\frac{\partial}{\partial x_i})
\end{displaymath}
since the $a_i's$ vanish at $0$.  Thus, the difference $vdf(\tilde{v}_1) -vdf(\tilde{v}_2)$ is an element of $df_p(T_pM)$, hence the vector $v(df(\tilde{v}))$ is well defined as an element of $T_p\R^n/df_p(T_p\R^m)$.  That is, we have a well-defined element of $T_pN/df_p(T_M)$ and the above construction gives us a map from $T_pM$ into $\hom(\ker(df_p),\operatorname{coker}(df_p)$.  As in the vector bundle example, this is the intrinsic derivative since we have simply differentiated how the derivative of $f$ acts on elements of $T_pM$ and then restricted to $\ker(df_p)$ followed by projection to $\operatorname{coker(df_p)}$.
\end{example}

\subsection{Sections With Fold Singularities}
We now generalize the definition of a submersion with folds, definitions 4.1 and 4.2, found in \cite{GG} to arbitrary fiber bundles.
\subsubsection{Definition of a Section with Fold Singularities}
\begin{definition}\label{def:folds}
Let $\pi:E \to M$ be a fiber bundle over a manifold with corners $M$.  Let $\chi$ be a connection on $E$, let $p_V:TE \to V$ be the induced projection to the vertical bundle of $E$, and let $S_r$ be the submanifold of $J^1(E)$ of jets of corank $r$.  We say a local section $(U,\phi)$ has a fold singularity at $p\in U$ if:

\begin{enumerate}
\item $j^1\phi_p \in S_1$ (the derivative drops rank by $1$, where rank is determined by the connection),
\item $j^1\phi \pitchfork_s S_1$ at $p$, meaning $Z= j^1\phi^{-1}(S_1)$ is a submanifold with corners near $p$, and
\item $\ker(p_V(d\phi_p))\pitchfork T_p(Z)$.
\end{enumerate}
If for each $p\in U$ we have $j^1\phi_p\in S_0$ or $j^1\phi_p\in S_1$ and the above three conditions are satisfied at such points, then we say that $(U,\phi)$ is a section with fold singularities.  We call $Z=(j^1\phi)^{-1}(S_1)$ the \emph{fold locus} of $\phi$.  It has codimension $(\dim(M)-q + 1)(\dim(F) - q + 1)$, where $q=\min(\dim(M),\dim(F))$, by proposition \ref{prop:Sr} and lemma \ref{lem:Lrmn}.
\end{definition}

\begin{remark}
Let $F$ bet the typical fiber of a fiber bundle $\pi:E \to M$ and let $(U,\phi)$ be a section with fold singularities.  Then, we necessarily have that $\dim(M) \ge \dim(F)$.  To see this, recall that $\chi$ gives us an isomorphism of fiber bundles:

\begin{displaymath}
\xymatrixcolsep{4pc}\xymatrixrowsep{.5pc}\xymatrix{
F_{\chi}: J^1(E) \ar[r] & \hom(\pi^*TM, V) \\
(j^1_p\phi) \ar[r]        &  (p,\phi(p),p_V\circ d\phi_p)
}
\end{displaymath}
If $\dim(F) > \dim(M)$, then $\operatorname{rank}(V) > \operatorname{rank}(\pi^*TM)$, hence any element of the fiber of $\hom(\pi^*TM, V)$ would necessarily have corank $\ge 1$.  This means that if $(U,\phi)$ is a local section then $\operatorname{corank}(j^1_p\phi)\ge 1$ for all $p\in U$, meaning $(U,\phi)$ either does not satisfy condition $1$ or it satisfies condition $1$ and violates condition $2$ since the image of $j^1\phi$ would be contained in $S_1$.

We therefore have that the fold locus $(j^1\phi)^{-1}(S_1)$ has codimension $(\dim(M)-\dim(F)+1)$ by proposition \ref{prop:Sr}.
\end{remark}

\begin{definition}\label{def:folds1}
Let $f:M \to N$ be a smooth map of manifolds with corners.  We will say that $f:M \to N$ is a submersion with folds if the section $\phi(m)=(m,f(m))$ of $M\times N$ is a section with fold singularities, where we equip $M\times N$ with the standard connection.  The fold locus, $Z=(j^1\phi)^{-1}(S_1)$ is a smooth submanifold with corners of $M$ of codimension $(\dim(M)-\dim(N) + 1)$.
\end{definition}

\begin{lemma}\label{lem:folds0}
Let $\pi:E \to M$ be a fiber bundle over a manifold with corners with connection $\chi$.  Let $(U,\phi)$ be a section with fold singularities.  Suppose $\pi':E' \to M$ is another fiber bundle with connection $\chi'$ and $\tau:E\to E'$ is a connection-preserving isomorphism of fiber bundles.  Then $(U,\tau\circ \phi)$ is a section with fold singularities.
\end{lemma}
\begin{proof} \mbox{ } \newline
Let $V$ and $V'$ be the vertical bundles of $E$ and $E'$ respectively.  The map $\tau$ induces an isomorphism of fiber bundles (over $M$):

\begin{displaymath}
\tilde{\tau}: \hom(\pi^*TM, V) \to \hom(\pi'^*TM, V')
\end{displaymath}
and by proposition \ref{prop:C2} we have a series of isomorphisms of fiber bundles (over $M$) :

\begin{displaymath}
\xymatrixcolsep{4pc}\xymatrix{
J^1(E) \ar[r]^{F_{\chi}} & \hom(\pi^*TM,V) \ar[r]^{\tilde{\tau}} & \hom(\pi'^*TM, V') \ar[r]^{F_{\chi'}^{-1}} & J^1(E')
}
\end{displaymath}
which sends $S_r\subset J^1(E)$ to $S_r'\subset J^1(E')$ by lemma \ref{lem:rank}.  Therefore, $\phi$ is a section with fold singularities if and only if $\tau(\phi)$ is a section with fold singularities.
\end{proof}

\begin{cor}\label{cor:folds-diffeo}
Let $M$ and $N$ be smooth manifolds with corners and let $f:M\to N$ be a submersion with folds.  Let $\tau:N \to P$ be a diffeomorphism of manifolds with corners.  Then $\tau \circ f$ is a submersion with folds.
\end{cor}

\begin{proof} \mbox{ } \newline
Apply lemma \ref{lem:folds0} with $E=M\times N$, $E'=M\times N_1$, and with $\chi, \chi'$ as the standard connections.  $f$ is a submersion with folds if and only if the section $\phi(m)=(m,f(m))$ is a section with folds if and only if $\tau(\phi)(m)=(m,\tau(f(m)))$ is a section with folds if and only if $\tau(f(m))$ is a submersion with folds.
\end{proof}
\subsubsection{Equidimensional Fold Maps and Computations}

\begin{definition}\label{def:folds2}
Let $\pi:E \to M$ be a fiber bundle over a manifold with corners with connection $\chi$ and typical fiber $F$ satisfying $\dim(F)=\dim(M)$.  We say a local section $(U,\phi)$ is an \emph{equidimensional section with fold singularities} if $(U,\phi)$ is a section with fold singularities (q.v. definition \ref{def:folds1}).  When the dimensions of $F$ and $M$ are understood to be equal, we simply refer to $(U,\phi)$ as a section with fold singularities.  We call $Z=(j^1\phi)^{-1}(S_1)$ the \emph{folding hypersurface} of $\phi$ since it has codimension $1$.
\end{definition}

\begin{definition}\label{def:folds3}
Let $M$ and $N$ be two $m$-dimensional manifolds with corners and let $E=M\times N$ be the trivial fiber bundle over $M$ with standard flat connection, $H=TM\oplus 0$.  We say $f:M \to N$ is an \emph{equidimensional map with fold singularities} if the section $\phi(m)=(m,f(m))$ is a section with fold singularities.  When the dimensions of $M$ and $N$ are understood to be equal, we simply refer to $f$ as a \emph{map with fold singularities}.  We call $Z=(j^1\phi)^{-1}(S_1)$ the \emph{folding hypersurface} of $f$ since it is a codimension $1$ submanifold with corners of $M$.
\end{definition}

\begin{remark}
Recall that if $A:V_1 \to V_2$ is a linear map of vector spaces then $A^n$ denotes the induced map between the $n^{th}$ exterior powers: $A^n:\Lambda^n(V_1) \to \Lambda^n(V_2)$.  Similarly, if $A:V_1 \to V_2$ is a map of vector bundles then we may write $A^n$ for the induced map $A:\Lambda^n(V_1) \to \Lambda^n(V_2)$ between maps of vector bundles.
\end{remark}

We have an important computational tool allowing us to detect when sections have fold singularities.

\begin{prop}\label{prop:folds4}
Let $\pi:E \to M$ be a fiber bundle over a manifolds with corners $M$ with connection $\chi$, induced projection $p_V:TE \to V$ onto the vertical bundle, and typical fiber $F$ satisfying $\dim(F)=\dim(M)=n$.  Let $\mathcal{O}$ be the zero section of $\hom(\Lambda^n(\pi^*TM), \Lambda^n(V))$ (a vector bundle over $E$).  Then $(U,\phi)$ satisfies
\begin{enumerate}
\item $\operatorname{corank}(j^1_\phi)>0$ $\iff$ $j^1_p\phi \in S_1$ and
\item $j^1\phi \pitchfork_s S_1$
\end{enumerate}
if and only if $(p_V(d\phi))^n:M \to \hom(\Lambda^n(\pi^*TM),\Lambda^n(V))$ satisfies $(p_V(d\phi))^n\pitchfork_s \mathcal{O}$.
\end{prop}

The proposition is a direct consequence of the following lemma.

\begin{lemma}\label{lem:folds1}
Let $M$ be a manifold with corners and let $\psi:M\to \hom(\R^n,\R^n)$ be a smooth map.  Let $L^r:=L^r_{nn} \subset \hom(\R^n,\R^n)$ be the subset of matrices of corank $r$ (rank $n-r$).  Then $\det(\psi):M\to \R$ is (strongly) transverse to $0$ if and only if the two conditions:

\begin{enumerate}
\item $\operatorname{corank}(\psi(p))>0$ if and only if $\psi(p)\in L^1$ and
\item $\psi \pitchfork_s L^1$.
\end{enumerate}
are satisfied
\end{lemma}

\begin{proof} \mbox{ } \newline
\begin{enumerate}
\item Assume that $\operatorname{corank}(\psi(p))>0$ if and only if $\psi(p)\in L^1$ and $\psi \pitchfork_s L^1$.  We show that $\det(\psi)\pitchfork_s 0$.
    \begin{itemize}
    \item First, note that at any point $A$ of $L^1$ $\det: \hom(\R^n, \R^n) \to \R$ is a submersion.  To see this, write $A=\begin{bmatrix} c_1 & \dots & c_n \end{bmatrix}$ and, without loss of generality, assume $c_n$ is in the span of the first $n-1$ columns.  Pick any vector $v$ not in the span of the first $n-1$ rows and consider $\gamma(t)=\begin{bmatrix} c_1& \dots &c_{n-1} & c_n+tv\end{bmatrix}$.  Then:
        \begin{displaymath}
        \frac{d}{dt} \det(\gamma(t)) = \frac{d}{dt} \det(c_1,\dots,c_{n-1},c_n+tv) = \frac{d}{dt}(\det(A) + t\det(c_1,\dots, c_{n-1},v)) = \det(c_1,\dots,c_{n-1},v) \ne 0
        \end{displaymath}
        which shows $\det$ is a submersion at $A$.
    \item Note that the restriction $\det \vert_{L^1}$ is identically zero since all elements in this subset have rank $n-1$, hence its derivative restricted to vectors tangent to $L^1$ vanishes.  We also note that $L^1$ is a hypersurface in $\hom(\R^n,\R^n)$ by lemma \ref{lem:Lrmn}.  Therefore, if $A \in L^1$ and $v\in T_A\hom(\R^n,\R^n)$ is transverse to $T_AL^1$, we must have that $d(\det)_A(v) \ne 0$ since it is surjective at $A$ by the first bullet.  Otherwise, the differential would vanish at $A\in L^1$ since it would vanish along directions tangent to a hypersurface and a direction transverse to that hypersurface.
    \item Let $p\in M$ be a point for which $\psi(p)\in L^1$.  $\psi \pitchfork_s L^1$ by assumption, so there exists a vector $v_0 \in T_pM$, tangent to the stratum containing $p$, so that $d\psi_p(v_0) \pitchfork T_{\psi(p)}L^1$.  Then $d(\det\circ \psi)p(v_0)\ne 0$ by the previous bullet.
    \item Since this is true for any $p\in M$ with $\psi(p)\in L^1$ and $\psi(p)\in L^r \iff \mbox{ }r=0,1$, we have $\det(\psi)\pitchfork_s 0$.
    \end{itemize}
\item Conversely, assume that $\det(\psi)\pitchfork_s 0$.
    \begin{itemize}
    \item We first show that $\psi$ cannot intersect the subset of matrices of corank $r$, $L^r$, unless $r=1$ or $r=0$.  Assume that $\psi(p)\in L^r$ and $r>1$.  Since $\det(\psi)\pitchfork_s 0$, $\det(\psi)^{-1}(0)$ is a smooth submanifold with corners of $M$ transverse to boundary strata and there are coordinates $(x_1,\dots,x_{m-1},t)$ around $p$ with $\det(\psi)^{-1}(0)$ identified with the zero set $t=0$ and $p$ identified with the origin.
    \item We write $\psi(\vec{x},t)$ in column form as $\psi(\vec{x},t) = \begin{bmatrix} c_1&\dots& c_{n-2}&c_{n-1}&c_n \end{bmatrix}$ for some smooth maps $c_i:M \to \R^n$.
    \item We are assuming $\psi(0)\in L^r$ with $r>1$ so, without loss of generality, $c_{n-1}(0)$ and $c_n(0)$ are in the span of the first $n-2$ columns at $0$.  That is, there are constants $a_i$ so that:
        \begin{displaymath}
        c_n(0)=\sum_{i=1}^{n-2} a_i c_i(0)
        \end{displaymath}
    \item Consider the curve $\gamma(t) = (0,t)$ passing through $p$.  Since $c_n(\gamma(t))-\sum_{i=1}^{n-2} a_i c_i(\gamma(t))$ vanishes at $t=0$, we have that $c_n(\gamma(t))= tG(\vec{x},t) + \sum_{i=1}^{n-2} a_i c_i(\gamma(t))$ for some smooth map $G(\vec{x},t)$.

    \item By assumption, $\det(\psi) \pitchfork_s 0$.  Since $Z=\det(\psi)^{-1}(0)$ is a hypersurface, this implies that for any direction $v$ transverse to $Z$ at $p$, $v(\det(\psi)) \ne 0$.  Therefore, $\frac{d}{dt}\det(\psi(\gamma(t)))$ should be nonzero at $t=0$.
    \item However, we compute:
    \begin{displaymath}
    \begin{array}{lcl}
    \det(\psi(\gamma(t)) &=& \det( c_1 , \dots , c_{n-2} , c_{n-1} , tG(\vec{x},t) + \sum_{i=1}^{n-2} a_i c_i(\gamma(t))) \\
                         &=& \det(c_1, \dots, c_{n-2}, c_{n-1}, tG(\vec{x},t)) \\
                         &=& t\det(c_1, \dots, c_{n-2}, c_{n-1}, G(\vec{x},t)) \\
    \end{array}
    \end{displaymath}
    hence $\frac{d}{dt} \det(\psi(\gamma(t)))\big\vert_{t=0} = \det(c_1(0),\dots, c_{n-2}(0), c_{n-1}(0),G(0)) =0$ since $c_{n-1}(0)$ is in the span of the first $n-2$ columns by assumption.  This contradicts the assumption that $\det(\psi)\pitchfork_s 0$, hence we must have $r\le 1$ and $\psi$ may only intersect $L^1$ if $r>0$.  That is $\psi(p) \in L^r \mbox{ } \iff \mbox{ } r=1,0$.
    \item Now, since $L^1$ is a hypersurface (q.v. lemma \ref{lem:Lrmn}), we can deduce that $\psi \pitchfork_s L^1$ fairly easily.  If $p\in \det(\psi)^{-1}(0)$ and $d\psi_p \not\pitchfork T_{\psi(p)}L^1$, then we must have that $d\psi_p(T_pM)\subset T_{\psi(p)}L^1$, since $L^1$ is a hypersurface by lemma \ref{lem:Lrmn}.  But then $d(\det\circ \psi)_p =0$ since the derivative of the determinant vanishes along directions tangent to $L^1$, meaning $\det(\psi) \not\pitchfork_s 0$ at $p$, contradicting our assumption.
    \item Therefore, $\psi$ intersects $L^r$ if and only if $r\le 1$ and $\psi \pitchfork_s L^1$.
    \end{itemize}
\end{enumerate}
\end{proof}

\begin{proof}[proof of proposition \ref{prop:folds4}]
Lemma \ref{lem:folds1} gives us a straightforward proof of proposition \ref{prop:folds4}.  Recall that in the setting of proposition \ref{prop:folds4} we have a fiber bundle $\pi:E \to M$ with connection $\chi$, which gives us a projection $p_V:TE \to V$ onto the vertical bundle, and we have $S_1\subset J^1(E)$, the subset of jets of corank $1$.  Our goal is to show that $j^1\phi \pitchfork_s S_1$ (and only intersects $S_1, S_0$) if and only if $(p_V(d\phi))^n\pitchfork_s \mathcal{O}$.  By proposition \ref{prop:C2}, the connection $\chi$ induces an isomorphism of fiber bundles:

\begin{displaymath}
F_{\chi}: J^1(E) \to \hom(\pi^*TM, V)
\end{displaymath}
Let $L_r$ denote the submanifold of $\hom(p^*TM, V)$ of maps of corank $r$.  As a note in what follows, $L_r$ is always the subfiber bundle of $\hom(p^*TM,V)$ of elements of corank $r$ while $L^r$ is the subset of $\hom(\R^n, \R^n)$ of elements of corank $r$.  The relationship is that $L^r$ is the typical fiber of $L_r$.  We have:
\begin{displaymath}
\begin{array}{lc}
\text{$j^1\phi$ intersects $S_r$ only when $r=1,0$ and $j^1\phi \pitchfork_s S_1$} & \iff \\
\text{$F_{\chi}(j^1\phi)=p_V(d\phi)$ intersects $L_r$ only when $r=1,0$ and $p_V(d\phi)\pitchfork_s L_1$} &  \iff \\
\text{The previous line is true in any trivialization of $\hom(\pi^*TM, V)$}                        &\iff \\
\text{it is true for $p_V(d\phi):U \to E\vert_U \times \hom(\R^n \times \R^n)$, $p_V(d\phi)(u)=(\phi(u),A(u))$  }  &\iff \\
\text{$A:U \to \hom(\R^n, \R^n)$ only intersects $L^1$ and $A\pitchfork_s L^1$ } &\iff \\
\text{$\det(A):U \to \R$ satisfies $\det(A)\pitchfork_s 0$ by lemma \ref{lem:folds1}}  &\iff \\
\text{$A^n:\Lambda^n(\R^n) \to \Lambda^n(\R^n)$ is transverse to the zero section $\mathcal{O}\in \Gamma(\hom(\Lambda^n(\R^n),\Lambda^n(\R^n)))$} & \iff \\
p_V(d\phi)^n \pitchfork_s \mathcal{O} \in  \Gamma(\hom(\Lambda^n(\pi^*TM),\Lambda^n(V)))
\end{array}
\end{displaymath}
where the first $5$ statements are a local restatement of what it means to only intersect $S_1$ and $S_0$ and to intersect $S_1$ transversally.  The last line is a global restatement of the previous line.
\end{proof}

\begin{cor}\label{cor:folds4-0}
Let $f:M \to N$ be a smooth map between two $m$-dimensional manifolds with corners.  Then $f$ is a map with fold singularities if and only if
\begin{enumerate}
\item The induced map $(df)^m: M \to \hom(\Lambda^m(TM), \Lambda^m(TN))$ is transverse (in the sense of manifolds with corners) to the zero section $\mathcal{O}$ of $\hom(\Lambda^m(TM),\Lambda^m(TN)) \to M \times N$, and
\item $\ker(df) \pitchfork ((df)^m)^{-1}(\mathcal{O})$.
\end{enumerate}
\end{cor}

\begin{proof} \mbox{ } \newline
This is just a restatement of the proposition in the case that $E$ is the trivial fiber bundle and $\chi$ is the standard connection.
\end{proof}

\begin{remark}
While corollary \ref{cor:folds4-0} follows easily from proposition \ref{prop:folds4}, its utility should not be overlooked.  What do we need to check to determine if a map of manifolds with corners $f:M^n \to N^n$ has fold singularities?

\begin{enumerate}
\item Compute the determinant $\det(df)$ in local coordinates and check that it vanishes transversally.  Note that we are dealing with manifolds with corners, so this means that we must check the condition $\det(df) \pitchfork_s 0$ on each stratum of $M$.
\item In local coordinates, check that $\ker(df_p)\pitchfork \det(df)^{-1}(0)$ for all $p\in \det(df)^{-1}(0)$.
\end{enumerate}
This makes the problem of understanding fold maps $f:M^n \to N^n$ very simple from a computational point of view.
\end{remark}

\begin{cor}\label{cor:folds4-2}
Let $f:M\to N$ be a smooth map with fold singularities between two $m$-dimensional manifolds with corners with folding hypersurface $Z\subset M$.  Suppose $\phi:M_1 \to M$ is a diffeomorphism of manifolds with corners.  Then $f\circ\phi:M_1 \to N$ is a (equidimensional) map with fold singularities with folding hypersurface $\phi^{-1}(Z)$.
\end{cor}

\begin{proof} \mbox{ } \newline
This statement follows from the definitions, but we list it here as a corollary to proposition \ref{prop:folds4} since the criteria for detecting fold singularities are now easier to describe.

\vspace{2mm}
We have that $(df)^m: M \to \hom(\Lambda^m(TM), \Lambda^m(TN))$ satisfies $(df)^m\pitchfork_s \mathcal{O}$, where $\mathcal{O}$ is the zero section, if and only if $(df\circ d\phi)^m : M_1 \to \hom(\Lambda^m(TM_1), \Lambda^m(TN))$ is transverse to the zero section since $\phi$ is a diffeomorphism, hence $(df\circ d\phi)^m$ is transverse the zero section and the degenerate hypersurface of $df\circ d\phi$ is $\phi^{-1}(Z)$.  $\ker(df\circ d\phi)=d\phi^{-1}(\ker(df))$ is transverse to $\phi^{-1}(Z)$ since $\ker(df)\pitchfork Z$ and $\phi$ is a diffeomorphism of manifolds with corners.  Thus, $f\circ \phi$ is a map with fold singularities.
\end{proof}

\begin{cor}\label{cor:folds4-1}
Let $f:M \to N$ be an equidimensional map with fold singularities.  Let $S\subset M$ be an $s$-dimensional submanifold with corners of $M$ and suppose $f(S) \subset R \subset N$, where $R$ is an $s$-dimensional submanifold with corners of $N$.  Then $f\vert_{S}:S \to R$ is an equidimensional map with fold singularities.

\vspace{2mm}

Furthermore, at points $p\in S$ where $\ker(df_p)\subset T_pS$, $df_p$ induces an isomorphism on the fibers of the normal bundles, $\nu(S)$ and $\nu(R)$, or $S$ and $R$.  That is:

\begin{displaymath}
df_p:\nu(S)_p \to \nu(R)_{f(p)}
\end{displaymath}
is an isomorphism.
\end{cor}

\begin{proof}\mbox{ } \newline
In light of proposition \ref{prop:folds4}, we need only check that the determinant of $d(f\vert_S)$ vanishes transversally in coordinates and the kernel of $d(f\vert_S)$ is transverse to the degenerate hypersurface.  By the definition of a submanifold with corners, we may choose coordinates $(x_1,\dots, x_m)$ near $S$ so that $S=\{x_1=0, \dots,x_{m-s}=0\}$, where some of these coordinates may be defined on half spaces $\R^+$, but it will be of no consequence.  Since $\dim(R)$ is also $s$, we may choose coordinates $(y_1,\dots,y_m)$ near $R$ so that $R=\{y_1=0,\dots,y_{m-s}=0\}$.  We now have three cases to consider:

\begin{enumerate}
\item If $Z\cap S =\emptyset$, then $d(f\vert_S)$ never drops rank and so it is trivially a map with fold singularities.
\item If $p\in Z\cap S$ but the one-dimensional subspace $\ker(df_p)$ is not contained in $T_pS$, then $d(f\vert_{S_1})$ has maximal rank at $p$.
\item If $p\in Z\cap S$ and $\ker(df_p)\subset T_pS$, then we may choose a curve $\gamma(t)$ in $S$ so that $\gamma(0)=p$ and $\gamma'(0)\ne 0$ is in $\ker(d\psi_p)$.  Along this curve, the differential $df$ has the form:

    \begin{equation}\label{eq:restriction}
    df_{\gamma(t)} =
    \begin{bmatrix}
    A(\gamma(t)) & 0 \\
    B(\gamma(t)) & d(f\vert_S)_{\gamma(t)}
    \end{bmatrix}
    \end{equation}
    Since $\gamma'(0)$ is transverse to the folding hypersurface $Z$ and $f$ has fold singularities, we must have that $\frac{d}{dt}\big\vert_0(\det(df_{\gamma(t)}))\ne 0$ by proposition \ref{prop:folds4}.  Using the form of $df_{\gamma(t)}$ in equation \ref{eq:restriction}, we find that:

    \begin{equation}\label{eq:restriction2}
    \begin{array}{lcl}
    \frac{d}{dt}\big\vert_0 \det(df_{\gamma(t)}) & = & \frac{d}{dt}\big\vert_0 \det(A(\gamma(t)))\det(d(f\vert_S)_{\gamma(t)}) \\
                                                 & = & \det(d(f\vert_S)_{\gamma(0)})\frac{d}{dt}\big\vert_0 \det(A(\gamma(t))) + \det(A(\gamma(0)))\frac{d}{dt}\big\vert_0\det(d(f\vert_S)_{\gamma(t)}) \\
                                                 & = & \det(A(p))\frac{d}{dt}\big\vert_0\det(d(f\vert_S)_{\gamma(t)}) \ne 0
    \end{array}
    \end{equation}
    Thus, neither $\det(A(p))$ or $\frac{d}{dt}\big\vert_0\det(d(f\vert_S)_{\gamma(t)})$ may be $0$.  In particular, $\det(d(f\vert_S))$ vanishes transversally.  The computation also reveals that the degenerate hypersurface of $f\vert_S$ is exactly $S\cap Z$ near $p$.  Since $\ker(df)\pitchfork TZ$, we have that $\ker(d(f\vert_S)_p)\pitchfork (T_p(Z\cap S))$, hence $f\vert_S$ is an equidimensional map with fold singularities.

    \vspace{2mm}
    To prove that $df_p$ induces an isomorphism on the fibers of the normal bundles, we appeal to equation \ref{eq:restriction} and our computation in \ref{eq:restriction2}.  In our chosen coordinates, we use the standard metric on $\R^n$ to see that $A(p)$ maps $(T_pS)^{\perp}$ to $(T_{f(p)}R)^{\perp}$.  Since $\det(A(p))\ne 0$ by \ref{eq:restriction2}, we have that $A(p)$ has maximal rank.  Since $\dim(S)=\dim(R)$ and $\dim(M)=\dim(N)$, the normal bundle $\nu(S)$ and $\nu(R)$ have the same rank.  Thus, $A(p)$ is an isomorphism.
\end{enumerate}
\end{proof}

\begin{example}
Here is an example of corollaries \ref{cor:folds4-0} and \ref{cor:folds4-1} in action.  Consider the map $f:\R^3 \to \R^3$ given by $f(x,y,z)=(x,y,z^2)$.  $\det(df)= 2z$ vanishes transversally at $z=0$ and $\ker(df_{z=0}) = \frac{\partial}{\partial z}$ is transverse to $z=0$, hence $f$ has fold singularities. Let $S$ be the paraboloid $y-z^2-x^2=0$ and let $R$ be the parabolic sheet $y-z-x^2=0$.  $f(S) \subset R$ and so $f\vert_S:S \to R$ is a map with fold singularities, which is straightforward to compute directly:

\vspace{3mm}
\noindent Both surfaces are the graphs of functions $y(x,z)$, so we may identify them with the $xz$-plane by projecting.  Under this identification, $f\vert_S(x,z)=(x,z^2)$ is a map with fold singularities with folding hypersurface given by $z=0$.  Thus, $f\vert_S$ has fold singularities along the parabola $y=x^2$, $z=0$ with kernel spanned by $\frac{\partial}{\partial z}$.
\end{example}

\begin{prop}\label{prop:folds5}
Let $f:M^n \to N^n$ be an equidimensional map with fold singularities, suppose $f$ is strata-preserving, and suppose for all $p\in Z$, the folding hypersurface, $\ker(df_p)$ is tangent to the stratum of $M$ containing $p$.  Then, for each $p\in Z$ there exist coordinates near $p$ and $f(p)$ so that:
\begin{displaymath}
f(x_1,\dots,x_{n-1}, t) = (x_1, \dots, x_{n-1}, 0 ) + t^2F(x_1,\dots,x_{n-1})
\end{displaymath}
where $t\in \mathbb{R}$ and $F(\vec{x})$ is a smooth map such that $\phi(\vec{x},t) := (\vec{x},0) + tF(\vec{x})$ is a strata-preserving diffeomorphism in a neighborhood of $0$. Hence, locally, every map $f$ satisfying the conditions of the proposition factors as a diffeomorphism composed with a strata-preserving fold map.
\end{prop}

\begin{proof} \mbox{ } \newline
First, we show that one may assume $M=N=Z\times \R$ and $f:Z\times \R \to Z\times \R$ is a map that folds along $Z\times \{0\}$ with kernel $\frac{\partial}{\partial t}$ and $f(z,0)=(z,0)$ for all $z\in Z$.
\begin{itemize}
\item Let $p\in Z$.  Because $\ker(df_x)$ is tangent to strata for all $x\in Z$, we may choose a local section of the bundle $\ker(df) \to Z$ near $p$ and extend it to a stratified vector field in a neighborhood of $p$.  Its flow then induces coordinates near $p$ of the form $(z,t)$, where $t\in \R$ and $z\in Z$, so we may assume that $M=Z \times \R$ where the folding hypersurface is $Z\times \{0\}$ and $\ker(df_{(z,0)})=\span\{\frac{\partial}{\partial t}\}$.  Furthermore, we may choose coordinates on $N$ near $f(p)$ and so we may assume $N=\R^n$.
\item To summarize, we have a map $f:Z\times \R \to \R^n$, $\dim(Z)=n-1$, with fold singularities at $Z\times \{0\}$ and $\ker(df)$ is spanned by $\frac{\partial}{\partial t}$.
\item Since $f(z,t)-f(z,0)$ vanishes at $t=0$ we can write $f(z,t)= f(z,0)+ tG(z,t)$ for some smooth map $G(z,t)$.  Since $df_{(z,0)}(\frac{\partial}{\partial t})=0$, we have that $G(z,0)=0$ for all $z\in Z$, hence $G(z,t)=tF(z,t)$ for some smooth map $F(z,t)$.  This gives us the general formula $f(z,t)=f(z,0) + t^2F(z,t)$. Note that since $f$ is strata-preserving, $f\big\vert_{t=0}$ is strata-preserving.  Since the $t$ coordinate is defined on $\R$ and not a half-space $\R^+$, this implies that if $k$ half-space coordinates of $f(z,t)$ vanish, the same $k$ coordinates of $f(z,0)$ must vanish, which means that the same $k$ coordinates of $t^2F(z,t)$ must vanish, meaning that $F(z,t)$ is a vector tangent to the stratum containing $f(z,t)$, at $f(z,t)$, for all $t\ne 0$.  Smoothness of $F(z,t)$ then implies that we may also interpret $F(z,0)$ as a vector tangent to the stratum containing $f(z,0)$, hence $F(z,t)$ may be interpreted as a vector tangent to the stratum containing $f(z,t)$ at the point $f(z,t)$.

\item Choose a local frame for $Z$ near $p$, $\{e_1,\dots, e_{n-1}\}$ and extend it to a local frame on $Z\times \R$ near $(p,0)$, $\{e_1, \dots, e_{n-1}, \frac{\partial}{\partial t}\}$.  Then $df$, near $(p,0)$ may be written:

\begin{displaymath}
df=
\begin{bmatrix}
df(e_1) & \dots & df(e_{n-1}) & 2tF + t^2dF(\frac{\partial}{\partial t}) \\
\end{bmatrix}
\end{displaymath}
Since $\det(df)\pitchfork_s 0$ by corollary \ref{cor:folds4-0}, we have:
\begin{displaymath}
\frac{\partial}{\partial t}\big\vert_{t=0}(\det(df)) = \det(df(e_1),\dots, df(e_{n-1}), F) \ne 0
\end{displaymath}
Which means that at points of $Z\times \{0\}$, the differential of the mapping $\phi(z,t)= f(z,0) + tF(z,0)$ has maximum rank.  We claim that $\phi$ is a stratified map in a neighborhood of any point $p=(z_0,0)$.  By corollary \ref{cor:folds4-1}, $df_p$ maps directions transverse to the stratum containing $p$ isomorphically onto the normal bundle of the stratum containing $f(p)$.  Since $f\big\vert_{t=0}$ is an immersion, this implies that $f\big\vert_{t=0}$ locally embeds $Z\times \{0\}$ as a codimension $1$ submanifold with corners of $N$ transverse to strata.  The vector $F(z,0)$ is tangent to the stratum containing $f(z,0)$ and is transverse to the embedded image of $Z\times \{0\}$.  Thus, there is a neighborhood of $p$ on which $\phi$ is stratified since is a map that embeds a hypersurface transverse to strata and sends a path $(z,t)$ to a path through $f(z,0)$ along the stratum containing $f(z,0)$, transverse to the embedded hypersurface.  Since $\phi$ is stratified and $d\phi\big\vert_{t=0}$ has maximal rank, there is a neighborhood of $p=(z_0,0)$ on which $\phi$ is a diffeomorphism of manifolds with corners.  Notice that $\phi(z,0)=f(z,0)$ for all $z\in Z$, hence $\phi^{-1}(f(z,0))=(z,0)$ and, in particular, $\phi^{-1}(f(z_0,0))=(z_0,0)$.

\item Thus, if we postcompose with $\phi^{-1}$ then we have a map $\phi^{-1} \circ f : V_1\subset \to V_2$, where $V_i\subset Z\times \R$ is an open neighborhood of $(z_0,0)$ for each $i$.  Consequently, we may assume that we have a map $f:Z\times \R \to Z \times \R$ that folds along $Z\times \{0\}$ with kernel $\frac{\partial}{\partial t}$ and $f(z,0)=(z,0)$ for all $z\in Z$.
\end{itemize}
To finish the proof, we may choose coordinates $\vec{x}=(x_1,\dots x_{n-1})$ near $p$ that identify $p$ with the origin.  We then have a map $f:\R^{n-1}\times \R \to \R^{n-1}\times \R$ satisfying $f(\vec{x},0)= (\vec{x},0)$.  Using the same tricks in the first part of the proof, we may write $f(\vec{x},t)= (\vec{x},0) + t^2F(\vec{x},t)$ where $\phi(\vec{x},t)= (\vec{x},0) + tF(\vec{x},0)$ is a diffeomorphism near $0$.  Let $\gamma(\vec{x},t)=(\vec{x},t^2)$ and note that it folds along the set $\{t=0\}$.  Furthermore, the map $\psi(\vec{x},t) = (\vec{x},0) + t^2F(\vec{x},0)$ is the composition $\psi= \phi \circ \gamma$.  The maps:

\begin{displaymath}
\begin{array}{ll}
\psi^+ = & \phi \circ \gamma\big\vert_{t\ge 0} \\
\psi^- = & \phi \circ \gamma\big\vert_{t\le 0}
\end{array}
\end{displaymath}
are homeomorphisms near $0$ since $\phi$ is a homeomorphism near $0$ and $\gamma(\vec{x},t)=(\vec{x},t^2)$ restricted to $t\ge 0$ or $t\le 0$ is a homeomorphism.  We now define a continuous map:

\begin{displaymath}
\Phi(\vec{x},t) =
\begin{cases}
((\psi^+)^{-1}\circ f)(\vec{x},t) & \mbox{if } t\ge 0 \\
((\psi^-)^{-1}\circ f)(\vec{x},t) & \mbox{if } t\le 0
\end{cases}
\end{displaymath}
Note that $(\psi^+)^{-1}(f(\vec{x},0)) = (\psi^+)^{-1}(\vec{x},0) = (\vec{x},0) = (\psi^-)^{-1}(\vec{x},0) = (\psi^-)^{-1}(f(\vec{x},0))$ so that $\Phi$ is well-defined and continuous at $t=0$.  It is strata-preserving since it is a composition of strata-preserving maps.

\vspace{5mm}

Our first claim is that $\Phi$ is a diffeomorphism of manifolds with corners near $0$.  Let $e_1, \dots, e_{n-1}$ be the standard basis vectors for $\R^{n-1}$ and observe that $\{e_1,\dots, e_{n-1},F(\vec{x},0)\}$ is a linearly independent set for $\vec{x}$ near $0$, which is equivalent to $\phi$ being a diffeomorphism near $0$ since the image of $d\phi_0$ using the standard basis is precisely this set.  Therefore, there exist smooth functions $a_1, \dots, a_n$ so that:

\begin{displaymath}
F(\vec{x},t) = (a_1, \dots, a_{n-1}, a_nF(\vec{x},0))
\end{displaymath}
near $0$, where $a_n \ne 0$ near $0$ since $F(\vec{x},0)$ and $F(\vec{x},t)$ agree at $t=0$.  Then, we may write:

\begin{displaymath}
f(\vec{x},t) = (\vec{x},0) + t^2F(\vec{x},t)= (x_1 +t^2a_1, \dots, x_{n-1}+t^2a_{n-1}, a_nt^2F(\vec{x},0))
\end{displaymath}
and if we apply $(\gamma\big\vert_{t\ge0})^{-1}$ or $(\gamma\big\vert_{t\le 0})^{-1})$ we obtain the formula:
\begin{displaymath}
\Phi(\vec{x},t)=
\begin{cases}
(\gamma\big\vert_{t\ge0})^{-1}(\phi^{-1}(f(\vec{x},t))) =  (x_1 +t^2a_1, \dots, x_{n-1}+t^2a_{n-1}, \sqrt{a_n}t) & t\ge 0 \\
(\gamma\big\vert_{t\le0})^{-1}(\phi^{-1}(f(\vec{x},t))) =  (x_1 +t^2a_1, \dots, x_{n-1}+t^2a_{n-1}, \sqrt{a_n}t) & t\le 0
\end{cases}
\end{displaymath}
where $\sqrt{a_n}$ is smooth near $0$ since $a_n$ is nonzero in a neighborhood of $p$.  Hence,

\begin{displaymath}
\Phi(\vec{x},t)= (x_1 +t^2a_1, \dots, x_{n-1}+t^2a_{n-1}, a_nt^2, \sqrt{a_n}t)
\end{displaymath}
 for some smooth functions $a_i$ with $a_n\ne 0$ near $0$, meaning $\Phi$ is a smooth map near $0$.  It is a diffeomorphism near $0$ since $\Phi$ restricted to the fold, $\{t=0\}$, is the identity map and $d\Phi_0(\frac{\partial}{\partial t}) = \sqrt{a_n(0)}\ne 0$.  Thus, $d\Phi_0$ is an isomorphism and there exists a neighborhood of $0$ on which it is a diffeomorphism since it is a strata-preserving map.

Our second claim is that $\Phi$ leads to a normal form and factorization of $f$ as a diffeomorphism composed with a fold map.  By definition of $\Phi$, if we apply the map $\gamma(z,t)=(z,t^2)$ on the left, we get:

\begin{displaymath}
\gamma \circ \Phi = \phi^{-1} \circ f
\end{displaymath}
where $\phi(\vec{x},t)=(\vec{x},0) + tF(\vec{x},0)$ is a diffeomorphism.  Since $\gamma$ has fold singularities and $\Phi$ is a diffeomorphism, $\gamma \circ \Phi$ is a map with fold singularities, hence $f=\phi \circ (\gamma \circ \Phi)$ is a factorization of $f$ into a diffeomorphism composed with a fold map, which is folded by corollary \ref{cor:folds-diffeo}.  If we precompose with $\Phi^{-1}$, we get:

\begin{displaymath}
(f\circ \Phi^{-1})(\vec{x},t) = (\phi\circ \gamma)(\vec{x},t) = (\vec{x},0) + t^2F(\vec{x},0)
\end{displaymath}
which proves the proposition.
\end{proof}

\begin{example}\label{ex:stdfoldmap}
The map $f:\R^n\times \R \to \R^n \times \R$ given by $f(\vec{x},t)=(\vec{x},t^2)$ is a (equidimensional) map with fold singularities.  This example of a fold map is, in some sense, the \emph{only} example.  By proposition \ref{prop:folds5}, one can write every fold map as $\psi(\vec{x},t)=(\vec{x},0) + t^2F(\vec{x})$, where $F$ is transverse to the image of $d\psi$ at $\{t=0\}$.  Furthermore, if the map $\psi$ is strata-preserving then, as we saw in the proof of \ref{prop:folds5}, $F(\vec{x})$ may be viewed as a vector tangent to the stratum containing $\psi(\vec{x},t)$.  Thus, we may define a new coordinate on the target space near $\psi(\vec{x},t)$ using $F(\vec{x})$.  Then the map $\psi$ has the form $\psi(\vec{x},t)=(\vec{x},t^2)$.
\end{example}

\begin{example}\label{ex:foldspheres}
Let $S^2$ be the $2$-sphere embedded into $\R^3$ as the level set $x^2+y^2+z^2=1$.  Then the projection map $p:S^2 \to \R^2$ given by $p(x,y,z)=(x,y)$ has fold singularities along the equator $S^1\subset S^2$.  For example, near the point $(1,0,0)$, the map looks like the composition:

\begin{displaymath}
(y,z) \to (\sqrt{1-y^2-z^2},y,z) \to (\sqrt{1-y^2-z^2},y)
\end{displaymath}
The determinant of the differential of this map is $\displaystyle \frac{-z}{\sqrt{1-y^2-z^2}}$, which vanishes transversally at $z=0$.  The kernel of the differential of the map is given by $\displaystyle \frac{\partial}{\partial z}$, which is transverse to the equator.  By corollary \ref{cor:folds4-0}, $p$ is a map with fold singularities along $S^1$.  One may perform the same construction for each sphere $S^n$ so that every sphere admits a fold map into $\R^n$ with folding hypersurface given by the equator $S^{n-1}$.  Note that one may also perform the same construction for closed surfaces of genus $g$.
\end{example}

\begin{example}
Let $W\subset \R^2$ be the half space $W=\{(x,y)\vert y\ge 0\}$ and consider the map $f:W \to \R^2$ given by $f(x,y)=(x,y^2)$.  As we discussed in the beginning of the chapter, this map has fold singularities in the traditional sense.  However, it does not have fold singularities according to definition \ref{def:folds3}.  According to corollary \ref{cor:folds4-0}, we can compute the determinant of $df$ to determine if $f$ has fold singularities.  The determinant is $2y$, which vanishes along the boundary $y=0$.  Thus, $\det(df)$ is not transverse (in the sense of manifolds with corners) to $0$ and so it does not have fold singularities.
\end{example}

\subsection{Example: Generalizing Morse Functions to Fiber Bundles}
\subsubsection{What is a Morse Function?}
We give a definition of a Morse function and relate it to $1$-jets for the purpose of generalizing Morse functions to sections of fiber bundles with $1$-dimensional fibers.  We do not give the standard definition where one defines the Hessian and requires it to be non-degenerate at critical points.  Instead, we give an equivalent, geometric formulation which will facilitate our discussion.

\begin{remark}
Let us assume throughout this section that all manifolds and fiber bundles are without boundary, hence they are simply $C^{\infty}$ manifolds.  We will use the term \emph{manifold} to mean a manifold without boundary.
\end{remark}

\begin{definition}\label{def:Morse}
Let $M$ be a manifold (without corners) and let $\mathcal{O}\in \Gamma(T^*M)$ be the zero section.  We say a smooth map $f:M\to \R$ is a \emph{Morse function} if $df \pitchfork \mathcal{O}$ as a map $df:M \to T^*M$.
\end{definition}

We now discuss how this definition relates to our study of jet bundles.  Let $E=M\times \R$ and let \newline $\pi:E\to M$ be the trivial fiber bundle over $M$ with standard flat connection.  That is, the horizontal bundle is $TM \oplus 0 \subset TM \times T\R$ and the vertical projection $p_V:TM \times T\R \to T\R$ is $p_V(X,Y) = Y$.  The first jet bundle $J^1(E)$ is $\hom(TM,T\R)$ (q.v. example \ref{ex:products}), which is canonically isomorphic to $T^*M \times \R$ as a fiber bundle over $M\times \R$.  A section $\phi(m)=(m,f(m))$ is the graph of a function $f:M\to \R$ and its $1$-jet $j^1\phi$ at $m$ is $j^1\phi(m)=(m,df_m,f(m))$ in $T^*M \times \R$.

The submanifold $S_r\subset J^1(E)$ of jets of corank $1$ is empty for $r\ge 1$ since $\operatorname{corank}(df)=0,1$.  The submanifold $S_1\subset J^1(E)$ is exactly the zero section of $T^*M \times \R \to M\times \R$:

\begin{displaymath}
\operatorname{corank}(j^1_m\phi) = 1 \iff \mbox{ } \operatorname{corank}(p_V(d\phi_m)) = 1 \iff \mbox{ } \operatorname{corank}(df_m)=1 \iff df_p=0
\end{displaymath}

Recall, we may view $1$-jet fields on $E$ as connections on $E$ by proposition \ref{prop:C1}.  We claim that the zero section $\chi:M \times \R \to T^*M \times \R$ of the first jet bundle is exactly the standard flat connection on $M\times \R$:

\begin{displaymath}
\begin{array}{lcl}
(X,Y)\in T_mM \times T_t\R \text{ is horizontal} & \iff & \text{for $\phi(m)=(m,f(m))$ representing $\chi(m,t)$}\\
                                                 &      & \text{we have $d\phi_m\circ d\pi_{(m,t)}(X,Y)=(X,Y)$} \\
                                                 & \iff & d\phi_m(X)=(X,Y)\\
                                                 & \iff & (X,df_m(X))=(X,Y) \\
                                                 & \iff & Y = df_m(X) = 0 \text{ (since $\chi$ is the zero section, we have $df_m=0$) }\\
                                                 & \iff & (X,Y)=(X,0) \\
                                                 & \iff & \text{The horizontal bundle is $\pi^*TM\oplus 0$.}
\end{array}
\end{displaymath}
By our definition of a Morse function, we have:

\begin{displaymath}
\begin{array}{lcl}
\text{$f$ is Morse} & \iff & df \pitchfork \mathcal{O}, \text{ where $\mathcal{O}$ is the zero section of $T^*M$,} \\
                    & \iff & j^1\phi \pitchfork \mathcal{O}\times \R, \text{ where $\phi(m)=(m,f(m))$} \\
                    & \iff & j^1\phi \text{ is transverse to the zero section, $\chi$, of $J^1(M\times \R)$} \\
                    & \iff & j^1\phi \pitchfork \chi \text{ (just a restatement of the previous line)}
\end{array}
\end{displaymath}
Thus, a Morse function $f:M\to \R$ may be viewed as a section (graph) $\phi:M \to M\times \R$ whose $1$-jet is transverse to the standard flat connection on $M\times \R \to M$, viewed as either a section of $\pi_{1,0}:J^1(M\times \R) \to M\times \R$ or a submanifold of $J^1(E)$.  This leads us to a possible definition of what a Morse section of a fiber bundle should be, where the fiber has dimension $1$.
\subsubsection{$\chi$-Morse Functions and Sections}
\begin{remark}
Again, we are assuming all manifolds appearing in this section are manifolds without boundary.
\end{remark}
\begin{definition}\label{def:Morse1}
Let $\pi:E \to M$ be a fiber bundle with typical fiber $F$, $\dim(F)=1$, and a connection $\chi$, which we view as a section $\chi:E \to J^1(E)$ by proposition \ref{prop:C1}.  A local section $(U,\phi)$ is $\chi$-Morse if $j^1\phi \pitchfork_s \chi$.  That is, the $1$-jet of $\phi$ is transverse to the connection $\chi$.
\end{definition}

\begin{definition}\label{def:Morse2}
Let $\pi:M\times \R \to M$ be the trivial fiber bundle with a connection $\chi$, which we view as a section $\chi:M\times \R \to J^1(M\times \R)$.  Then a $\chi$-\emph{Morse function} $f:M\to \R$ is a function such that $\phi_f(m)=(m,f(m))$ is a $\chi$-Morse section of $M\times \R$.
\end{definition}

We have a nice geometric interpretation of what it means to be $\chi$-Morse.

\begin{lemma}\label{lem:Morse1}
Let $\pi:E \to M$ be a fiber bundle with typical fiber $F$, $\dim(F)=1$, and a connection $\chi:E \to J^1(E)$.  Let $p_V:TE \to V$ be the projection onto the vertical bundle.  Let $(U,\phi)$ be a local section of $E$. Then,
\begin{displaymath}
\phi \pitchfork \chi \mbox{ } \iff \mbox { } p_V(d\phi) \pitchfork \mathcal{O}
\end{displaymath}
where $\mathcal{O}\in \Gamma(\hom(\pi^*TM,V))$ is the zero section.
\end{lemma}

\begin{proof}
Proposition \ref{prop:C2} shows that the connection $\chi$ induces an isomorphism of fiber bundles:

\begin{displaymath}
\xymatrixcolsep{4pc}\xymatrixrowsep{3pc}\xymatrix{
J^1(E) \ar[r]^{F_{\chi}} \ar[d]^{\pi_{1,0}} & \hom(\pi^*TM, V) \ar[d] \\
E \ar[d]^{\pi} \ar@/^/[u]^{\chi} \ar[r]^{id}&  E \ar[d]^{\pi}         \\
M \ar@/^2pc/[uu]^{j^1\phi}   \ar[r]^{id}         &  M \ar@/_2pc/[uu]_{p_V(d\phi)}
}
\end{displaymath}
where $F_{\chi}(j^1_m\phi)=p_V(d\phi_m)$.  Furthermore, $F_{\chi}\circ \chi$ is the zero section by proposition \ref{prop:C3}.  Thus, $j^1\phi \pitchfork \chi$ if and only if $F_{\chi}(j^1\phi) \pitchfork \mathcal{O}$ if and only if $p_V(d\phi) \pitchfork \mathcal{O}$.
\end{proof}

We may now make sense of a critical point.
\begin{definition}\label{def:Morse3}
Let $\pi:E \to M$ be a fiber bundle with typical fiber $F$, $\dim(F)=1$, and a connection $\chi: E \to J^1(E)$.  Let $p_V:TE \to V$ be the projection onto the vertical bundle.  Let $(U,\phi)$ be a local section.  We say $m\in M$ is a \emph{$\chi$-critical point} of $M$ if $p_V(d\phi_m)=0$.  Equivalently, $\operatorname{corank}(j^1_m\phi)=1$, hence $j^1_m\phi$ intersects $S_1$.
\end{definition}

\begin{example}
Let $M=\R$ and let $E=\R \times \R$ with coordinates $(s,t)$ and let $\chi$ be the connection defined by the vertical projection:

\begin{displaymath}
\begin{array}{c}
p_V:T\R \times T\R \to T\R \\
\displaystyle p_V(a\frac{\partial}{\partial s}, b\frac{\partial}{\partial t}) = (b-a)\frac{\partial}{\partial t}
\end{array}
\end{displaymath}
with corresponding horizontal projection $\displaystyle p_H(a\frac{\partial}{\partial s}, b\frac{\partial}{\partial t}) = a\frac{\partial}{\partial s} + a\frac{\partial}{\partial t}$.  A section $\phi(m)=(m,f(m))$ is $\chi$-Morse if $p_V(d\phi)=(df-ds)$ is transverse to the zero section $\iff$ $(\frac{\partial f}{\partial s} -1) \pitchfork_s 0$ $\iff$ the slope of the graph of $f$ passes through $1$ transversally, which is true if and only if $\frac{\partial^2 f}{\partial^2 s} \ne 0$ at $\chi$-critical points.

Note that the horizontal distribution is involutive, hence the Frobenius theorem allows us to to integrate this distribution and obtain smooth submanifolds parameterized by the fiber $\R$, $S_c\subset \R \times \R$, where $c\in \R$.  These submanifolds are simply the graphs of $f(s)=s+c$, $c\in \R$.  At a $\chi$-critical point $p$ of a Morse section $\phi$, the section becomes tangent to $S_c$ for some $c$.  Since $\frac{\partial^2 f}{\partial^2 s}(p) \ne 0$, the section $\phi$ is concave up or concave down, meaning it locally looks like the parabola $x^2$ intersecting $f(x)=x$ at $x=1$ or the parabola $-x^2$ intersecting $f(x)=x$ at $x=-1$.
\end{example}

\begin{example}
Building on the previous example, we may consider a flat connection $\chi$ on fiber bundle $\pi:E \to M$ with fiber $F$, $\dim(F)=1$.  Then the horizontal bundle $H$ is involutive and we may integrate it to obtain a foliation of $E$ by submanifolds $S_p$, parameterized by $p\in F$, of dimension $\dim(M)$ transverse to the fibers of $\pi$ whose tangent spaces are exactly the horizontal subspaces.

A local section $(U,\phi)$ of $E$ is $\chi$-Morse if whenever $\phi(U)$ is tangent to $S_p$, its vertical differential $p_V(d\phi)$ vanishes transversely in all directions.  Hence, in any direction through a $\chi$-critical point, $\phi$ will look like a the graph of a parabola intersecting a submanifold above or below it tangentially.
\end{example}

\subsubsection{$\chi$-Morse Functions are generic}

There are several questions one can, and probably should, ask about $\chi$-Morse functions $f:M\to \R$, where $\chi$ is a connection on $M\times \R$:

\begin{enumerate}
\item First and foremost, how useful are they?  For a choice of connection $\chi$, do we get a homology theory using the $\chi$-critical points of $f$ as we do in the standard case of a Morse function?

\item To this end, how do we define the index of a $\chi$-critical point $p\in M$?  Presumably, the connection $\chi$ will allow us to produce a non-degenerate quadratic form on $T_pM$ whose index will be the index of the $\chi$-critical point.  But, how will this work and what will it tell us?

\item And, if one is to produce a homology theory from a $\chi$-Morse function, then one ought to be able to show $\chi$-Morse functions \emph{exist}.  Do they exist?  And, are they generic?
\end{enumerate}

We address the last question.  To this end, we will need the Thom Transversality Theorem, which is  Theorem 4.9 in \cite{GG}.  The general statement involves $k$-jets: we will only provide the statement in the case where $k=1$, but the reader should note that Thom proved a more powerful theorem than the one we give.

\begin{theorem}\label{thm:Thom}
Let $M$ and $N$ be smooth manifolds (without corners), let $M\times N \to M$ be the trivial fiber bundle, and for each $f\in C^{\infty}(M,N)$ let $\phi_f(m)=(m,f(m))$ denote the induced section of $M\times N$.  Let $W \subset J^1(M\times N)$ be a smooth submanifold of the first jet bundle and let:

\begin{displaymath}
T_W = \{f\in C^{\infty}(M,N) \vert \mbox{ } j^1\phi_f \pitchfork W\}
\end{displaymath}
Then $T_W$ is a residual subset of $C^{\infty}(M,N)$ in the $C^{\infty}$ topology and it is open and dense if $W$ is a closed subset.
\end{theorem}

Almost for free, we obtain:

\begin{prop}\label{prop:Morse}
Let $\pi:M \times \R \to M$ be the trivial fiber bundle over $M$ and let $\chi: M\times \R \to J^1(M\times \R)$ be a connection on $M\times \R$.  Then the set:

\begin{displaymath}
C^{\infty}_{\chi}(M) = \{f\in C^{\infty}(M) \vert \mbox{ } \text{f is $\chi$-Morse} \}
\end{displaymath}
is an open, dense subset of $C^{\infty}(M)$.
\end{prop}

\begin{proof} \mbox{ } \newline
Let $W$ be the submanifold $\chi(M\times \R)$ of $J^1(M\times \R)$.  $W$ is then a closed submanifold and $f$ is $\chi$-Morse if and only if the induced section $\phi_f$ satisfies $j^1\phi_f \pitchfork \chi$ $\iff$ $j^1\phi_f \pitchfork W$.  Therefore, the set $T_W$ described in the hypotheses of the Thom transversality theorem is exactly the set $C^{\infty}_{\chi}(M)$.  Since $W$ is closed, $T_W$ is open and dense in $C^{\infty}(M)$.  We therefore have that $C^{\infty}_{\chi}(M)\subset C^{\infty}(M)$ is an open dense subset.
\end{proof}

As a last exercise, we show that if $f\in C^{\infty}(M)$, then it is $\chi$-Morse for some connection $\chi$ on $M\times \R$.

\begin{lemma}\label{lem:Morse2}
Let $M$ be a smooth manifold (without corners) and let $f\in C^{\infty}(M)$ be a smooth function.  Then there exists a connection $\chi$ on the bundle $M\times \R \to M$ so that $f$ is $\chi$-Morse.
\end{lemma}

\begin{proof}
Given a smooth function $f:M \to \R$, its differential $df$ defines a section of $T^*M$ which we may use to define a connection $\chi_f$ on $M\times \R$. Recall that $J^1(M\times \R) = T^*M \times \R$ as a bundle over $M\times \R$.  The connection $\chi_f:M\times \R \to T^*M \times \R$ is then given as:

\begin{displaymath}
\chi_f(m,t) = (m,df_m, t)
\end{displaymath}
By proposition \ref{prop:Morse} there exists a $\chi_f$-Morse function $g:M \to \R$, meaning, if $\phi_g$ is the graph of $g$, $j^1\phi_g \pitchfork \chi_f$, but this is true if and only if $dg \pitchfork df$ as sections of $T^*M$.  Using this fact, we define a connection:

\begin{displaymath}
\chi_g(m,t)=(m,dg_m,t)
\end{displaymath}
and note that $df \pitchfork dg$ if and only if $j^1\phi_f \pitchfork_s \chi_g$ if and only if $f$ is $\chi_g$-Morse.
\end{proof}

Thus, with our new definition of $\chi$-Morse, every function may be realized as a $\chi$-Morse function for some $\chi$ and therefore has a home inside a generic subset of $C^{\infty}(M)$.

\section{Folded Symplectic Manifolds}
We introduce the notion of a folded-symplectic manifold and show how they arise in a very fundamental way by dualizing a standard construction in $b$-symplectic geometry.  Some constructions will involve manifolds with corners while others will use manifolds without boundary: we will always explicitly state which type of manifold we are using in the definitions, lemmas, and propositions.  Our main goals in this section, listed in order of desirability, are:

\begin{enumerate}
\item to develop a normal form for the folding hypersurface inside a folded-symplectic manifold,
\item to show that being folded-symplectic is equivalent to inducing an isomorphism of sheaves $\sigma^\#: \Gamma(TM) \to S$, where $S$ is a distinguished sheaf of $1$-forms on $M$, discussed below, and finally
\item to develop a Moser-type argument for deformations of folded-symplectic structures.
\end{enumerate}
Of course, we will see that these goals are all related: we will need a Moser argument to develop the normal form for the fold and, to this end, we will need to understand when one can solve Moser's equation, which is intimately related to the second goal.

\subsection{Definition and Examples}

\begin{definition}\label{def:fsform}
Let $M$ be a $2m$-dimensional manifold with corners.  We say $\sigma\in\Omega^2(M)$ is \emph{folded-symplectic} if
\begin{enumerate}
\item $d\sigma = 0$
\item $\sigma^m \pitchfork_s \mathcal{O}$, where $\mathcal{O} \subset \Lambda^{2m}(T^*M)$ is the zero section, hence $Z=(\sigma^m)^{-1}(\mathcal{O})$ is a codimension $1$ submanifold with corners intersecting the strata of $M$ transversally.
\item If $i_Z:Z\hookrightarrow M$ is the inclusion, $i_Z^*\sigma$ has maximal rank, $2m-2$.
\end{enumerate}
We say $(M,\sigma)$ is a \emph{folded-symplectic manifold} with corners and we call $Z\subset M$ the \emph{fold} or the \emph{folding hypersurface}.
\end{definition}

\begin{definition}\label{def:fsmaps}
Let $(M,\sigma_1)$, $(N,\sigma_2)$ be two folded-symplectic manifolds with corners.  We say a smooth map $\phi:M\to N$ is \emph{folded-symplectic} if $\phi^*\sigma_2 = \sigma_1$.  If $\phi$ is a diffeomorphism, we say it is a \emph{folded-symplectomorphism}.
\end{definition}

\begin{definition}\label{def:nullbundles}
Let $(M,\sigma)$ be a folded-symplectic manifold with corners and let $Z\subset M$ be the folding hypersurface of $\sigma$ with inclusion $i_Z:Z\hookrightarrow M$.  Assume that $Z$ is nonempty.  We define two vector subbundles of $i_Z^*TM$:

\begin{enumerate}
\item $\ker(\sigma) \to Z$ is $2$-plane bundle over $Z$ whose fiber at a point $z\in Z$ is:
\begin{displaymath}
\ker(\sigma_z) = \{v \in T_zZ \vert \mbox{ } i_v\sigma_z=0\}
\end{displaymath}

\item $\ker(i_Z^*\sigma) \to Z$ is the rank $1$ vector bundle over $Z$ defined to be the intersection $\ker(i_Z^*\sigma):= \ker(\sigma) \cap TZ$.  It may be viewed as a vector subbundle of $i_Z^*TM$ or $TZ$.
\end{enumerate}
\end{definition}

\begin{remark}
Note that since $\ker(i_Z^*\sigma)$ is a rank $1$ vector subbundle of $TZ$, it is trivially involutive.  Hence, by the Frobenius theorem, it defines an integrable distribution on $Z$ and we obtain a foliation of $Z$ by $1$-dimensional leaves.
\end{remark}

\begin{definition}\label{def:nullfibration}
Let $(M,\sigma)$ be a folded-symplectic manifold with corners and let $Z\subset M$ be the folding hypersurface of $\sigma$.  Assume that $Z$ is nonempty and let $\ker(i_Z^*\sigma)$ be the rank $1$ vector subbundle of $TZ$ of definition \ref{def:nullbundles}.  Let $\mathcal{F}$ be the foliation of $Z$ induced by this bundle.  We refer to this foliation as the \emph{null-foliation}.
\end{definition}

\begin{example}\label{ex:fs1}
Let $(M,\omega)$ be a symplectic manifold (with or without corners).  It is trivially folded-symplectic with folding hypersurface $Z=\emptyset$.  Since $Z$ is empty, the bundles $\ker(\sigma)$ and $\ker(i_Z^*\sigma)$ are not defined, though one could view them as just empty sets.  Consequently, there is no null-foliation to consider.
\end{example}

\begin{example}\label{ex:fs2}
Let $(M^{2n},\omega)$ be a symplectic manifold with corners and let $\psi:N^{2n} \to M^{2n}$ be a map with folds.  Then $\sigma=\psi^*\omega$ is a folded-symplectic form on $M$.  To see why this is true, note that in any choice of coordinates near $p\in M$ $\det(d\psi) \pitchfork_s 0$ by corollary \ref{cor:folds4-0}, hence in any choice of coordinates coordinates $(\psi^*\omega)^n = \det(d\psi) \omega^n \pitchfork_s \mathcal{O}$.  This computation also shows that the folding hypersurface $Z$ of $\psi$ is the folding hypersurface of $\psi^*\omega$.  Furthermore:

\begin{enumerate}
\item The bundle $\ker(\sigma)$ has fiber: $\ker(\sigma_z) = (d\psi_z)^{-1}(d\psi_z(T_zM)^{\omega})$.
\item In particular, it contains $\ker(d\psi_z)$.  Since $\ker(d\psi_z) \pitchfork T_zZ$ and $Z$ is a hypersurface, we may write:
\begin{displaymath}
\ker(\sigma_z) = (d\psi_z)^{-1}(d\psi_z(T_zZ)^{\omega})
\end{displaymath}
Since $Z$ is a hypersurface and $d\psi_z\big\vert_{TZ}$ is injective, $d\psi_z(T_zZ)$ is a codimension $1$ subspace of $T_{\psi(z)}M$, meaning it is co-isotropic.  Therefore, the $1$-dimensional subspace $d\psi_z(T_zZ)^{\omega}$ is contained in $d\psi_z(T_zZ)$ and has a unique $1$-dimensional preimage in $T_zZ$, which shows that $i_Z^*(\psi^*\omega)$ has a $1$-dimensional kernel at any point $z\in Z$, meaning it has maximal rank.  This verifies the third condition of definition \ref{def:fsform}.
\item Since $\omega$ is closed, $\psi^*\omega$ is closed, which completes the verification of all three conditions in definition \ref{def:fsform}.
\end{enumerate}
\end{example}

\begin{example}\label{ex:fs3}
Let $\sigma \in \Omega^2(\R^{2n})$ be defined by
\begin{displaymath}
\sigma = x_1dx_1\wedge dx_2 + dx_3\wedge dx_4 + \dots + dx_{2n-1}\wedge dx_{2n}
\end{displaymath}
Then $\sigma$ is folded-symplectic with fold $Z$ defined by $x_1=0$.
\begin{enumerate}
\item The bundle $\ker(\sigma)\to Z$ is framed by the vector fields $\displaystyle \{\frac{\partial}{\partial x_1}, \frac{\partial}{\partial x_2}\}$ along $Z$.
\item The bundle $\ker(i_Z^*\sigma)$ is framed by the vector field $\displaystyle \frac{\partial}{\partial x_2}$ along $Z$.
\end{enumerate}
\end{example}

\begin{example}\label{ex:fs4}
Let $\pi: S^{2n} \to \R^{2n}$ be the projection $\pi(\vec{x},z) =\vec{x}$ and let $\omega_{\R^{2n}}$ be the standard symplectic form on $\R^{2n}$.  $\pi$ is a map which folds along the equator $S^{2n-1}$, hence $\sigma=\pi^*\omega_{\R^{2n}}$ is a folded-symplectic form on $S^{2n}$ by example \ref{ex:fs2}.  \begin{enumerate}
\item The bundle $\ker(\sigma) \to S^{2n-1}$ is spanned by $\displaystyle \frac{\partial}{\partial z}$ (the kernel of $d\pi$ at the equator) and the vector field(s) induced by the diagonal action of $S^1$ on $S^{2n-1}$, $\lambda\cdot (z_1, \dots, z_n) = (\lambda z_1, \dots , \lambda z_n)$, where $z_i$ are complex numbers.  This is because the kernel of the restriction of $\omega_{\R^{2n}}$ to $S^{2n-1}$ is spanned by this vector field.

\item Consequently, the null foliation on the fold $Z=S^{2n-1}$ is given by the orbits of the diagonal $S^1$ action on $S^{2n-1}$.
\end{enumerate}
\end{example}

\begin{example}\label{ex:fs5}
For $m>0$, let $G=\mathbb{Z}/2\mathbb{Z}$ act on $\R^{m}$ by reflection: $\vec{x} \to -\vec{x}$.  Then the projection $\pi:S^{2n} \to \R^{2n}$ of example \ref{ex:fs4} is $G$-equivariant.  Since $\omega_{\R^{2n}}$ is $G$-invariant, $\sigma=\pi^*\omega_{\R^{2n}}$ is $G$-invariant.  Let $p:S^{2n} \to \R P^{2n} (=S^{2n}/G)$ be the orbit map.  Since $\sigma$ is $G$-invariant, we have $\sigma=p^*\bar{\sigma}$ for some $\bar{\sigma}\in \Omega^2(\R P^{2n})$.  Since $p$ is a covering map, it is a local diffeomorphism and $\bar{\sigma}$ has fold singularities along $\R P ^{2n-1}$.

\begin{enumerate}
\item The null-foliation on $\R P^{2n-1}$ is given by the image of the orbits of the diagonal action of $S^1$ on $S^{2n-1}$.
\item The null direction transverse to $\R P^{2n-1}$ at a point $[(\vec{x},0)]$ is given by $dp_{(\vec{x},0)}(\frac{\partial}{\partial z})$.
\end{enumerate}
\end{example}

The following lemma demonstrates how one may construct a folded-symplectic manifold from an odd-dimensional manifold equipped with a closed $2$-form of maximal rank.  We will also see how this lemma allows us to decide when some $2$-forms are folded-symplectic or not.

\begin{lemma}\label{lem:symplectize}
Let $Z^{2n-1}$ be an odd-dimensional manifold with corners.  Let $\omega \in \Omega^2(Z)$ be a closed $2$-form of maximal rank $2n-2$ and let $\ker(\omega)\to Z$ be the bundle whose fiber is $\ker(\omega)_z:= \{v\in T_zZ \vert \mbox{ } i_v\omega=0\}$.  Let $\mu \in \Omega(Z\times \R)$ and let $p:Z\times \R \to Z$ be the projection.  Then

\begin{displaymath}
\sigma :=p^*\omega + t\mu
\end{displaymath}
is folded-symplectic in a neighborhood of $Z\times \{0\}$ with fold $Z\times \{0\}$ if and only if for all $z\in Z$ we have $\mu\big\vert_{\ker{p^*\omega}}$ is nondegenerate in a neighborhood of $Z\times \{0\}$.
\end{lemma}

\begin{proof}
The top power of $\sigma$ is:

\begin{displaymath}
\sigma^n = t(p^*\omega)^{n-1} \wedge \mu + t^2 \beta
\end{displaymath}
where $\beta \in \Omega^{2n}(Z\times \R)$ is some $2n$-form, which shows that $\sigma^n$ vanishes at $Z\times \{0\}$.  We have $\sigma^n \pitchfork_s \mathcal{O}$, where $\mathcal{O}$ is the zero section, if and only if $\mu$ doesn't vanish on the directions on which $p^*\omega$ vanishes.  That is, $\mu$ doesn't vanish on $\ker(p^*\omega)$ at points of $Z\times \{0\}$, hence it doesn't vanish on $\ker(p^*\omega)$ in a neighborhood of $Z\times \{0\}$.  In particular, ince $p^*\omega$ vanishes on $\frac{\partial}{\partial t}$ and $p^*\ker(\omega)$, we have:

\begin{displaymath}
\sigma^n \pitchfork_s \mathcal{O} \mbox{} \iff \text{ $\mu_{(z,0)}(\frac{\partial}{\partial t}, v)\ne 0$ for all $v\in \ker(\omega_z)$}
\end{displaymath}
By definition of $\sigma$, its restriction to the hypersurface $Z\times \{0\}$ has maximal rank, since it is essentially $\omega$ on $Z$, which shows that it is folded-symplectic if and only if the hypotheses of the lemma are satisfied.
\end{proof}

\begin{cor}\label{cor:symplectize}
Let $Z^{2n-1}$ be a manifold with corners and suppose $\sigma \in \Omega^2(Z\times \R)$ is folded-symplectic with fold $Z\times \{0\}$.  Furthermore, suppose $\ker(\sigma)$ contains the subbundle framed by $\frac{\partial}{\partial t}$ along $Z\times \{0\}$.  Let $i:Z \to Z\times \R$ be the inclusion as the zero section and let $p:Z\times \R \to Z$ be the projection.  Then:

\begin{enumerate}
\item $\sigma = p^*i^*\sigma + t\mu$ for some $2$-form $\mu\in \Omega^2(Z\times \R)$.
\item $\mu\vert_{\ker(\sigma)}$ is non-degenerate.
\end{enumerate}
\end{cor}

\begin{proof}
\begin{enumerate}
\item Consider the difference $\sigma -p^*i^*\sigma$.  Since both $p^*i^*\sigma$ and $\sigma$ vanish on $\frac{\partial}{\partial t}$ along $Z\times \{0\}$, this difference vanishes at $Z\times \{0\}$.  Thus, $\sigma - p^*i^*\sigma = t\mu$ for some $2$-form $\mu\in \Omega^2(Z\times \R)$ by lemma \ref{lem:had}, which shows $\sigma = p^*i^*\sigma + t\mu$.
\item $(Z,i^*\sigma)$ is an odd-dimensional manifold with a closed $2$-form of maximal rank.  By lemma \ref{lem:symplectize}, $\sigma$ is folded-symplectic with fold $Z\times \{0\}$ if and only if $\mu$ doesn't vanish on $\ker(p^*i^*\sigma)$ at $Z\times \{0\}$.  Since $\ker(p^*i^*\sigma)$ is spanned by $\frac{\partial}{\partial t}$ and elements of $\ker(i^*\sigma)$ along $Z\times \{0\}$, we have that $\ker(p^*i^*\sigma)\big\vert_{Z\times \{0\}} = \ker(\sigma)$.  Since $\sigma$ is folded, we must have that $\mu\big\vert_{\ker{\sigma}}$ is non-degenerate.
\end{enumerate}
\end{proof}

\begin{example}\label{ex:symplectize1}
Consider $S^1$ equipped with the $2$-form $0\in \Omega^2(S^1)$, which is a closed $2$-form of maximal rank.  Then the cylinder $S^1\times \R$ may be given a fold structure using $\sigma = 0 + t(2dt\wedge d\theta) = d(t^2d\theta)$.
\end{example}

\begin{example}\label{ex:symplectize2}
Take any contact manifold $(Z,\alpha)$, where $\alpha$ is a contact $1$-form on $Z$.  Let $p:Z\times \R \to Z$ be the projection.  Then:

\begin{displaymath}
\sigma := p^*(d\alpha) + t(2dt\wedge p^*\alpha) + t^2p^*d\alpha = d(1+t^2)(p^*\alpha)
\end{displaymath}
is folded-symplectic in a neighborhood of $Z\times \{0\}$ with folding hypersurface $Z\times \{0\}$.
\end{example}

\begin{example}\label{ex:symplectize3}
To generalize example \ref{ex:symplectize2}, consider any oriented, odd-dimensional manifold $Z$ with a closed $2$-form $\omega$ of maximal rank.  Since $Z$ is oriented, the bundle $\ker(\omega)$ is oriented, meaning we may choose a non-vanishing $1$-form $\alpha \in \Omega^1(Z)$ that doesn't vanish on $\ker(\omega)\to Z$.  Then $\sigma = p^*\omega + d(t^2p^*\alpha)$ is folded-symplectic in a neighborhood of $Z\times \{0\}$ since $dt\wedge p^*\alpha$ is non-degenerate on $\ker(p^*\omega)$ along $Z\times \{0\}$.
\end{example}

\begin{remark}
We will show that if $(M,\sigma)$ is a folded-symplectic manifold with folding hypersurface $Z\subset M$, then $Z$ is canonically oriented by $\sigma$, hence $(Z,i_Z^*\sigma)$ is an odd-dimensional, oriented manifold with a closed $2$-form of maximal rank.  Thus, \emph{every} folding hypersurface ever conceived is the type of hypersurface discussed in example \ref{ex:symplectize3}.  Furthermore, each such hypersurface admits a symplectization and, by example \ref{ex:symplectize3}, a folded-symplectization.  We will see that neighborhoods of co-orientable folding hypersurfaces look like the symplectization of the folding hypersurface pulled back by a folding map.
\end{remark}

We will now give some more utility to the constructions of examples \ref{ex:symplectize1}, \ref{ex:symplectize2}, and \ref{ex:symplectize3} by providing a partial normal form for a neighborhood of the fold $Z$ of a folded-symplectic manifold $(M,\sigma)$ in the case that $Z$ is co-orientable.  We will need this normal form when we discuss vector bundle and sheaf-theoretic characterizations of folded-symplectic forms.

\begin{lemma}\label{lem:fsnormal}
Let $(M,\sigma)$ be a folded-symplectic manifold with corners and suppose for all $p\in Z$ $\ker(\sigma_p)$ is tangent to the stratum of $M$ containing $p$.  Furthermore, suppose $Z$ is co-orientable.  Then there exists a a neighborhood $U$ of $Z$, a neighborhood $V$ of the zero section of $Z\times \R$, and an isomorphism $\phi:V \to U$ so that:
\begin{displaymath}
\phi^*\sigma = p^*i^*\sigma + t\mu
\end{displaymath}
where $i:Z \hookrightarrow Z\times \R$ is the inclusion as the zero section, $p:Z\times \R \to Z$ is the projection, and $\mu\in \Omega^2(U)$ is some $2$-form satisfying $\mu\big\vert_{\ker(\phi^*\sigma)}$ is non-degenerate.
\end{lemma}

\begin{remark}\label{rem:whatismu}
We will use \ref{lem:fsnormal} to construct the induced orientation on the fold, first discovered by Martinet in \cite{M}, and we will also use it when we develop a sheaf-theoretic characterization of the notion of being folded-symplectic.  The most important detail we would like to emphasize is that the $2$-form $\mu$ appearing in lemma \ref{lem:fsnormal} functions as the intrinsic derivative of the contraction map $C_{\sigma}:TM \to T^*M$, $C_{\sigma}(p,X)=i_X\sigma_p$, at points of the folding hypersurface.  We'll discuss this more in the future.
\end{remark}

\begin{proof}
We will construct $\phi$ by considering the flow of a stratified vector field on $M$ whose values at $Z$ lie inside $\ker(\sigma)$ and whose image under $\phi^{-1}$ will be $\frac{\partial}{\partial t}$. We have a short exact sequence of vector bundles:

\begin{displaymath}
0 \to \ker(i_Z^*\sigma) \to \ker(\sigma) \to \ker(\sigma)/\ker(i_Z^*\sigma) \to 0
\end{displaymath}
The bundle $\ker(\sigma)/\ker(i_Z^*\sigma)$ is canonically isomorphic to the normal bundle $\nu(Z)= i_Z^*TM/ TZ$ of $Z$ via the inclusion $j:\ker(\sigma) \hookrightarrow i_Z^*TM$.  The kernel of the inclusion followed by projection to $\nu(Z)$ is exactly $\ker(i_Z^*\sigma)$, hence $j$ descends to an isomorphism $\ker(\sigma)/\ker(i_Z^*\sigma) \simeq \nu(Z)$.

Since $Z$ is co-orientable, $\nu(Z)$ is trivializable.  Thus, there exists a non-vanishing section $\bar{w} \in \Gamma(\ker(\sigma)/\ker(i_Z^*\sigma))$ which lifts to a non-vanishing section $w \in \Gamma(\ker(\sigma))$ so that $w(z)\notin T_zZ$ for all $z\in Z$.  Since $\ker(\sigma)$ is tangent to the strata of $M$, $w$ is stratified, hence we may extend $w$ to a stratified vector field $\tilde{w}$ on $M$.  Since stratified vector fields may be integrated, we may consider the flow $\tilde{\phi}$ of $\tilde{w}$, which is defined on an open subset of $M\times \R$ containing $M\times \{0\}$.  In particular, it is defined on an open neighborhood $\tilde{V}$ of the zero section of $Z\times \R$.  We may then define:

\begin{displaymath}
\xymatrixcolsep{3pc}\xymatrixrowsep{.5pc}\xymatrix{
\phi: V\subset Z\times \R \ar[r] &  M \\
(z,t) \in V \ar[r]               & \tilde{\phi}(z,t)}
\end{displaymath}
It is a bijection in a neighborhood of $Z\times \{0\}$ since $\tilde{w}$ is non-vanishing in a neighborhood of $Z\times \{0\}$ and integral curves do not intersect.  It is a strata-preserving map since $\tilde{w}$ is a stratified vector field.  Hence, it is locally a diffeomorphism of manifolds with corners in a neighborhood of $Z\times \{0\}$ since $d\phi$ has maximal rank at $Z\times \{0\}$:
\begin{enumerate}
\item $\phi\big\vert_{Z\times \{0\}}(z,0)= z$ and
\item $d\phi_{(z,0}(\frac{\partial}{\partial t}) = \tilde{w}(z)$
\end{enumerate}
This means that there is a neighborhood $V\subset \tilde{V}$ of $Z\times \{0\}$ and a neighborhood $U=\phi(V)$ of $Z$ so that $\phi:V \to U$ is a diffeomorphism of manifolds with corners.

\vspace{5mm}

Since $d\phi_{(z,0)}(\frac{\partial}{\partial t})= \tilde{w}(z) \in \ker(\sigma_z)$, $\frac{\partial}{\partial t} \in \ker(\phi^*\sigma)_{(z,0)}$ for all $z\in Z$.  By corollary \ref{cor:symplectize}, $\phi^*\sigma = p^*i^*\sigma + t\mu$ for some $2$-form $\mu$ and $\mu\big\vert_{\ker(\phi^*\sigma)}$ is non-degenerate.
\end{proof}

\begin{cor}\label{cor:fsnormal}
Let $(M^{2n},\sigma)$ be a folded-symplectic manifold with corners with folding hypersurface $Z\subset M$.  Suppose that for all $z\in Z$ $\ker(\sigma_z)$ is tangent to the stratum of $M$ containing $z$.  Then if $Z$ is co-orientable, it is orientable.
\end{cor}

\begin{proof}
Let $p:Z\times \R \to Z$ be the projection and let $i:Z \to Z\times \R$ be the inclusion as the zero section.  By lemma \ref{lem:fsnormal}, we may assume a neighborhood of $Z$ is a neighborhood of the zero section $V\subset Z\times \R$ with fold form $p^*i^*\sigma + t\mu$, where $\mu\big\vert_{\ker(p^*i^*\sigma)}$ is non-degenerate in a neighborhood of $Z\times \{0\}$.  Consequently, $i^*(i_{\frac{\partial}{\partial t}}\mu)$ is a $1$-form on $Z$ that does not vanish on $\ker(i_Z^*\sigma)$.  We may then define an orientation on $Z$ using the form:

\begin{displaymath}
\Omega= (i_Z^*\sigma)^{n-1}\wedge \alpha
\end{displaymath}
\end{proof}

\begin{remark}
We will see that the orientation defined in the proof of corollary \ref{cor:fsnormal} may be canonically defined using the the folded symplectic form $\sigma$ and its kernel bundle $\ker(\sigma) \to Z$.  In fact, we will see that if $(M,\sigma)$ is any folded-symplectic manifold, the folding hypersurface $Z\subset M$ is endowed with an orientation induced by $\sigma$, regardless of co-orientability.
\end{remark}

\subsection{Cotangent Bundles and Folded Structures}
While we now have the basic facts and examples of folded-symplectic manifolds established, we are left wondering where one might find a naturally occurring folded-symplectic structure.  To answer this question, let us consider for a moment the case of $b$-symplectic structures on manifolds without boundary.  A $b$-symplectic manifold is a $2n$-dimensional manifold (without boundary) equipped with a Poisson structure $\Pi \in \Gamma(\Lambda^2(TM))$ such that $\Pi^n \pitchfork \mathcal{O}$, where $\mathcal{O}$ is the zero section of $\Gamma(\Lambda^{2n}(TM))$.  Thus, folded-symplectic manifolds are, in a sense, dual to the notion of a $b$-symplectic manifold where we consider $T^*M$ instead of $TM$.

Now, $b$-symplectic structures occur quite naturally (see \cite{GMP,GMPS} for details).  Given a manifold without boundary $M$ and a hypersurface $Z\subset M$, one may form the $Z$-tangent bundle $T_ZM$, specified up to isomorphism, whose space of global sections is isomorphic to the space of vector fields $\Gamma_Z(TM)$ that are tangent to $Z$ at points of $Z$.  Dualizing, one obtains the $Z$-cotangent bundle $T_Z^*M := (T_ZM)^*$.  There is a map $i:T_ZM \to TM$ induced by the inclusion $\Gamma_Z(TM) \hookrightarrow \Gamma(TM)$ into the space of vector fields on $M$.  Its dual $i^*$ gives us a map $i^*:T^*M \to T^*_ZM$ and the pushforward of the canonical Poisson structure on $T^*M$ gives us a $b$-symplectic structure on $T^*_ZM$.  In what follows, we dualize this construction to produce vector bundles that \emph{always} have a folded-symplectic structure, albeit perhaps non-canonically.

\begin{remark}
We work with manifolds without boundary for the sake of convenience.  However, in what follows, one need only assume all vector bundles are stratified subbundles of $TM$ in order to generalize to the case of manifolds with corners.  There is one instance where we use a tubular neighborhood theorem for manifolds without boundary (q.v. proof of lemma \ref{lem:Vcotangent3}): it is also true in the case of manifolds with corners but the proof appears to be outside the scope of this document (q.v. \cite{MN} p. 4).
\end{remark}

\subsubsection{Constructing the bundles $T^*_V M$}

\begin{definition}\label{def:Vcotangent}
Let $M$ be a manifold without boundary and let $Z\subset M$ be a closed hypersurface, i.e. a codimension $1$ submanifold of $M$.  Let $V$ be a rank $1$ vector subbundle of $i_Z^*TM$ so that for all $z\in Z$ the fiber $V_z$ over $z$ is transverse to $T_zZ$.  For each open set $U\subset M$ we define:
\begin{displaymath}
\Omega_V^1(U):= \{\alpha \in \Omega^1(U) \vert \mbox{ } \alpha\vert_V = 0\}
\end{displaymath}
to be the space of all $1$-forms on $U$ vanishing on $V$.  If $U\cap Z = \emptyset$, then this is just $\Omega^1(U)$.  The restriction maps are defined using the pullbacks $i^*$ by the inclusion maps $i:U \hookrightarrow V$.
\end{definition}

\begin{lemma}\label{lem:Vcotangent1}
Let $M$ be a manifold without boundary, let $Z\subset M$ be a closed hypersurface, and let $V$ be a rank $1$ vector subbundle of $i_Z^*TM$ transverse to $TZ$.  Then $\Omega_V^1(\cdot)$ is a sheaf of $C^{\infty}(M)$ modules on $M$.
\end{lemma}

\begin{proof} \mbox{ } \newline
$\Omega_V^1(\cdot)$ is a sub-presheaf of $\Omega^1(M)$, hence we need only check the conditions that ensure it's a sheaf.

\begin{itemize}
\item If $U\subset M$ is open, $\{U_i\}$ is a cover of $U$, and we have $\omega \in \Omega_V^1(U)$ such that $\omega\big\vert_{U_i}=0$ for all $i$, then $\omega=0$ since $\omega$ evaluated at any point $p\in U$ must be zero.
\item Using the same cover, if we are given $\omega_i \in \Omega_V^1(U_i)$ for each $i$ so that $\omega_i\big\vert_{U_i\cap U_j} = \omega_j\big\vert_{U_i\cap U_j}$, then we may choose a partition of unity $\{\psi_i\}$ subordinate to $\{U_i\}$, refining the cover to guarantee local finiteness if necessary, and define $\omega= \sum_i \psi_i \omega_i$.  If $p\in Z\cap U$, then for each $v\in V_p$ we have:

    \begin{displaymath}
    \omega_p(v) = \sum_i (\psi_i\omega_i)(v) = 0
    \end{displaymath}
    since either $\psi_i(p)=0$ or $(\omega_i)_p(v)=0$.  Therefore, the $\omega_i's$ glue together to give a unique section $\omega \in \Omega_V^1(U)$.
\end{itemize}
\end{proof}

\begin{lemma}\label{lem:Vcotangent2}
Let $M$ be a manifold without boundary, $Z\subset M$ a closed hypersurface, and $V$ a rank $1$ vector subbundle of $i_Z^*TM$ transverse to $TZ$.  Let $\Omega_V^1(\cdot)$ be the corresponding sheaf of $1$-forms vanishing on $V$ along $Z$.  Then $\Omega_V^1(\cdot)$ is a locally free sheaf of $C^{\infty}(M)$ modules.  Furthermore, for each $z\in Z$ there exists a neighborhood $U$ of $z$ and coordinates $(x_1,\dots, x_{n-1},t)$ on $U$ so that $\Omega_V^1(U)$ is generated by $\{dx_1,\dots, dx_{n-1}, tdt\}$.
\end{lemma}

\begin{proof}
Let $p\in M$ be a point.  If $p\in M\setminus Z$ then there exists a neighborhood $U$ of $p$ such that $U\cap Z = \emptyset$ and an isomorphism $\phi:U\to V$ onto a neighborhood $V$ of the origin in $\R^n$, where $n=\dim(M)$.  $\Omega_V^1(U)$ is then generated by the pullbacks of $dx_1,\dots,dx_n$ under $\phi$.

If $p\in Z$, then we may choose a local, nonvanishing section $w$ of $V$ and extend it to a nonvanishing vector field $\tilde{w}$ in a neighborhood $U$ of $p$.  Furthermore, we may shrink $U$ so that the intersection $U\cap Z$ has coordinates defined on it.  If we shrink $U$ sufficiently and require that its closure is compact, then the flow $\tilde{\phi}$ of $\tilde{w}$ gives us an isomorphism:

\begin{displaymath}
\phi: U\cap Z \times (-\epsilon, \epsilon) \to U
\end{displaymath}
for some $\epsilon > 0$, defined by $\phi(z,t)= \tilde{\phi}(z,t)$.  Here, we are using the fact that $V$ is transverse to $TZ$ and is thus identifiable with the normal bundle $\nu(Z)$ to $Z$.  Furthermore, $d\phi_{(z,0)}(\frac{\partial}{\partial t})$ is a nonzero element of $V_z$, hence the fibers of $V$ are identified with the span of $\frac{\partial}{\partial t}$ along $(U\cap Z)\times \{0\}$.  Choosing coordinates $x_1,\dots, x_{n-1}$ on $U\cap Z$, we have that $\Omega^1_V(U)$ is spanned by the pullbacks of $\{dx_1,\dots dx_{n-1},tdt\}$.
\end{proof}

\begin{cor}\label{cor:Vcotangent}
Let $M$ be a manifold without boundary, $Z\subset M$ a closed hypersurface, and $V\to Z$ a rank $1$ subbundle of $i_Z^*TM$ transverse to $TZ$.  Then there exists a vector bundle $T_V^*M \to M$ whose rank is $\dim(M)$ and whose space of global sections $\Gamma(T_V^*M)$ is isomorphic to $\Omega^1_V(M)$.  Furthermore, the inclusion $i:\Omega^1_V(M) \to \Omega^1(M)$ induces a map $i:T_V^*M \to T^*M$.
\end{cor}

\begin{proof} \mbox{ }\newline
The category of locally free sheaves of $C^{\infty}(M)$ modules on $M$ and the category of vector bundles over $M$ are equivalent by the non-compact version of the Serre-Swan theorem.  For a discussion of this equivalence, see \cite{R}.
\end{proof}

\begin{remark}
The bundle $T_V^*M$ is defined up to isomorphism, hence once we have fixed a representative $T_V^*M$ the map $i:T_V^*M \to T^*M$ is unique up to a precomposition $i\circ \Phi$ with an automorphism $\Phi:T_V^*M \to T_V^*M$.  In coordinates around $p\in Z$, we may frame $T_V^*M$ by $\{dx_1,\dots,dx_{n-1},tdt\}$ (q.v. lemma \ref{lem:Vcotangent2}) and the map $i:T_V^*M\to T^*M$ is:

\begin{displaymath}
i(x_1,\dots,x_{n-1},t,p_1,\dots,p_{n-1},p_n) = (x_1,\dots,x_{n-1},t,p_1,\dots,p_{n-1},tp_n)
\end{displaymath}
Note that this map is \emph{not} a map with fold singularities: while its determinant vanishes transversally, its kernel $\displaystyle \frac{\partial}{\partial p_n}$ is always tangent to the degenerate hypersurface, $\{t=0\}$.
\end{remark}

\subsubsection{Uniqueness up to Isomorphism}
The isomorphism class of $T_V^*M$ is independent of $V$.  Indeed, we have the following lemma:

\begin{lemma}\label{lem:Vcotangent3}
Let $M$ be a manifold without boundary and let $Z\subset M$ be a closed hypersurface.  Suppose we have chosen two vector subbundles $V_1$, $V_2$ of $i_Z^*TM$ satisfying $V_i \pitchfork TZ$ for $i=1,2$.  There exists an isomorphism $\Phi:TM \to TM$ inducing an isomorphism of locally free sheaves $\Phi^*:\Omega_{V_2}^1(\cdot) \to \Omega_{V_1}^1(\cdot)$.
\end{lemma}

\begin{proof}
We produce an isomorphism $\Phi_1:TU \to TU$ in a neighborhood $U$ of $Z$ that induces an isomorphism $\Phi\big\vert_Z:V_1 \to V_2$ and then glue it to the identity map $id:TM \to TM$ using a partition of unity.  For the first step, we'll need to extend the bundles $V_1$ and $V_2$ to neighborhoods of $Z$.

\begin{enumerate}
\item For each $i\in \{1,2\}$, $V_i$ is transverse to $TZ$, hence the tubular neighborhood theorem gives us a neighborhood $U$ of $Z$, neighborhoods $W_i\subseteq V_i$ of the zero sections, and isomorphisms $\phi_i:W_i \to U$ satisfying:

    \begin{itemize}
    \item $\phi_i(z,0)=z$ and
    \item $d\phi_{(z,0)}(e_i) \in (V_i)_z$ for all $e_i$ tangent to the fiber of $V_i$ at $(z,0)$.
    \end{itemize}

    This isomorphism allows us to extend the bundles $V_i$ to the neighborhood $U$ as follows.  If $\pi_i:V_i \to Z$ is the projection, then we define the extension of $V_i$ to be $d\phi_i(\ker(d\pi_i))$, where we restrict $\ker(d\pi_i)$ to $W_i$.

\item Thus, we have a neighborhood $U$ of $Z$ and extensions $\tilde{V}_i$ of $V_i$ to vector subbundles of $TU$.  Using the same technique as in step $1$, we may extend $TZ \to Z$ to a vector subbundle of $TU$ (this requires a choice of connection).  Let us call this extension $E$.  Now, $E$ is complementary to $\tilde{V}_i$ at points of $Z$ for each $i$, hence we may shrink $U$ so that $E$ is complementary to both $\tilde{V}_i's$ on $U$, giving us two splittings:
    \begin{itemize}
    \item $TU \simeq E \oplus \tilde{V}_1$ with projection $p_1:TU \to \tilde{V}_1$ and inclusion $i_1:\tilde{V}_1 \to E\oplus \tilde{V}_1$,
    \item $TU \simeq E \oplus \tilde{V}_2$ with projection $p_2:TU \to \tilde{V}_2$ and inclusion $i_2:\tilde{V}_2 \to E \oplus \tilde{V}_2$.
    \end{itemize}
    We may use these maps to define a an isomorphism $\phi_1:\tilde{V}_1 \to \tilde{V}_2$ given by $p_2 \circ i_1$.  To see this, note that at each point $p\in U$, $p_2(i_1(v))= 0$ if and only if $i_1(v)\in E_p$ if and only if $v=0$, hence it is an injective map of line bundles and is therefore an isomorphism.
\item Now, define $\Phi_1:TU \to TU$ by $\Phi_1 = Id_E \oplus \phi_1$.
\item We finish the construction by choosing a partition of unity $\psi_1, \psi_2$ subordinate to the cover $\{U, M\setminus Z\}$ and defining $\Phi:TM \to TM$ by $\Phi= \psi_1(\Phi_1) + \psi_2 Id$.  $\Phi$ is an isomorphism of vector bundles since it is an isomorphism over $M\setminus U$ and on $U$ we have:
    \begin{itemize}
    \item $\Phi\vert_E = Id_E$ and
    \item $\Phi\vert_{V_1} = \phi_1$
    \end{itemize}
    hence $\Phi\vert_U$ is an isomorphism of $TU$.
\end{enumerate}
Now, note that along $Z$, $\Phi:V_1 \to V_2$ is an isomorphism.  Therefore, $\Phi^*:T^*M \to T^*M$ induces a map of sections $\Phi^*:\Omega^1(M) \to \Omega^1(M)$ which restricts to a map of $C^{\infty}(M)$-modules $\Phi^*: \Omega_{V_1}^1(M) \to \Omega_{V_2}^1(M)$.  Since $\Phi$ is invertible, $\Phi^*$ is invertible and is therefore an isomorphism of modules.  Restricting $\Phi^*$ to open subsets of $M$ gives the requisite isomorphism of sheaves.
\end{proof}

\begin{cor}\label{cor:Vcotangent1}
Let $M$ be a manifold without boundary, let $Z\subset M$ be a closed hypersurface, and let $V_1$, $V_2$ be two rank $1$ vector subbundles of $i_Z^*TM$ transverse to $TZ$.  Then there exists an isomorphism $\Phi: T_{V_2}^*M \to T_{V_1}^*M$.
\end{cor}

\begin{proof} \mbox{ } \newline
By lemma \ref{lem:Vcotangent3}, there exists an isomorphism of locally free sheaves of $C^{\infty}(M)$ modules $\Phi^*:\Omega_{V_2}^1(\cdot) \to \Omega_{V_1}^1(\cdot)$.  Since the category of locally free sheaves of $C^{\infty}(M)$ modules and the category of vector bundles over $M$ are equivalent (q.v. \cite{R}), this map induces an isomorphism between any choice of representatives $T_{V_1}^*M$ and $T_{V_2}^*M$ of the vector bundles isomorphism classes assocatiated to $\Omega_{V_1}^1(M)$ and $\Omega_{V_2}^1(M)$, respectively.
\end{proof}

\subsubsection{Folded Structures on $T_V^*M$}
\begin{lemma}\label{lem:Vcotangent4}
Let $M$ be a manifold without boundary, let $Z\subset M$ be a closed hypersurface, and let $V$ be a rank $1$ subbundle of $i_Z^*TM$ transvserse to $TZ$.  Let $\pi_V:T_V^*M\to M$ be a representative of the isomorphism class of vector bundles associated to $\Omega_V^1(M)$ and let $i:T_V^*M \to T^*M$ be the map induced by the inclusion $\Omega_V^1(M) \to \Omega^1(M)$.  Let $\omega_{T^*M}$ be the canonical symplectic structure on $T^*M$.  Then $i^*\omega_{T^*M}$ is a folded-symplectic form with folding hypersurface $\pi_V^{-1}(Z)$.
\end{lemma}

\begin{proof} \mbox{ } \newline
Away from $Z\subset M$, $i$ is an isomorphism of vector bundles, hence it is a local diffeomorphism, and $i^*\omega_{T^*M}$ is nondegenerate.  At a point $p\in Z$, we have seen that we may choose coordinates $x_1,\dots,x_{n-1}, t$ near $p$ so that $V$ is framed by $\frac{\partial}{\partial t}$ along $Z=\{t=0\}$ and $T_V^*(M)$ is framed by $\{dx_1,\dots, dx_{n-1},tdt\}$.  That is, $T_V^*M$ is isomorphic to a trivial vector bundle spanned by these sections.  Furthermore, in these coordinates, the map $i$ is given by:

\begin{displaymath}
i(x_1,\dots,x_{n-1},t,p_1,\dots, p_{n-1},p_n) = (x_1,\dots,x_{n-1},t,p_1,\dots,p_{n-1},tp_n)
\end{displaymath}
In these coordinates, the canonical symplectic form $\omega_{T^*M}$ is $\omega_{T^*M}= (\sum_{i=1}^{n-1}dq_i\wedge dx_i) + dq_n\wedge dt$ and we have:

\begin{displaymath}
i^*\omega_{T^*M}= (\sum_{i=1}^{n-1}dp_i\wedge dx_i) + tdp_n \wedge dt
\end{displaymath}
which is the folded-symplectic form of example \ref{ex:fs3} whose folding hypersurface (in these coordinates) is $\{t=0\}=\pi_V^{-1}(Z)$.  Thus, $i^*\omega_{T^*M}$ is a folded-symplectic form whose folding hypersurface is $\pi_V^{-1}(Z)$.  As a final note, $i^*\omega_{T^*M}$ vanishes on $\frac{\partial}{\partial t}$ when $t=0$, illustrating that, locally, null directions transverse to the folding hypersurface are given by the fibers of $V$.
\end{proof}

\begin{remark}
As a closing remark to the construction, consider the map $i:T_V^*M \to T^*M$.  We know that $T_V^*M$ is unique up to isomorphism, hence the map $i$ is canonical up to a precomposition with an automorphism of $T_V^*M$.  Precomposing with any automorphism $\Phi:T_V^*M\to T_V^*M$ gives us a new map $i\circ \Phi: T_V^*M \to T^*M$ which again induces a folded-symplectic structure $(i\circ \Phi)^*\omega_{T^*M}$.  However, $(\Phi^{-1})^*$ is a folded-symplectomorphism between $(i\circ \Phi)^*\omega_{T^*M}$ and $i^*\omega_{T^*M}$, hence the folded-symplectic structure we have induced on $T_V^*M$ is canonical up to folded-symplectomorphism.
\end{remark}

\subsection{Isotopies Between Fold Forms}
Now that we know where one can find a plethora of examples of folded-symplectic manifolds, let us determine how to construct isomorphisms between folded-symplectic structures.  Our method of choice will be isotopies.  Suppose we are given a folded-symplectic manifold with corners $(M,\sigma)$ and a $1$-form $\beta$.  When can we solve the equation

\begin{equation}\label{eq:Moser1}
i_X\sigma = \beta
\end{equation}
for some smooth vector field $X$ on $M$?  Certainly, if $\sigma$ has a nonempty folding hypersurface $Z$, then one cannot always solve equation \ref{eq:Moser1} for $X$.  Indeed, let $M=\R^2$, $\sigma=x_1dx_1\wedge dx_2$, and $\beta= dx_1$.  Then the solution to equation \ref{eq:Moser1} is $X=\displaystyle \frac{1}{x_1}\frac{\partial}{\partial x_2}$, which is undefined at the fold $Z=\{x_1=0\}$.  The reasons for understanding equation \ref{eq:Moser1} are two-fold:

\begin{enumerate}
\item First, we would like to understand when a closed $2$-form is folded-symplectic in terms of maps of vector bundles, sheaves, or modules.  For example, a closed $2$-form $\omega\in \Omega^2(M)$ is symplectic if and only if the map $\omega^\#(X)=i_X\omega$ from $TM$ to $T^*M$ is an isomorphism of vector bundles.  Equivalently, $\omega$ induces an isomorphism of sheaves $\omega^\#:\operatorname{Vec}(\cdot) \to \Omega^1(\cdot)$ or $C^{\infty}(M)$ modules, $\omega^\#:\operatorname{Vec}(M) \to \Omega^1(M)$.  What is the corresponding condition for folded-symplectic forms?

\item Secondly, and perhaps most importantly, we would like to develop a Moser-type argument to produce folded-symplectomorpisms from families of fold forms, thereby providing a convenient way to construct folded-symplectomorphisms.  We'll use the discussion of $1$ to determine when solutions to Moser's equation (eq. \ref{eq:Moser1}) exist.
\end{enumerate}
\subsubsection{Isotopies and Moser's argument}
We begin by reviewing Moser's argument in the context of symplectic geometry (q.v. \cite{MS} for details).  Given a manifold with corners $M$ and a smooth path of symplectic forms $\omega_t\in \Omega^2(M)$, we would like to produce a family of diffeomorphisms $\phi_t:M \to M$ satisfying:

\begin{equation}\label{eq:Moser2}
\psi_t^*\omega_t=\omega_0
\end{equation}
Generally, this cannot be done.  But, if we assume the derivative is exact:

\begin{equation}\label{eq:Moser3}
\frac{d}{dt}\omega_t=d\beta_t
\end{equation}
where $\beta_t\in \Omega^1(M)$ is a smooth path of $1$-forms, then we have a better chance of producing $\psi_t$ by representing it as a flow of a time-dependent vector field $X_t$ on $M$.  If we assume $\displaystyle \frac{d}{dt}\psi_t = X_t\circ \psi_t$ and $\psi_0=Id_M$, then differentiating equation \ref{eq:Moser3} gives us:

\begin{equation}\label{eq:Moser4}
0=\psi_t^*(\frac{d}{dt}\omega_t + \pounds_{X_t}\omega_t) = \psi_t^*(\dot{\omega}_t + d(i_{X_t}\omega_t))
\end{equation}
Since $\psi_t$ is assumed to be a diffeomorphism, this will be true if and only if:

\begin{equation}\label{eq:Moser5}
\dot{\omega}_t + d(i_{X_t}\omega_t) =0 \mbox{ } \iff \mbox{ } d(i_{X_t}\omega_t) = -d\beta_t
\end{equation}
where we are using the assumption $\dot{\omega}_t = d\beta_t$.  Equation \ref{eq:Moser5} may be solved by solving:

\begin{equation}\label{eq:Moser6}
i_{X_t}\omega_t = -\beta_t
\end{equation}
Since we are assuming $\omega_t$ is symplectic for all $t$, we have a smooth path of vector bundle isomorphisms $\omega_t^\#:TM \to T^*M$ with smooth inverse $\omega_t^{\flat}:T^*M \to TM$.  Equation \ref{eq:Moser6} is then solved by setting $X_t = -\omega_t^{\flat}(\beta_t)$.  If $X_t$ is stratified and $M$ is compact, we may integrate $X_t$ to obtain a path of diffeomorphisms $\psi_t:M\to M$ satisfying $\psi_t^*\omega_t=\omega_0$.

In general, $M$ is not compact and so $X_t$ needn't have a globally defined flow $\psi_t$.  To circumvent this problem, we use the following lemma:

\begin{lemma}\label{lem:Moser1}
Let $M$ be a manifold with corners and suppose $X_t$ is a stratified, time-dependent vector field on $M$ (i.e. it is tangent to the boundary strata at points of $\partial(M)$).  Let $N\subset M$ be a submanifold with corners so that $X_t(p)=0$ for all $p\in N$ and for all $t$.  Then there exists a neighborhood $U$ of $N$ on which the flow $\psi_t$ is defined for all $t\in [0,1]$.  Thus, we have a smooth path of open embeddings $\psi_t:U \to M$ satisfying $\psi_t(n)=n$ for all $n\in N$.
\end{lemma}

\begin{remark}
We impose the condition that $X_t$ is stratified so that its flow defines diffeomorphisms of manifolds with corners.
\end{remark}

\begin{proof}
The trick is to pass to the product $M\times \R$ and consider the vector field $Y=(X_t, \frac{\partial}{\partial t})$.  We claim that its flow will give us the time-dependent flow of $X_t$.  Its flow $F_s$ has the form
\begin{displaymath}
F_s(m,t)=(G_s(m,t), t+s)
\end{displaymath}
and we have:
\begin{displaymath}
F_{-s}(F_s(m,t))=(G_{-s}(G_s(m,t), t+s),t)=(m,t)
\end{displaymath}
Hence, for fixed $s$ and $t$ the map $m \to G_s(m,t)$ is a diffeomorphism with smooth inverse given by $y \to G_{-s}(y, s+t)$.  Furthermore, we have:
\begin{displaymath}
\frac{d}{ds}(G_s(m,0),s)=\frac{d}{ds}F_s(m,0) = (Y \circ F_s)(m,0) = (X_s(G_s(m,0)), \frac{\partial}{\partial t})
\end{displaymath}
hence $\displaystyle \frac{d}{ds}G_s(m,0) = X_s(G_s(m,0))$ and so $m \to G_s(m,0)$ is the flow of $X_s$.  Now, if the flow $F_s$ exists at a point $(m,0)$ for all $s\in [0,1]$, then there exists a neighborhood $V_m$ of $(m,0)$ so that $F_s$ is defined on $U$ for all $t\in[0,1]$ since the flow is defined on an open domain and $(m,0)\times [0,1]$ is compact.  If we define $U_m=V_m\cap (M\times \{0\})$, then $U_m$ defines a neighborhood of $m\in M$ on which the flow $G_s(p,0)$ exists for all time $s\in [0,1]$.  We form such neighborhoods for each $m\in N$ and take the union $U=\bigcup_{m\in N} U_m$.  Then the flow $\psi_s(p):=G_s(p,0)$ is defined for all time $s\in [0,1]$ on $U$.  For each $s$, $\psi_s:U \to M$ is an open embedding into $M$.  Furthermore, since the vector field $X_s$ satisfies $X_s(m)= 0$ for all $m\in N$, $\psi_s(m)=G_s(m,0)=m$ for all $s$, hence $\psi_s\vert_N = Id_N$.
\end{proof}

Now, returning to Moser's argument, let us consider the folded-symplectic setting.  Suppose we have a $2m$-dimensional manifold with corners $M$ and a path $\sigma_t\in \Omega^2(M)$ of folded-symplectic forms satisfying:

\begin{displaymath}
Z=((\sigma_t)^m)^{-1}(\mathcal{O}) \text{ is independent of $t$.}
\end{displaymath}
That is, we assume the folding hypersurface is fixed in time.  Furthermore, let us assume that the kernel bundles $\ker(\sigma_t)\to Z$ have fibers tangent to the faces of $M$ for all $t$.  Following the setup for Moser's argument in the symplectic setting, we assume:

\begin{displaymath}
\frac{d}{dt}\sigma_t = d\beta_t \text{ for a family of smooth $1$-form $\beta_t\in \Omega^1(M)$.}
\end{displaymath}
If we begin with the equation $\psi_t^*\sigma_t = \sigma_0$ and differentiate, then we arrive at the equation:
\begin{equation}\label{eq:Moser7}
-d\beta_t = di_{X_t}\sigma_t
\end{equation}
where $X_t$ is the time-dependent vector field generating $\psi_t$.  It is therefore sufficient to solve:
\begin{equation}\label{eq:Moser8}
-\beta_t=i_{X_t}\sigma_t \text{ (just as in the symplectic case)}
\end{equation}
Thus, we have arrived at the question we originally posed:
\begin{itemize}
\item Given a folded-symplectic form $\sigma$ and a $1$-form $\beta$, can we find a smooth vector field $X$ such that $i_X\sigma=\beta$?
\end{itemize}
This is the subject of the next section.
\subsubsection{Solving $i_X\sigma= \beta$ to Obtain Isotopies}
Suppose $(M,\sigma)$ is a folded-symplectic manifold with corners with folding hypersurface $Z\subset M$.  Let $\ker(\sigma)\to Z$ be the kernel bundle over $Z$ and suppose it is stratified.  That is, suppose its fibers are tangent to the boundary strata of $M$.  Similar to the construction of the folded cotangent bundle $T_V^*M$, we will construct a module of $1$-forms vanishing on $\ker(\sigma)$ at points of $Z$, which will give us a vector bundle $T_{\ker(\sigma)}^*M$.  We will use this vector bundle to characterize folded-symplectic forms on $M$ whose fold is $Z$ and whose kernel bundle is $\ker(\sigma)$.

\begin{remark}
In much of what follows, it is not necessary to assume that $\ker(\sigma)$ is stratified.  It will only be used when we show that $\sigma^\#:TM \to T_{\ker(\sigma)}^*M$ is an isomorphism if and only if $\sigma$ is folded-symplectic.  Furthermore, we needn't begin any constructions assuming we have a folded-symplectic form.  The main ingredients in what follows are:
\begin{itemize}
\item A hypersurface $Z\subset M$ inside a manifold with corners, transverse to strata.
\item A $2$-plane bundle $E \to Z$ so that $E\cap TZ$ is a rank $1$ vector bundle.
\end{itemize}
which is devoid of any folded-symplectic geometry: we just have a hypersurface and a vector bundle.  One may then study all forms whose degenerate set is $Z$ and whose kernel distribution on $Z$ contains the bundle $E\to Z$.  What we will effectively show is that if $\sigma$ is such a form, then it is folded-symplectic with folding hypersurface $Z$ and kernel bundle $E\to Z$ if and only if it induces an isomorphism of sheaves $\sigma^\#:\operatorname{Vec}(\cdot) \to \Omega^1_E(\cdot)$, where $\operatorname{Vec}(\cdot)$ is the sheaf of vector fields and $\Omega^1_E(\cdot)$ is the sheaf of $1$-forms vanishing on $E$ at $Z$.  However, on a first run through the constructions, it is perhaps easier to begin with a folded-symplectic form and a fixed kernel bundle $\ker(\sigma)$ so that one has a grounding in the folded-symplectic world.
\end{remark}

\begin{definition}\label{def:Ecotangent1}
Let $(M,\sigma)$ be a folded-symplectic manifold with corners with folding hypersurface $Z\subset M$.  Let $\ker(\sigma)\to Z$ be the kernel bundle over the fold $Z$ and suppose it is stratified.  That is, for all $p\in Z$, $\ker(\sigma)_p$ is tangent to the stratum of $M$ containing $p$.  We define a presheaf $\Omega_{\ker(\sigma)}^1(\cdot)$ of $C^{\infty}(M)$ modules on $M$ as follows.  Given an open subset $U\subset M$,

\begin{displaymath}
\Omega_{\ker(\sigma)}^1(U) = \{\beta \in \Omega^1(U) \vert \mbox{ } \beta\big\vert_{\ker(\sigma)} = 0 \text{ at $Z\cap U$}\}
\end{displaymath}
to be the space of $1$-forms on $U$ that vanish on $\ker(\sigma)$ at points of $Z\cap U$.  The restriction maps are given by pullbacks $i^*$ induced by inclusions $i:V \hookrightarrow U$.
\end{definition}

\begin{remark}
We take a moment to make sense of everything stated in the definition.  Note that vanishing on $\ker(\sigma)$ is a linear condition, hence $\Omega_{\ker(\sigma)}^1(U)$ is a submodule of $\Omega^1(U)$ for all $U\subset M$. Since the restriction maps of $\Omega_{\ker(\sigma)}^1(\cdot)$ are induced by those of $\Omega^1(\cdot)$, we have that $\Omega_{\ker(\sigma)}^1(\cdot)$ is a sub-presheaf of $\Omega^1(\cdot)$, hence it is a presheaf.
\end{remark}

\begin{lemma}\label{lem:Ecotangent1}
Let $(M,\sigma)$ be a folded-symplectic manifold with corners with folding hypersurface $Z\subset M$.  Suppose that the fibers of $\ker(\sigma)\to Z$ are tangent to the strata of $M$.  Then the presheaf $\Omega^1_{\ker(\sigma)}(\cdot)$ of forms vanishing on $\ker(\sigma)$ is a sheaf of $C^{\infty}(M)$ modules on $M$.
\end{lemma}

\begin{proof}
Let $U\subseteq M$ be open and let $\{U_i\}_{i\in I}$ be an open cover of $U$.
\begin{enumerate}
\item Suppose we have $\tau \in \Omega_{\ker(\sigma)}^1(U)$ such that $\tau\big\vert_{U_i} =0 $ for each $U_i$.  Then $\tau=0$ since it is $0$ at every point of $U$.
\item Suppose we have a collection of forms $\{\tau_i\in \Omega_{\ker(\sigma)}^1(U_i)\}_{i\in I}$ satisfying:
    \begin{displaymath}
    \tau_i\big\vert_{U_i\cap U_j} = \tau_j\big\vert_{U_i\cap U_j}
    \end{displaymath}
    for all $i,j\in I$.  Then, refining the cover if necessary, we may choose a partition of unity $\{\psi_i\}$ subordinate to the cover $\{U_i\}$ and define:
    \begin{displaymath}
    \tau:= \sum_{i\in I} \psi_i \tau_i
    \end{displaymath}
    Since vanishing on $\ker(\sigma)$ is a linear condition, $\tau$ vanishes on $\ker(\sigma)$ at points of $U\cap Z$.
\end{enumerate}
\end{proof}

\begin{lemma}\label{lem:Ecotangent2}
Let $(M,\sigma)$ be a manifold with corners with folding hypersurface $Z\subset M$ and let $\ker(\sigma) \to Z$ be the kernel bundle of $\sigma$.  Suppose that the fibers of $\ker(\sigma)$ are tangent to the strata of $M$.  Let $\Omega^1_{\ker(\sigma)}(\cdot)$ be the sheaf of $1$-forms vanishing on $\ker(\sigma)$ along $Z$.  Then $\Omega^1_{\ker(\sigma)}(\cdot)$ is a locally free sheaf of $C^{\infty}(M)$ modules.
\end{lemma}

\begin{proof}
Let $2m=\dim(M)$.  Suppose $p\in M\setminus Z$.  Then there exists a neighborhood $U$ of $p$ diffeomorphic to an open subset of the quadrant $(\R^+)^{2m}$ satisfying $U\cap Z=\emptyset$.  Then $\Omega_{\ker(\sigma)}^1(U)$ is spanned by the pullbacks of $dx_1,\dots,dx_{2m}$.

\vspace{5mm}

Now, suppose that $p\in Z$.  Since $Z$ is transverse to the boundary strata $\partial^k(M)$ for all $k\ge 0$, we may choose a stratified vector field transverse to $Z$ whose flow gives us a diffeomorphism of a neighborhood $U$ of $p$ with $U\cap Z \times (-\epsilon, \epsilon)$, where $U\cap Z$ is identified with the zero section $U\cap Z \times \{0\}$.  This gives us coordinates $(z,t)$ around $p$, where $z$ lies in the folding hypersurface.  Choose a local frame $\{f_1,f_2\}$ of $\ker(\sigma)$ and extend it to a local frame $\{e_1,\dots, e_{2m-2}, f_1,f_2\}$ of $i_Z^*TM$ on $U\cap Z$.  Extend this frame to a frame on $U\cap Z \times (-\epsilon, \epsilon)$ by lifting the vector fields on $U\cap Z$ back to the product.  We redefine the $e_i$'s and $f_i$'s to be these vector fields.  Then the set:

\begin{displaymath}
\{e_1^*, \dots, e_{2m-2}^*, tf_1^*, tf_2^*\}
\end{displaymath}
is a basis for $\Omega_{\ker(\sigma)}^1(U)$, where $t$ is the coordinate on $(-\epsilon, \epsilon)$.  This is because $\{e_1^*,\dots,e_{2m-2}^*, f_1^*, f_2^*\}$ is a coframe for $T^*U$ and any form vanishing on $\ker(\sigma)$ at $t=0$ must have vanishing $f_1^*$ and $f_2^*$ terms at $t=0$.
\end{proof}

\begin{cor}\label{cor:Ecotangent1}
Let $(M,\sigma)$ be a folded-symplectic manifold with corners, let $Z\subset M$ be its folding hypersurface, let $\ker(\sigma)\to Z$ be the kernel bundle of $\sigma$ and suppose it is stratified (its fibers are tangent to the strata of $M$), and let $\Omega_{\ker(\sigma)}^1(\cdot)$ be the locally free sheaf of $1$-forms vanishing on $\ker(\sigma)$ along $Z$.  Then there exists a vector bundle $T_{\ker(\sigma)}^*M$ whose space of global sections is isomorphic to $\Omega_{\ker(\sigma)}^1(\cdot)$.  Furthermore, $\sigma$ induces a map of locally free sheaves:

\begin{displaymath}
\sigma^\#: \operatorname{Vec}(\cdot) \to \Omega_{\ker(\sigma)}^1(\cdot)
\end{displaymath}
given by $X \to i_X\sigma$.  Hence $\sigma$ induces a map $\sigma^\#:TM \to T_{\ker(\sigma)}^*M$.
\end{cor}

\begin{proof} \mbox{ } \newline
Since $\Omega_{\ker(\sigma)}^1(\cdot)$ is a locally free sheaf of $C^{\infty}(M)$ modules, there exists a vector bundle $T_{\ker(\sigma)}^*M$ whose space of sections is isomorphic to $\Omega_{\ker(\sigma)}^1(M)$.  Now, if $X\in \operatorname{Vec}(U)$ is a vector field on an open subset $U\subset M$, then for all $z\in U\cap Z$ and for all $v\in \ker(\sigma)_z$, we have:

\begin{displaymath}
(i_X\sigma)(v)=\sigma(X,v)=0
\end{displaymath}
hence $i_X\sigma \in \Omega_{\ker(\sigma)}^1(U)$, thus $\sigma^\#(X) = i_X\sigma$ defines a map of sheaves since it commutes with restrictions.  It is a map of locally free sheaves since $\sigma^\#$ is $C^{\infty}(U)$ linear for each open subset $U\subset M$.  It therefore induces a map of vector bundles, which we also denote $\sigma^\#$:

\begin{displaymath}
\sigma^\#: TM \to T_{\ker(\sigma)}^*M
\end{displaymath}
\end{proof}

\begin{prop}\label{prop:Ecotangent1}
Let $(M,\sigma)$ be a folded-symplectic manifold with corners, $Z\subset M$ its folding hypersurface, $\ker(\sigma)\to Z$ its kernel bundle, where we assume $\ker(\sigma)$ is stratified (its fibers are tangent to the strata of $M$), and let $\Omega^1_{\ker(\sigma)}(M)$ be the space of $1$-forms vanishing on $\ker(\sigma)$ along $Z$.  Then,

\begin{enumerate}
\item The map $\sigma^\#:\operatorname{Vec}(\cdot) \to \Omega_{\ker(\sigma)}^1(\cdot)$ is an isomorphism of locally free sheaves and therefore induces an isomorphism of vector bundles $\sigma^\#:TM \to T_{\ker(\sigma)}^*M$. In particular, $\sigma^\#:\Gamma(TM) \to \Omega_{\ker(\sigma)}^1(M)$ is an isomorphism of $C^{\infty}(M)$ modules.
\item If $\sigma_0$ is another closed $2$-form on $M$ so that the degenerate set of $\sigma_0$ is $Z$ and $\sigma_0$ vanishes on $\ker(\sigma)$, then $\sigma_0$ is folded-symplectic if and only if the induced map $\sigma_0^\#$ is an isomorphism.
\end{enumerate}
\end{prop}

\begin{remark}
We are showing that $\sigma_0^n \pitchfork_s 0$ and $i_Z^*\sigma_0$ has maximal rank if and only if $\sigma_0^\# :\Gamma(TM) \to \Omega^1_{\ker(\sigma)}$ is an isomorphism of locally free sheaves of $C^{\infty}(M)$ modules.
\end{remark}

\begin{proof}
\begin{enumerate}
\item Let $U\subset M$ be open.  If $U\cap Z=\emptyset$, then there is nothing to prove since $\sigma\vert_U$ is non-degenerate, hence $\sigma^\#\vert_U$ is invertible.  Now, assume $U\cap Z\ne \emptyset$.  We show that $\sigma^\#: \operatorname{Vec}(U) \to \Omega^1_{\ker(\sigma)}(U)$ is injective, hence local solutions to the equation $i_X\sigma = \beta$ are unique.  We use this uniqueness to show surjectivity by gluing local solutions together.
    \begin{enumerate}
    \item Suppose $X,Y \in \operatorname{Vec}(U)$ and $\sigma^\#(X) = \sigma^\#(Y)$.  Since $\sigma^\#$ is a map of sheaves, this identity holds on $U\setminus Z$, which is open and dense in $U$.  We have:

        \begin{displaymath}
        \begin{array}{lcl}
        \sigma^\#(X)=\sigma^\#(Y) & \iff & i_X\sigma = i_Y\sigma \\
                                  & \iff & i_{X-Y}\sigma = 0  \\
                                  & \iff & i_{X-Y}\sigma = 0 \text{ on $U\setminus Z$} \\
                                  & \iff & X-Y = 0  \text{ on $U\setminus Z$ since $\sigma\vert_{U\setminus Z}$ is symplectic}\\
                                  & \iff & X=Y \text{ on $U\setminus Z$} \\
                                  & \iff & X=Y \text{ on $U$ since $U\setminus Z$ is dense.}
        \end{array}
        \end{displaymath}
        Thus $\sigma^\#$ is injective on $U$.
    \item Now, let $\beta \in \Omega_{\ker(\sigma)}^1(U)$.  If $V\subset U$ is an open subset, then the injectivity of $\sigma^\#$ implies that local solutions $i_X\sigma\vert_V = \beta\vert_V$ are unique.  Hence, if $\{V_i\}$ is a cover of $U$ with a set $\{X_i\}$ of solutions to $i_X\sigma\vert_{V_i} = \beta\vert_{V_i}$ on each $V_i$, then we have:

        \begin{displaymath}
        X_i\big\vert_{V_i\cap V_j} = X_j \big\vert_{V_i\cap V_j}
        \end{displaymath}
        hence these local solutions glue to give a global solution $X\in \operatorname{Vec}(U)$ to the equation $i_X\sigma = \beta$.  We show that such a cover of $U$ exists.
        \begin{itemize}
        \item On $U\setminus Z$, $\sigma$ is symplectic and so the equation $i_X\sigma\big\vert_{U\setminus Z} = \beta\big\vert_{U\setminus Z}$ has a solution.
        \item For each $z\in U\cap Z$, $\ker(\sigma)$ being stratified implies that we may choose a neighborhood $V_z$ of $z$ and a diffeomorphism

        \begin{displaymath}
        \phi: V_z\cap Z \times (-\epsilon, \epsilon) \to V_z
        \end{displaymath}
        so that $\phi^*\sigma=p^*i^*\sigma +t\mu$ for some $2$-form $\mu$, where $p:V_z\cap Z \times (-\epsilon, \epsilon) \to V_z\cap Z$ is the projection and $i:V_z\cap Z \to V_z \cap Z \times (-\epsilon, \epsilon)$ is the inclusion as the zero section.

        \vspace{3mm}
        Choose a local frame $f$ of $\ker(i_Z^*\sigma)$ near $p$ and extend it to a local frame $\{e_1,\dots,e_{2m-2},f\}$ of $Z$ near $z$.  We may lift these vector fields to vector fields on the product $V_z\cap Z \times (-\epsilon, \epsilon)$ (using the standard connection), where we use the same notation $e_i$ and $f$, and extend to a local frame:
        \begin{displaymath}
        \{e_1, \dots, e_{2m-2},f,\frac{\partial}{\partial t}\}
        \end{displaymath}
        where $f$ and $\frac{\partial}{\partial t}$ satisfy:

        \begin{displaymath}
        \begin{array}{l}
        i_f\sigma = i_f(p^*i^*\sigma + t\mu) = 0 + ti_f\mu = t(i_f\mu) \\
        i_{\frac{\partial}{\partial t}}\sigma = i_{\frac{\partial}{\partial t}}(p^*i^*\sigma + t\mu) = t(i_{\frac{\partial}{\partial t}}\mu)
        \end{array}
        \end{displaymath}
        Now, since $\sigma$ is folded-symplectic, lemma \ref{lem:symplectize} implies that $\mu$ is non-degenerate on $\ker(\phi^*\sigma)$, hence:

        \begin{displaymath}
        \mu(f,\frac{\partial}{\partial t}) \ne 0
        \end{displaymath}
        in a neighborhood of $(V_z\cap Z) \times \{0\}$.  Hence, the set:

        \begin{displaymath}
        \{i_{e_1}\sigma, \dots, i_{e_{2m-2}}\sigma, i_f\mu, i_{\frac{\partial}{\partial t}}\mu\}
        \end{displaymath}
        is a local frame of the cotangent bundle in a neighborhood of $(z,0)$.  Hence, we may write:

        \begin{displaymath}
        \beta = \sum_{i=1}^{2m-2}a_i(i_{e_i}\sigma) + b(i_f\mu) + c(i_{\frac{\partial}{\partial t}}\mu)
        \end{displaymath}
        for some smooth functions $a_i, b,$ and $c$.  Since $\beta$ vanishes on $f$ and $\frac{\partial}{\partial t}$ at $t=0$, we must have that $b=tb_0$ and $c=tc_0$ for some smooth functions $b_0$ and $c_0$.  Define:

        \begin{displaymath}
        X=\sum_{i=1}^{2m-2} a_ie_i + b_0f + c_0\frac{\partial}{\partial t}
        \end{displaymath}
        Then,

        \begin{displaymath}
        i_X\sigma = \sum_{i=1}^{2m-2} a_i(i_{e_i}\sigma) + tb_0(i_f\mu) + tc_0(i_{\frac{\partial}{\partial t}}\mu) = \beta
        \end{displaymath}
        Thus, for each $z\in U\cap Z$ there exists a neighborhood $U_z$ and a solution $X_z$ of the equation $i_X\sigma =\beta$ on $U_z$.  Since $\{U_z\}_{z\in U\cap Z} \cup U\setminus Z$ covers $U$ and we have solutions on each open subset, we may glue them together by uniqueness to obtain a global solution $X \in \operatorname{Vec}(U)$ to $i_X\sigma = \beta$.  Thus, $\sigma^\#$ is surjective.
        \end{itemize}
    \item Since $\sigma^\#$ is a map of locally free sheaves of $C^{\infty}(M)$-modules and it is an isomorphism on each open subset of $U$ it is an isomorphism of locally free sheaves of $C^{\infty}(M)$-modules.  Therefore, it induces an isomorphism of vector bundles:
    \begin{displaymath}
    \sigma^\#: TM \to T_{\ker(\sigma)}^*(M)
    \end{displaymath}
    \end{enumerate}
\item Now, suppose $\sigma_0$ is another closed $2$-form on $M$ that degenerates at all points of $Z$ (and only at $Z$) and vanishes on $\ker(\sigma)$.  Part $1$ of the proposition implies that $\sigma_0^\#:\operatorname{Vec}(\cdot) \to \Omega^1_{\ker(\sigma)}(\cdot)$ is an isomorphism if $\sigma_0$ is folded-symplectic.  We must show the \emph{only if} portion, so we assume that $\sigma_0^\#$ is an isomorphism.  We first show that $i_Z^*\sigma_0$ has maximal rank and then use this to argue it is folded-symplectic.
    \begin{enumerate}
    \item Consider a point $p\in Z$.  Since $\ker(\sigma)$ is stratified and transverse to $Z$, we may choose a neighborhood $U$ of $p$ and assume that $M = U\cap Z \times (-\epsilon,\epsilon)$ where $U\cap Z$ is identified with the zero section and $\frac{\partial}{\partial t}$ is in $\ker(\sigma)$ (q.v. lemma \ref{lem:fsnormal}).  Hence, we may assume $M=Z\times (-\epsilon, \epsilon)$ and drop the intersection notation $U\cap Z$ from our discussion.  Since this is a local calculation, we may also assume that $Z$ is contractible so that any vector bundle over $M$ is trivializable.  Lastly, since $\sigma_0$ vanishes on $\ker(\sigma)$, it vanishes on $\frac{\partial}{\partial t}$ so that:

        \begin{displaymath}
        \sigma_0 = p^*i^*\sigma_0 + t\mu
        \end{displaymath}

        \vspace{3mm}
        Choose a local frame $f$ of $\ker(i_Z^*\sigma)$, extend it to a local frame $\{e_1,\dots, e_{2m-2},f\}$ of $Z$, and then extend this to a local frame of $M$ by lifting the vector fields to the product $Z\times \R$ using the standard connection.  We are first going to show that $i^*\sigma_0$, the restriction of $\sigma_0$ to $Z$, has maximal rank.

        Since $e_i^*$ vanishes on $\ker(\sigma)$ at $Z\times \{0\}$, it is an element of $\Omega^1_{\ker(\sigma)}(M)$.  Since $\sigma_0^\#$ is an isomorphism (by assumption), we know there exists a vector field $X_i$ satisfying:

        \begin{displaymath}
        i_{X_i}\sigma_0=e_i^*
        \end{displaymath}

        If we write $\displaystyle X_i=\sum_{j=1}^{2m-2} a_{ij}e_j + b_if + c_i\frac{\partial}{\partial t}$ and use the fact that $\displaystyle i_{\frac{\partial}{\partial t}}\sigma_0\big\vert_Z = i_f \sigma_0\big\vert_Z=0$, we have:

        \begin{equation}\label{eq:surjects}
        i_{X_i}\sigma_0 = e_i^* \to \mbox{ } i^*(i_{X_i}\sigma_0) = e_i^* \to \mbox{ } i_{\sum a_{ij}e_j}(i^*\sigma_0) = e_i^*\big\vert_Z \in \Gamma(TZ)
        \end{equation}
        where $i:Z\hookrightarrow Z\times \R$ is the inclusion as the zero section.  Since $\sum a_{ij}e_j$ a section of $TZ$ at points of $Z\times \{0\}$, equation \ref{eq:surjects} shows that $e_i^*$ is in the image of the contraction mapping $(i^*\sigma_0)^\#:TZ \to T^*Z$.  Since we can choose any $e_i$ and equation \ref{eq:surjects} holds, we have that the map $(i^*\sigma_0)^\#:TZ \to T^*Z$ sends each fiber of $TZ$ onto a $2m-2$-dimensional subspace of the corresponding fiber of $T^*Z$, meaning $i^*\sigma_0$, the restriction of $\sigma_0$ to the fold, has maximal rank.

    \item Now, recall that we have written $\sigma_0$ as $\sigma_0=p^*i^*\sigma_0 + t\mu$ for some $2$-form $\mu$.  Since $i^*\sigma_0$ has maximal rank, lemma \ref{lem:symplectize} implies that $\sigma_0$ is folded if and only if $\mu\big\vert_{\ker(\sigma)}$ is non-degenerate, hence we would like to show $\mu$ is non-degenerate.  We are again assuming that we have a frame $\{e_1,\dots,e_{2m-2},f,\frac{\partial}{\partial t}\}$, where $f\in \ker(\sigma)$ at points of $Z\times \{0\}$.  Assume towards a contradiction that there is a point $(p,0)\in Z\times \{0\}$ where $\mu(f,\frac{\partial}{\partial t})=0$.  Then we claim there is no vector field $X$ satisfying $i_X\sigma_0 = tdt$.  Indeed, $X$ would need to satisfy:

        \begin{itemize}
        \item $(i_X\sigma_0)(e_i)=0$ for all $i$ and
        \item $(i_X\sigma_0)(\frac{\partial}{\partial t}) = t\mu(X,\frac{\partial}{\partial t})=t$
        \end{itemize}
        The first condition implies that at points $(z,0)$ we have $X\in \ker(\sigma)$, hence $\displaystyle X=af + b\frac{\partial}{\partial t}$ for some constants $a$ and $b$ at $(z,0)$.  The second condition implies that $\mu(X,\frac{\partial}{\partial t})=1$, hence $a\ne 0$.  At $(z,0)$, we then have that
        \begin{displaymath}
        1=\mu(X,\frac{\partial}{\partial t})=a\mu(f,\frac{\partial}{\partial t})=0
        \end{displaymath}
        since we are assuming $\mu$ is degenerate on the $\ker(\sigma)$ at $(z,0)$.  Thus, $\mu(X,\frac{\partial}{\partial t})=0\ne 1$, which means $tdt$ cannot be in the image of $\sigma_0^\#$, which contradicts the assumption that $\sigma_0^\#: \operatorname{Vec}(\cdot)\to \Omega^1_{\ker(\sigma)}(\cdot)$ is an isomorphism of sheaves since it cannot be surjective near $(z,0)$.

        \vspace{5mm}

        This means that $\mu\big\vert_{\ker(\sigma)}$ is non-degenerate and lemma \ref{lem:symplectize} implies that $\sigma_0=p^*i^*\sigma_0 + t\mu$ is folded-symplectic in a neighborhood of $Z\times \{0\}$.  Since this is true at each point $z\in Z$ and, by assumption, $\sigma_0$ only degenerates at $Z$, we have that $\sigma_0$ is folded-symplectic on $M$.
    \end{enumerate}
\end{enumerate}
\end{proof}

\begin{cor}\label{cor:Ecotangent2}
Let $(M,\sigma)$ be a folded-symplectic manifold with corners with folding hypersurface $Z\subset M$ and stratified kernel bundle $\ker(\sigma)\to Z$ (i.e. the fibers of $\ker(\sigma)$ are tangent to the boundary strata of $M$).  Let $\beta \in \Omega^1(M)$ be a $1$-form.  Then there is a smooth vector field $X\in \Gamma(TM)$ satisfying $i_X\sigma = \beta$ if and only if $\beta$ vanishes on $\ker(\sigma)$ at points of $Z$.  That is, if and only if $\beta \in \Omega^1_{\ker(\sigma)}(M)$.
\end{cor}

\begin{proof} \mbox{ } \newline
By proposition \ref{prop:Ecotangent1}, the map $\sigma^\#:\operatorname{Vec}(\cdot) \to \Omega^1_{\ker(\sigma)}(\cdot)$ is an isomorphism of sheaves, hence $\sigma^\#:\Gamma(TM) \to \Omega^1_{\ker(\sigma)}(M)$ is an isomorphism of global sections.  Thus, $i_X\sigma = \beta$ if and only if $\sigma^\#(X)=\beta$ if and only if $\beta \in \Omega^1_{\ker(\sigma)}(M)$.
\end{proof}

\begin{prop}\label{prop:Moser}
Let $(M,\sigma)$ be a $2m$-dimensional folded-symplectic manifold with corners and suppose $\sigma_t$ is a smooth path of folded-symplectic forms satisfying:

\begin{enumerate}
\item $\sigma_0 = \sigma$,
\item The fold $Z=(\sigma_t^m)^{-1}(\mathcal{O})$ is independent of $t$, and
\item The bundles $\ker(\sigma_t)$ are stratified (their fibers are tangent to the boundary strata of $M$).
\end{enumerate}

Then, if $\dot{\sigma}_t=d\beta_t$ for some path of $1$-forms $\beta_t$ so that $\beta_t \in \Omega^1_{\ker(\sigma_t)}(M)$ for all $t$, there exists a smooth time-dependent vector field $X_t$ satisfying $i_{X_t}\sigma_t = -\beta_t$.  If $X_t$ is stratified and vanishes at $Z$ for all $t$, then we may integrate $X_t$ in a neighborhood $U$ of $Z$ to obtain a path of open embeddings:

\begin{displaymath}
\phi_t:U \to M
\end{displaymath}
satisfying $\phi_t(z)=z$ and $\phi_t^*\sigma_t = \sigma_0 =\sigma$ for all $t\in [0,1]$.
\end{prop}

\begin{proof} \mbox{ } \newline

For each $t\in \R$, $\sigma_t$ defines an isomorphism of $C^{\infty}(M)$ modules:

\begin{displaymath}
\sigma_t^\#: \Gamma(TM) \to \Omega_{\ker(\sigma_t)}^1(M)
\end{displaymath}
by proposition \ref{prop:Ecotangent1}.  Since $\beta_t\in \Omega_{\ker(\sigma_t)}^1(M)$ for all $t\in \R$, we have that $X_t:=(\sigma_t^\#)^{-1}(-\beta_t)$ defines a smooth vector field on $M$ for each $t\in \R$.  Even though the target space of $\sigma_t^\#$ varies in $t$, we claim that $X_t$ is also smooth in $t$.  However, we put this issue on hold momentarily and address the remaining claim of the proposition: assume $X_t$ is smooth, stratified, and vanishes on $Z$.  Then lemma \ref{lem:Moser1} implies that there exists a neighborhood $U$ of $Z$ on which the flow $\phi_t$ of $X_t$ exists for all $t\in [0,1]$.  $X_t$ satisfies:

\begin{displaymath}
i_{X_t}\sigma=-\beta_t
\end{displaymath}
which is equation \ref{eq:Moser8}, hence its flow satisfies $\phi_t^*\sigma_t = \sigma_0$.

\vspace{4mm}

We now address the smoothness of $X_t$.  At first glance, this would appear to be obvious since $\sigma_t^\#$ and $\beta_t$ are both smooth in $t$.  However, the target space $\Omega_{\ker(\sigma_t)}^1(M)$ varies in $t$, so it isn't immediately clear that the composite $(\sigma_t^\#)^{-1}(\beta_t)$ is smooth in $t$ unless we say something like ``$\Omega_{\ker(\sigma_t)}^1(M)$ varies smoothly in $t$.''  But, its not clear what we mean by such a statement unless we have built some machinery.  If the reader is convinced that $X_t$ is smooth because $\Omega_{\ker(\sigma_t)}^1(M)$ is a smoothly varying submodule of $\Omega^1(M)$, then he or she may pass on to the next section.

\vspace{4mm}

We show smoothness in $t$ by adapting our constructions on $M$ to the product $M\times \R$.  Also, we only check smoothness of $X_t$ in $t$ near points of the fold $Z\subset M$ since each $\sigma_t$ is symplectic away from $Z$, hence $(\sigma_t^\#)^{-1}(-\beta_t)$ is smooth away from $Z$.  The family of $1$-forms $\sigma_t$ defines a smooth map $\Phi$ of vector bundles over $M\times \R$:

\begin{displaymath}
\xymatrix{
TM \times \R \ar[r]^\Phi \ar[d]^{\tau_M\times Id} & T^*M\times \R \ar[d]^{\tau^*_M \times Id} \\
M\times \R \ar[r]                                 &   M\times \R
}
\end{displaymath}
where $\Phi$ is defined at a point $(m,t)$ as $\Phi(m,X,t) = i_X(\sigma_t)_m$, the contraction of $X$ with $\sigma_t$.  Let $\tilde{\sigma}$ be the $2$-form $\tilde{\sigma}\in \Omega^2(M\times \R)$ defined at a point $(m,t)$ by:

\begin{displaymath}
\tilde{\sigma}_{(m,t)}(X,\frac{\partial}{\partial t}) = (\sigma_t)_m(X)
\end{displaymath}
where $X\in T_mM$ is a tangent vector.  Let $\tilde{\beta}\in \Omega^1(M\times \R)$ be the $1$-form defined similarly as:

\begin{displaymath}
\tilde{\beta}_{(m,t)}(X,\frac{\partial}{\partial t}) = (\beta_t)_m(X)
\end{displaymath}
Now, $TM\times \R$ is a vector subbundle of $TM\times T\R$, hence $E:=\ker(\Phi\vert_{Z\times \R})$ defines a vector subbundle of $i_Z^*TM \times T\R \to Z\times \R$.  Note that $E\to Z\times \R$ has fiber $E_{(z,t)} = \ker(\sigma_t)_z$, hence $E$ is a stratified vector bundle over $Z\times \R$ and gathers the bundles $\ker(\sigma_t)\to Z$ into a smooth vector bundle over $Z\times \R$.  Note that $\beta$ is a $1$-form on $Z\times \R$ that vanishes on $\frac{\partial}{\partial t}$ everywhere on $M\times \R$ and vanishes on $E$ at $Z\times \R$.

\vspace{3mm}

Since $\ker(\sigma_t)\pitchfork TZ$ and $\ker(\sigma_t) = E\vert_{Z\times \{t\}}$, we have that $E \pitchfork TZ \times T\R$.  Therefore, near a point $(z,t) \in Z\times \R$, we may choose a nonvanishing section $w\in \Gamma(E)$ transverse to $Z\times \R$.  The section $w$ is stratified and tangent to the leaves $Z\times \{t\}$ for all $t\in \R$.  This is because $E$ is stratified and tangent to the leaves $Z\times \{t\}$ for all $t\in \R$.  Consequently, we may extend it to a non-vanishing stratified vector field $\tilde{w}$ in a neighborhood $U$ of $(z,t)$ that is tangent to $M\times \{t\}$ for all $t\in \R$.  Its flow then defines a diffeomorphism (on a perhaps smaller neighborhood of $(z,t)$ in $U$):

\begin{displaymath}
\phi:U\cap(Z\times \R)\times (-\epsilon,\epsilon) \to U
\end{displaymath}
where the coordinates are $(z,t,s)$ and $\frac{\partial}{\partial s}$ is identified with $\tilde{w}$ under $d\phi$.  Consequently, we may assume that our manifold is $Z\times \R \times (-\epsilon, \epsilon)$ with coordinates $(z,t,s)$, $M\times \{t\}$ is identified with $Z\times \{t\}\times (-\epsilon, \epsilon)$, $Z\times \R$ is identified with $Z\times \R \times \{0\}$, and $E \to Z\times \R \times \{0\}$ has fiber containing $\frac{\partial}{\partial s}$ at all points.

\vspace{3mm}

Let $p:Z\times \R \times (-\epsilon,\epsilon) \to Z\times \R$ be the projection and $i:Z\times \R \to Z \times \R \times (-\epsilon, \epsilon)$ the inclusion as the zero section.  $i_{\frac{\partial}{\partial s}}\tilde{\sigma}$ vanishes at $s=0$, we may write:

\begin{displaymath}
\tilde{\sigma} = p^*i^*\tilde{\sigma} + s\mu
\end{displaymath}
for some $2$-form $\mu$ where:

\begin{itemize}
\item $\tilde{\sigma}_t = p^*i^*\tilde{\sigma}_t + s\mu_t = \sigma_t + p^*i^*\sigma_t +s\mu_t$ is folded on $Z\times \{t\} \times (-\epsilon, \epsilon)$ for all $t\in \R$,
\item $\mu\big\vert_E$ is non-degenerate on $E$ since $\mu_t \big\vert_{\ker(\sigma)_t}$ is non-degenerate by lemma \ref{lem:symplectize}, and
\item $i_{\frac{\partial}{\partial t}}\mu = 0$ and $i_{\frac{\partial}{\partial t}}\sigma =0$.
\end{itemize}

Now, choose a local frame $\{e_1,\dots, e_{2m-2}, f, \frac{\partial}{\partial t}\}$ of $Z\times \R$ near $(z,t)$, where the $e_i's$ are tangent to the leaves $Z\times \{t\}$ and $f \in \Gamma(E)$ is a section of the kernel bundle tangent to the leaves $Z\times \{t\}$.  Extend this frame to the product:

\begin{displaymath}
\{e_1,\dots, e_{2m-2},f,\frac{\partial}{\partial t}, \frac{\partial}{\partial s}\}
\end{displaymath}
by lifting each $e_i$ and $\frac{\partial}{\partial t}$ to the product and appending $\frac{\partial}{\partial s}$.

\vspace{3mm}

Since $f$ and $\frac{\partial}{\partial s}$ span $E$ near $(z,t)$ and $\mu\big\vert_E$ is non-degenerate, we have $\mu(f,\frac{\partial}{\partial s})\ne 0$. We then obtain a coframe:

\begin{displaymath}
\{i_{e_1}\sigma, \dots, i_{e_{2m-2}}\sigma, i_f\mu, i_{\frac{\partial}{\partial s}}\mu, dt\}
\end{displaymath}
near $(z,t)$.  Thus, $\tilde{\beta}$ may be written:

\begin{displaymath}
\tilde{\beta} = \sum_{i=1}^{2m-2} a_i(i_{e_i}\sigma) + c_1(i_f\mu) + c_2(i_{\frac{\partial}{\partial s}}\mu) + c_3dt
\end{displaymath}
for some choice of smooth function $a_i$ and $c_i$ near $(z,t)$ on $Z\times \R \times (-\epsilon, \epsilon)$.

\begin{enumerate}
\item Since $\tilde{\beta}(\frac{\partial}{\partial t})=0$, $c_3=0$.
\item Since $\tilde{\beta}(\frac{\partial}{\partial s})$ and $\tilde{\beta}(f)$ vanish at $s=0$, the condition $\mu(\frac{\partial}{\partial s}, f)\ne 0$ at $Z\times \R \times \{0\}$ implies that $c_2$ and $c_1$ vanish at $s=0$.  Thus, $c_2=s\tilde{c}_2$ and $c_1=s\tilde{c}_2$ for some smooth function $\tilde{c}_1$ and $\tilde{c}_2$.
\item The time dependent vector field is then given by:

\begin{displaymath}
X = \sum_{i=1}^{2m-2}a_ie_i + \tilde{c}_1f + \tilde{c}_2\frac{\partial}{\partial s}
\end{displaymath}
restricted to the leaves $Z\times \{t\} \times (-\epsilon, \epsilon)$.  This is because each of the vectors appearing in the defining equation of $X$ are tangent to the leaves $Z\times \{t\} \times (-\epsilon, \epsilon)$ by construction.

\item Since $a_i$, $\tilde{c}_1$, and $\tilde{c}_2$ are smooth function of $z$, $t$, and $s$, we have that $X$ defines a smooth, time-dependent vector field on $Z\times (-\epsilon,\epsilon)$, which shows that $X_t$ is smooth at points of the fold.
\end{enumerate}
\end{proof}

\subsection{Local Structure Near the Fold}
We now seek to strengthen corollary \ref{cor:fsnormal} to include non-orientable folded-symplectic manifolds and refine lemma \ref{lem:fsnormal} to a statement about folded-symplectic manifolds without boundary.  We will prove the following two propositions:

\begin{prop}\label{prop:orientation}
Let $(M,\sigma)$ be a folded-symplectic manifold with corners.  Let $Z\subset M$ be the folding hypersurface.  Then $Z$ possesses a canonical orientation induced by $\sigma$.  Equivalently, the null bundle $\ker(\sigma)\cap TZ$ possesses a canonical orientation induced by $\sigma$.
\end{prop}

\begin{remark}
This strengthens corollary \ref{cor:fsnormal} by showing that there are no choices in defining the orientation on $Z$ \emph{and} the orientation exists even in the case where $Z$ is not co-orientable.
\end{remark}

\begin{definition}\label{def:orientation}
Let $(M,\sigma)$ be a folded-symplectic manifold with corners.  Let $Z\subset M$ be the folding hypersurface.  We call the orientation on $Z$ afforded by proposition \ref{prop:orientation} the $\sigma$-induced orientation or the \emph{orientation induced by} $\sigma$.  We also refer to the orientation on $\ker(\sigma)\cap TZ$ as the \emph{orientation induced by} $\sigma$.
\end{definition}

\begin{prop}\label{prop:fsnormal}
Let $(M,\sigma)$ be a folded-symplectic manifold without boundary.  Let $Z\subset M$ be the folding hypersurface with kernel bundle $\ker(\sigma)\to Z$ and suppose $Z$ is co-orientable.  Then there exists a neighborhood $U$ of $Z$, a neighborhood $V\subset Z\times \R$ of the zero section, and a diffeomorphism $\phi: V \to U$ of manifolds (without boundary) satsifying:

\begin{enumerate}
\item $\phi(z,0)= z$ for all $z\in Z$ and
\item $\phi^*\sigma = p^*i^*\sigma + d(t^2p^*\alpha)$, where $p:Z\times \R \to Z$ is the projection, $i:Z \to Z\times \R$ is the inclusion as the zero section, and $\alpha\in \Omega^1(Z)$ is a $1$-form that doesn't vanish on $\ker(i_Z^*\sigma)$ (and orients it in the canonical way, necessarily).
\end{enumerate}
\end{prop}

\begin{remark}
Proposition \ref{prop:fsnormal} is a slight generalization of theorem $1$ in \cite{CGW}.  In particular, we do not require $M$ to be compact and orientable.  The equivariant version of this proposition is proposition \ref{prop:eqfsnormal}.
\end{remark}

The advantage of strengthening lemma \ref{lem:fsnormal} to the form of proposition \ref{prop:fsnormal} may be found in its corollary:

\begin{cor}\label{cor:fsnormal1}
Let $(M,\sigma)$ be a folded-symplectic manifold without boundary.  Let $Z\subset M$ be the folding hypersurface with kernel bundle $\ker(\sigma) \to Z$ and suppose $Z$ is co-orientable.  Let $\psi:Z\times \R \to Z\times \R$ be the fold map $\psi(z,t)=(z,t^2)$.  Then there exists a neighborhood $U$ of $Z$, a neighborhood $V$ of the zero section of $Z\times \R$, and a diffeomorphism of manifolds (without boundary) $\phi:V \to U$ so that:

\begin{displaymath}
\phi^*\sigma = \psi^*\omega
\end{displaymath}
for some symplectic form $\omega\in \Omega^2(V)$.
\end{cor}

Hence, if $Z$ is co-orientable then there is a neighborhood $U$ of $Z$, a symplectic form $\omega\in \Omega^2(U)$ and a fold map $f:U \to U$ satisfying $\sigma = f^*\omega$.  This means that every folded-symplectic form with co-orientable folding hypersurface, $Z$, looks like the pullback of a symplectic structure by a fold map in the neighborhood of $Z$.
\subsubsection{Canonical Orientation on the Fold}

Before we begin, let us recall a construction of the Hessian in Morse Theory.  Given a manifold without boundary, a smooth function $f:M \to \R$, and a critical point $p\in M$, one may define the Hessian at $p$ $Hf_p:T_pM \times T_pM \to \R$ as follows:

\begin{displaymath}
Hf_p(X,Y) = X(df(\tilde{Y}))
\end{displaymath}
where $\tilde{Y}$ is any extension of the vector $Y\in T_pM$ to a local vector field.  It turns out that the degeneracy of $df$ at $p$ is enough to guarantee that the Hessian is independent of the choice of extension $Y$.  Note that, using our discussion of the intrinsic derivative, we may also interpret the Hessian as the intrinsic derivative of $f$ at $p$.

\vspace{5mm}

Given a folded-symplectic manifold with corners $(M,\sigma)$, we will perform a similar construction using the degeneracies of $\sigma$ in order to ensure that we have a well-defined result.  The spirit of the proof is very similar to that of the proof that $Hf_p$ is well-defined, though it is slightly more technical since we now have a $2$-form instead of a $1$-form $df$.

\begin{remark}\label{rem:intrinsic}
In what follows, we are computing an intrinsic derivative and using it to define the orientation on $Z$.  However, since the reader may not be familiar with the intrinsic derivative (q.v. section 2.1.5), we simply compute everything directly.  We outline the intrinsic derivative approach in this remark.

Consider the map of vector bundles given by contraction with $\sigma$ (we avoid calling it $\sigma^\#$ since we have reserved that notation for a different map):

\begin{displaymath}
\xymatrixrowsep{.5pc}\xymatrixcolsep{3pc}\xymatrix{
C_{\sigma}: TM \ar[r] & T^*M \\
 C_{\sigma}(p,X) \ar[r] & i_X\sigma_p
}
\end{displaymath}
This is a map of vector bundles over $M$, hence we may view it as a section of $\hom(TM,T^*M)$ and compute its intrinsic derivative at points $z\in Z$.  Its kernel at $z\in Z$ is just $\ker(\sigma_z)$ and its cokernel may be identified with any $2$-dimensional subspace $V$ of $T_z^*Z$ transverse to $\ker(\sigma_z)^o$.  The intrinsic derivative at $z$ gives us a map:

\begin{displaymath}
\xymatrixcolsep{3pc}\xymatrix{
(DC_{\sigma})_z:T_zM \ar[r] &  \hom(\ker(\sigma_z), V)
}
\end{displaymath}
Because $\operatorname{corank}(C_{\sigma}) = 2$ along $Z$, we have that $d(C_{\sigma})_z(T_zZ)$ is tangent to $L^2(TM, T^*M)_{C_{\sigma_z}}$, hence projecting to the normal bundle of $L^2(TM,T^*M)$ gives us the zero map.  Thus, $(Df)_z:T_zM \to \hom(\ker(\sigma_z),V)$ descends to a map:

\begin{displaymath}
\xymatrixcolsep{3pc}\xymatrix{
(DC_{\sigma})_z:\nu(Z)_z \ar[r] & \hom(\ker(\sigma_z), V)
}
\end{displaymath}
where $\nu(Z)=TM\big\vert_Z / TZ$ is the normal bundle.  We can define an orientation at $z$ as follows: pick a direction $w\in T_zM$ transverse to $z$, which gives us a nonzero element of $\nu(Z)_z$.  We then define a nonzero element $v$ of $\ker(\sigma_z)\cap T_zZ$ to be \emph{positively oriented} if:

\begin{displaymath}
((DC_{\sigma})_z([w])(w))(v) > 0
\end{displaymath}
that is, if the element $(Df)_z([w])(w) \in V\subset T_z^*Z$ evaluated on $v$ gives a positive number.  There is something to prove here.  Namely, one must show that $((DC_{\sigma})_z([w])(w))(v)$ is nonzero, but this follows from the fact that $\sigma$ induces a symplectic structure on $\ker(\sigma)$ by differentiating in a direction normal to $Z$ (q.v. lemma \ref{lem:fsnormal}).  If the reader is convinced, he or she may skip the proof of proposition \ref{prop:orientation}.  On the other hand, if the reader isn't convinced, then let us offer a more explicit proof below.
\end{remark}

\newpage

\begin{proof}[Proof of proposition \ref{prop:orientation}]
\mbox{ }
Let $\dim(Z)=2m-1$.  We define the orientation on the fold $Z$ as follows.  Let $p\in Z$ be any point in the fold and choose a local non-vanishing section $w\in \Gamma(\ker(\sigma))$ near $p$ transverse to $Z$.  Fix an extension $\tilde{w}$ of $w$ to a local vector field on $M$ near $p$.  We then define a $1$-form $\alpha\in \Omega^1(Z)$ on $Z$ near $p$ by the equation:

\begin{equation}\label{eq:alpha}
\alpha_x(v) = w_x(\sigma(\tilde{w},\tilde{v}))
\end{equation}
where $x\in Z$ is a point near $p$ and $\tilde{v}$ is any extension of $v\in T_xZ$ to a local vector field on $M$ near $p$.  This is just differentiating the local function $\sigma(\tilde{w},\tilde{v})$ in the direction $w_x$.  We claim that the $2m-1$ form:

\begin{equation}\label{eq:Omega}
\Omega := (i_Z^*\sigma)^{m-1}\wedge \alpha
\end{equation}
defines an orientation near $p$ which is independent of the choices we have made along the way.  So, given $p\in Z$, we have several goals:

\begin{enumerate}
\item Show that the $1$-form $\alpha$ defined near $p$ by equation \ref{eq:alpha} does not depend on the choice of extension $\tilde{v}$.
\item Show that the induced orientation in a neighborhood of $p$ defined by $\Omega$ in equation \ref{eq:Omega} is independent of the choice of $w$ and its extension $\tilde{w}$.
\item Show that the preceding two facts allow us to glue the orientations together on $Z$ to give a globally defined orientation.
\end{enumerate}

We proceed in the above order:

\begin{enumerate}
\item To show that $\alpha$ is well-defined, we use the local model of lemma \ref{lem:fsnormal}.  That is, we assume $M$ is a neighborhood of the zero section of $Z\times \R$ and $\sigma=p^*i^*\sigma + t\mu$, where $p:Z\times \R \to Z$ is the projection, $i:Z \to Z\times \R$ is the inclusion as the zero section, and $\mu$ is some $2$-form that is non-degenerate on $\ker(\sigma)$.  In this local model, we have that $\frac{\partial}{\partial t} \in \Gamma(\ker(\sigma))$ at points of $Z\times \{0\}$.  Therefore, any section $w\in \Gamma(\ker(\sigma))$ transverse to $Z\times \{0\}$ has the form:

    \begin{displaymath}
    w=g\frac{\partial}{\partial t} + X, \text{ where $X\in \ker(i_Z^*\sigma)$ and $g$ is non-vanishing}.
    \end{displaymath}
    Its extension $\tilde{w}$ then has the form:

    \begin{displaymath}
    \tilde{w} = \tilde{g}\frac{\partial}{\partial t} + \tilde{X}
    \end{displaymath}
    where $\tilde{X}$ is an extension of $X$ to a local vector field in a neighborhood of $Z\times \{0\}$ and $\tilde{g}$ is a local extension of $g$ to a smooth function in a neighborhood of $Z\times \{0\}$.  Note that, by definition of extension, $\tilde{X}(z,0)=X(z)\in \ker(i_Z^*\sigma)$ and $\tilde{g}(z,0)=g(z) \ne 0$.  We compute:

    \begin{displaymath}
    \begin{array}{l}
    w(\sigma(\tilde{w},\tilde{v})) =  \\
    \displaystyle (g\frac{\partial}{\partial t}\big\vert_{t=0} + X(z))(p^*i^*\sigma(\tilde{w},\tilde{v}) + t\mu(\tilde{w},\tilde{v})) = \\
    \displaystyle g\frac{\partial}{\partial t}\big\vert_{t=0}(\tilde{g}t\mu(\frac{\partial}{\partial t},\tilde{v}) + p^*i^*\sigma(\tilde{X},\tilde{v}) + t\mu(\tilde{X},\tilde{v}) = \\
    \displaystyle (g^2\mu(\frac{\partial}{\partial t}, v) + g\mu(X,v))\big\vert_{t=0} + g\frac{\partial}{\partial t}\big\vert_{t=0} p^*i^*\sigma(\tilde{X},\tilde{v})
    \end{array}
    \end{displaymath}
    where the third line follows since $X(z)$ is tangent to $Z$ and the quantity $\sigma(\tilde{w},\tilde{v})_{(z,0)} = \sigma(w,\tilde{v})_{(z,0)}=0$ vanishes along $Z\times \{0\}$.  Now, $\tilde{X}\big\vert_{t=0} = X \in \ker(i_Z^*\sigma)$, hence $\tilde{X}=p^*X + tY$, where $p^*X=(X,0)$ is the pullback of the vector field $X\in \Gamma(TZ)$ to the product $Z\times \R$ and $Y$ is some vector field.  We then have:

    \begin{displaymath}
    g\frac{\partial}{\partial t}\big\vert_{t=0}p^*i^*\sigma(\tilde{X},\tilde{v}) = gp^*i^*\sigma(Y,v)\big\vert_{t=0}
    \end{displaymath}
    which depends only on the extension $\tilde{w}$ of $w$.  But, we fixed this extension, hence each of the terms of:

    \begin{equation}\label{eq:alpha1}
    \alpha(v)=w(\sigma(\tilde{w},\tilde{v})) = (g^2\mu(\frac{\partial}{\partial t}, v) + g\mu(X,v) + gp^*i^*\sigma (Y,v))\big\vert_{t=0}
    \end{equation}
    do not depend on the extension of $v$ to $\tilde{v}$ and we obtain a well-defined $1$-form $\alpha$ on $Z$ in a neighborhood of $p$.  Furthermore, if $v\in \Gamma(\ker(i_Z^*\sigma))$ lies in the kernel of $\sigma$, then equation \ref{eq:alpha1} yields:

    \begin{displaymath}
    \alpha(v) = g^2\mu(\frac{\partial}{\partial t},v)\big\vert_{t=0}
    \end{displaymath}
    which is nonzero since $\mu$ is non-degenerate on the kernel bundle $\ker(\sigma) \to Z$.  Thus, $\alpha\big\vert_{\ker(i_Z^*\sigma)}$ is nonvanishing, which allows us to define the orientation:

    \begin{equation}\label{eq:Omega1}
    \Omega:= (i_Z^*\sigma)^{m-1}\wedge \alpha = g^2(i_Z^*\sigma)^{m-1}\wedge \mu(\frac{\partial}{\partial t}, \cdot)
    \end{equation}
    To see the rightmost equality, simply choose a local frame $v\in \Gamma(\ker(i_Z^*\sigma))$, extend to a local frame of $Z$ near $p$, and evaluate $\Omega$ on this frame.  Many of the terms will vanish since $p^*i^*\sigma(v,\cdot) = 0$.  Since $g$ is nowhere vanishing on $Z$, $g^2$ is nowhere vanishing and $\Omega$ defines an orientation form on $Z$.  Note that we are using the fact that $i_Z^*\sigma$ has maximal rank $2m-2$.

\item Now, if $w_0$ is any other section of $\ker(\sigma)$ transverse to $Z\times \{0\}$, then we may write:
    \begin{displaymath}
    w_0  = f\frac{\partial}{\partial t} + X_0
    \end{displaymath}
    for some $X_0\in \ker(i_Z^*\sigma)$ and some smooth function $f\in C^{\infty}(Z)$.  Fix an extension $\tilde{w}_0$ of $w_0$ and let $\alpha_0$ be the $1$-form generated by our construction.  If we let $v\in \Gamma(\ker(i_Z^*\sigma))$ be a section of the kernel and use equation \ref{eq:alpha1}, we obtain:

    \begin{displaymath}
    \alpha_0(v) = f^2\mu(\frac{\partial}{\partial t}, v)
    \end{displaymath}
    hence the induced orientation form $\Omega_0$ is:

    \begin{displaymath}
    \Omega_0 = (i_Z^*\sigma)^{m-1} \wedge \alpha_0 = f^2(i_Z^*\sigma)^{m-1}\wedge \mu(\frac{\partial}{\partial t} ,\cdot)
    \end{displaymath}
    That is, all orientation forms constructed in this manner are positive multiples of each other since they are positive multiples of:

    \begin{displaymath}
    (i_Z^*\sigma)^{m-1} \wedge \mu(\frac{\partial}{\partial t},\cdot)
    \end{displaymath}
    Thus, given $p\in Z$, we have defined an orientation on a neighborhood $U_p\subset Z$ of $p$ that is independent of our choices.
\item Now, cover $Z$ by the neighborhoods $\{U_p\}_{p\in Z}$, where each $U_p$ is equipped with our orientation.  On an overlap $U_{p_0}\cap U_{p_1}$, these orientations must agree since any section $w_i$ of $\ker(\sigma)$ on $U_i$ restricts to a section of $\ker(\sigma)$ on the intersection $U_{p_0}\cap U_{p_1}$.  Hence, if we perform the construction of the orientation on $U_{p_0}\cap U_{p_1}$ using $w_0$ and $w_1$, we find that the induced orientations agree since they are independent of the choices of sections of $\ker(\sigma)$.  That is, the orientation of $U_{p_0}$ agrees with that of $U_{p_1}$ on the intersection.  Thus, our construction produces a global orientation on $Z$.
\end{enumerate}
\end{proof}
\subsubsection{Normal Form for the Fold}

We will need the following weak version of the Poincar\'{e} lemma to prove proposition \ref{prop:fsnormal}.

\begin{lemma}\label{lem:poincare}
Let $Z$ be a manifold with corners and let $\sigma\in \Omega^k(Z\times \R)$ be a $k$-form such that $i_Z^*\sigma = 0$, where $i_Z:Z\times \{0\} \to Z\times \R$ is the inclusion as the zero section.  Let $r_s(z,t)=(z,st)$ be the deformation retract of $Z\times \R$ onto $Z\times \{0\}$.  Then:

\begin{displaymath}
\sigma = d\int_0^1 \frac{1}{s}r^*_s(i_{t\frac{\partial}{\partial t}}\sigma)ds = d(t\int_0^1 r_s^*(i_{\frac{\partial}{\partial t}}\sigma)ds
\end{displaymath}
\end{lemma}

\begin{proof} \mbox{ } \newline
This is lemma 31.16 in \cite{Mi1}.
\end{proof}

\begin{proof}[Proof of proposition \ref{prop:fsnormal}] \mbox{ }
By assumption, we have a folded-symplectic maniold $(M,\sigma)$ without boundary and a co-orientable folding hypersurface $Z\subset M$.  By lemma \ref{lem:fsnormal}, we may assume that a neighborhood of $Z$ is a neighborhood $W$ of the zero section of $Z\times \R$ with folded-symplectic form:

\begin{displaymath}
\sigma = p^*i^*\sigma + t\mu
\end{displaymath}
where $p:Z\times \R \to Z$ is the projection, $i:Z \to Z\times \R$ is the inclusion as the zero section, and $\mu$ is a $2$-form on $W$.  Let $\alpha \in \Omega^1(Z)$ be the $1$-form defined by:

\begin{displaymath}
\alpha := i^*(i_{\frac{\partial}{\partial t}}\mu)
\end{displaymath}
By lemma \ref{lem:symplectize}, $\mu\big\vert_{\ker(\sigma)}$ is non-degenerate, hence $\alpha \big\vert_{\ker(i_Z^*\sigma)}$ is non-vanishing.  In fact, it defines the canonical orientation of $Z$ in proposition \ref{prop:orientation}.  Define:

\begin{displaymath}
\sigma_0:= p^*i^*\sigma + d(\frac{t^2}{2}p^*\alpha)
\end{displaymath}
and
\begin{displaymath}
\sigma_1:= p^*i^*\sigma + t\mu
\end{displaymath}
and let $\sigma_s$ be the path:

\begin{equation}\label{eq:foldpath}
\sigma_s:= (1-s)\sigma_0 + s\sigma_1= p^*i^*\sigma + t((1-s)(dt\wedge p^*\alpha + \frac{t}{2}p^*d\alpha) + s\mu)
\end{equation}
Note that $\ker(\sigma_s)=\ker(\sigma)$ is independent of $s$.  Define $\mu_s:= (1-s)(dt\wedge p^*\alpha + \frac{t}{2}p^*d\alpha) + s\mu$ and note that for any oriented section $f \in \Gamma(\ker(i_Z^*\sigma))$, we have:

\begin{displaymath}
\mu_s\big\vert_{t=0}(\frac{\partial}{\partial t}, f) = \mu(\frac{\partial}{\partial t},f) \ne 0
\end{displaymath}
hence $\mu_s\big\vert_{\ker(\sigma)} = \mu_s\big\vert_{\ker(\sigma_s)}$ is non-degenerate.  Lemma \ref{lem:symplectize} implies that $\sigma_s$ is folded for each $s$.  Thus, there exists a neighborhood $\tilde{V}$ of $Z\times \{0\}$ on which $\sigma_s$ is folded for all $s\in [0,1]$.  Now, $\sigma_1-\sigma_0$ is closed and vanishes on $Z\times \{0\}$, hence lemma \ref{lem:poincare} implies:

\begin{displaymath}
\dot{\sigma}_s = \sigma_1-\sigma_0 = d(t\int_0^1r_s^*(i_{\frac{\partial}{\partial t}}(\sigma_1-\sigma_0))ds)
\end{displaymath}
where $r_s(z,t)=(z,st)$ is the deformation retract onto $Z\times \{0\}$.  Since $i_{\frac{\partial}{\partial t}}(\sigma_0-\sigma_1)$ vanishes at $Z\times \{0\}$, we may write it as $t\eta$ for some $1$-form $\eta$, hence:

\begin{displaymath}
\dot{\sigma}_s = d(t\int_0^1r_s^*(t\eta)ds) = d(t^2\int_0^1 sr_s^*(\eta)ds)
\end{displaymath}
Let $\beta_s = t\int_0^1 sr_s^*(\eta)ds$ and note that it vanishes at $Z\times \{0\}$, hence it vanishes on $\ker(\sigma)$ and therefore gives us an element $\beta_s \in \Omega^1_{\ker(\sigma_s)}(M) = \Omega^1_{\ker(\sigma)}(M)$ for all $s$.  By proposition \ref{prop:Moser}, there exists a smooth time-dependent vector field $X_s$ on $\tilde{V}$ satisfying:

\begin{displaymath}
i_{X_s}\sigma_s = -\beta_s
\end{displaymath}
defined as $X_s:=(\sigma_s^\#)^{-1}(-\beta_s)$.  Now, $\dot{\sigma}_s= t\beta_s$, hence $tX_s$ is the unique vector field satisfying:

\begin{displaymath}
i_{tX_s}\sigma_s = -t\beta_s = -\dot{\sigma}_s
\end{displaymath}
Hence, the flow $\phi_s$ of $tX_s$ satisfies $\phi_s^*\sigma_s =\sigma_0$ by proposition \ref{prop:Moser}.  Since $tX_s$ vanishes at $Z\times \{0\}$, we may again use proposition \ref{prop:Moser} to obtain a neighborhood $V\subset{\tilde{V}}$ of $Z\times \{0\}$ on which the flow exists for all time $s\in [0,1]$.  We then have an open embedding:

\begin{displaymath}
\phi_1: V \to Z\times \R
\end{displaymath}
satisfying $\phi_1(z,0)=(z,0)$ and $\displaystyle\phi_1^*\sigma_1 = \phi_1^*\sigma = \sigma_0 = p^*i^*\sigma + d(\frac{t^2}{2}p^*\alpha)$.
\end{proof}

\begin{remark}
There is a significant difficulty in generalizing proposition \ref{prop:fsnormal} to manifolds with corners.  Namely, the time-dependent vector field $X_s$ constructed in the proof may $not$ be stratified, meaning its flow does $not$ generate diffeomorphisms of manifolds with corners.  Consider the following example:

\vspace{3mm}
\begin{example}
This example discusses the difficulties inherent in producing stratified, time-dependent vector fields.  Let $M= (\R^+)^2 \times \R^2$ with coordinates $(x_1,x_2, t, y)$ and let

\begin{displaymath}
\begin{array}{lcl}
\sigma_0 & =& dx_1\wedge dx_2 + tdt\wedge dy +dx_1 \wedge dy \\
\sigma_1 & =& dx_1\wedge dx_2 + tdt\wedge dy +dx_1 \wedge dy + 3t^2dt\wedge dx_1 \\
\end{array}
\end{displaymath}
and let $\sigma_s$ be the linear path $\sigma_s = (1-s)\sigma_0 + s\sigma_1$, which we may write as:

\begin{displaymath}
\sigma_s = \sigma_0 + s3t^2dt\wedge dx_1 = dx_1\wedge dx_2 + tdt\wedge dy +dx_1 \wedge dy + 3st^2dt\wedge dx_1
\end{displaymath}
Then $\sigma_s$ is folded for all $s$ since:

\begin{itemize}
\item $\sigma_s^2 = tdx_1\wedge dx_2 \wedge dt\wedge dy$, hence $Z=\{t=0\}$ and
\item $i_Z^*\sigma_s= dx_1\wedge dx_2 + dx_1\wedge dy$ has rank $2$ on $(\R^+) \times \{0\}\times \R$.
\end{itemize}
The derivative $\dot{\sigma}_s$ is:

\begin{displaymath}
\dot{\sigma}_s = 3t^2dt\wedge dx_1
\end{displaymath}
There are infinitely many candidates for primitives $\beta_s$ of $\dot{\sigma}_s$, where \emph{primitive} means $d\beta_s = \dot{\sigma}_s$.  We consider two of them and show that the vector fields produced from the equation $i_{X_s}\sigma_s = -\beta_s$ are not stratified.

\begin{enumerate}
\item In the proof of proposition \ref{prop:fsnormal}, the primitive of $\dot{\sigma}_s$ would be $\beta_s= t^3dx_1$.  The vector field satisfying the equation $i_{X_s}\sigma_s = -\beta_s$ is then $ct^3\frac{\partial}{\partial x_2}$, which is not a stratified vector field on $M$.  Indeed, it is transverse to the stratum $\{x_2=0\}$ when $t\ne 0$, meaning its flow is not strata-preserving and therefore does not induce diffeomorphisms of manifolds with corners.

\item Let us try using the primitive $\beta_s= -3x_1t^2dt$.  Then the vector field $X_s$ satisfying $i_{X_s}\sigma_s = -\beta_s$ is:
\begin{displaymath}
X_s = 3x_1t\frac{\partial}{\partial y} + 3x_1 t \frac{\partial}{\partial x_2}
\end{displaymath}
which is not stratified since it is transverse to the set $\{x_2=0\}$ when $x_1\ne 0$ and $t\ne 0$.
\end{enumerate}

The point is that one must find a family of primitives $\beta_s$ for $\dot{\sigma}_s$ so that for each $s\in \R$ $\beta_s$ lies in the image of the submodule of stratified vector fields under the map:

\begin{displaymath}
\sigma_s^\#:\Gamma(TM) \to \Omega_{\ker(\sigma_s)}^1(M)
\end{displaymath}
and it is not clear that this is always possible.
\end{example}
\end{remark}





\pagebreak
\section{Hamiltonian Group Actions on Folded-Symplectic Manifolds}
The purpose of this chapter is two-fold.  We first review group actions and symplectic group actions for the benefit of the reader, standardizing our notation and recalling some important results about both general group actions and torus actions.  This material comprises the bulk of the first two sections and could rightfully be placed in an appendix.  However, because Hamiltonian group actions are an integral component of our study, we wish to place the background alongside our development of new theory for folded-symplectic manifolds.  In particular, we review symplectic representations of tori, Hamiltonian actions on symplectic manifolds, and the normal form of Guillemin and Sternberg \cite{GS2} for isotropic orbits in Hamiltonian $G$-manifolds.  We will need all of these tools in order to discuss the local (equivariant) invariants of a toric, folded-symplectic manifold.  Along the way, we will show that every folding hypersurface admits an equivariant symplectization and an equivariant folded-symplectization where the fold is co-orientable.  Coupling the normal form of Guillemin-Sternberg with the symplectization of the folding hypersurface produces a local uniqueness result for folding hypersurfaces in toric, folded-symplectic manifolds (q.v. lemma \ref{lem:locunique}), which we will extend to a local uniqueness statement for toric, folded-symplectic manifolds.  After we have reviewed the requisite group action theory, we will develop an equivariant analog of the normal form proposition \ref{prop:fsnormal} for the folding hypersurface, which will facilitate our proofs of structural results about toric, folded-symplectic manifolds.

\subsection{Group Actions on Manifolds}
We give a definition of a group action on manifolds \emph{with corners}.  However, most subsequent lemmas, definitions, theorems, and propositions will consider only group actions on manifolds \emph{without corners}.  We will distinguish between manifolds with corners and those without when necessary.

\begin{definition}\label{def:action}
Let $G$ be a Lie group and let $M$ be a manifold with corners.  An action of $G$ on $M$ is a homomorphism:

\begin{displaymath}
\tau:G \to \operatorname{Diff}(M)
\end{displaymath}
so that the action map $\mathcal{A}:G\times M \to M \times M$ given by $\mathcal{A}(g,m)=(\tau(g)(m),m)$ is smooth.  For the sake of notation, for each $g\in G$ we write$\tau_g:=\tau(g)$ and refer to it as \emph{the action of g on M}.  We often write $g\cdot p$ for the action of an element $g\in G$ on $p\in M$.  We reserve the right to switch between $g\cdot p$ and $\tau_g(p)$ freely.
\end{definition}

\begin{definition}\label{def:orbit}
Let $G$ be a Lie group that acts on a manifold without corners $M$.  Let $p\in M$ be a point.  We define the \emph{orbit of p} to be the set:

\begin{displaymath}
G\cdot p:= \{m\in M\vert \mbox{ } \exists g\in G \text{ such that } m=g\cdot p\}
\end{displaymath}
We define the \emph{stabilizer of p} to be the subgroup:

\begin{displaymath}
G_p:=\{ g\in G \vert \mbox{ }g\cdot p=p\}
\end{displaymath}
\end{definition}

\begin{definition}\label{def:proper}
Let $G$ be a Lie group and let $M$ be a manifold with corners.  Suppose $\tau:G \to \operatorname{Diff}(M)$ is an action of $G$ on $M$.  We say the action is \emph{proper} if the action map $\mathcal{A}:G\times M \to M \times M$ is proper.  That is, if $K\subset M\times M$ is compact, then $\mathcal{A}^{-1}(K)$ is also compact.
\end{definition}

The following lemma is standard in the study of group actions on manifolds.  Its proof may be found in \cite{Ka, Mi1}.

\begin{lemma}\label{lem:orbits}
Let $G$ be a Lie group acting properly on a manifold without corners $M$.  For each $p\in M$, the stabilizer $G_p$ is a compact subgroup of $G$, hence it is a Lie subgroup of $G$, $G/G_p$ is a smooth manifold, and the action map $\mathcal{A}:G\times M \to M \times M$ induces a map $\mathcal{A}': G/G_p \to M$ given by $\mathcal{A}'([g])=g\cdot p$.  The map $\mathcal{A}'$ is a smooth embedding whose image is the orbit $G\cdot p$, hence the orbits of a proper Lie group action are embedded submanifolds.
\end{lemma}

We will need the analog of lemma \ref{lem:orbits} in the case where $M$ is a manifold with corners.

\begin{cor}\label{cor:orbits}
Let $G$ be a Lie group acting properly on a manifold with corners.  Then for each $p\in M$, the stabilizer, $G_p$ is a closed subgroup of $G$, $G/G_p$ is a smooth manifold, and the orbit $G\cdot p$ is an embedded submanifold of the stratum containing $p$.  In particular, they are embedded submanifolds with corners of $M$ whose boundary is empty.
\end{cor}

\begin{proof} \mbox{ } \newline
Let $\partial^k(M)$ be the stratum of $M$ containing $p$.  By definition of a group action, $G$ acts by diffeomorphisms of manifolds with corners, hence it preserves the $k$-boundary $\partial^k(M)$ for each $k\ge 0$.  This means that we may consider the restricted action map $\mathcal{A}:G\times \partial^k(M) \to \partial^k(M) \times \partial^k(M)$, which is also proper since any compact subset $K\subset \partial^k(M)\times \partial^k(M)$ is also a compact subset of $M\times M$, hence $\mathcal{A}^{-1}(K)\subset G\times \partial^k(M)$ is compact.  It follows from lemma \ref{lem:orbits} that $G\cdot p$ is an embedded submanifold of $\partial^k(M)$, hence it is an embedded submanifold with corners of $M$ whose boundary is empty.
\end{proof}

\begin{definition}\label{def:induced}
Let $G$ be a Lie group acting properly on a manifold $M$ with corners.  Let $\fg=\operatorname{Lie}(G)$ and let $\exp:\fg \to G$ be the exponential map.  For each $X\in \fg$ we define the induced vector field $X_M\in \Gamma(TM)$ pointwise by the equation:

\begin{displaymath}
X_M(p):=\frac{d}{dt}\big\vert_0 \exp(tX)\cdot p
\end{displaymath}
If we write $\mathcal{A}_p:G \to M$ for the map given by $\mathcal{A}_p(g)=g\cdot p$ then $X_M(p)= d(\mathcal{A}_p)_e(X)$.
\end{definition}

The following statement is a corollary to lemma \ref{lem:orbits}.
\begin{cor}\label{cor:orbits1}
Let $G$ be a Lie group acting properly on a manifold $M$ with corners.  Let $\fg=\operatorname{Lie}(G)$ be the Lie algebra of $G$.  Then for each point $p\in M$, the tangent space to the orbit is:

\begin{displaymath}
T_p(G\cdot p) := \{X_M(p) \vert \mbox{ } X\in \fg\}
\end{displaymath}
That is, the tangent space to the orbit at $p$ is generated by the induced vector fields at $p$.
\end{cor}

\begin{proof}\mbox{ } \newline
Let $p\in M$ and let $G_p$ be its stabilizer.  The map $\mathcal{A}_p:G \to G\cdot p$ given by $\mathcal{A}_p(g)=g \cdot p$ factors through the projection $\pi:G \to G/G_p$:

\begin{displaymath}
\xymatrix{
G \ar[r]^{\mathcal{A}_p} \ar[d]^{\pi} & G\cdot p \\
G/G_p \ar[ur]^{\mathcal{A}'}
}
\end{displaymath}
where $\mathcal{A}'$ is the embedding of $G/G_p$ into $M$ given by lemma $\mathcal{A}'([g])=g\cdot p$ (q.v. lemma \ref{lem:orbits}).  Since $\mathcal{A}'$ is an embedding and $\pi$ is a submersion, the composite $\mathcal{A}'\circ \pi$ is a submersion, hence $\mathcal{A}_p$ is a submersion.  The image of $d\mathcal{A}_p$ is therefore a surjection, meaning the tangent space $T_p(G\cdot p)$ is generated by the image of $d(\mathcal{A}_p)_e$, which is the space of induced vector fields at $p$.
\end{proof}

\begin{definition}\label{def:diffslice}
Let $G$ be a Lie group acting properly on a manifold without corners $M$.  Let $p\in M$ and let $T_p(G\cdot p)$ be the tangent space to the orbit at $p$, which is well-defined by lemma \ref{lem:orbits}.  We define the \emph{differential slice} $W$ at $p$ to be the quotient space
\begin{displaymath}
W:=T_pM/(T_p(G\cdot p).
\end{displaymath}
We have a linear action of $G_p$ on $T_pM$ given by

\begin{displaymath}
g \cdot v = d\tau_g(v)
\end{displaymath}
which turns $T_pM$ into a representation of $G_p$.  That is, we have a homomorphism $\rho:G_p \to GL(T_pM)$.  Since $T_p(G\cdot p)$ is an invariant subspace of this representation, the representation $\rho:G_p \to GL(T_pM)$ descends to a representation $\bar{\rho}:G_p \to W$, hence we refer to $W$ as the \emph{differential slice representation}.
\end{definition}

\begin{remark}
Let $G$ be a Lie group acting acting properly on a manifold without corners $M$.  Let $p\in M$, $H=G_p$, and let $W$ be the differential slice representation at $p$.  We may form the vector bundle $G\times_H W$, which is the quotient of $G\times W$ by the action $h\cdot(g,v)=(gh^{-1},h\cdot v)$.  It inherits an action of $G$ given by $g_0\cdot [(g,v)]=[(g_0g,v)]$.  We will see that this vector bundle equipped with this action of $G$ is a local model for the action of $G$ on $M$ near $p$.
\end{remark}

The following theorem, due to Palais, is critical in the study of group actions on manifolds and reduces the study of neighborhoods of orbits to an exercise in representation theory.  A proof can be found in \cite{ACL} (in the case where $G$ is compact), \cite{Ka}, \cite{Me}, or \cite{Mi1}.

\begin{theorem}[The Differential Slice Theorem]\label{thm:slice}
Let $G$ be a Lie group acting properly on a manifold without corners $M$ and let $p\in M$ be a point with stabilizer $G_p$ and differential slice representation $W$.  There exists an invariant neighborhood $U_1$ of $p$, an invariant neighborhood $U_2$ of the zero section of $G \times _{G_p} W$, and an equivariant diffeomorphism $\phi:U_1 \to U_2$ satisfying $\phi(p)=[(e,0)]$, where $e\in G$ is the identity.
\end{theorem}

\begin{definition}\label{def:orbittype}
Let $G$ be a Lie group acting on a manifold without corners $M$.  Let $H\le G$ be a Lie subgroup and let $(H)$ denote the conjugacy class of $H$.  We define:
\begin{displaymath}
\begin{array}{l}
M_H = \{p\in M \vert \mbox{ } G_p=H\} \\
M_{(H)} =\{p \in M \vert \mbox{ } G_p\in (H)\}
\end{array}
\end{displaymath}
We refer to $M_{(H)}$ as the \emph{orbit-type stratum of type} $H$.
\end{definition}

\begin{remark}\label{rem:abelian}
Suppose $G$ is an abelian Lie group acting on a manifold with corners.  For any Lie subgroup $H\le G$ we have $(H)=\{H\}$, hence $M_H=M_{(H)}$.  Thus, when we begin to consider toric actions, we will refer to $M_H$ as an \emph{orbit-type stratum}.  Note that, in general, the relationship between $M_H$ and $M_{(H)}$ is that $G\cdot M_H = \{g\cdot p \vert \mbox{ } p\in M_H, \mbox{ } g\in G\}= M_{(H)}$ since the stabilizers of points in the same orbit are related by conjugation.
\end{remark}

The following is a standard corollary to the slice theorem (q.v. \cite{Ka}).

\begin{cor}\label{cor:slice1}
Let $G$ be a Lie group acting properly on a manifold without corners $M$.  Let $H\le G$ be a Lie subgroup.  Then $M_H$ is a smooth submanifold of $M$ and $M_{(H)}$ is a smooth submanifold of $M$, hence $\bigsqcup_{H\le G} M_{(H)}$ is a decomposition of $M$ into disjoint, smooth submanifolds.
\end{cor}

\begin{proof} \mbox{ } \newline
Let $p$ be a point in $M$ with stabilizer $H=G_p$ and let $W$ be the differential slice representation.  Let $W^H$ be the invariant subspace of vectors fixed by the action of $H$.  We will show that, around $p$, $M_H$ is isomorphic to $\nu(H) \times_H W^H$ and $M_{(H)}=G\times_H W^H$.
\begin{enumerate}
\item Let $p\in M_H$.  By theorem \ref{thm:slice} there is an invariant neighborhood of $p$ equivariantly diffeomorphic to a neigbhorhood of the zero section in $G\times_H W$, where $W=T_pM/(T_p(G\cdot p))$ is the differential slice at $p$.  Let $W^H$ be the subspace of $W$ fixed by $H$ and let $\nu(H)$ be the normalizer of $H$ in $G$.  We claim that in these coordinates $M_H$ corresponds to $\nu(H)\times_H W^H$.  We first show that $M_H \subset \nu(H)\times_H W^H$.  Indeed, suppose $[(g,v)]\in G\times_H W$ is a point with stabilizer $H$ and $h\in H$ is an arbitrary element.  Then, we have:

    \begin{equation}\label{eq:subman}
    h\cdot[(g,v)] = [(hg,v)]=[(g,v)] \iff \text{There exists $g_0\in H$ such that $(hgg_0^{-1},g_0\cdot v)=(g,v)$.}
    \end{equation}
    Thus, $g^{-1}hg=g_0\in H$.  Since $h$ was arbitrary, we have $g^{-1}hg\in H$ for all $h\in H$, hence $g$ is in the normalizer of $H$: $g\in \nu(H)$.  We still need to show that $v$ is fixed by $H$.  We have
    \begin{equation}\label{eq:subman1}
    h\cdot([g,v]) = ([hg,v])=[(g(g^{-1}hg),v)]=[(g,gh^{-1}g^{-1}v)]=[(g,v)]
    \end{equation}
    which is true if and only if there exists $g_1\in H$ such that $gg_1^{-1}=g$ and $g_1gh^{-1}g^{-1}v=v$.  The first equation implies $g_1=e$, hence $gh^{-1}g^{-1}v=v$.  Since $h$ was arbitrary and $g\in \nu(H)$, we have that $H$ fixes $v$.  Thus, $[(g,h)]\in \nu(H)\times_H W^H$ and $M_H \subseteq \nu(H)\times_H W^H$.

     Let us show the reverse inclusion.  If we begin with a point $[(g,v)]\in \nu(H)\times_H W^H$ then for all $h\in H$, $g\in \nu(H)$ and the calculation of equation \ref{eq:subman1} modulo the last equality demonstrates that $h\cdot [(g,v)] = [(g,v)]$, hence $H$ is contained in the stabilizer of $[(g,v)]$.  If $\eta\in G$ is some element such that $\eta \cdot[(g,v)]=[(g,v)]$, then the calculation of equation \ref{eq:subman} demonstrates that $\eta g= gh^{-1}$ for some $h\in H$, hence $\eta = gh^{-1}g^{-1}$, which is an element of $H$ since $g$ lies in the normalizer of $H$.  Thus, near $p$, $M_H$ is equivariantly diffeomorphic to a neighborhood of the zero section of $\nu(H) \times_H W^H$.

\item By remark \ref{rem:abelian}, the set $M_{(H)}$ is simply the set $G\cdot M_H$.  Thus, in the model $G\times_H W$ we have that $M_{(H)}$ corresponds to $G\cdot(\nu(H) \times_H W^H)$. Since the action of $G$ on itself is transitive, this set is $G\times _H W^H$, which gives $M_{(H)}$ the structure of a smooth submanifold.
\end{enumerate}
\end{proof}

\begin{cor}\label{cor:slice2}
Let $M$ be a manifold without corners and let $G$ be a Lie group acting properly on $M$.  Let $H\le G$ be a subgroup and consider $M_H$.  For each $p\in M_H$, we have:
\begin{displaymath}
T_pM_H= (T_pM)^H
\end{displaymath}
where $(T_pM)^H$ is the subspace of vectors fixed by the action of $H$.
\end{cor}

\begin{remark}
This is proposition 3.3 in \cite{GS1}.
\end{remark}

\begin{proof}\mbox{ } \newline
Since $H$ acts trivially on $M_H$, its induced action on $T_pM_H$ is trivial, so we certainly have $T_pM_H\subseteq (T_pM)^H$.  Thus, we focus on showing the reverse inclusion.  Using an $H$-invariant metric on $T_pM$, we may write $T_pM = T_p(G\cdot p ) \oplus E$, where $E$ is some invariant complimentary subspace isomorphic to the differential slice representation.  Then $(T_pM)^H = T_p(G\cdot p)^H \oplus E^H$.  The vectors in $T_p(G\cdot p)$ fixed by $H$ are given by $T_p(\nu(H)\cdot p)$, where $\nu(H)$ is the normalizer of $H$.

By the slice theorem, a neighborhood of the orbit in $M_H$ is isomorphic to a neighborhood of the zero section of $\nu(H)\times_H W^H = \nu(H)/H \times W^H$, where $W\simeq E$ is the differential slice.  The point $p$ corresponds to $[(e,0)]$ and the tangent space at $[(e,0)]$ is exactly $T_{[e]}\nu(H)/H \oplus W^H$, which demonstrates that the dimension of the tangent space of $T_pM_H$ is exactly the dimension of $(T_pM)^H = T_p(G\cdot p)^H \oplus E^H$, hence the two spaces are the same.
\end{proof}

A highly nontrivial consequence of the slice theorem is the existence of what is known as a \emph{principal orbit type} in $M$.

\begin{prop}\label{prop:princ}
Let $G$ be a compact Lie group acting on a manifold $M$ without corners.  If the orbit space $M/G$ is connected, then there exists a unique conjugacy class $(H)$ of subgroups of $G$ such that the orbit type stratum $M_{(H)}$ is open and dense in $M$.
\end{prop}

\begin{proof}[Summary of the proof] \mbox{ } \newline
This is theorem 4.27 in \cite{Ka}.  The proof is several pages, so we opt to summarize it for the purpose of explaining how it is derived from the slice theorem.  A condensed version of the proof may be found in Eckhard Meinrenken's notes, \cite{Me}.  The strategy is as follows: one first fixes a point $p\in M$ with stabilizer $G_p$ and shows that such an orbit-type stratum exists in the model $G\times_{G_p} W$, where $W$ is the differential slice.  This is done using induction on the dimension of $M$ and constructing an equivariant retraction of $G\times_{G_p}(W\setminus \{0\})$ onto the unit sphere bundle $G\times_{G_p} S(V)$, where $S(V)$ is the unit sphere in $V$ for some choice of invariant metric.  One then covers $M$ by invariant neighborhoods isomorphic to the local models $G\times_H W$, where $H$ and $W$ depend on the choice of point $p\in M$.  A unique, open, dense orbit type stratum exists on each neighborhood.  Uniqueness guarantees that the orbit-type strata agree on overlaps, hence there is a unique orbit-type stratum $M_{(H)}\subset M$ that is open and dense.
\end{proof}

We'll need proposition \ref{prop:princ} in our discussion of effective, abelian group actions.

\begin{definition}\label{def:effective}
Let $G$ be a Lie group acting on a manifold with corners $M$.  We say that the action of $G$ is \emph{effective} if the homomorphism $\tau:G \rightarrow \operatorname{Diff}(M)$ is injective.  Equivalently, for each $g\in G$ with $g\ne e$ there exists a point $p\in M$ such that $\tau_g(p)\ne p$.
\end{definition}

We have several structural lemmas pertaining to effective, abelian group actions.

\begin{lemma}\label{lem:eff1}
Let $M$ be a manifold without corners and suppose $G$ is a compact abelian group acting effectively on $M$.  Then the set $M_{e}$ is open and dense in $M$.  That is, the set where the action of $G$ is free is open and dense in $M$.
\end{lemma}

\begin{proof} \mbox{ } \newline
By proposition \ref{prop:princ}, there exists an orbit-type stratum $M_{(H)}$ that is open and dense in $M$.  Since $G$ is abelian, $M_{(H)}=M_H$, hence there is a subgroup $H\le G$ such that $M_H$ is open and dense in $M$.  If $H$ is not $\{e\}$, then there exists $h\in H$, $h\ne e$ and $\tau_h$ acts trivially on $M_H$.  Since $M_H$ is open and dense and the action is smooth, $\tau_h$ fixes all of $M$.  This means that the action is not effective, contradicting our assumptions.
\end{proof}

\begin{lemma}\label{lem:eff2}
Let $M$ be a manifold without corners and suppose $G$ is a compact abelian group acting effectively on $M$.  Then $G$ acts effectively on every invariant, open subset of $M$.
\end{lemma}

\begin{proof}\mbox{ } \newline
Assume that there is an invariant open subset $U\subset M$ so that the action of $G$ is not effective.  Then there exists $g\in G$, $g\ne e$, such that $\tau_g\vert_U = id_U$, which means the stabilizer of every point in $U$ is nontrivial.  However, lemma \ref{lem:eff1} implies that $G$ acts freely on an open dense subset of $M$, hence it acts freely on an open dense subset of $U$.  Thus, we have arrived at a contradiction and we must have that $G$ acts effectively on \emph{all} invariant, open subsets of $M$.
\end{proof}

\begin{lemma}\label{lem:eff3}
Let $M$ be a manifold without corners and suppose $G$ is a compact abelian group acting effectively on $M$.  Let $p\in M$, let $H=G_p$ be the stabilizer, and let $W=T_pM/T_p(G\cdot p)$ be the differential slice representation of $H$ at $p$.  Then $H$ acts effectively on $W$.  Equivalently, the representation $\rho:H \to GL(W)$ is faithful.
\end{lemma}

\begin{proof}\mbox{ } \newline
Suppose there is an element $h\in H$ such that $h\ne e$ and $h$ fixes all elements of $W$.  By the slice theorem, an invariant neighborhood of $p$ is isomorphic to an invariant neighborhood of the zero section of $G\times_H W$.  We then have that for any element $[(g,v)]\in G\times_H W$:

\begin{displaymath}
h\cdot [(g,v)] = [(hg,v)]= [(gh,v)] = [(g,h^{-1}\cdot v)]= [(g,v)]
\end{displaymath}
which means that $h$ fixes an invariant neighborhood of $p$.  This contradicts the conclusion of lemma \ref{lem:eff3}, hence we must have that the action of $h$ on $W$ is effective and the representation $\rho: H \to GL(W)$ is faithful.
\end{proof}

\subsection{Hamiltonian Actions in the Symplectic Case}
We introduce Hamiltonian actions for symplectic manifolds, study symplectic representations on vector spaces, and introduce a local normal form for proper, Hamiltonian actions on symplectic manifolds.  The two key components of this sections are the study of symplectic weights for representations of tori and the existence of a local normal form.  Everything in this section is review.

First, recall that a symplectic manifold is a manifold $M$ equipped with a closed, non-degenerate two form $\omega$.  Equivalently, it is a folded-symplectic manifold $(M,\omega)$ where the folding hypersurface $Z\subset M$ is empty.  A symplectic vector space is a vector space $V$ equipped with a non-degenerate element $\omega \in \Lambda^2(V^*)$, hence $(V,\omega)$ is a symplectic manifold since the form is constant, hence closed.  The group $\operatorname{Symp}(V,w)$ is the group of linear automorphisms of $V$ preserving $\omega$.  Let us recall some basic constructions:

\begin{enumerate}
\item If $V_1\subseteq V$ is a subspace, then $V_1^{\omega} = \{v\in V \mbox{ } \vert \mbox { } \omega(v,v_1)=0 \text{ for all $v_1\in V_1$}\}$
\item A subspace $V_1\subseteq V$ is isotropic if $V_1\subseteq V_1^{\omega}$.
\item A subspace $V_1 \subseteq V$ is coisotropic if $V_1^{\omega} \subseteq V_1$.
\item A subspace is Lagrangian if $V_1=V_1^{\omega}$.  Lagrangian subspaces are maximally isotropic: there are no larger isotropic subspaces in which they are contained.
\item Finally, if $(M,\omega)$ is a symplectic manifold, then a submanifold $N$ is isotropic/coisotropic/Lagrangian if the fibers of $TN$ are isotropic/coisotropic/Lagrangian in $TM\big\vert_N$
\end{enumerate}

\begin{definition}\label{def:Ham}
Let $(M,\omega)$ be a symplectic manifold and let $G$ be a Lie group which acts properly on $M$.  Let $\fg$ be the Lie algebra of $G$ and let $\fg^*$ be its dual.  We say the action is \emph{Hamiltonian} if:

\begin{enumerate}
\item for all $g\in G$, $\tau_g^*\omega=\omega$, where $\tau_g:M \to M$ is the action of $g$ on $M$, and
\item there exists an equivariant map $\mu:M\to \fg^*$ satisfying:
\begin{displaymath}
i_{X_M}\omega = -d\langle \mu, X\rangle, \text{ for all $X\in \fg$}
\end{displaymath}
\end{enumerate}
That is, the action preserves the symplectic form $\omega$ and the induced vector fields $X_M$ are Hamiltonian vector fields for the functions $\langle \mu, X \rangle$.
\end{definition}

\subsubsection{Symplectic Representations of Tori}

\begin{definition}\label{def:symrep}
Let $G$ be a Lie group.  A \emph{symplectic representation} of $G$ is a homomorphism $\rho:G \to \operatorname{Symp}(V,\omega)$, where $(V,\omega)$ is some symplectic vector space.
\end{definition}

\begin{lemma}\label{lem:symrep}
Let $G$ be a Lie group and let $\rho:G \to \operatorname{Symp}(V,\omega)$ be a symplectic representation.  Then the action of $G$ on $V$ is Hamiltonian with moment map defined by:

\begin{displaymath}
\langle \mu(V), X \rangle = -\frac{1}{2}\omega(d\rho_e(X)v,v)
\end{displaymath}
where $X\in \fg$ is a Lie algebra element.
\end{lemma}

\begin{proof}\mbox{ } \newline
We first note that the vector field induced by $X\in \fg$ is $X_V(v)= \frac{d}{dt}\big\vert_0 \rho(\exp(tX))v=d\rho_e(X)v$.  Fix a vector $v\in V$  and let $\eta \in T_vV$.  We compute:

\begin{displaymath}
\begin{array}{lcl}
\langle d\mu_v(\eta),X \rangle & = & \frac{d}{dt}\big\vert_0 (-\frac{1}{2}\omega(d\rho_e(X)(v+t\eta), v+t\eta)) \\
                               & = & -\frac{1}{2}(\omega(d\rho_e(X)v, \eta) + \omega(d\rho_e(X)\eta,v)) \\
                               & = & -\frac{1}{2}(\omega(d\rho_e(X)v,\eta) - \omega(\eta, d\rho_e(X)v) \\
                               & = & -\frac{1}{2}(\omega(d\rho_e(X)v,\eta)+\omega(d\rho_e(X)v,\eta) \\
                               & = & -\omega(d\rho_e(X)v,\eta) \\
                               & = & -(i_{X_V(v)}\omega)(\eta)
\end{array}
\end{displaymath}
Thus, $\mu$ is a moment map for the action of $G$.
\end{proof}

We now discuss symplectic representations of tori and their weights in a modicum of detail.  These representations may be viewed as complex representations of tori and the theory of these representations has been well studied (q.v. \cite{Ad,Th}), hence the material in this section is review.  We'll need the material in this section when we study orbit spaces of toric, folded-symplectic manifolds.

\begin{definition}
Let $G$ be a torus with Lie algebra $\fg$.  Then the exponential map $\exp:\fg \to G$ is a group epimorphism, i.e. it is surjective.  We define its kernel $\mathbb{Z}_G:=\ker(\exp)$ to be the \emph{integral lattice of} $G$.  The dual lattice $\mathbb{Z}_G^*$ is the set of all elements $\beta\in \fg^*$ for which $\beta(X)\in \mathbb{Z}$ for every $X\in \mathbb{Z}_G$.
\end{definition}

\begin{remark}
Given a torus $G$ with exponential map $\exp:\fg \to G$, we can define an action of $G$ on $\C^n$ if we are given a (multi)set of weights $\{\beta_1, \dots,\beta_n\}$.  We have:

\begin{displaymath}
\exp(X) \cdot (z_1, \dots, z_n) := (e^{2\pi i \beta_1(X)}z_1, \dots, e^{2\pi i \beta_n(X)} z_n)
\end{displaymath}
which gives us a well-defined action of $G$ since the integral lattice $\mathbb{Z}_G$ is mapped to the identity via this recipe, hence $\ker(\exp)$ is sent to the identity transformation and so the above recipe gives a well-defined action of $\fg/\ker(\exp) = G$ on $\C^n$.
\end{remark}

\begin{definition}\label{def:weights}
Let $G$ be a torus.  A \emph{character} is a homomorphism $\chi:G \to U(1)$ and a \emph{weight} is a differential of a character at the identity, $\beta=d\chi_e$.  Note that there is a one-to-one correspondence between weights and characters which can be seen from the commutative diagram:

\begin{displaymath}
\xymatrix{
\fg \ar[r]^{\beta} \ar[d]^{\exp_G} & \operatorname{Lie}(U(1)) \ar[d]^{\exp_{U(1)}} \\
G \ar[r]^{\xi}                     & U(1)
}
\end{displaymath}
If we identify $\operatorname{Lie}(U(1))$ with $\R$, then a weight $\beta$ is an element of $\fg^*$.  The character associated to $\beta$ is a homomorphism of tori, hence $\beta$ maps $\mathbb{Z}_G$ to $\mathbb{Z}\subset \R$ and so $\beta \in \mathbb{Z}_G^*$.  Thus, we refer to the set $\mathbb{Z}_G^*$ as the set of weights for $G$.
\end{definition}

\begin{example}\label{ex:complexrep}
Let $\rho:G \to GL(\C^n)$ be a complex representation of a torus $G$.  Then the standard theory of toric representations tells us that there is a multiset of weights associated to $\rho$, $\{\beta_1,\dots, \beta_n\}$, which specifies the representation up to isomorphism.  These weights exist because every complex representation of a torus splits as a direct sum of $1$-dimensional complex representations by Schur's lemma (q.v. proposition 3.7 in \cite{Ad}) and the action of $G$ preserves an invariant metric on each summand, hence $\rho$ is really a map $\rho:G \to \oplus_{i=1}^n U(1)$.  Projection onto the $i^{th}$ factor gives us a character, which gives us a corresponding weight $\beta_i$. The action of $G$ on a summand $\C$ is given explicitly by $\exp(X)\cdot z = e^{2\pi\beta_i(X)}z$, where $X\in \fg$ is a Lie algebra element.  Thus, the weight specifies the representation up to isomorphism.

If we include the symplectic structure $\omega_{\C}=\displaystyle \frac{i}{2\pi}dz \wedge d\bar{z}$ on $\C$ into our calculations, then the action $\exp(X)\cdot z = e^{2\pi\beta_i(X)}z$ is an Hamiltonian action and the moment map is $\mu(z)=\vert z \vert^2\beta_i$.  Thus, if we consider the action of $G$ on $\C^n$ with weights $\{\beta_1, \dots, \beta_n\}$, the action is Hamiltonian with moment map:

\begin{displaymath}
\mu(z_1,\dots,z_n)=\sum_{i=1}^n \vert z_i \vert^2 \beta_i
\end{displaymath}
\end{example}

\begin{lemma}\label{lem:sympweights}
Let $G$ be a torus.  There exists a one-to-one correspondence between multisets of weights $\{\beta_1, \dots, \beta_n\}$, $\beta_i\in \mathbb{Z}_G^*$, and isomorphism classes of symplectic representations $\rho:G \to \operatorname{Symp}(V,\omega)$, where $\dim(V)=2n$.
\end{lemma}

\begin{proof}\mbox{ } \newline
Choose an invariant almost complex structure $J:V \to V$ compatible with $\omega$, so that $\omega(\cdot, J\cdot)$ is an invariant metric.  The choice of $J$ identifies $V$ with $\C^n$, hence we have a complex representation of a torus.  Using the metric, for example, we may split $\C^n$ into a direct sum of $1$-dimensional complex representations: $V\simeq \oplus_{i=1}^n V_i$.  This is the combination of the facts that all representations split into a direct sum of unique irreducibles (unique up to permutation of summands) and that $G$ is a torus (q.v. lemma 3.25 and theorem 3.24 in \cite{Ad}).  These $V_i's$ will also be irreducible \emph{real} representations (q.v. \cite{B,Th}).  Now, the $V_i's$ are mutually orthogonal with respect to the metric, hence they are symplectically orthogonal as well since the metric is $\omega(\cdot, J\cdot)$.  $J$ restricts to an almost complex structure on each $V_i$, hence we have a splitting $(V,\omega) = \oplus_{i=1}^n (V_i,\omega_i)$ where each factor is linearly, symplectically isomorphic to $(\C, \frac{i}{2\pi}dz\wedge d\bar{z})$ via the choice of $J$.  We note that this splitting is independent of the choice of $J$ since the splitting of a representation into isotypicals is unique up to a reordering of the factors (q.v. \cite{Ad,B,Th}).  Since the action preserves the metric, each representation $\rho_i: G \to \operatorname{Symp}(V_i,\omega_i)$ is really a character $\chi_i:G \to  U(1)$.  Thus, a choice of a $J$ gives us a multi-set of weights $\{\beta_1, \dots, \beta_n\}$.

\vspace{5mm}

The weights do not depend on the choice of $J$.  Indeed, the space of all invariant almost complex structures is contractible since it is diffeomorphic to the space of invariant metrics, which is an affine space.  Thus, for two different choices of invariant almost complex structures, $J_1$ and $J_2$, there is a continuous path $\gamma(t)$ connecting them.  For each $(V_i,\omega_i)$, the path of almost complex structures, $\gamma(t)$, gives us a continuous path of characters $\chi_t: G\to U(1)$, hence a continuous path of weights $(\beta_i)_t$.  Since the set of weights is discrete, we have that $(\beta_i)_t$ does not depend on $t$, hence the weights do not depend on the choice of invariant almost complex structure.

\vspace{5mm}

Now, because the splitting of $(V,\omega)$ into isotypicals $\oplus_{i=1}^n(V_i,\omega_i)$ is unique up to reordering of the factors, the multiset of weights corresponding to $(V,\omega)$ is unique.  That is, $G$ cannot act on a summand $(V_i,\omega_i)$ with two different weights, hence the correspondence $\text{Representations} \to \text{Weights}$ is injective.  Furthermore, if we are given a multiset of weights $\{\beta_1,\dots,\beta_n\}$ then we may construct a symplectic representation of $G$ using the recipe $\exp(X)\cdot(z_1,\dots, z_n) = (e^{2\pi i\beta_i(X)}z_1,\dots,e^{2\pi i\beta_n(X)}z_n)$, where $X\in \fg$ is a Lie algebra element.  By our discussion, every symplectic representation of $G$ is isomorphic to a complex representation of this form and the multiset of weights corresponding to this representation is unique.  Thus, the correspondence between symplectic representations of tori and multisets of weights is both injective and surjective.
\end{proof}

\begin{cor}\label{cor:sympweights}
Let $\rho: G \to \operatorname{Symp}(V,\omega)$ be a symplectic representation of a torus $G$.  Then there exist weights $\beta_1,\dots, \beta_n$ associated to $\rho$ and an isomorphism of symplectic representations $\phi:(V,\omega) \to (\C^n,\omega_{\C^n})$, where $G$ acts on the $i^{th}$ factor of $\C^n$ with weight $\beta_i$.  Furthermore, the moment map for the action of $G$ on $\C^n$ is:

\begin{displaymath}
\mu(z_1,\dots,z_n)=\sum_{i=1}^n \vert z_i\vert^2\beta_i
\end{displaymath}
\end{cor}

\begin{proof} \mbox{ } \newline
This result follows from lemma \ref{lem:sympweights} and example \ref{ex:complexrep}.
\end{proof}

\begin{lemma}\label{lem:sympweights2}
Suppose $\rho:G \to \operatorname{Symp}(V,\omega)$ is a faithful symplectic representation of a compact abelian group, $G$.  Then
\begin{enumerate}
\item $\dim(G)\le \frac{1}{2}\dim(V)$ and
\item if $\dim(G)=\frac{1}{2}\dim(V)$, then $G$ is connected, hence it is a torus.  The weights corresponding to this representation form a $\mathbb{Z}$-basis for the set of weights $\mathbb{Z}_G^*$.
\end{enumerate}
\end{lemma}

\begin{proof} \mbox{ } \newline
\begin{enumerate}
\item Let $\dim(V)=2n$.  A choice of invariant almost complex structure $J$ on $V$ compatible with $\omega$ allows us to view $\rho$ as a homomorphism $\rho:G \to U(n)$, where $U(n)$ is the unitary group, since $G$ preserves the metric $\omega(\cdot, J\cdot)$.  Let $G^0$ be the connected component of the identity in $G$, which is a torus since $G$ is compact and abelian.  The image $\rho(G^0)$ of $G^0$ is then a subtorus of $U(n)$, which lies inside a maximal torus of $U(n)$.  All such tori are conjugate (q.v. corollary 4.23 in \cite{Ad}), so we may assume it is the standard $n$-torus.  Thus,
    \begin{displaymath}
    \dim(G)=\dim(G_0) = \dim(\rho(G^0)) \le n
    \end{displaymath}
    where we have used that $\rho$ is injective.  Since $n=\frac{1}{2}\dim(V)$, we have $\dim(G) \le \frac{1}{2}\dim(V)$.

\item If $\dim(G)=\frac{1}{2}\dim(V)=n$, then $\rho(G^0)$ is a torus of dimension $n$ in $U(n)$, where $G^0$ is the connected component of the identity.  Since all maximal tori in $U(n)$ are of dimension $n$ and $\rho(G^0)$ has dimension $n$, we conclude that $\rho(G^0)$ is itself a maximal torus, hence $\rho\vert_{G^0}$ surjects onto an $n$-dimensional torus in $U(n)$.  Since $\rho(G)$ is contained in this torus and $\rho$ is injective, we must have $G^0=G$, else there are two points sent to the same element in $U(n)$.

\item Finally, if we decompose $(V,\omega)$ into its $2$-dimensional isotypical subspaces $(V,\omega)=\oplus_{i=1}^n (V_i,\omega_i)$, then $\rho$ is a map $\rho:G \to U(1)^n$, the characters $\chi_i:G \to U(1)$ are given by projection onto the $i^{th}$ factor of $U(1)^n$, and the weights are $\beta_i=(d\chi_i)_e$.  By part $2$, if we assume that $\dim(G)=n$, then $\rho$ is an isomorphism of tori, meaning $d\rho_e$ maps $\ker(\exp_G)$ isomorphically onto $\ker(\exp_{U(1)^n})=\mathbb{Z}^n$.  Since $d\rho_e= (\beta_1,\dots,\beta_n)$, this can only happen if the $\mathbb{Z}$ span of $\{\beta_1,\dots,\beta_n\}$ is all of $\mathbb{Z}_G^*$.
\end{enumerate}
\end{proof}

\subsubsection{Symplectic Normal Form}

We would now like to state a few important results in the theory of Hamiltonian actions on symplectic manifolds, which we will use to study the local uniqueness of toric, folding hypersurfaces and, consequently, toric folded-symplectic manifolds..  The following lemma is due to Guillemin and Sternberg and is a staple of the study of Hamiltonian group actions.  It is theorem 3.5 in $\cite{GS1}$.
\begin{lemma}\label{lem:symporbit}
Let $(M,\omega)$ be a symplectic manifold with a proper, Hamiltonian action of a Lie group $G$ and corresponding moment map $\mu:M\to \fg^*$.  Let $H\le G$ be a subgroup.  Then the set $M_H$ is a symplectic submanifold of $M$.
\end{lemma}
\begin{proof} \mbox{ } \newline
We need only show that the tangent space to $M_H$ is symplectic.  By corollary \ref{cor:slice2}, for $p\in M_H$ the tangent space $T_pM_H$ is $(T_pM)^H$, the subspace of vectors fixed by the action of $H$.  We claim this subspace is symplectic.  We may assume $H$ is compact since the action is proper, hence we may choose an invariant almost complex structure $J$ compatible with $\omega$.  Then, for $v\in (T_pM)^H$ with $v\ne 0$, $(d\tau_h)_p(v)=v$ implies $(d\tau_h)_p(Jv)=J(d\tau_h)_p(v)=Jv$, hence $Jv$ is also fixed by the action of $H$.  Since $\omega(v,Jv)>0$, we see that $\omega$ is nondegenerate on $(T_pM)^H$, hence it is nondegenerate on $T_pM_H$.  Thus, $(M_H, i^*\omega)$ is a symplectic submanifold of $M$.
\end{proof}

Since $M_H$ is a symplectic submanifold, the bundle $TM_H^{\omega}$ is a symplectic vector bundle over $TM_H$.  Since it is complementary to $TM_H$ in $TM\big\vert_{M_H}$, it is canonically isomorphic to the normal bundle $\nu(M_H)$ and we see that there is a canonical symplectic structure on $\nu(M_H)$. Instead of studying the symplectic normal bundle to $M_H$, we could study the symplectic normal bundle to an orbit $G\cdot p$, the fibers of which are called the \emph{symplectic slice representation}.

\begin{definition}\label{def:sympslice}
Let $(M,\omega)$ be a symplectic manifold with a proper, Hamiltonian action of a Lie group $G$ and let $p\in M$ be a point with stabilizer $H:=G_p$.  We define the \emph{symplectic slice} to be:

\begin{displaymath}
V:= \frac{T_p(G\cdot p)^{\omega}}{T_p(G\cdot p) \cap T_p(G\cdot p)^{\omega}}
\end{displaymath}
Since the form $\omega$ is $G$ invariant, it is $H$-invariant and $T_p(G\cdot p)^{\omega}$ is an invariant subspace of $T_pM$.  $T_p(G\cdot p)$ is also an invariant subspace since the orbits are invariant under the action of $G$, hence their intersection $T_p(G\cdot p)^{\omega}\cap T_p(G\cdot p)$ is invariant and $V$ inherits the structure of a representation of $H$.  Since $H$ preserves $\omega$, the representation is symplectic, hence we often call $V$ the \emph{symplectic slice representation}.
\end{definition}

The significance of the symplectic slice representation is that its data along with the pullback of the symplectic form $\omega$ to the orbit $G\cdot p$ determines an invariant neighborhood of $G\cdot p$ up to equivariant symplectomorphism.  This is the content of the equivariant constant rank embedding theorem, which we do not discuss here (q.v. \cite{LS, Me}).  We will be concerned with group actions for which the orbits are isotropic.  That is, the pullback to the orbit is $0$: $i^*\omega =0$.  In this case, the equivariant constant rank embedding theorem is the equivariant version of Weinstein's isotropic embedding theorem.  There is a convenient normal form for neighborhoods of such orbits discovered by Guillemin and Sternberg (\cite{GS1}).

\begin{theorem}\label{thm:sympnorm}
Let $(M,\omega)$ be a symplectic manifold with a proper, Hamiltonian action of a Lie group $G$ with moment map $\mu$ and let $p\in M$ be a point with stabilizer $G_p=H$.  Suppose the orbit $G\cdot p$ is an isotropic submanifold of $M$ ($i^*\omega =0$) and let $V$ be the symplectic slice representation at $p$.  Let $\fg$ be the Lie algebra of $G$ and choose an $H$-equivariant splitting $\fg^*=\frak{h}^*\oplus \frak{h}^o$, where $\frak{h}$ is the Lie algebra of $h$ and $\frak{h}^o$ is its annihilator.  Then there exists a symplectic structure on the total space of
\begin{displaymath}
E=G\times_H \frak{h}^o \oplus V,
\end{displaymath}
an invariant neighborhood $U_1$ of the zero section of $E$, an invariant neighborhood $U_2$ of $G\cdot p$, and an equivariant symplectomorphism $\phi:U_1\to U_2$ satisfying $\phi([e,0,0])=p$ so that the moment map $\mu \circ \phi$ is given by:

\begin{equation}
\mu\circ \phi(g,\eta,v) = Ad^*(g)(\eta + \Phi_V(v)) + \mu(p)
\end{equation}
where the action of $G$ on $E$ is given by $g_0\cdot [(g,\eta,v)] = [(g_0g,\eta,v)]$ and $\Phi_V$ is the canonically defined moment map for the action of $H$ on $V$ (q.v. lemma \ref{lem:symrep}).
\end{theorem}

\begin{proof}[Sketch of a Proof] \mbox{ } \newline
We provide a short sketch of how one might prove theorem \ref{thm:sympnorm}.  Suppose $(M_i,\omega_i)$, $i=1,2$ are two symplectic manifolds, each with a proper Hamiltonian action of $G$ and moment maps $\mu_i$.  Let $p_i\in M_i$ be a point in $M_i$ for $i=1,2$ and suppose the stabilizers $G_{p_1}$, $G_{p_2}$ agree, say $G_{p_i}=H$.  Furthermore, suppose that the orbits $G_\cdot p_1$ and $G\cdot p_2$ are isotropic and that $\mu_1(p_1)=\mu_2(p_2)$.  The equivariant constant rank embedding theorem states that there is an equivariant symplectomorphism between neighborhoods $U_i$ of $G\cdot p_i$ if and only if the symplectic slice representations are linearly, symplectically isomorphic.  Thus, to prove the theorem it suffices to prove that
\begin{itemize}
\item there exists a symplectic structure on $E$,
\item the orbit $G\cdot ([e,0,0])$ is isotropic, and
\item the symplectic slice representations for $G\cdot p$ and $G\cdot([e,0,0])$ are isomorphic.
\end{itemize}
The symplectic structure comes from the fact that $E$ is the reduced space $(T^*G \times V)//_0H$.  We will discuss symplectic reduction in chapter 6, so the reader may need to take this as a black box for now.  The orbit of $([e,0,0])$ is then the image of the $0$ section of $T^*G \times V$ in the reduced space.  Since the zero section in $T^*G \times V$ is isotropic, its image in the reduced space is isotropic.  By construction, the symplectic slice representation at $([e,0,0])$ is precisely $V$.  Thus, the equivariant constant rank embedding theorem gives us an isomorphism between a neighborhood of $G\cdot p$ in $M$ and a neighborhood of the zero section of $G\times_H \frak{h}^o \oplus V$.  The equation for the moment map comes from computing the moment map for the reduced space.
\end{proof}
We will come back to this normal form theorem when we discuss toric actions on folded-symplectic manifolds.

\subsection{Hamiltonian Actions on Folded-Symplectic Manifolds}
We define Hamiltonian actions for folded-symplectic manifolds and construct an equivariant analog of proposition \ref{prop:fsnormal} for folded-symplectic manifolds with co-orientable folding hypersurfaces, which gives us a useful normal form for a neighborhood of the folding hypersurface.  Our normal form generalizes theorem $1$ in \cite{CGW} since we do not require compactness or orientability of the folded-symplectic manifold.  We then use this normal form to study the structure of folded-symplectic manifolds equipped with Hamiltonian actions of Lie groups.  Along the way, we prove a new result that reveals that every folding hypersurface in a folded-symplectic Hamiltonian $G$-manifold may be realized as:

\begin{enumerate}
\item a hypersurface in a \emph{symplectic}, Hamiltonian $G$-manifold or
\item a co-orientable folding hypersurface in a folded-symplectic, Hamiltonian $G$-manifold.
\end{enumerate}
That is, regardless of whether or not a folding hypersurface is co-orientable in the original ambient manifold, we can always extract it and equivariantly embed it into a folded-symplectic manifold as a co-orientable folding hypersurface.

We begin with the definition of an Hamiltonian action.

\begin{definition}\label{def:fsham}
Let $(M,\sigma)$ be a folded-symplectic manifold without corners and let $G$ be a Lie group.  We say an action of $G$ on $M$ is Hamiltonian if:

\begin{enumerate}
\item $\tau_g^*\sigma = \sigma$ for all $g\in G$ and
\item There exists an equivariant map $\mu:M\to \fg^*$ satisfying
    \begin{displaymath}
    i_{X_M}\sigma = -d\langle \mu,X\rangle, \text{ for all lie algebra elements $X\in \fg$.}
    \end{displaymath}
\end{enumerate}
We call $\mu$ the moment map and refer to $(M,\sigma)$ with the action of $G$ as a \emph{folded-symplectic Hamiltonian $G$-manifold}.
\end{definition}

\begin{lemma}\label{lem:preserves}
Let $G$ be a Lie group and let $(M,\sigma)$ be a folded-symplectic Hamiltonian $G$-manifold with folding hypersurface $Z$.  Then the action of $G$ preserves the fold, $Z$.  That is, $Z$ is an invariant submanifold of $M$.  Furthermore, the action of $G$ preserves $\ker(\sigma)\to Z$ and if $Z$ is equipped with its orientation induced by $\sigma$, then the action of $G$ on $Z$ is orientation-preserving.
\end{lemma}

\begin{proof}\mbox{ } \newline
We first show that the action preserves $Z$ and then show it is orientation preserving.
\begin{itemize}
\item For all $g\in G$, we have $\tau_g^*\sigma = \sigma$ by the definition of an Hamiltonian action, where $\tau_g$ is the action of $g$ on $M$.  Let $2m=\dim(M)$ be the dimension of $M$.  Then $\tau_g^*\sigma = \sigma$ implies $\tau_g^*(\sigma^m)=\sigma^m$.  Since $\tau_g$ is a diffeomorphism, we must have that $\tau_g(Z) \subset Z$.  Since the same logic holds for $g^{-1}$ and $\tau_g^{-1}=\tau_{g^{-1}}$, we have that $\tau_{g}^{-1}(Z) \subset Z$ and so $\tau_g(Z)=Z$.

\item To see that the action of $G$ preserves $\ker(\sigma)$, we pick a point $p\in Z$, a vector $v\in \ker(\sigma_p)$, an element $g\in G$, and we compute:

    \begin{displaymath}
    i_{d\tau_g(v)}\sigma = \sigma_{g\cdot p}(d\tau_g(v), \cdot ) = \sigma_p(v, d\tau_{g^{-1}}(\cdot)) = 0
    \end{displaymath}
    hence $\ker(\sigma)$ is an invariant subbundle of $TM\big\vert_Z$.  Since $Z$ is invariant, we also have that $\ker(\sigma)\cap TZ$ is an invariant subbundle of $TZ$.

\item  We now want to show the action is orientation preserving.  It is enough to show the action on $\ker(\sigma)\cap TZ$ is orientation preserving since the orientation on $Z$ is equivalent to a choice of orientation on $\ker(\sigma)\cap TZ$.  Let $p\in Z$ and choose a co-orientable neigbourhood $U\subseteq Z$ of $p$.  We may choose any non-vanishing section $w$ of $\ker(\sigma)$ on $U$ that is transverse to $Z$ and extend it to a local vector field $\tilde{w}$ in a neighborhood $\tilde{U}$ of $p$.  An element $v$ of $\ker(\sigma_p)\cap T_pZ$ is then positively oriented if, for any extension of $v$ to a vector field $\tilde{v}$, we have:
    \begin{equation}\label{eq:recipe}
    w_p(\sigma(\tilde{w},\tilde{v}))= d(\sigma(\tilde{w},\tilde{v}))_p(w_p) >0
    \end{equation}
    Fix an element $g\in G$.  We define a function:
    \begin{displaymath}
    \xymatrixrowsep{.5pc}\xymatrixcolsep{3pc}\xymatrix{
    f:\tilde{U} \ar[r] & \R \\
    f(x) \ar[r] & \sigma_x(\tilde{w}(x),\tilde{v}(x))
    }
    \end{displaymath}
    and a function $h:\tau_g(\tilde{U})\to \R$ given by:
    \begin{displaymath}
    \xymatrixrowsep{.5pc}\xymatrixcolsep{3pc}\xymatrix{
    h:\tau_g(\tilde{U}) \ar[r] & \R \\
    h(\tau_g(x)) \ar[r] & \sigma_{\tau_g(x)}(d\tau_g(\tilde{w}(x)),d\tau_g(\tilde{v}(x)))
    }.
    \end{displaymath}
    Then $h\circ \tau_g = f$ by construction.  We then have:
    \begin{displaymath}
    d(\sigma(\tilde{w},\tilde{v}))_p(w_p) = df(w_p) = dh(d\tau_g(w_p))= d(\sigma(d\tau_g(\tilde{w}),d\tau_g(\tilde{v})))_{g\cdot p}(d\tau_g(w_p))
    \end{displaymath}
    Since $w_p$ is an element of $\ker(\sigma_p)$, $d\tau_g(w_p)$ is an element of $\ker(\sigma_p)$ by part $2$.  Since $w_p$ is transverse to $Z$ and $\tau_g$ is a diffeomorphism that preserves $Z$, $d\tau_g(w_p)$ is transverse to $Z$.  Since $\ker(\sigma)\cap TZ$ is an invariant subbundle of $TZ$, $v \in \ker(\sigma_p)\cap T_pZ$ if and only if $d\tau_g(v) \in \ker(\sigma_{g\cdot p})\cap T_{g\cdot p}Z$.  Lastly, we note that $d\tau_g(\tilde{w})$ and $d\tau_g(\tilde{v})$ are extensions of $d\tau_g(w_p)$ and $d\tau_g(v)$, respectively.  Thus, $d\tau_g(v)$ is positively oriented if and only if:

    \begin{displaymath}
    d(\sigma(d\tau_g(\tilde{w}),d\tau_g(\tilde{v})))_{g\cdot p}(d\tau_g(w_p)) >0
    \end{displaymath}
    by equation \ref{eq:recipe}.  But $d(\sigma(d\tau_g(\tilde{w}),d\tau_g(\tilde{v})))_{g\cdot p}(d\tau_g(w_p))>0$ if and only if $d(\sigma(\tilde{w},\tilde{v}))_p(w_p) = df(w_p)>0$, as we have shown.  Hence $v$ is positively oriented if and only if $d\tau_g(v)$ is positively oriented, which shows that the action is orientation-preserving.
\end{itemize}
\end{proof}

\begin{cor}\label{cor:preserves}
Let $G$ be a compact Lie group and let $(M,\sigma)$ be a folded-symplectic Hamiltonian $G$-manifold with folding hypersurface $Z$ and moment map $\mu:M \to \fg^*$.  Then there exists an equivariant co-isotropic embedding of $Z$ into a symplectic, Hamiltonian $G$-manifold $(M_0,\omega_0)$.  There exists an equivariant fold map $\psi:M_0 \to M_0$ which folds along the image of $Z$, hence $Z$ also embeds into the folded-symplectic Hamiltonian $G$-manifold $(M_0,\psi^*\omega_0)$ as a co-orientable folding hypersurface.
\end{cor}

\begin{proof}\mbox{ } \newline
Since $\ker(\sigma)\cap TZ$ is oriented, we may choose a non-vanishing section $v$ and choose a corresponding $1$-form $\alpha \in \Omega^1(TZ)$ so that $\alpha(v)=1$.  Since $G$ is compact, we may average $\alpha$ via the equation:

\begin{displaymath}
\frac{1}{\vert G\vert} \int_G \tau_g^*\alpha dg
\end{displaymath}
so that it is $G$ invariant.  Since the action of $G$ is orientation preserving, $\alpha(v)\ne 0$.  We then form the $G$-space:

\begin{displaymath}
(Z\times \R, \omega_0=p^*(i_Z^*\sigma) + d(tp^*\alpha)), \text{ $g\cdot (z,t)=(g\cdot z, t)$}.
\end{displaymath}
where $p:Z\times \R \to Z$ is the projection.  Since $\sigma$ is closed, $\omega_0$ is closed.  Since $\sigma$ is $G$-invariant and $\alpha$ is $G$- invariant, $\omega_0$ is $G$-invariant.  We have $\displaystyle dt\wedge p^*\alpha(\frac{\partial}{\partial t},v) >0$, hence $\omega_0$ is non-degenerate at $Z\times \{0\}$ and is non-degenerate in an invariant neighborhood $M_0$ of the zero section $Z\times \{0\}$.  A moment map for the action of $G$ on $M_0$ is given by $\mu_s(z,t)= \mu\big\vert_Z(z) + tF(z)$, where $F(z)$ is defined by the equation $\langle F(z), X \rangle = \alpha(X_Z)$, hence the action of $G$ on $M_0$ is Hamiltonian.  We therefore have an equivariant co-isotropic embedding into a symplectic, Hamiltonian $G$-manifold.  To obtain the fold map, simply restrict $M_0$ to the set of points where $\vert t \vert \le 1$ and set $\psi(z,t)=(z,t^2)$.
\end{proof}

\begin{lemma}\label{lem:preserves1}
Suppose $G$ is a connected Lie group and $(M,\sigma)$ is a folded-symplectic Hamiltonian $G$-manifold, where the action of $G$ is proper.  Then for each $p\in Z$, $\ker(\sigma_p)$ is a representation of $G_p$.  If the fold is co-orientable, this representation is trivial.
\end{lemma}

\begin{proof} \mbox{ } \newline
For all $g\in G$, we have $\tau_g^*\sigma = \sigma$ by definition of an Hamiltonian action.  Thus, for $p\in Z$, $v\in T_pM$, and $h\in G_p$, we have $i_v\sigma_p = 0$ if and only if $i_{(d\tau_h)_pv}\sigma_p =0$, hence $\ker(\sigma_p)$ is an invariant subspace of $T_pM$ for the action of $G_p$.  Thus, $\ker(\sigma_p)$ is a representation of $G_p$.

\vspace{3mm}

By lemma \ref{lem:preserves}, the action of $G$ preserves $\ker(\sigma)\cap TZ$, hence $\ker(\sigma_p)\cap T_pZ$ is a representation of $G_p$.  Lemma \ref{lem:preserves} also implies that $G$ acts on $\ker(\sigma)\cap TZ$ by orientation-preserving isomorphisms, hence the action of $G_p$ on $\ker(\sigma_p)\cap T_pZ$ is orientation preserving.  Since the action of $G$ is proper, $G_p$ is compact by lemma \ref{lem:orbits}.  $\ker(\sigma_p)\cap T_pZ$ is $1$-dimensional, hence we may assume that we have a representation $\rho: G_p \to GL(\R)=\R\setminus \{0\}$ of a compact Lie group on $\R$ that is also orientation-preserving.  For each $h\in G_p$, $\rho(h)=\pm 1$.  Otherwise, the set $\{\rho(h)^n\vert n\in \mathbb{Z}\}$ is unbounded if $\vert\rho(h)\vert>1$, hence the image of $\rho$ is not compact, or its closure contains $0$ if $\vert\rho(h)\vert<1$ and the image is similarly non-compact.  In either case, we have a contradiction since $G_p$ itself is compact.  Since the action is orientation preserving, $\rho(G_p)=\{1\}$, hence the representation of $G_p$ on $\ker(\sigma_p)\cap T_pZ$ is trivial.

Similarly, if $Z$ is co-orientable then $G$ preserves a choice of co-orientation on $Z$.  Co-orientability of $Z$ is equivalent to the orientability of the bundle $\ker(\sigma)/\ker(i_Z^*\sigma)$, since this bundle is isomorphic to the normal bundle of $Z$.  Thus, if we fix a co-orientation on $Z$ and choose an invariant complement $V$ to $\ker(\sigma_p)\cap T_pZ$ in $\ker(\sigma_p)$, then $V$ is an oriented $1$-dimensional real representation of $G_p$ and the action is orientation preserving.  Thus, the representation is trivial on $V$.  Since $\ker(\sigma_p)\simeq V \oplus (\ker(\sigma_p))\cap T_pZ$, we have that the representation $\ker(\sigma_p)$ of $G_p$ is trivial.
\end{proof}

The following is an equivariant analog of proposition \ref{prop:fsnormal}.  As with proposition \ref{prop:fsnormal}, it is a generalization of theorem 1 in \cite{CGW}: we do not require that the folded-symplectic manifold be orientable and we do not require compactness.

\begin{prop}\label{prop:eqfsnormal}
Let $G$ be a compact, connected Lie group and suppose $(M,\sigma)$ is a folded-symplectic Hamiltonian $G$-manifold with moment map $\mu:M\to \fg^*$, where $\fg$ is the Lie algebra of $G$.  If the folding hypersurface $Z$ is co-orientable, then there exists an invariant neighborhood $U_1$ of the zero section of $Z\times \R$, an invariant neighborhood $U_2$ of $Z$ in $M$, and an equivariant diffeomorphism:

\begin{displaymath}
\phi:U_1 \to U_2, \text{ where the action on $U_1$ is $g\cdot (z,t)=(g\cdot z, t)$}
\end{displaymath}
such that $\phi^*\sigma = p^*i^*\sigma + d(t^2p^*\alpha)$ and $\phi(z,0)=z$ for all $z\in Z$.  Here, $p:Z\times \R \to Z$ is the projection, $i:Z\to Z\times \R$ is the inclusion as the zero section, and $\alpha\in \Omega^1(Z)^G$ is an invariant $1$-form that does not vanish on $\ker(i_Z^*\sigma)$ and orients it in the canonical way.
\end{prop}

\begin{proof} \mbox{ } \newline
The proof is essentially the same as the proof of proposition \ref{prop:fsnormal}, except that we must ensure all constructions are equivariant.  We therefore walk the reader through the relevant steps and assume all claims unrelated to invariance or equivariance have been proven (as they have been in proposition \ref{prop:fsnormal}).
\begin{itemize}
\item We first note that $G$ preserves the orientation on $\ker(\sigma)\cap TZ$ induced by $\sigma$ (q.v. lemma \ref{lem:preserves}).  Furthermore, since $G$ is connected, it preserves any choice of co-orientation on the fold.
\item Since $\ker(\sigma)\cap TZ$ is oriented, we may choose a $1$-form $\alpha \in \Omega^1(Z)$ so that $\alpha\big\vert_{\ker(\sigma)\cap TZ}$ is nonvanishing and positive for any oriented section.  Since $G$ is compact, we may average $\alpha$ using the formula:
    \begin{displaymath}
    \frac{1}{\vert G\vert} \int_G \tau_g^*\alpha dg
    \end{displaymath}
    hence we may assume $\alpha$ is $G$-invariant.  Since $G$ preserves the orientation on $\ker(\sigma)\cap TZ$, we have that this invariant form is still non-vanishing on $\ker(\sigma)\cap TZ$.
\item Now, as in proposition \ref{prop:fsnormal}, we choose a vector $w$ on $M$ so that for each point $z\in Z$ we have $w(z) \in \ker(\sigma_z)$ and $w(z)$ is transverse to $T_zZ$.  We may then average $w$ via the formula:
    \begin{displaymath}
    \frac{1}{\vert G \vert} \int_G d\tau_g(w_{(\tau_g^{-1}(z))})dg
    \end{displaymath}
    hence we may assume $w$ is $G$-invariant.  Since $G$ preserves the co-orientation of $Z$ induced by $w$, the averaged vector field is still transverse to $Z$ at points of $Z$.  Since the action of $G$ preserves the subbundle $\ker(\sigma)$, the averaged vector field still has values in $\ker(\sigma)$ at points of $Z$.
\item Since $w$ is $G$-invariant, its flow $\Phi(z,t)$ is $G$-equivariant.
\item We use the flow $\Phi$ of $w$ to define $\tilde{\phi}:Z\times \R \to M$ by
    \begin{equation}\label{eq:flowdiff}
    \phi(z,t)=\Phi(z,t)
    \end{equation}
     which satisfies $\phi(z,0)=z$.  Furthermore, it is a diffeomorphism in a neighborhood $U$ of $Z\times \{0\}$.  By the slice theorem, every neighborhood of a point in $Z\times \R$ contains an invariant neighborhood of the same point, hence we may shrink the neighborhood $U$ to assume it's invariant.
\item $\tilde{\phi}^*\sigma$ and $p^*i^*\sigma + d(t^2p^*\alpha)$ are $G$-invariant, agree at $Z\times \{0\}$, and induce the same orientation on $Z$, hence the linear path
    \begin{displaymath}
    \sigma_s := s\tilde{\phi}^*\sigma + (1-s)(p^*i^*\sigma + d(t^2p^*\alpha)=p^*i^*\sigma + (1-s)(d(t^2p^*\alpha) + s\mu= p^*i^*\sigma +\mu_s
    \end{displaymath}
    is folded-symplectic in a neighborhood of $Z\times \{0\}$ by lemma \ref{lem:symplectize} ($\mu_s$ is non-degenerate on the kernel bundle $\ker(\sigma_s)=\ker(\sigma_0)$) and invariant under the action of $G$.
\item As discussed in proposition \ref{prop:fsnormal}, the derivative $\dot{\sigma}_s$ is exact: $\dot{\sigma}_s = -d\beta_s$.  Furthermore, we may choose $\beta_s$ so that it vanishes to second order at $Z\times \{0\}$.  Since the group $G$ is compact, we may average $\beta_s$:

    \begin{displaymath}
    \tilde{\beta}_s=\frac{1}{\vert G \vert} \int_G \tau_g^*\beta_s dg
    \end{displaymath}
    We have
    \begin{displaymath}
    -d\tilde{\beta}_s = -\frac{1}{\vert G \vert} \int_G \tau_g^*d\beta_s dg = \frac{1}{\vert G \vert} \int_G \sigma_s dg = \dot{\sigma}_s
    \end{displaymath}
    hence we may assume that the primitive for $\dot{\sigma}_s$ is $G$-invariant.  Since the action of $G$ preserves the fold and $\beta_s$ vanishes to second order at $Z\times \{0\}$, we have that $\tilde{\beta}_s$ also vanishes to second order at $Z\times \{0\}$.
\item Since $\tilde{\beta}_s$ vanishes on $\ker(\sigma_s)=\ker(\sigma_0)$ at $Z\times \{0\}$, proposition \ref{prop:Moser} gives us a unique time-dependent vector field $X_s$ so that $i_{X_s}\sigma_s = -\tilde{\beta}$.  We claim that $X_s$ is $G$-invariant:
    \begin{displaymath}
    \sigma_s(d\tau_g(X_s),\cdot )_{g\cdot p} = \sigma_s(X_s,d\tau_{g^{-1}} \cdot)_p = (i_{X_s}\sigma_s)(d\tau_{g^{-1}}\cdot)_p = -(\tau_{g^{-1}}^*\tilde{\beta_s})_p = -(\tilde{\beta_s})_p
    \end{displaymath}
    Since $X_s$ is unique, we must have $d\tau_g(X_s)=X_s$.
\item The flow of $X_s$ is therefore equivariant and gives us an equivariant isotopy $\phi_s$ so that $\phi_s^*\sigma_s = \sigma_0$.  Taking $\phi=\tilde{\phi}\circ \phi_1$, where $\tilde{\phi}$ is defined using a flow in equation \ref{eq:flowdiff}, gives us our requisite diffeomorphism from a neighborhood of the zero section of $Z\times \R$ onto a neighborhood of $Z$ in $M$.
\end{itemize}
\end{proof}

\begin{cor}\label{cor:eqfsnormal1}
Let $G$ be a compact, connected Lie group and let $(M,\sigma)$ be a folded-symplectic Hamiltonian $G$-manifold with moment map $\mu:M\to \fg^*$, where $\fg$ is the Lie algebra of $G$.  If the folding hypersurface $Z$ is co-orientable then there exists an invariant neighborhood $U$ of $Z$, a symplectic form $\omega\in \Omega^2(U)$ for which the action of $G$ is Hamiltonian, an equivariant map with fold singularities $\psi:U \to U$ which folds along $Z$ so that $\psi^*\omega= \sigma$, and a symplectic moment map $\mu_s$ so that $\psi^*\mu_s = \mu$ on $U$.
\end{cor}

\begin{proof}\mbox{ } \newline
According to the equivariant normal form proposition \ref{prop:eqfsnormal}, a local model for an invariant neighborhood of $Z$ is given by a neighborhood $V$ of the zero section of $Z\times \R$ equipped with the fold form
\begin{displaymath}
\sigma_0=p^*i^*\sigma + d(t^2p^*\alpha),
\end{displaymath}
where $\alpha\in \Omega^1(Z)$ positively orients the bundle $\ker(\sigma)\cap TZ$ and $\ker(\sigma_0) = (\ker(\sigma)\cap TZ) \oplus (\span(\frac{\partial}{\partial t}))$.  As a reminder, $p:Z\times \R \to Z$ is the projection and $i:Z \to Z\times \R$ is the inclusion as the zero section.  The form:

\begin{displaymath}
\omega = p^*i^*\sigma + d(tp^*\alpha)
\end{displaymath}
is non-degenerate in a neighborhood of $Z\times \{0\}$ since $dt \wedge p^*\alpha$ is symplectic on $\ker(\sigma_0)$.  It is closed since both of the summands are closed $2$-forms.  Thus, $\omega$ is symplectic in a neighborhood $V_1$ of $Z\times \{0\}$.

Take the intersection of $V\cap V_1$ with the set of points where $\vert t \vert <1$ and redefine $V$ to be this set.  Then the map $\psi(z,t)=(z,t^2)$ maps $V$ to $V$ and folds along $Z\times \{0\}$, $\omega$ is symplectic on $V$, and $\psi^*\omega = \sigma$.  A symplectic moment map for the action of $G$ is given by:

\begin{displaymath}
\mu_s(z,t) = \mu_Z(z) + tF
\end{displaymath}
where $\mu_Z = \mu\vert_Z$ and $F$ is defined by $\langle F, X \rangle = \alpha(X_M)$.  The folded-symplectic moment map is given by
\begin{displaymath}
\mu(z,t)=\mu_Z(z) + t^2F
\end{displaymath}
hence $\psi^*\mu_s = \mu$.  To summarize, we have a commutative diagram:

\begin{displaymath}
\xymatrix{
(U,\sigma) \ar[r]^\psi \ar[dr]^\mu &(U,\omega) \ar[d]^{\mu_s} \\
                                   & \fg^*
}
\end{displaymath}
\end{proof}

The following isn't entirely a corollary of proposition \ref{prop:eqfsnormal}: one can show the orbit-type strata are transverse to $Z$ without constructing the normal form.  However, the normal form makes this fact obvious, so we list it as a corollary.

\begin{cor}\label{cor:eqfsnormal2}
Let $G$ be a compact, connected Lie group and let $(M,\sigma)$ be a folded-symplectic Hamiltonian $G$-manifold (without corners) with moment map $\mu:M \to \fg^*$, where $\fg$ is the Lie algebra of $G$.  If the folding hypersurface $Z$ is co-orientable, then the orbit-type strata $M_{(H)}$ intersect $Z$ transversely.  Furthermore, if $H\le G$ is a subgroup, then $M_H \pitchfork Z$ and $(M_H, i^*\sigma)$ is a folded-symplectic manifold.
\end{cor}

\begin{proof}\mbox{ } \newline
In the local model of proposition \ref{prop:eqfsnormal}, the orbit type strata are given by $M_{(H)}=Z_{(H)} \times \R$, hence they intersect the fold $Z\times\{0\}$ transversely.  To prove the second claim, we begin by noting that $M_H = Z_H\times \R$ in the local model, hence $M_H\pitchfork Z$.  By lemma \ref{lem:symporbit}, $M_H\setminus Z$ is a symplectic submanifold of $M\setminus Z$, hence we need only check that $(M_H,i^*\sigma)$ is folded-symplectic near $Z$.  By corollary \ref{cor:eqfsnormal1} we have a commutative diagram:
\begin{displaymath}
\xymatrix{
(U,\sigma) \ar[r]^\psi \ar[dr]^\mu &(U,\omega) \ar[d]^{\mu_s} \\
                                   & \fg^*
}
\end{displaymath}
where $U$ is an invariant neighborhood of the fold, $\omega$ is symplectic, and $\mu_s$ is a symplectic moment map for the action of $G$.  The map $\psi$ is an equivariant fold map.  Equivariance implies that $\psi$ restricts to a map $\psi:M_H\cap U \to M_H\cap U$.  By corollary \ref{cor:folds4-1}, $\psi$ is a map with fold singularities.  Note that $\psi$ is guaranteed to have singularities along $M_H\cap Z$ since $\omega$ is non-degenerate, $\sigma = \psi^*\omega$, and $\sigma$ has singularities along $M_H\cap Z$.  Now, lemma \ref{lem:symporbit} implies that $(M_H\cap U,i^*\omega)$ is a symplectic submanifold, hence $(M_H\cap U, \psi^*i^*\omega)$ is folded-symplectic.  However, $\psi \circ i = i \circ \psi$, where $i:M_H\cap U \to M$ is the inclusion, hence $\psi^*i^*\omega = i^*\psi^*\omega = i^*\sigma$ and $(M_H\cap U, i^*\sigma)$ is folded-symplectic.  Thus, $(M_H,i^*\sigma)$ is a folded-symplectic submanifold of $M$.
\end{proof}

The following proposition allows us to study the normal bundle of $M_H$ and argue that it is naturally a symplectic vector bundle.  This proposition will also give us a canonical, invariant complement to the bundle $TM_H$ inside $TM\vert_{M_H}$, which we will use when we study orbit spaces.
\begin{prop}\label{prop:normbundle}
Let $G$ be a compact, connected Lie group and let $(M,\sigma)$ be a folded-symplectic Hamiltonian $G$-manifold with moment map $\mu:M\to \fg^*$, where $\fg$ is the Lie algebra of $G$.  Suppose the folding hypersurface $Z$ is co-orientable.  Let $H\le G$ be a subgroup and suppose $M_H$ is nonempty.  Then there exists a vector bundle $\widetilde{(TM_H)}^{\sigma}\to M_H$ with the following properties:

\begin{enumerate}
\item $\widetilde{(TM_H)}^{\sigma}$ is a subbundle of $TM\big\vert_{M_H}$.
\item The restriction $\widetilde{(TM_H)}^{\sigma}\big\vert_{M\setminus Z}$ to the symplectic portion of $M$ is the vector bundle $T(M_H \setminus Z)^{\sigma}\to (M_H\setminus Z)$.
\item $TM\big\vert_{M_H}$ splits $H$-equivariantly as $TM\big\vert_{M_H} = TM_H \oplus \widetilde{(TM_H)}^{\sigma}$.
\item $\widetilde{(TM_H)}^{\sigma}$ equipped with the restriction of $\sigma$ is a symplectic vector bundle over $M_H$.
\item $\widetilde{(TM_H)}^{\sigma}\big\vert_{Z_H}$ is a subbundle of $TZ_H$.
\end{enumerate}
In other words, the symplectic normal bundle to $M_H\setminus Z$ extends across the fold $Z$ to give us a symplectic normal bundle to $M_H$ and, at points of the intersection $Z_H=M_H\cap Z$, it is tangent to the fold.
\end{prop}

\begin{remark}
There's a straightforward way to see why this bundle should exist.  By corollary \ref{cor:eqfsnormal1}, we have a local model for a neighborhood $U$ of the folding hypersurface:
\begin{displaymath}
\xymatrix{
(U,\sigma) \ar[r]^{\psi} & (U,\omega)
}
\end{displaymath}
where $\psi$ is an equivariant fold map that folds along $Z$, $\omega$ is symplectic, and $\psi^*\omega=\sigma$.  The bundle $TM_H^{\omega}$ is well-defined and complementary to $TM_H$ in $TM \big\vert_H$.  By corollary \ref{cor:folds4-1}, $\psi:M_H \to M_H$ induces an isomorphism on each fiber of the normal bundle $d\psi:\nu(M_H)_p \to \nu(M_H)_{\psi(p)}$.  Since $\nu(M_H) \simeq TM_H^{\omega}$, we are led to believe that $\psi$ will allow us to simply pull back the bundle $TM_H^{\omega}$.  We do so using the recipe:

\begin{displaymath}
(\widetilde{M_H})^{\sigma}_p :=
\begin{cases}
d\psi_{\psi(p)}^{-1})(TM_H^\omega) & \text{if } p\notin Z_H \\
(d\psi\big\vert_Z)_{\psi(p)}^{-1}(TM_H^{\omega}) & \text{if } p\in Z_H
\end{cases}
\end{displaymath}
Of course, it's not immediately obvious that this works, which is why we need to prove something.  Furthermore, we don't want to use a model to define a bundle which appears to arise from the data intrinsic to a folded-symplectic Hamiltonian $G$-manifold, so we define this symplectic normal bundle without appealing to fold maps and symplectic forms, which require choices.
\end{remark}

\begin{proof}\mbox{ } \newline
Let $p\in M_H$.  We define the fiber of $\widetilde{(TM_H)}^{\sigma}$ at $p$ to be the elements of $T_pM_H$ that extend to local sections of the restricted tangent bundle $TM\big\vert_{M_H}$ with values in the distribution $(TM_H)^{\sigma}$.  It is straightforward to check that this is a vector subspace of $(T_pM_H)^{\sigma}$.  We argue that the dimension of $E_p$ is constant and, since all elements of $E_p$ extend to local sections with values in $(TM_H)^{\sigma}$, this fact implies $\widetilde{(TM_H)}^{\sigma}$ is a vector subbundle of $TM\big\vert_{M_H}$.

If $p\in M_H \setminus Z$, then by definition of $\widetilde{(TM_H)}^{\sigma}$, the fiber $\widetilde{(TM_H)}^{\sigma}_p$ is just $(T_pM_H)^{\sigma}$, which proves the second claim of the proposition.  Thus, we need only consider the case where we  are at the fold: $p\in M_H\cap Z$.  To this end, we may use proposition \ref{prop:eqfsnormal} to assume that our manifold is $Z\times \R$ with folded-symplectic form $\sigma=p^*i^*\sigma + d(t^2p^*\alpha)$ for some $1$-form $\alpha \in \Omega^1(Z)$ orienting $\ker(\sigma)\cap TZ$ positively.  Here, $p:Z\times \R \to Z$ is the projection and not the point $p$.  Let us fix some notation.  The kernel of $\sigma$ at $Z\times \{0\}$ is:

\begin{displaymath}
\ker(\sigma) = \span(\frac{\partial}{\partial t}) \oplus (\ker(\sigma)\cap TZ)
\end{displaymath}
and the orientability of $\ker(\sigma)\cap TZ$ means we may choose an invariant, nonvanishing section $v\in\Gamma(\ker(\sigma)\cap TZ)$ so that:
\begin{equation}\label{eq:kersigma}
\ker(\sigma) = \span(\frac{\partial}{\partial t}) \oplus \span(v)
\end{equation}
where $v$ extends to a globally defined  vector field on $M_H=Z_H\times \R$ via the recipe $\tilde{v}(z,t)=v(z)$.  We'll refer to the extension as $v$ from now on and note that $v$ takes values in $\ker(\sigma)\cap TZ$ at $Z\times \{0\}$.  We may also assume that $p^*\alpha(v)=1$ since $p^*\alpha(v)$ is nonvanishing.

As a reminder, the connectedness of the group $G$ implies that the representation of $H$ on $\ker(\sigma_p)$ is trivial by lemma \ref{lem:preserves1}.  Thus, $H$ fixes $\ker(\sigma)\big\vert_{M_H\cap Z}$ and corollary \ref{cor:slice2} implies that $\ker(\sigma)\big\vert_{M_H\cap Z}$ is a subbundle of $T(M_H\cap Z)$.  That is, $\ker(\sigma)$ is tangent to $M_H$ at points of $M_H\cap Z$.  Now, we wish to prove the following claim which gives a precise description of when a vector in $T_p(M_H\cap Z)$ admits an extension to a local section of $TM\big\vert_{M_H}$ with values in $TM_H^{\sigma}$:

\vspace{3mm}

\noindent\textbf{Claim:}Let $p\in Z\times \{0\}$. A vector $X\in \widetilde{(TM_H)}^{\sigma}_p$ extends to a local section of $TM\big\vert_{M_H}$ with values in $(TM_H)^{\sigma}$ if and only if the covector $i_X(dt\wedge p^*\alpha_p)$ vanishes on $\ker(\sigma_p)$.  In particular, $\ker(\sigma_p)\cap \widetilde{(TM_H)}^{\sigma}_p= \{0\}$.

\vspace{3mm}
To prove the \emph{only if} portion, we show that if the covector $i_X(dt\wedge p^*\alpha_p)$ doesn't vanish on $\ker(\sigma_p)$, then any extension of $X$ to a local vector field will have values not in $(TM_H)^{\sigma}$ in arbitrarily small neighborhoods of $p$.  If the covector $i_X(dt\wedge p^*\alpha_p)$ doesn't vanish on $\ker(\sigma_p)$, then there is an element $Y\in \ker(\sigma_p)$ such that $(dt\wedge p^*\alpha)(X,Y) \ne 0$.  As discussed above, $Y$ is a linear combination of $v(p)\in \ker(\sigma_p)\cap T_pZ$ and $\frac{\partial}{\partial t}$, which are both tangent to $M_H$ at $p$, hence $Y$ admits an extension $\tilde{Y}$ to a vector field on $M_H$ which is a linear combination of $v$ and $\frac{\partial}{\partial t}$.  In particular, the extension of $Y$ satisfies:

\begin{equation}
i_{\tilde{Y}}p^*i^*\sigma = i_{gv + f\frac{\partial}{\partial t}}p^*i^*\sigma = 0
\end{equation}
hence its contraction with $\sigma$ is:
\begin{equation}\label{eq:contraction}
i_{\tilde{Y}}\sigma = i_{\tilde{Y}}d(t^2p^*\alpha)
\end{equation}
Any extension $\tilde{X}$ of $X$ satisfies $(dt\wedge p^*\alpha)(\tilde{X},\tilde{Y})\ne 0$ in some neighborhood of $p\in M_H$, which depends on the choice of extension.  Thus, by equation \ref{eq:contraction} we have:
\begin{equation}
\sigma(\tilde{X},\tilde{Y}) = 2t(dt\wedge p^*\alpha)(\tilde{X},\tilde{Y}) + t^2p^*d\alpha(\tilde{X},\tilde{Y})
\end{equation}
which vanishes transversally at $t=0$ since $(dt\wedge p^*\alpha(\tilde{X},\tilde{Y})\big\vert_{t=0}=(dt\wedge p^*\alpha)(X,Y)\ne 0$, hence $\sigma(\tilde{X},\tilde{Y})$ is nonzero for arbitrarily small, nonzero values of $t$.  Thus, $\tilde{X}$ does not take values in $(TM_H)^{\sigma}$ in a neighborhood of $p$.

\vspace{3mm}
For the \emph{if} part, we'll need to construct a vector bundle on $M_H\cap Z$ first.  We begin by noting that the local model for a neighborhood of the fold gives us a diagram:

\begin{displaymath}
\xymatrixcolsep{5pc}\xymatrix{
(Z\times \R,\sigma=p^*i^*\sigma + d(t^2p^*\alpha) \ar[r]^{\psi(z,t)=(z,t^2)} & (Z\times \R, \omega=p^*i^*+d(tp^*\alpha))
}
\end{displaymath}
where $\psi^*\omega=\sigma$.  At a point $(z,0)\in Z_H\times \{0\}=M_H\cap Z)$ we must have that $(T_{(z,0)}M_H)^{\omega}\subset T_{(z,0)}(Z\times \{0\})$.  This is because both null directions $\frac{\partial}{\partial t}$ and $v$ are tangent to $M_H$ and $(i_v\omega)_{(z,0)}=-dt$, hence any vector $X$ transverse to $Z$ at $(z,0)$ pairs with $v$ to give $\omega(v,X)\ne 0$.

Since $(T_{(z,0)}M_H)^{\omega}\subset T_{(z,0)}(Z\times \{0\})$ for any $z\in Z_H$ and $\psi\big\vert_{Z\times \{0\}}$ is the identity, we can form the bundle $d\psi\big\vert_Z^{-1}(TM_H^{\omega})\vert_Z=TM_H^\omega\big\vert_Z$ along $M_H\cap Z$.  Since $\psi^*\omega=\sigma$, we have that the fibers of $TM_H^{\omega}\big\vert_{M_H\cap Z}$ are subspaces of the fibers of $TM_H^{\sigma} \to M_H\cap Z$.

Now, to finish the proof of our claim, we note that if $X\in (T_pM_H)^{\sigma}$ satisfies $i_X(dt\wedge p^*\alpha)\big\vert_{\ker(\sigma_p)}=0$, then $X$ must be an element of the fiber of $d\psi^{-1}(TM_H^{\omega})\vert_Z$.  Certainly, it must lie in $T_pZ$ since it cannot be transverse to $Z$ else $i_X(dt\wedge p^*\alpha)$ would not vanish on $v$.  We then have:

\begin{displaymath}
i_{d\psi(X)}\omega = i_X\omega = (i_Xp^*i^*\sigma + ti_Xp^*d\alpha)\big\vert_{t=0} = i_Xp^*i^*\sigma = i_X\sigma
\end{displaymath}
hence $i_X\omega$ vanishes on $T_pM_H$ because $i_X\sigma$ vanishes on $T_pM_H$ by assumption.  Thus, we may extend $X$ to a local section of the bundle $TM_H^{\omega}\big\vert_{M_H\cap Z}$ and then extend it to a section of $TM\big\vert_{M_H}$ via $\tilde{X}(z,t)=X(z)$.  This extension does not necessarily have values in $(TM_H)^{\sigma}$, so we correct it.  The contraction is:

\begin{equation}
i_{\tilde{X}}\sigma = i_{\tilde{X}}p^*i^*\sigma -2tdtp^*\alpha(\tilde{X}) + t^2i_{\tilde{X}}p^*d\alpha
\end{equation}
We then compute:
\begin{equation}
i_{\tilde{X} - p^*\alpha(\tilde{X})v} \sigma = i_{\tilde{X}}p^*i^*\sigma +t^2i_{\tilde{X} - p^*\alpha(\tilde{X})v}dp^*\alpha
\end{equation}
Note that the $1$-form $t^2i_{\tilde{X} - p^*\alpha(\tilde{X})v}dp^*\alpha$ vanishes at the fold, hence proposition \ref{prop:Ecotangent1} guarantees the existence of a vector field $\Gamma$ such that $i_{\Gamma}=-t^2i_{\tilde{X} - p^*\alpha(\tilde{X})v}dp^*\alpha$.  Then we have:

\begin{equation}
i_{\tilde{X}-p^*\alpha(\tilde{X})v + \Gamma}\sigma = i_{\tilde{X}}p^*i^*\sigma
\end{equation}
which vanishes on $M_H$ if and only if it vanishes on $TM_H$ at points of $Z\times \{0\}$ since
\begin{itemize}
\item $(i_{\tilde{X}}p^*i^*\sigma)_{(z,t)} = (i_Xp^*i^*\sigma)_{(z,0)}$ by definition of $\tilde{X}$ and
\item $M_H=Z_H\times \R$, hence the tangent bundle is $TZ_H \times T\R$, hence a $1$-form that is independent of $t$ and vanishes on $TZ_H \times T\R$ at points of $Z\times \{0\}$ will vanish on $TZ_H\times T\R$.
\end{itemize}
Thus, we have shown that
\begin{itemize}
\item if $i_X(dt\wedge p^*\alpha)\big\vert_{\ker(\sigma)}=0$ at a point $(z,0)$, then it lies in $TM_H^{\omega}$ at $(z,0)$ and
\item if $X$ is in the fiber of $TM_H^{\omega}$ at $(z,0)$, then it extends to a local section of $TM\big\vert_{M_H}$ with values in $TM_H^{\sigma}$,
\end{itemize}
which finishes the proof of the claim.  Now, from our discussion, we have show that the dimension of the subspace of vectors $X\in (T_pM_H)^{\sigma}$ satisfying $i_X(dt\wedge p^*\alpha)\big\vert_{\ker(\sigma)}=0$ is exactly the rank of $(TM_H)^{\omega}$, which is fixed, hence the dimension of $\widetilde{(TM_H)}^{\sigma}_p$ is independent of the choice of $p\in M_H$ and $\widetilde{(TM_H)}^{\sigma}$ is a vector bundle over $M_H$.

\vspace{3mm}
To prove the remaining claims of the proposition, note that there is nothing to prove away from $Z$ since the bundle away from $Z$ is $TM_H^{\sigma}$, which is a symplectic vector bundle and complementary to $M_H$.  Thus, we consider a point $p\in Z$.  Using the local model, we have shown that the fiber at $p$ is $T_pM_H^{\omega}$, where $\omega$ is some symplectic form for which the action is Hamiltonian.  Since $T_pM_H^{\omega}$ is complementary to $T_pM_H$, we have that $\widetilde{(TM_H)}^{\sigma}$ is complementary to $TM_H$.  Secondly, we have shown that $\sigma$ restricted to $T_pM_H^{\omega}$ is $\omega$ restricted to $T_pM_H^{\omega}$.  Since $\omega$ is symplectic, $\sigma$ must be symplectic.  Thus, $\widetilde{(TM_H)}^{\sigma}$ is a symplectic vector bundle.
\end{proof}

\begin{cor}\label{cor:normbundle}
Let $G$ be a compact, connected Lie group and let $(M,\sigma)$ be a folded-symplectic Hamiltonian $G$-manifold with moment map $\mu:M\to \fg^*$, where $\fg$ is the Lie algebra of $G$.  Suppose the folding hypersurface $Z$ is co-orientable.  Let $H\le G$ be a subgroup and suppose $M_H$ is nonempty.  Then the normal bundle $\nu(M_H) = (TM\vert_{M_H})/TM_H$ is canonically a symplectic vector bundle.
\end{cor}

\begin{proof}\mbox{ } \newline
According to proposition \ref{prop:normbundle}, the bundle $\widetilde{(TM_H)}^\sigma$ is a symplectic vector bundle complementary to $TM_H$ inside the restricted tangent bundle $TM\big\vert_{M_H}$.  The projection:

\begin{displaymath}
p: TM\big\vert_{M_H} \to TM\big\vert_{M_H}/(TM_H)
\end{displaymath}
therefore restricts to an isomorphism
\begin{displaymath}
p: \widetilde{(TM_H)}^\sigma \to \nu(M_H).
\end{displaymath}
We define the symplectic structure on $\nu(M_H)$ to be $\omega$ such that $p^*\omega = \sigma \big\vert_{\widetilde{(TM_H)}^\sigma}$, which must be non-degenerate since $p$ is an isomorphism.
\end{proof}

\begin{cor}\label{cor:normbundle1}
Let $G$ be a torus and let $(M,\sigma)$ be a folded-symplectic Hamiltonian $G$-manifold with moment map $\mu:M\to \fg^*$, where $\fg$ is the Lie algebra of $G$.  Suppose the folding hypersurface $Z$ is co-orientable.  Let $H\le G$ be a subgroup and suppose $M_H$ is nonempty.  If the action of $G$ is effective, then the action of $H$ on the fibers of $\widetilde{(TM_H)}^{\sigma}$ is effective, hence the representations $\rho:H\to GL(\widetilde{(TM_H)}^{\sigma}_p)$ are faithful, symplectic representations of $H$.
\end{cor}

\begin{proof}\mbox{ } \newline
Let $p\in M_H$.  The action of $H$ on the differential slice $T_pM/T_p(G\cdot p)$ is effective by lemma \ref{lem:eff3}, hence the action of $H$ on $T_pM$ is effective.  Since the tangent space $T_pM$ splits as $\tilde{(TM_H)}^{\sigma}_p \oplus T_pM_H$ and the action of $H$ on $T_pM_H$ is trivial, we must have that the action of $H$ on $\widetilde{(TM_H)}^{\sigma}_p$ is effective.  Equivalently, the representation $\rho:H \to GL(\widetilde{(TM_H)}^{\sigma}_p)$ is faithful.
\end{proof}

\begin{cor}\label{cor:normbundle2}
Let $(M,\sigma)$ be a folded-symplectic manifold with an Hamiltonian action of a Lie group $G$ and moment map $\mu:M\to \fg*$, where $\fg=\operatorname{Lie}(G)$.  Consider the fold $Z$.  For each subgroup $H\le G$ there exists a vector subbundle of $TZ\big\vert_{Z_H}$, $(\widetilde{TZ_H})^{\sigma}$ which is a symplectic vector bundle and such that $TZ\big\vert_{Z_H}= TZ_H \oplus (\widetilde{Z_H})^{\sigma}$, $H$-equivariantly.
\end{cor}

\begin{proof}
By corollary \ref{cor:preserves}, we can equivariantly embed $Z$ into a folded-symplectic Hamiltonian $G$ manifold $(M,\sigma)$ as a co-orientable folding hypersurface, where the action of $G$ preserves the co-orientation. The result then follows by taking $(\widetilde{TM_H})^{\sigma}\big\vert_{Z_H}$ since the fibers along $Z_H$ are tangent to $Z$.  Note that we aren't assuming $G$ is connected, but we only did this in the statement of proposition \ref{prop:normbundle} to ensure that the action of $G$ preserved the co-orientation of $Z$.  This is automatically built into the equivariant embedding of $Z$ into $(M,\sigma)$, hence we can freely drop the connectedness assumption.
\end{proof}

\begin{remark}
Let $(M,\sigma)$ be a folded-symplectic manifold with an Hamiltonian action of a Lie group compact, connected Lie group $G$.  Corollary \ref{cor:normbundle2} essentially states that the data encoded in the fibers of the generalized symplectic normal bundle $(\widetilde{TM_H})^{\sigma}$ along $Z_H = M_H\cap Z$ are intrinsic to the folding hypersurface.  Again, as in corollary \ref{cor:preserves}, it appears as if the folding hypersurface does not seem to mind how it is embedded into a folded-symplectic manifold.
\end{remark}

\pagebreak
\section{Toric Folded-Symplectic Manifolds}
We now turn our attention to toric folded-symplectic manifolds.  These manifolds have a maximal number of commuting Hamiltonian functions whose differentials become linearly dependent at the folding hypersurface.  Thus, toric folded-symplectic manifolds may be viewed as somewhat tractable examples of degenerate, completely integrable systems.  We define these manifolds in the case that they do not have corners.  We'll need a separate definition for corners when we introduce the category $\mathcal{B}_{\psi}$ and we choose to delay it for now to avoid confusion between the cases with corners and the cases without corners.

The main results of this section are as follows.  We show that the orbit space of a toric, folded-symplectic manifold with a co-orientable folding hypersurface is a manifold with corners.  To prove this result, we first show that the stabilizers in a toric folded-symplectic manifold are tori and then argue that these manifolds are locally standard.  We then show that the moment map descends to what we call a unimodular map with folds.  We spend a fair amount of time studying what information one can read from the orbital moment map, which turns out to be a great deal.  In particular, one can reproduce the null foliation on the folding hypersurface \emph{and} its induced orientation directly from the orbital moment map.  These results all represent original work which is motivated by results seen in the studies of origami manifolds (q.v. \cite{CGP, HP}).

\subsection{Definitions and Basic Properties}

\begin{definition}\label{def:TFSmanifold}
A \emph{toric folded-symplectic manifold} (without corners) is a \emph{connected} folded-symplectic manifold $(M,\sigma)$ with an effective, Hamiltonian action of a torus $G$, where $\dim(G)=\frac{1}{2}\dim(M)$.  We denote a toric, folded-symplectic manifold as a triple $(M,\sigma,\mu:M \to \fg^*)$, where $\mu$ is a moment map for the action of $G$.  We often to omit $G$ from the notation as it is usually implied that we have fixed a torus, $G$.
\end{definition}

\begin{remark}
We are going to classify toric, folded-symplectic manifolds with co-orientable folding hypersurfaces.  Hence, throughout much of this section, the reader will see the phrase \emph{with co-orientable folding hypersurface} appear in the hypotheses of the lemmas and propositions.
\end{remark}

\begin{lemma}\label{lem:isoorbits}
Let $(M,\sigma,\mu:M\to \fg^*)$ be a toric, folded-symplectic manifold.  Let $p\in M$ be a point and let $G\cdot p$ be the orbit through $p$.  Then $G\cdot p$ is istotropic: $i_{G\cdot p}^*\sigma =0$.
\end{lemma}

\begin{proof}
By corollary \ref{cor:orbits1}, the tangent space $T_p(G\cdot p)$ is generated by the induced vector fields: $T_p(G\cdot p)= \{X_M(p) \vert \mbox{ } X\in \fg\}$. Thus, we need only show that $\sigma_p(X_M(p), Y_M(p))$ for every pair of induced vector fields.  Let $X,Y \in\fg$ be Lie algebra elements.  We compute:

\begin{equation}\label{eq:isotropic}
\sigma(X_M(p),Y_M(p)) = (i_{X_M}\sigma)(Y_M(p)) = -d(\langle \mu, X \rangle)_p(Y_M(p))
\end{equation}
The action of $G$ on $\fg^*$ is trivial since $G$ is abelian, hence the equivariance of $\mu$ implies $G$-invariance of $\mu$: $\mu(g\cdot p ) =Ad^*(g)\mu(p)=\mu(p)$.  $G$-invariance of $\mu$ implies the $G$-invariance of $\langle \mu, X \rangle$, hence $\langle \mu, X\rangle$ is constant along orbits and its derivative vanishes along directions tangent to orbits.  Thus, the right-hand side of equation \ref{eq:isotropic} is zero and we have:

\begin{displaymath}
\sigma(X_M(p),Y_M(p))=0
\end{displaymath}
for any choice of $X,Y\in \fg$, which means $\sigma$ restricted to orbits is $0$.
\end{proof}

\begin{lemma}\label{lem:stabtori}
Let $(M,\sigma,\mu:M\to \fg^*)$ be a toric, folded-symplectic manifold.  Suppose the folding hypersurface $Z\subset M$ is co-orientable.  Then the stabilizer of a point $p\in M$ is a subtorus $H=G_p$ of $G$ and $\dim(M_H)=2(\dim(G)-\dim(H))$.  Consequently, $(M_H,i_{M_H}^*\sigma, \mu\big\vert_{M_H}:M_H \to \frak{h}^o)$ is a toric, folded-symplectic $G/H$ manifold, where $\frak{h}^o$ is the annihilator of $\frak{h}=Lie(H)$ in $\frak{g}^*$.  The kernel bundle $i^*_{M_H}\sigma \to Z_H$ is given by $\ker(\sigma)\big\vert_{Z_H} \to Z_H$.
\end{lemma}

\begin{proof}
Suppose $H\le G$ and $M_H$ is nonempty.  We argue that $\dim(M_H)=2(\dim(G)-\dim(H))$, which means that the symplectic normal bundle $(\widetilde{TM_H})^{\sigma}$ has rank $\dim(H)$ by proposition \ref{prop:normbundle}.  The symplectic representation of $H$ on the fibers of $(\widetilde{TM_H})^{\sigma}$ is faithful by corollary \ref{cor:normbundle1}, hence lemma \ref{lem:sympweights2} will imply that $H$ is a torus.  We now compute the dimension of $M_H$.

By corollary \ref{cor:eqfsnormal2}, $M_H \pitchfork_s Z$, hence we may choose a point $p\in M_H\setminus Z$.  The torus $G$ is abelian by definition, hence $H\le G$ fixes $G\cdot p$ and so it fixes $T_p(G\cdot p)$.  By corollary \ref{cor:slice2}, we have $T_p(G\cdot p)\subseteq T_pM_H$.  Since $M\setminus Z$ is symplectic and $(M\setminus Z)_H=M_H\setminus Z$, lemma \ref{lem:symporbit} implies $M_H\setminus Z$ is symplectic.  Since $T_p(G\cdot p)$ is isotropic in $T_p(G\cdot p)$ (q.v. lemma \ref{lem:isoorbits}), we must have $\dim(M_H)\ge 2\dim(G\cdot p) = 2(\dim(G)-\dim(H))$.

By proposition \ref{prop:normbundle}, the symplectic normal bundle $(\widetilde{TM_H})^\sigma$ is complementary to $TM_H$ in $TM\big\vert_{M_H}$, hence its rank $r$ is $r=\dim(M)-\dim(M_H) \le 2\dim(G) - 2(\dim(G)-\dim(H)) = 2\dim(H)$.  Thus, $r\le 2\dim(H)$.  By corollary \ref{cor:normbundle1}, the representation of $H$ on a fiber of $(\widetilde{TM_H})^{\sigma}$ is symplectic and faithful, hence $2\dim(H) \le r$ by lemma \ref{lem:sympweights2}.  We therefore have:

\begin{displaymath}
2\dim(H) \le \operatorname{rank}((\widetilde{TM_H})^{\sigma}) \le 2\dim(H)
\end{displaymath}
hence $\operatorname{rank}((\widetilde{TM_H})^{\sigma})=2\dim(H)$, $H$ is a torus by lemma \ref{lem:sympweights2}, and $\dim(M_H) = \dim(H) - \operatorname{rank}((\widetilde{TM_H})^{\sigma}) = 2\dim(G) - 2\dim(H) = 2(\dim(G)-\dim(H))$.

\vspace{3mm}
To see that $(M_H,i_{M_H}^*\sigma, \mu\big\vert_{M_H})$ is a toric, folded-symplectic manifold we first invoke corollary \ref{cor:eqfsnormal1}, which states that $(M_H,i_{M_H}^*\sigma)$ is folded-symplectic, where the kernel bundle is simply $\ker(i_{M_H}^*\sigma) = \ker(\sigma)\big\vert_{Z_H}$ by inspection of the normal form in proposition \ref{prop:eqfsnormal}.  Now, $G$ is a torus and $H$ is a torus, hence $G/H$ is a compact, connected abelian group.  That is, it is also a torus.  Since the action of $G$ preserves $\sigma$, the action of $G/H$ on $M_H$ preserves $i_{M_H}^*\sigma$.  Since $H$ fixes $M_H$, we have that for each element $X \in \frak{h}$ of the Lie algebra of $H$

\begin{displaymath}
-\langle d(\mu\big\vert_{M_H}), X \rangle = i_{M_H}^*(-i_{X_M}\sigma) = 0
\end{displaymath}
since $X_M$ is zero along $M_H$.  Thus, $d\mu$ maps into $\frak{h}^o$, hence each connected component of $M_H$ must map into an affine subspace $\eta +\frak{h}^o$, which is isomorphic to $(\frak{g}/\frak{h})^* = Lie(G/H)^*$ via projection, hence $\mu\big\vert_{M_H}$ gives a moment map for the action of $G/H$ on $M_H$.
\end{proof}

\begin{cor}\label{cor:stabtori}
Let $(M,\sigma,\mu:M\to \fg^*)$ be a toric folded-symplectic manifold with co-orientable folding hypersurface $Z$.  Let $M_{(H)}=M_H$ be an orbit-type stratum.  Then,
\begin{enumerate}
\item There exist well-defined weights $\beta_i$, $1\le i \le h=\dim(H)$, for the symplectic representations of $H$ on the fibers of $(\widetilde{TM_H})^{\sigma}$ and these weights do not change along the connected components of $M_H$.
\item The weights $(\beta_1,\dots,\beta_h)$ form a $\mathbb{Z}$-basis for the weight lattice of $H$.
\end{enumerate}
\end{cor}
\begin{proof}
By lemma \ref{lem:stabtori}, $H$ is a subtorus of $G$, where $G$ is the torus acting on $M$.  By corollay \ref{cor:normbundle}, the representations of $H$ on the fibers of $(\widetilde{TM_H})^{\sigma}$ are faithful and symplectic.  By lemma \ref{lem:sympweights}, symplectic representations of tori have well-defined weights which specify the representation up to isomorphism, which gives us a (multi)set of weights $\{\beta_1,\dots,\beta_h\}$.  By lemma \ref{lem:stabtori}, $\operatorname{rank}((\widetilde{TM_H})^{\sigma}) = 2\dim(H)$.  By lemma \ref{lem:sympweights2}, we have that $h=\dim(H)$, the weights $\{\beta_1,\dots,\beta_h\}$ are distinct, and they form a $\mathbb{Z}$-basis for the weight lattice $\mathbb{Z}_H^*$ of $H$.
\end{proof}

\begin{remark}
The following fixed-point corollary of lemma \ref{lem:stabtori} can be proven independently of lemma \ref{lem:stabtori} by studying the folding hypersurface directly.  However, lemma \ref{lem:stabtori} helps us to make the claim obvious, so we list it as a corollary.  The reader should also be aware of the fact that one needn't require co-orientability of the fold, but for our proof to work it is necessary.
\end{remark}

\begin{cor}\label{cor:stabtori}
Let $(M,\sigma,\mu:M\to \fg^*)$ be a toric, folded-symplectic manifold with co-orientable folding hypersurface $Z\subset M$.  Then there are no points in $Z$ fixed by the torus $G$.
\end{cor}

\begin{proof}
Let $M_G$ be the fixed-point stratum of $M$.  By lemma \ref{lem:stabtori}, $\dim(M_G)=2(\dim(G)-\dim(G))=0$, hence it has codimension $\dim(M)$ in $M$.  By corollary \ref{cor:eqfsnormal1}, $M_G\pitchfork Z$, hence the intersection has codimension $\dim(M)$.  But, $Z$ is a hypersurface in $M$, hence it has dimension $\dim(M)-1$ and the maximum codimension would be $\dim(M)-1$.  Thus, $M_G\cap Z$ must be empty.
\end{proof}

There is another more interesting way to see corollary \ref{cor:stabtori}.  We will show that the null foliation on the folding hypersurface $Z$ is generated by the group action.  Towards the end of this chapter, we will show that the generators of the null foliation fit together to give us a vector bundle over $Z$, which will be an invariant of a toric-folded symplectic manifold, albeit a superfluous invariant.  As in corollary \ref{cor:stabtori}, one needn't require co-orientability of the fold $Z$, but we are focusing on such manifolds so we require it.

\begin{lemma}\label{lem:generator}
Let $(M,\sigma,\mu:M\to \fg^*)$ be a toric, folded-symplectic manifold with co-orientable folding hypersurface $Z$.  Let $p\in Z$ be a point in the fold and let $\ker(\sigma)\cap TZ$ be the line bundle on which $i_Z^*\sigma$ vanishes.  Then the fiber $\ker(\sigma_p)\cap T_pZ$ is tangent to the orbit through $p$: $\ker(\sigma_p)\cap T_pZ \subseteq T_p(G\cdot p)$, hence the torus action generates the null foliation.
\end{lemma}

\begin{proof}
Let $p\in Z$, let $H=G_p$ be its stabilizer, and suppose the claim is false at $p$.  Since $\ker(\sigma)\cap TZ$ is a line bundle, this means we are assuming $\ker(\sigma_p)\cap T_pZ \cap T_p(G\cdot p)=\{0\}$.  We will show that this leads to the dimension of $Z_H=M_H\cap Z$ being too large.  By lemma \ref{lem:stabtori}, $\dim(M_H) = 2(\dim(G)-\dim(H))$ and by corollary \ref{cor:eqfsnormal1} $M_H \pitchfork Z$, hence $\dim(Z_H) = \dim(M_H\cap Z) = 2(\dim(G)-\dim(H))-1$.

By corollary \ref{cor:eqfsnormal1}, $M_H$ is folded-symplectic with folding hypersurface $Z_H=M_H\cap Z$.  The kernel bundle of $i_{Z_H}^*\sigma$ is just $\ker(i_Z^*\sigma)\big\vert_{Z_H}$ since $\ker(\sigma)$ is tangent to $M_H$ at points of $M_H$, which is true because $H$ acts trivially on the fibers of $\ker(\sigma)$ by lemma \ref{lem:preserves}.  Notice that since the orbit is contained in $Z$ and $H$ fixes all elements in $T_p(G\cdot p)$, which is true since $G$ is abelian, we have $T_p(G\cdot p)\subset T_pZ_H$.  We therefore have that $\ker(\sigma_p)\cap T_pZ$ and $T_p(G\cdot p)$ are two non-intersecting subspaces of $T_pZ_H$.

Now, $i_{Z_H}^*\sigma$ has maximal rank and the kernel is \emph{not} contained in the isotropic subspace $T_p(G\cdot p)$ by assumption, hence there must be a subspace $V_p\subseteq T_pZ_H$ complementary to $T_p(G\cdot p)+ (\ker(\sigma_p)\cap T_pZ)$ so that $i_{Z_H}^*\sigma$ is nondegenerate on $V_p + T_p(G\cdot p)$.  Since $T_p(G\cdot p)$ is isotropic, $\dim(V_p)$ is at least $\dim(G\cdot p)=\dim(G)-\dim(H)$.  But, this gives us:

\begin{displaymath}
\dim(Z_H) = \dim(V_p) + \dim(T_p(G\cdot p) + \dim(\ker(\sigma_p)\cap T_pZ) \ge 2(\dim(G)-\dim(H)) + 1 > \dim(Z_H)
\end{displaymath}
which is a contradiction, so we must have that $\ker(\sigma_p)\cap T_pZ$ is contained in $T_p(G\cdot p)$.
\end{proof}

Corollary \ref{cor:stabtori} can now be seen as follows: if $G$ fixes a point $p\in Z$, then we have $\ker(\sigma_p)\cap T_pZ \subseteq T_p(G\cdot p) = \{0\}$ by lemma \ref{lem:generator}. But, $\ker(\sigma_p)\cap T_pZ$ is $1$-dimensional by definition of folded-symplectic, hence $G$ cannot fix any point in $Z$.  The following corollary is a computational restatement of lemma \ref{lem:generator}.  We will need it to see how one may recover the null foliation on $Z$ using the moment map.

\begin{cor}\label{cor:generator}
Let $(M,\sigma,\mu:M\to \fg^*)$ be a toric folded-symplectic manifold and suppose the folding hypersurface $Z\subset M$ is co-orientable.  Then for all $p\in Z$ there exists a lie algebra element $X\in \fg$ such that $X_M(p)\ne 0$ and $X_M(p)\in \ker(\sigma_p)\cap T_pZ$.  That is, the induced vector field $X_M$ generates the tangent space to the leaf of the null foliation on $Z$ at $p$.
\end{cor}

\subsection{Toric Symplectic Local Normal Form and the Fold}
We now seek to describe the local structure of a toric symplectic manifold, which will allow us to describe the local structure of the folding hypersurface inside a toric folded-symplectic manifold quite easily.  Let us make a few remarks before stating the toric symplectic normal form proposition.

\begin{remark}\label{rem:complement}
Let $G$ be a torus and suppose $H\le G$ is a subtorus with $\dim(H)<\dim(G)$ so that $H\ne G$.  Then one can find a complementary subtorus $K\le G$ so that $G\simeq H\times K$ as Lie groups.  This works as follows:

\begin{itemize}
\item Since $H\le G$ is a subgroup, $Lie(H)=\frak{k}$ is a Lie subalgebra of $\fg$, hence the integral lattice $\mathbb{Z}_H$ is a sublattice.
\item Recall that an element $\eta$ of $\mathbb{Z}_G$ is \emph{primitive} if there is no element $\eta_0\in \mathbb{Z}_G$ and no positive integer $n$ such that $\eta=n\eta_0$.  Since $H$ is a subtorus, the integral lattice $\mathbb{Z}_H$ must contain $h=\dim(H)$ elements, $\eta_1,\dots,\eta_h$, that are primitive in both $\mathbb{Z}_H$ and $\mathbb{Z}_G$.
\item Thus, there exist $k=\dim(G)-\dim(H)$ other primitive elements, $v_1,\dots,v_k$, in $\mathbb{Z}_G$ so that $\{\eta_1,\dots,\eta_h,v_1,\dots,v_k\}$ is a $\mathbb{Z}$-basis for $\mathbb{Z}_G$.  Consequently, $K=\exp(\span\{v_1,\dots,v_k\})$ is a subtorus of $G$.
\end{itemize}
\end{remark}

\begin{remark}\label{rem:slicerep}
Let $(M,\sigma,\mu:M\to \fg^*)$ be a toric, folded-symplectic manifold.  Choose a point $p\in M\setminus Z$ and let $H=G_p$ be its stabilizer, which must be a subtorus of $G$ by lemma \ref{lem:stabtori} .  By lemma \ref{lem:isoorbits}, the orbit $G\cdot p$ is isotropic, hence the symplectic slice at $p$ is:

\begin{displaymath}
V_p = \frac{T_p(G\cdot p)^{\sigma}}{T_p(G\cdot p)}
\end{displaymath}
We claim that this is canonically isomorphic, as a representation of $H$, to the fiber of $(\widetilde{TM_H})^\sigma$.  Indeed, we have:

\begin{itemize}
\item $T_p(G\cdot p)^{\sigma}\cap T_pM_H = T_p(G\cdot p)$ since $T_p(G\cdot p)$ is Lagrangian in $T_pM_H$ by lemmas \ref{lem:isoorbits} and \ref{lem:stabtori}.
\item $T_p(G\cdot p)^{\sigma} \cap (\widetilde{TM_H})^{\sigma}_p = (\widetilde{TM_H})^{\sigma}_p$ since, by definition, $(\widetilde{TM_H})^{\sigma}_p$ is the $\sigma$ perpendicular of $T_pM_H$ at points of $M\setminus Z$, hence it must vanish on $T_p(G\cdot p)\subset T_pM_H$.
\end{itemize}
By a dimension count, facilitated by lemma \ref{lem:stabtori}, we have the decomposition into invariant subspaces:

\begin{displaymath}
T_p(G\cdot p)^{\sigma} = (\widetilde{TM_H})^{\sigma}_p \oplus T_p(G\cdot p)
\end{displaymath}
hence $T_p(G\cdot p)^{\sigma}/T_p(G\cdot p) = (\widetilde{TM_H})^{\sigma}_p$ as representations of $H$.  Thus, for toric, folded-symplectic manifolds, the existence of the bundle $(\widetilde{TM_H})^{\sigma}$ implies that the notion of a symplectic slice extends across the fold: at a point $z\in Z$ in the fold, we could define the fiber $(\widetilde{TM_H})^{\sigma}_z$ to be the symplectic slice.
\end{remark}

\begin{remark}\label{rem:slicerep2}
Let $(M,\sigma,\mu:M\to \fg^*)$ be a toric, folded-symplectic manifold.  Let $p\in M$, let $H=G_p$ be the stabilizer of $p$, and consider the representation of $H$ on $(\widetilde{TM_H})^{\sigma}$.  By lemma \ref{lem:stabtori}, $H$ is a torus and corollary \ref{cor:stabtori} tells us that there exist $h=\dim(H)$ weights, $\{\beta_1,\dots,\beta_h\}$, associated to the representation $(\widetilde{TM_H})^{\sigma}_p$ that form a basis for the integral lattice $\mathbb{Z}_H^*$ of $H$.  By lemma \ref{lem:sympweights}, there exists an isomorphism of symplectic representations between $(\widetilde{TM_H})^{\sigma}_p$ and the representation of $H$ on $\C^h$ given by:

\begin{displaymath}
\exp(X)\cdot(z_1,\dots,z_h) = (e^{2\pi i\beta_1(X)}z_1, \dots, e^{2\pi i\beta_h(X)},z_h)
\end{displaymath}
\end{remark}

The following normal form proposition is an application of theorem \ref{thm:sympnorm} and remarks \ref{rem:complement} and \ref{rem:slicerep2}.  It is lemma B.5 in \cite{KL}.

\begin{prop}\label{prop:torsympnorm}
Let $(M,\omega,\mu:M\to \fg^*)$ be a toric \emph{symplectic} manifold.  That is, assume it is folded-symplectic with empty folding hypersurface.  Let $p$ be a point in $M$ and let $H=G_p$ be its stabilizer.

\begin{enumerate}
\item Let $\mathbb{T}^h=\R^h/\mathbb{Z}^h$ be the standard torus.  There exists an isomorphism $\tau_H: H \to \mathbb{T}^h$ of Lie groups such that the symplectic slice representation at $p$ is isomorphic to the action of $H$ on $\C^h$ obtained from the composition of $\tau_H$ with the standard action of $\mathbb{T}^h$ on $\C^h$, which is:
    \begin{equation}\label{eq:action}
    [t_1,\dots,t_h]\cdot (z_1,\dots, z_h)=(e^{2\pi i t_1}z_1, \dots, e^{2\pi i t_h}z_h)
    \end{equation}
    and this isomorphism can be constructed using the weights $\{\beta_1,\dots,\beta_h\}$ of the symplectic slice representation of $H$.
\item Choose a complementary subtorus $K\le G$ and let $\tau:G \to K \times H$ be an isomorphism of Lie groups such that $\tau(a)=(a,e)$ for all $a\in K$.  Then there exists a $G$-invariant open neighbourhood $U$ of $p$ in $M$ and a $\tau$-equivariant open symplectic embedding
    \begin{displaymath}
    j:U \hookrightarrow T^*K \times \C^h
    \end{displaymath}
    with $j(G\cdot p)=K\times \{0\}$.  Here, $K$ acts on $T^*K$ by the lift of the left multiplication and $H$ acts on $\mathbb{C}^h$ by the recipe of equation \ref{eq:action} composed with $\tau_H$.  The moment map is:

    \begin{equation}\label{eq:momentmap}
    \mu\vert_U = \mu(p) + \tau^* \circ \phi \circ j
    \end{equation}
    where $\phi:T^*K \times \C^h \to \frak{k}^* \oplus \frak{h}^*$ is given by
    \begin{equation}\label{eq:stdmomentmap}
    \phi((\lambda,\eta),(z_1,\dots,z_h))=(\eta,\sum_{i=1}^h \vert z_j \vert^2\beta_j).
    \end{equation}
    where the $\beta_j's$ are the weights for the representation of $H$ on $\C^h$ and $\tau^*:\frak{k}^* \times \frak{h}^* \to \fg^*$ is the isomorphism on the duals of the Lie algebras induced by $\tau$.
\end{enumerate}
\end{prop}

\begin{proof}[Sketch of Proof]
By theorem \ref{thm:sympnorm}, a neighborhood of the orbit $G\cdot p$ in $M$ is isomorphic to a neighborhood of the zero section of:

\begin{displaymath}
G\times_H (\frak{h}^o\oplus V) = (T^*G \times V)//_0 H
\end{displaymath}
with its natural symplectic structure.  The slice representation has weights $\{\beta_1,\dots,\beta_h\}$ associated to it and these weights form a basis for the weight lattice $\mathbb{Z}_H^*$ of $H$. Consequently, they define an isomorphism $\tau_H$ between $H$ and the standard torus $\mathbb{T}^h$:

\begin{displaymath}
\xymatrixcolsep{5pc}\xymatrix{
\frak{h} \ar[r]^{(\beta_1,\dots,\beta_h)} \ar[d]^{\exp_H} & \mathbb{R}^h \ar[d]^{\exp_{\mathbb{T}^h}} \\
H \ar[r]^{\tau_H}                                         & \mathbb{T}^h
}
\end{displaymath}
By remark \ref{rem:slicerep2}, $V$ is isomorphic, as a symplectic representation, to $\C^h$ where $H$ acts via:

\begin{displaymath}
\exp(X)\cdot(z_1,\dots,z_h) = (e^{2\pi i\beta_1(X)}z_1, \dots, e^{2\pi i\beta_h(X)},z_h)
\end{displaymath}
This is exactly the action generated by the standard torus action composed with the map $\tau_H$.  If we choose a complementary subtorus $K$ so that $G\simeq K \times L$, then the reduced space $(T^*G \times \C^h)//_0 H$ becomes:

\begin{displaymath}
(T^*K \times T^*H \times \C^h)//_0 H = T^*K \times (T^*H \times \C^h)//_0 H = T^*K \times C^h
\end{displaymath}
and the moment map for the $K\times H$ action is
\begin{equation}\label{eq:momentmap1}
\phi(\lambda,\eta, z_1,\dots,z_h)= \eta + \sum_{i=1}^h \vert z_i \vert^2 \beta_i + \mu(p)
\end{equation}
Hence, a neighborhood of $G\cdot p$ is isomorphic to a neighborhood of the zero section of $T^*K \times C^h$ with its standard symplectic structure and moment map given by equation \ref{eq:momentmap1}.  The isomorphism $\tau:G \to K \times H$ induces a map $\tau^*:\frak{h}^*\times \frak{k}^* \to \frak{g}^*$ and it is straightforward to check that if $\phi$ is the moment map in equation \ref{eq:stdmomentmap}, then $\mu\big\vert_U=\tau^*\circ \phi + \mu(p)$.
\end{proof}

The following corollary states that we can realize $Z$ as a unique embedded hypersurface in the standard symplectic model of proposition \ref{prop:torsympnorm}.  This may be taken as a consequence of corollary \ref{cor:preserves}, which shows we can always form an equivariant symplectization of the folding hypersurface.  It may also be taken as a consequence of corollary \ref{cor:normbundle2}, which states that the symplectic slice data is intrinsic to the fold.

\begin{cor}\label{cor:foldnorm}
Let $(M,\sigma,\mu:M\to \fg^*)$ be a toric, folded-symplectic manifold with co-orientable folding hypersurface $Z\subset M$.  Let $p\in Z$ and let $H=G_p$ be its stabilizer, which is a subtorus by lemma \ref{lem:stabtori}, with $h:=\dim(H)$.  Let $K\le G$ be a subtorus of $G$ complementary to $H$ in $G$.  Then there exists a neighborhood $\mathcal{U}\subseteq Z$ of $p$ and a $K\times H$ equivariant, co-isotropic embedding:

\begin{displaymath}
j_Z:\mathcal{U} \hookrightarrow T^*K \times \C^h
\end{displaymath}
so that $\mu\vert_{\mathcal{U}}(z)= (\phi \circ j_Z)(z) + \mu(p)$ where
\begin{equation}\label{eq:mmtmap}
\phi(\lambda, \eta, z_1,\dots, z_h)= \eta + \sum_{i=1}^{h}\vert z_i \vert^2 \beta_i
\end{equation}
and $\{\beta_1,\dots,\beta_h\}$ is the set of weights for the action of $H$ on $(\widetilde{TM_H})^{\sigma}_p$.  The image of $\mathcal{U}$ in $T^*K \times \C^h$ is uniquely specified by the image of the moment map $\mu$:
\begin{displaymath}
j_Z(Z)=\phi^{-1}(\mu(\mathcal{U})-\mu(p)).
\end{displaymath}
Here, $\mu(\mathcal{U})-\mu(p)$ is the set of points in the image of $\mu$ shifted by the value $\mu(p)$.
\end{cor}

\begin{proof}
We begin by showing that $Z$ admits an equivariant, co-isotropic embedding into a toric, symplectic manifold and that the symplectic slice data at the fold is intrinsic to the fold: the symplectic slice obtained in the image of the embedding is independent of the embedding.  By proposition \ref{prop:eqfsnormal}, there exists an invariant neighborhood $U$ of $Z$ and a commutative diagram:

\begin{equation}\label{eq:model}
\xymatrix{
(U,\sigma) \ar[r]^{\psi} \ar[dr]^{\mu} & (U,\omega) \ar[d]^{\mu_s} \\
                                       &  \fg^*
}
\end{equation}
where $\omega$ is a symplectic structure for which the action of $G$ is Hamiltonian with moment map $\mu_s$, $\psi$ is an equivariant map with fold singularities that folds along $Z$ with $\psi\big\vert_Z=id_Z$, and $\psi^*\omega=\sigma$.  Note that $(U,\omega)$ is a toric symplectic manifold.  We want to study the symplectic slice representation $T_z(G\cdot z)^{\omega}/T_z(G\cdot z)$ at points $z\in Z$ and show that it is isomorphic to $\C^h$ and independent of our choice of symplectization.  By proposition \ref{prop:normbundle}, the fiber of the symplectic normal bundle at $z$, $(\widetilde{TM_H})^{\sigma}_z$, is contained in $T_zZ$.  Since the restriction of the fold map $\psi\big\vert_Z$ is the identity on $Z$ and $\psi^*\omega=\sigma$, we have

\begin{displaymath}
d\psi_z((\widetilde{TM_H})^{\sigma}_z) \subset (TM_H)^{\omega}_z.
\end{displaymath}
Since they have the same dimension and $d\psi_z\big\vert_{T_zZ}=id_{T_zZ}$ is injective, we have that

\begin{displaymath}
d\psi_z((\widetilde{TM_H})^{\sigma}_z) = (TM_H)^{\omega}_z
\end{displaymath}
By remark \ref{rem:slicerep}, the fiber $(TM_H)^{\omega}_z$ is isomorphic to the symplectic slice representation at $z$.  We therefore have an equivariant coisotropic embedding:

\begin{displaymath}
j\circ \psi\vert_Z:(Z,i_Z^*\sigma) \hookrightarrow (U,\omega)
\end{displaymath}
into a toric symplectic manifold $(U,\omega)$ and the symplectic slice representation $V$ of the stabilizer $H$ at the point $z\in U$ depends only on the representation $(\widetilde{TM_H})^{\sigma}_z$, hence the weights $\{\beta_1, \dots, \beta_h\}$ associated to $V$ are independent of our choice of model in equation \ref{eq:model}.

\vspace{2mm}
Now, apply proposition \ref{prop:torsympnorm} to a neighborhood $U_1$ of $z\in U$.  We obtain an equivariant, open symplectic embedding:

\begin{displaymath}
j:U_1 \hookrightarrow T^*K \times \C^h
\end{displaymath}
where $K\le G$ is a subtorus complementary to $H$ in $G$.  Precomposing with the restriction of the fold map $\psi\big\vert_Z$ and defining $\mathcal{U}:=U_1\cap Z$, we obtain the requisite embedding:
\begin{displaymath}
j_Z:\mathcal{U} \hookrightarrow T^*K \times \C^h
\end{displaymath}
Lastly, we show the hypersurface is uniquely specified by the image of $\mu$.  Let $\phi(\lambda, \eta, z_1,\dots, z_h) = \eta + \sum_{i=1}^h \vert z_i \vert^2 \beta_i$ be the moment map for the action of $K\times H$ on $T^*K \times \C^h$.  Let $t_1,\dots, t_h \in \R^+$ be nonnegative real numbers.  Then,

\begin{displaymath}
\phi^{-1}(\eta,t_1^2\beta_1, \dots, t_h^2\beta_h) = (K\times H) \cdot (\eta,t_1,\dots,t_h)
\end{displaymath}
hence the inverse images of points are orbits.  Thus, $(\phi \circ j_Z)(z) + \mu(p) = (\eta,t_1^2\beta_1, \dots, t_h^2\beta_h)$ if and only if $j_Z(z)$ is in the orbit $\phi^{-1}((\eta,t_1^2\beta_1, \dots, t_h^2\beta_h)-\mu(p))$.  Since the fold $Z$ is $G$-invariant (q.v. lemma \ref{lem:preserves}), the image $j_Z(Z)$ contains the entire orbit, hence
\begin{displaymath}
\phi^{-1}(\mu(\mathcal{U})-\mu(p)) \subseteq j_Z(z).
\end{displaymath}
The reverse inclusion follows from the fact that the moment maps satisfy $\phi\circ j_Z + \mu(p) = \mu$:
\begin{displaymath}
\phi \circ j_Z = \mu-\mu(p) \mbox{ }\rightarrow \mbox{ } j_Z(\mathcal{U}) \subseteq \phi^{-1}(\mu(\mathcal{U})-\mu(p))
\end{displaymath}
hence $\phi^{-1}(\mu(\mathcal{U})-\mu(p))=j_Z(\mathcal{U})$ and the hypersurface is uniquely determined by the moment map image.
\end{proof}
\subsubsection{Unimodular Maps with Folds}
We are almost ready to describe the invariants of a toric, folded-symplectic manifold $(M,\sigma, \mu:M \to \fg^*)$ for the purposes of classifying them up to isomorphism.  However, before we can describe these invariants, we'll need a few definitions so we can give them a name.  The following definitions and facts about unimodular local embeddings are taken, nearly verbatim, from \cite{KL}.  The results and discussions about unimodular maps with folds are generalizations of those related to unimodular local embeddings found in \cite{KL}.

\begin{definition}\label{def:unicone}
Let $G$ be a torus and let $\fg$ be its Lie algebra.  A \emph{unimodular cone} in $\fg^*$ is a subset $C$ of the form:

\begin{displaymath}
C= \{\eta \in \fg^* \vert \mbox{ } \langle \eta - \epsilon, v_i \rangle \ge 0 \text{ for all $1\le i \le k$}\}
\end{displaymath}
where $\epsilon$ is a point in $\fg^*$, $k$ is an integer greater than $0$, and $\{v_1,\dots,v_k\}$ is a $\mathbb{Z}$-basis of the integral lattice of a subtorus of $G$.  We write:

\begin{displaymath}
C=C_{(v_1,\dots,v_k),\epsilon}
\end{displaymath}
when we wish to make the dependence on the $v_i$ and $\epsilon$ explicit.  The \emph{closed facets} of $C$ are the sets:

\begin{displaymath}
F_i = \{\eta \in C \mbox{ } \vert \mbox{ } \langle \eta - \epsilon, v_i \rangle =0\}, \mbox{ } 1\le i\le k
\end{displaymath}
which are subsets of affine hyperplanes and we call $v_i$ the \emph{inward pointing primitive normal} to $F_i$.
\end{definition}

\begin{remark}
A unimodular cone $C$ in $\fg^*$ is a manifold with corners.  The $k$-boundary $\partial^k(C)$ is given by $k$-fold intersections of the closed facets $F_i$ of $C$.  That is, the boundary of $C$ is determined by the intersections of affine hyperplanes in $\fg^*$.
\end{remark}

\begin{lemma}\label{lem:face}
Let $G$ be a torus and $\fg$ its Lie algebra.  Let $C$ be a unimodular cone in $\fg^*$.  Then the primitive inward pointing normal $v_i$ to a facet $F_i$ of $C$ is uniquely determined by any open neighborhood of a point $x\in F_i$ in $C$.
\end{lemma}

\begin{proof}
A facet of $C$ is the closure of an open set inside an affine hyperplane inside $\fg^*$, hence we may assume that $F=\epsilon + H$, where $\epsilon \in \fg^*$, $H$ is a codimension $1$ subspace, and $\epsilon + H$ is the set of elements of the form $\epsilon + \eta$, where $\eta\in H$.  The annihilator $H^o\subset \fg$ is a unique $1$-dimensional subspace of $\fg$, which is the Lie algebra $\frak{k}$ of a subtorus $K$ of $G$ by the definition of a unimodular cone.  Thus, we may choose a primitive normal $v_i$ in the integral lattice $\mathbb{Z}_K\subset \mathbb{Z}_G$ of $K$.  This normal is determined up to sign and the sign is determined by the convention that $v_i$ points into the cone.
\end{proof}

\begin{definition}\label{def:ule}
Let $W$ be a manifold with corners and let $G$ be a torus with Lie algebra $\fg$.  A smooth map $\bar{\mu}:W\to \fg^*$ is a \emph{unimodular local embedding} if for each $w\in W$ there exists an open neighbourhood $\mathcal{T}\subset W$ of $w$ and a unimodular cone $C\subset \fg^*$ such that $\bar{\mu}(\mathcal{T})\subset C$ and $\bar{\mu}\big\vert_{\mathcal{T}}:\mathcal{T} \hookrightarrow C$ is an open embedding.  That is, $\bar{\mu}\big\vert_{\mathcal{T}}:\mathcal{T} \to \bar{\mu}(\mathcal{T})$ is a diffeomorphism of manifolds with corners.
\end{definition}

\begin{lemma}\label{lem:attach}
Let $G$ be a torus with Lie algebra $\fg$ and let $\bar{\mu}:W\to \fg^*$ be a unimodular local embedding.  Then $\bar{\mu}$ attaches to every point $w\in W$ a subtorus $K_w$ of $G$, a basis $\{v_1^{w}, \dots, v_k^w\}$ of its integral lattice, and a faithful symplectic representation of $K_w$ on $\C^k$ with weights $\{(v_1^w)^*, \dots, (v_k^w)^*\}$.  This assignment is constant on the stratum of $W$ containing $w$.  Moreover, for each $w\in W$ there exists a neighborhood $U_w$ such that $\bar{\mu}$ restricts to an open embedding:

\begin{displaymath}
\bar{\mu}\big\vert_{U_w}:U_w \hookrightarrow C_{(v_1^w,\dots,v_k^w),\bar{\mu}(w)}
\end{displaymath}
\end{lemma}

\begin{proof}
By definition of a unimodular local embedding, for each point $w\in W$ there exists a neighborhood $\mathcal{T}\subseteq W$ of $w$ and a unimodular cone $C$ so that $\bar{\mu}(\mathcal{T})\subseteq C$ and $\bar{\mu}\big\vert_{\mathcal{T}}:\mathcal{T} \to \bar{\mu}(\mathcal{T})$ is a diffeomorphism of manifolds with corners.  By shrinking $\mathcal{T}$, we may assume that $w$ lies in the intersection of the closures of exactly $k=\operatorname{depth}_W(w)$ distinct components of $\partial^1(W)$.  Call these components $B_1,\dots, B_k$.

Since $\mathcal{T}$ is open and $\bar{\mu}\big\vert_{\mathcal{T}}$ is a diffeomorphism, $\bar{\mu}(\mathcal{T}\cap B_i)$ is an open subset of a face $F_i$ of $C$.  Furthermore, the sets $\bar{\mu}(\mathcal{T}\cap B_i)$ are distinct for distinct values of $i$ since the components $B_1,\dots, B_k$ are all distinct and $\bar{\mu}$ is a diffeomorphism, hence the faces $F_i$ are all distinct.

By lemma \ref{lem:face}, $\bar{\mu}(\mathcal{T}\cap B_i)$ determines the inward pointing normal $v_i$ of $F_i$ uniquely, hence we obtain a linearly independent set $\{v_1^w,\dots, v_k^w\}$ from the normals of $C$.  By the definition of a unimodular cone, this set is a $\mathbb{Z}$-basis for the integral lattice of a subtorus $K_w\subset G$, where

\begin{displaymath}
K_w := \exp(\span\{v_1^w,\dots, v_k^w\})
\end{displaymath}
This basis does not depend on the cone $C$.  Indeed, the normals only depend on the images of the strata $\bar{\mu}(\mathcal{T}\cap B_i)$, which are fixed by $\bar{\mu}$, hence the assignment is independent of our choices.  The faithful, symplectic representation of $K_w$ on $\C^k$ is given by the representation with weights $\{(v_1^w)^*,\dots, (v_k^w)^*\}$.  Explicitly, $K_w$ acts via

\begin{displaymath}
\exp(X)\cdot (z_1, \dots, z_k) = (e^{2\pi i (v_1^w)^*(X)}z_1,\dots,e^{2\pi i (v_k^w)^*(X)}z_k)
\end{displaymath}
Now, if $p$ is another point in $\mathcal{T}$ in the same stratum as $w$, then $p\in \mathcal{T} \bigcap \cap_{i=1}^k(\bar{B}_i)$.  That is, the stratum of $p$ is determined by the same codimension $1$ strata $B_i$ that determine the stratum containing $w$.  Hence, the subtorus and lattice basis assigned to $p$ is determined by the images of the codimension $1$ strata $\bar{\mu}(\mathcal{T}\cap B_i)$.  We therefore obtain the same lattice basis and same subtorus constructed for $w$ since we only used the normals to the affine hyperplanes containing $\bar{\mu}(\mathcal{T}\cap B_i)$.  Thus, the assignments are locally constant along strata, hence they are constant on each connected component of the stratification:

\begin{displaymath}
W = \bigsqcup_{j=1}^{\dim(W)} \partial^j(W)
\end{displaymath}
\end{proof}

\begin{remark}
Suppose we choose another point $w'\in U_w$, where $U_w$ is the neighborhood of $w\in W$ constructed in lemma \ref{lem:attach}.  The subtorus $K_{w'}$, lattice basis $\{v_1^{w'}, \dots, v_{k'}^{w'}\}$, and representation are determined by the boundary structure of $U_w$ along with the set $\{v_1^w, \dots, v_k^w\}$.  Explicitly, we have:

\begin{displaymath}
\{v_1^{w'},\dots, v_{k'}^{w'}\} = \{v_i^w \mbox{ } \vert \mbox{ } \langle \bar{\mu}(w')-\bar{\mu}(w), v_i^w \rangle =0\}
\end{displaymath}
and
\begin{displaymath}
K_{w'}=\exp(\span_{\R}(\{v_i^{w} \mbox{ } \vert \mbox{ } \langle \bar{\mu}(w')-\bar{\mu}(w), v_i^w \rangle =0\}))
\end{displaymath}
The representation is determined by taking the duals to the lattice vectors $v_i^{w'}$.
\end{remark}

We now generalize the unimodular local embeddings of \cite{KL} by introducing fold singularities.  The reason is that one of the invariants of a toric, folded-symplectic manifold will be what is known as its orbital moment map, which will be a unimodular map with folds.

\begin{definition}\label{def:umf}
Let $G$ be a torus with Lie algebra $\fg$ and let $W$ be a manifold with corners.  A smooth map $\psi:W \to \fg^*$ is a \emph{unimodular map with folds} if:
\begin{enumerate}
\item It is a map with fold singularities, hence the folding hypersurface $\hat{Z}$ is a codimension $1$ submanifold with corners transverse to the strata of $W$.  We require that $\hat{Z}$ be co-orientable.
\item The fibers of the bundle $\ker(d\psi) \to \hat{Z}$ are tangent to the strata of $W$.  That is, if $z\in \partial^k\cap \hat{Z}=\partial^k(\hat{Z})$, then $\ker(d\psi_z)\subset T_z\partial^k(W)$.
\item The map away from the fold, $\psi\big\vert_{W\setminus \hat{Z}}$, is a unimodular local embedding.
\end{enumerate}
\end{definition}

Since $\psi$ only has fold singularities, we have an analog of lemma \ref{lem:attach} for unimodular maps with folds, which we will use to produce local factorizations of $\psi$ into unimodular local embeddings composed with fold maps.  This will be useful for when we perform folded-symplectic reduction on manifolds with corners.

\begin{lemma}\label{lem:attach1}
Let $G$ be a torus with Lie algebra $\fg$, let $W$ be a manifold with corners, and let $\psi:W \to \fg^*$ be a unimodular map with folds with folding hypersurface $\hat{Z}\subset W$.  $\psi$ attaches to every point $w\in W$ a subtorus $K_w$ of $G$, a $\mathbb{Z}$-basis $\{v_1^w,\dots,v_k^w\}$ of its integral lattice, and a faithful symplectic representation of $K_w$ on $\C^k$ with weights defined by the duals of each $v_1^w$.
\end{lemma}

\begin{proof}
Since $\psi\big\vert_{W\setminus \hat{Z}}$ is a unimodular local embedding, we know that such an assignment exists on $W\setminus \hat{Z}$ by lemma \ref{lem:attach}, hence lemma \ref{lem:attach1} will be proven if we can show this assignment extends across the fold $\hat{Z}$.  To this end, we will show that the connected components of the codimension $1$ stratum $\partial^1(W)$ are mapped to affine hyperplanes under $\psi$.  Then, at a point $z\in \hat{Z}$ which is in the intersection of the closures of $k$ distinct components $B_i$ of $\partial^1(W)$, i.e. $z\in \cap_{i=1}^k\bar{B}_i$, we'll have that the normals corresponding to the affine hypersurfaces $\psi(B_i)$ comprise a $\mathbb{Z}$-basis for a the integral lattice of a subtorus $K_w$ of $G$ since they form such a basis away from $\hat{Z}$, which will prove the theorem.

We let $B$ be a connected component of the codimension $1$ boundary $\partial^1(W)$ and consider $z\in \hat{Z}\cap B$.  Since $\hat{Z}\pitchfork_s B$, we may choose a neighborhood $U$ of $z$ in $B$ such that $U\setminus \hat{Z}$ has two connected components, $U^+$ and $U^-$.  Choose a curve $\gamma(t)$ in $U$ satisfying:

\begin{enumerate}
\item $\gamma(0)=z$, i.e. $\gamma$ passes through $z$,
\item $\gamma'(0)\in \ker(d\psi_z)$ with $\gamma'(0)\ne 0$, and
\item $\gamma(t) \in U^+$ for $t>0$ and $\gamma(t) \in U^-$ for $t<0$.
\end{enumerate}

Since $\psi\big\vert_{W\setminus \hat{Z}}$ is a unimodular local embedding, lemma \ref{lem:attach} reveals that $\psi(U^+)$ is an open subset of an affine hyperplane $\psi(w) + H^+$, where $H^+\subset \fg^*$ is a codimension $1$ subspace, and $\psi(U^-)$ is an open subset of an affine subspace $\psi(w) + H^-$, where $H^-\subset \fg^*$ is a codimension $1$ subspace.  We let $v^+$ and $v^-$ denote the primitive normal vectors to these hypersurfaces.  Our goal is to show they are the same, hence $H^+ = H^-$ and $\psi(w) + H^+ =\psi(w)+H^-$.  Now, if we identify the tangent space of a vector space at a point with the vector space itself, we have the following series of calculations:

\begin{itemize}
\item $d\psi_z(T_zZ)$ is contained in both $H^-$ and $H^+$ as a codimension $1$ subspace.  Indeed, $d\psi_{\gamma(t)}(T_zZ)$ is a smoothly varying subspace of $H^-$ for $t<0$ and is a smoothly varying subspace of $H^+$, hence in the limit as $t$ goes to $0$ it lies in both the subspace $H^-$ and $H^+$.  If we can show that $H^-$ and $H^+$ both contain a common nonzero vector transverse to $d\psi_z(T_zZ)$, then we will have that $H^+=H^-$.
\item We have $d\psi(\gamma'(t)) \in H^+$ for $t>0$ and $d\psi(\gamma'(t)) \in H^-$ for $t<0$.  Hence $d\psi(\gamma'(t))$ is a path in $H^+$ for $t>0$ and a path in $H^-$ for $t<0$, meaning its derivative satisfies:
\item $\frac{d}{dt}(d\psi(\gamma'(t)))$ is in $H^+$ for $t>0$ and $H^-$ for $t<0$.
\item Now, $\frac{d}{dt}\big\vert_{t=0}(d\psi(\gamma'(t)))$ is transverse to the image of $d\psi_z$ since $\gamma'(0)\in \ker(d\psi_z)$ and $\psi$ has fold singularities, hence $H^+$ and $H^-$ both contain $\frac{d}{dt}\big\vert_{t=0}(d\psi(\gamma'(t)))$, which is transverse to $d\psi_z(T_zZ)$ in $H^+$ and $d\psi_z(T_zZ)$ in $H^-$.
\item Thus, $H^+=H^-$ and there is a unique affine hypersurface $\psi(w) + H$ into which $B$ is mapped and there is a unique primitive element $v \in \mathbb{Z}_G$ corresponding to $B$.
\end{itemize}

As we have discussed, if $z$ is contained in the intersection of the closure of exactly $k$ codimension $1$ strata, $z\in \cap_{i=1}^k \bar{B}_i$ and $\operatorname{depth}_W(z)=k$, then there are $k$ primitive normal vectors $\{v_1^z, \dots ,v_k^z\}$ obtained from reading the normals to the images of the $B_i$ away from $\hat{Z}$.  Away from $\hat{Z}$, these normals form a $\mathbb{Z}$-basis for the integral lattice of a subtorus of $G$ since $\psi$ is a unimodular local embedding away from $\hat{Z}$.  As we have shown, these normals do not change as we cross the fold, $\hat{Z}$, hence they form a $\mathbb{Z}$-basis for the integral lattice of a subtorus $K_z$ of $G$.
\end{proof}

\begin{cor}\label{cor:attach}
Let $G$ be a torus with Lie algebra $\fg$, let $W$ be a manifold with corners, and let $\psi:W\to \fg^*$ be a unimodular map with folds, where the folding hypersurface is denoted $\hat{Z}$.  Then for each $z\in \hat{Z}$, there exists a neighborhood $U_z$, a map with fold singularities $\gamma:U_z \to U_z$ with folding hypersurface $\hat{Z}\cap U_z$, and a unimodular local embedding $\bar{\mu}:U_Z \to \fg^*$ so that
\begin{displaymath}
\psi\big\vert_{U_z} = \bar{\mu}\circ \gamma.
\end{displaymath}
Thus, the image of a unimodular map with folds is locally described by a unimodular cone folded across a hypersurface transverse to its faces.
\end{cor}

\begin{proof}
Let $z\in \hat{Z}$ be a point in the fold.  Corollary \ref{cor:folds5} implies that there exists a neighborhood $U_z$ and a factorization $\psi\big\vert_{U_z} =\bar{\mu}\circ \gamma$, where $\gamma:U_z \to U_z$ is a fold map which folds across $U_z \cap \hat{Z}$ and $\bar{\mu}$ is an open embedding.  By lemma \ref{lem:attach1}, the codimension $1$ strata $B\in \partial^1(W)$ are mapped into affine hyperplanes whose normals are primitive elements of $\mathbb{Z}_G$.  Furthermore, at the intersections of their closures, $\cap_{i=1}^k\bar{B}_i$, these normals fit together to give a basis for the integral lattice of a subtorus of $G$.  We may shrink $U_z$, if necessary, to assume that it is a manifold with faces.  That is, it is diffeomorphic to an open subset of $\R^m \times (\R^+)^n$ for some $m,n\in \mathbb{N}$.  Thus, $\bar{\mu}$ maps the faces of $U_z$ to affine hyperplanes whose normals form a basis for the integral lattice of a subtorus, hence $\bar{\mu}$ is a unimodular local embedding.
\end{proof}

\subsection{The Orbit Space of a Toric, Folded-Symplectic Manifold and Invariants}
We are now ready to produce the two primary invariants of a toric, folded-symplectic manifold $(M,\sigma,\mu:M\to \fg^*)$.  The first is the orbit map $\pi:M \to M/G$, where $M/G$ has the structure of a manifold with corners.

\begin{remark}\label{rem:invntfxns}
It will be useful to know about some of the invariant structures on the standard representation $\C$ of the circle $S^1$.  If a function $f:\C \to \R$ is invariant under the action of $S^1$, then its values are determined by its restriction to the real line $f\big\vert_{\R}$.  The restriction is invariant under reflections across the origin since this is the action of $\{1,-1\}\subset S^1$ restricted to the real line.  Thus, every $S^1$-invariant smooth function $f(z)$ on $\mathbb{C}$ may be written as $f(z) = f\big\vert_{\R}(\vert z \vert) = g(\vert z \vert ^2)$ for some smooth function $g: \R^+ \to \R$.  Conversely, any function that can be written as $g(\vert z \vert ^2)$ is $S^1$-invariant, hence the map

\begin{displaymath}
\xymatrixrowsep{.5pc}\xymatrix{
\C^{\infty}(R^+) \ar[r] & C^{\infty}(\C)^{S^1} \\
g \ar[r] & g(\vert z \vert^2)
}
\end{displaymath}
is surjective.  It is also injective and is therefore an isomorphism.  This is the content of theorem $1$ of \cite{Sch} in the special case of $S^1$ acting on $\C$.  Consequently, $\R^+$ may be identified with the orbit space $\C/S^1$ and the quotient map $q:\C \to \R^+$ is given by $q(z) = \vert z \vert ^2$.

Of course, all of this applies to the product version of this problem.  If $\mathbb{T}^h$ acts on $\C^h$ in the standard way via rotations, then the space of invariant smooth functions $C^{\infty}(\C^h)^{\mathbb{T}^h}$ is isomorphic to $C^{\infty}((R^+)^h)$, where the map sends a function on $(\R^+)^h$ $g(t_1,\dots, t_h)$ to $g(\vert z_1 \vert ^2, \dots, \vert z_h \vert ^2)$.
\end{remark}

\begin{prop}\label{prop:corners}
Let $(M,\sigma,\mu:M\to \fg^*)$ be a toric, folded-symplectic manifold with co-orientable folding hypersurface $Z$.  Then $M/G$ is naturally a manifold with corners and $Z/G$ is a co-orientable submanifold with corners transverse to the strata of $M/G$.  Furthermore, $\pi:M \to M/G$ is a quotient map and $\mu:M \to \fg^*$ descends to a smooth map $\psi:M/G \to \fg^*$.
\end{prop}

\begin{proof} \mbox{ } \newline
\begin{enumerate}
\item We first show $M/G$ is a manifold with corners.  We will show that it is locally a manifold with corners by constructing charts using the slice theorem.  We will then show that the transition maps are smooth by applying Theorem 1 of \cite{Sch}.  The key is that $M$ is locally standard: the differential slice representation decomposes as a trivial representation and a faithful, symplectic representation of a torus $H$ on $\C^h$, where $h=\dim(H)$.

    More precisely, let $p\in M$ and let $H=G_p$ be the stabilizer.  By proposition \ref{prop:normbundle}, the tangent space $T_pM$ splits as a representation of $H$ into $T_pM= T_pM_H \oplus (\widetilde{TM_H})^{\sigma}_p$, where $(\widetilde{TM_H})^\sigma$ is the symplectic normal bundle to $M_H$ constructed in proposition \ref{prop:normbundle}.  Thus, the differential slice, as a representation, is:
    \begin{displaymath}
    T_pM/T_p(G\cdot p) = (T_pM_H/T_p(G\cdot p)) \oplus (\widetilde{TM_H})^{\sigma}_p
    \end{displaymath}
    $H$ acts trivially on $T_pM_H/T_p(G\cdot p)$ since $H$ fixes $T_pM_H$ and the dimension of this space is $\dim(M_H)-\dim(G\cdot p) = 2\dim(\dim(G)-\dim(H)) -\dim(G) + \dim(H) = \dim(G)-\dim(H)$.  Thus, $T_pM_H/T_p(G\cdot p)$ is isomorphic, as a representation, to the annihilator of $\frak{h}=\operatorname{Lie}(H)$, $\frak{h}^o\subset \fg^*$.

    By corollary \ref{cor:stabtori}, there exist linearly independent weights $\{\beta_1,\dots,\beta_h\}$ for the faithful, symplectic representation of $H$ on $(\widetilde{TM_H})^{\sigma}_p$, hence $(\widetilde{TM_H})^{\sigma}_p$ is isomorphic to $\C^h$ where $H$ acts by rotations.  Explicitly,

    \begin{displaymath}
    \exp(X)\cdot (z_1,\dots,z_h) = (e^{2\pi i\beta_1(X)}z_1, \dots, e^{2\pi i\beta_n(X)}z_n)
    \end{displaymath}
    If we choose a complementary subtorus $K\le G$ and specify an isomorphism $\tau: G \to K\times H$ of Lie groups, then a neighborhood of $G\cdot p$ is $G$-equivariantly isomorphic to a neighborhood of the zero section of:
    \begin{displaymath}
    G\times_H \frak{h}^o \oplus \C^h = K\times \frak{h}^o \oplus \C^h = T^*K \times \C^k
    \end{displaymath}
    The orbit space of this neighborhood is then $(T^*K\times \C^k)/G = \frak{h}^o \times C^h/H$.  Now, by remark \ref{rem:invntfxns} any $S^1$-invariant function $f(z)$ on $\C$ may be written as $g(\vert z\vert^2)$, where $g$ is a smooth function on $\R$.  This extends to the product $\C^h$: a $\mathbb{T}^h$-invariant function $f(z_1,\dots,z_h)$ may be written as $g(\vert z_1 \vert^2, \dots, \vert z_h \vert^2)$ for some smooth map $g:\R^h \to \R$.  That is, the functions $p_i(\vec{z})=\vert z_i \vert^2$ generate the ring of $\mathbb{T}^h$ invariant functions on $\C^h$.  By theorem 1 of \cite{Sch}, the space $C^h/\mathbb{T}^h$ is $(\R^+)^h$ and the quotient map $q$ is given by:
    \begin{displaymath}
    q(z_1,\dots,z_h) = (\vert z_1 \vert^1, \dots, \vert z_h \vert^2)
    \end{displaymath}
    Consequently, $\frak{h}^o \times C^h/H \simeq \R^k \times (\R^+)^h$ and the orbit space $M/G$ is a manifold with corners.  We have thus shown that for each point $[p]\in M/G$ there is a neighborhood $U$ of $[p]$ and a homeomorphism $\phi:U \to V$ onto an open subset $V$ of $\R^k \times (\R^+)^h$, where $h$ and $k$ depend on $p$.  $\phi$ satisfies the following key property:

    \begin{displaymath}
    \phi^*C^{\infty}(V) = C^{\infty}(U)
    \end{displaymath}
    where $C^{\infty}(U)$ is the ring of functions on $U\subset M/G$ that lift to smooth, invariant functions on $\pi^{-1}(U)$, hence $C^{\infty}(U)$ is $C^{\infty}(\pi^{-1}(U))^G$, the ring of invariant smooth functions on $\pi^{-1}(U)$.  Thus, for any two neighborhoods $U_1$, $U_2$ and homeomorphisms $\phi_1$, $\phi_2$, the transition maps $\phi_1\circ \phi_2^{-1}$ satisfy:
    \begin{displaymath}
    (\phi_1\circ\phi_2^{-1})^*C^{\infty}(V_1)=C^{\infty}(V_2)
    \end{displaymath}
    hence they send smooth functions to smooth functions, meaning they are smooth maps.  Thus, the transition maps are diffeomorphisms of manifolds with corners and $M/G$ inherits the structure of a manifold with corners.
\item The proof that $Z/G$ is a submanifold with corners of $M/G$ is very similar, but we first equivariantly identify a neighborhood of $Z$ with a neighborhood of the zero section of $Z\times \R$, so that we may assume $M=Z\times \R$ where $G$ doesn't act on the second factor.  Consider the differential slice representation $T_Z/T_z(G\cdot z)$.  By proposition \ref{prop:normbundle}, we have $T_zZ = T_zZ_H \oplus (\widetilde{TM_H})^{\sigma}_z$.  Hence the slice representation is:

    \begin{displaymath}
    T_zZ/T_z(G\cdot z) = (T_zZ_H/T_z(G\cdot z)) \oplus (\widetilde{TM_H})^{\sigma}_z
    \end{displaymath}
    where the first summand is a trivial representation.  The same arguments used to show $M/G$ is a manifold with corners now also show that $Z/G$ is a manifold with corners.  Thus, the quotient space $(Z\times \R)/G = Z/G \times \R$ is a manifold with corners and $Z/G$ intersects the strata transversally:

    \begin{displaymath}
    \partial^k(Z/G\times \R) = \partial^k(Z/G) \times \R
    \end{displaymath}
    Since a neighborhood of $Z$ in $M$ is isomorphic to a neighborhood of the zero section of $Z\times \R$, we have that a neighborhood $\bar{U}$ of $Z/G$ in $M/G$ is isomorphic to a neighborhood of the zero section of $Z/G \times \R$, hence $Z/G$ intersects the strata of $M/G$ transversally.

\item Finally, the fact that $\pi:M \to M/G$ is a quotient map follows from the fact that the map
    \begin{displaymath}
    q:\C^h \to (\R^+)^h, \text{ $q(z_1,\dots,z_h)=(\vert z_1 \vert^2, \dots, \vert z_h \vert^2)$},
    \end{displaymath}
    is a quotient map for the standard $\mathbb{T}^h$ action on $\C^h$.  Since $M$ is locally isomorphic to $T^*K \times \C^h$, where $\C^h$ is a representation of a subtorus $H$ of $G$ and $K$ is a complementary subtorus, we have that $\pi:M \to M/G$ is locally isomorphic to $p \times q: T^*K \times \C^h \to \frak{h}^o \times (\R^+)^h$, where $p:T^*K \to \frak{h}^o$ is the projection.  Since $\mu$ is $G$-invariant and $\pi:M \to M/G$ is a quotient map, there exists a smooth map $\psi:M/G \to \fg^*$ such that $\mu =\psi \circ \pi$.
\end{enumerate}
\end{proof}

\begin{definition}\label{def:orbitalmomentmap}
Let $(M,\sigma,\mu:M\to \fg^*)$ be a toric folded-symplectic manifold.  By proposition \ref{prop:corners}, $M/G$ is a manifold with corners and $\mu:M \to \fg^*$ descends to $\psi:M/G \to \fg^*$.  We call $\psi$ the \emph{orbital moment map}.
\end{definition}

\begin{prop}\label{prop:umf}
Let $(M,\sigma,\mu:M\to \fg^*)$ be a toric, folded-symplectic manifold with co-orientable folding hypersurface $Z\subset M$.  Let $\psi:M/G \to \fg^*$ be the orbital moment map.  Then $\psi$ is a unimodular map with folds whose folding hypersurface is $Z/G$.  Furthermore, at a point $[z]\in Z/G$, $\ker(d\psi_{[z]})$ is the image of $\ker(\sigma_z)$ under the differential $d\pi_z$ of the quotient map $\pi:M\to M/G$ at $z\in Z$.
\end{prop}

\begin{proof}
Let us first consider the case where $(M,\sigma,\mu:M\to \fg^*)$ is a toric \emph{symplectic} manifold.  Let $p\in M$, let $H=G_p$ be its stabilizer with $h=\dim(H)$, let $\{\beta_1,\dots,\beta_h\}$ be the weights of the symplectic slice representation of $H$, and let $K$ be a complementary subtorus in $G$.  By lemma \ref{prop:torsympnorm}, we have an invariant neighborhood $U$ of $G\cdot p$ and a commutative diagram:

\begin{displaymath}
\xymatrix{
(U,\sigma) \ar[r]^j \ar[dr]^{\mu} & T^*K \times \C^h \ar[d]^{\varphi} \\
                                  &  \fg^* = \frak{k}^* \oplus \frak{h}^* }
\end{displaymath}
where $j$ is an open, equivariant symplectic embedding and $\varphi(\lambda, \eta, z_1,\dots, z_h) = \eta + \sum_{i=1}^h \vert z_i \vert^2 \beta_i$.    The image of $\mu\big\vert_U$ is therefore an open subset of a unimodular cone $C$ in $\fg^*$ whose normals are given by $\beta_1^*, \dots, \beta_h^*$, where the dual makes sense as an element of $\fg^*$ since we have chosen a splitting $\fg^* = \frak{k}^* \oplus \frak{h}^*$.  The ring of $\mathbb{T}^h$ invariant functions on $\C^h$ is generated by the functions $p_i(\vec{z}) = \vert z_i \vert^2$, $1\le i\le h$, hence the ring of $H$ invariant functions on $\C^h$ are also generated by the $p_i$'s since $H$ acts via:

\begin{displaymath}
\exp(X)\cdot (z_1,\dots,z_h)=(e^{2\pi i \beta_1(X)} z_1, \dots, e^{2\pi i \beta_h(X)} z_h)
\end{displaymath}
Thus, the map $\varphi:T^*K \times \C^h \to C$ is a quotient map by theorem 1 in \cite{Sch}.  Consequently, the commutativity of the diagram implies $\mu: U \to C$ is the quotient map and so it descends to the identity map $id_C:C \to C$, hence $\mu$ descends to a unimodular local embedding.

Now, if $(M,\sigma,\mu:M\to \fg^*)$ is toric folded-symplectic, then $(M\setminus Z, \sigma,\mu\big\vert_{M\setminus Z})$ is a toric \emph{symplectic} manifold, hence the orbital moment map $\psi$ restricts to a unimodular local embedding on $(M\setminus Z)/G = M/G \setminus Z/G$.  At the fold, $Z$, proposition \ref{prop:eqfsnormal} gives us an invariant neighborhood $U$ of $Z$ and a commutative diagram:

\begin{equation}\label{eq:bigdiagram}
\xymatrixcolsep{4pc}\xymatrix{
(U,\sigma) \ar[r]^-{\phi} \ar[d]^{\pi} &  (Z\times \R, p^*i^*\sigma + d(t^2p^*\alpha)) \ar[r]^{\tilde{\gamma}(z,t)=(z,t^2)}\ar[d] & (Z\times \R, p^*i^*\sigma +d(tp^*\alpha)) \ar[dr]^{\mu_s} \ar[d] \\
U/G \ar[r]^{\bar{\phi}}              & Z/G \times \R \ar[r]^{\gamma([z],t)=([z],t^2)}                                       & Z/G\times \R \ar[r]^{\bar{\mu}_s} & \fg^*
}
\end{equation}
where
\begin{itemize}
\item $\phi$ is an equivariant open embedding satisfying $\phi(z)=(z,0)$, $\bar{\phi}$ is an open embedding,
\item $\phi^*(p^*i^*\sigma+d(t^2p^*\alpha))=\sigma$,
\item $p^*i^*\sigma + d(tp^*\alpha)$ is nondegenerate in a neighborhood of $Z\times\{0\}$,
\item $\mu_s$ is a symplectic moment map in a neighborhood of $Z\times \{0\}$, and
\item $\bar{\mu}_s$ is the induced map on the quotient space.
\end{itemize}
From our study of the toric \emph{symplectic} case, $\bar{\mu}_s$ is a unimodular local embedding.  The composition of the arrows along the top row with $\mu_s$ gives the moment map $\mu\vert_U$, hence the composition of the arrows in the bottom row gives us $\psi\big\vert_{U/G}$.  Let $\gamma:Z/G \times \R \to Z/G \times \R$ be the map $\gamma([z],t)=([z],t^2)$.  Since $Z/G$ is a manifold with corners by proposition \ref{prop:corners}, $\gamma$ is a smooth map with fold singularities along $Z/G \times \{0\}$.  We therefore have:

\begin{displaymath}
\psi = \bar{\mu}_s \circ \gamma \circ \phi
\end{displaymath}
where $\bar{\mu}_s$ is a unimodular local embedding, $\gamma$ is a fold map, and $\phi$ is an open embedding of manifolds with corners.  Since $\phi$ is an open embedding, the map $\gamma \circ \phi$ is a map with fold singularities that folds along $Z/G$ by corollary \ref{cor:folds4-2}.  Since $\bar{\mu}_s$ is a local embedding, the composition $\bar{\mu}_s \circ \gamma \circ \phi$ is a map with fold singularities that folds along $Z/G$ by corollary \ref{cor:folds-diffeo}.  Thus, $\psi$ has fold singularities at points of $Z/G$.

We can read the kernel of $d\psi$ at $Z/G$ from the diagram.  The kernel of $p^*i^*\sigma + d(t^2p^*\alpha)$ contains $\frac{\partial}{\partial t}$ at points of $Z\times \{0\}$.  The kernel of $d\gamma$ at $Z/G \times \{0\}$ is given by $\frac{\partial}{\partial t}$, which is the image of the kernel of the fold form under the projection map, hence $d\pi_z(\ker(\sigma_z))=\ker(d\psi_z)$.
\end{proof}

\begin{remark}
Let $(M,\sigma,\mu:M\to \fg^*)$ be a toric, folded-symplectic manifold with co-orientable folding hypersurface $Z$.  By proposition \ref{prop:umf}, the kernel of the differential of the orbital moment map, $d\psi$, at points of $Z/G$ is given by the image of the fibers of the bundle $\ker(\sigma) \to Z$ under the projection $d\pi:TM \to T(M/G)$.  The bundle $\ker(\sigma)\to Z$ has rank $2$, while the bundle $\ker(d\psi) \to Z/G$ has rank $1$.  The reason that the dimensions are not preserved under $d\pi$ is because the fibers of the bundle $\ker(\sigma)\cap TZ$ are generated by the group action by lemma \ref{lem:generator}.  Thus, a $1$-dimensional subspace of each fiber of $\ker(\sigma)$ is eliminated by the differential of the orbit map $d\pi$.  On the other hand, we are about to show that one may recover $\ker(\sigma)\cap TZ$ with its orientation induced by $\sigma$ from the orbital moment map.
\end{remark}

\begin{definition}\label{def:annihilator}
Let $(M,\sigma,\mu:M\to \fg^*)$ be a toric, folded-symplectic manifold with co-orientable folding hypersurface $Z$.  Then $M/G$ is a manifold with corners by proposition \ref{prop:corners} and $\mu$ descends to $\psi:M/G \to \fg^*$, a unimodular map with folds.  We may define two vector bundles over $Z/G$.  We have the rank $1$ \emph{kernel bundle} of $\psi$:

\begin{displaymath}
\ker(d\psi) \to Z/G, \mbox{ } \ker(d\psi)_{[z]}:= \{v \in T_{[z]}M/G \mbox{ } \vert \mbox{ } d\psi_{[z]}(v)=0\}
\end{displaymath}
and the rank $1$ \emph{annihilator bundle} of $\psi$, which functions as the cokernel of the differential:
\begin{displaymath}
\operatorname{Im}(d\psi)^o \to Z/G, \mbox{ } \operatorname{Im}(d\psi)^o_{[z]}:= \{\eta \in \fg \mbox{ } \vert \mbox{ } \langle d\psi_{[z]}, \eta \rangle =0 \}
\end{displaymath}
where we have identified $T\fg^* = \fg^* \times \fg^*$.
\end{definition}

\begin{remark}
There is a fast way to see that $\operatorname{Im}(d\psi)^o \to Z/G$ is in fact a rank $1$ vector bundle over $Z/G$.  Consider the trivial bundle $Z/G \times \fg$ over $Z/G$ and the cotangent bundle $T^*(Z/G)$.  The differential $d\psi$ coupled with the canonical pairing $\langle \cdot, \cdot \rangle$ between $\fg^*$ and $\fg$ gives us a map of vector bundles over $Z/G$:

\begin{displaymath}
\xymatrixcolsep{4pc}\xymatrix{
Z/G \times \fg \ar[r]^{\langle d\psi, \cdot \rangle} & T^*(Z/G)
}
\end{displaymath}
where the map is given pointwise by $([z],X) \to \langle d\psi_{[z]},X \rangle$, which is a covector in $T^*(Z/G)$.  The annihilator bundle $\operatorname{Im}(d\psi)^o \to Z/G$ is then the kernel of this map.
\end{remark}

We focus our attention on $\operatorname{Im}(d\psi)^o \to Z/G$ for now and show there is an orientation on $\operatorname{Im}(d\psi)^o \to Z/G$ induced by $\psi$.  We will then show that there is a smooth, orientation-preserving isomorphism of vector bundles $(\cdot)_Z:\pi^*\operatorname{Im}(d\psi)^o \to (\ker(\sigma)\cap TZ)$, where $\pi:M \to M/G$ is the quotient map.

\begin{lemma}\label{lem:orientation1}
Let $(M,\sigma,\mu:M\to \fg^*)$ be a toric, folded-symplectic manifold and let $\psi:M/G \to \fg^*$ be the orbital moment map.  Let $\operatorname{Im}(d\psi)^o \to Z/G$ be the annihilator bundle.  There exists an orientation on $\operatorname{Im}(d\psi)^o \to Z/G$ induced by $\psi$.
\end{lemma}

\begin{proof}
We mimic the construction of the orientation on $\ker(\sigma)\cap TZ$.  As in the case of the orientation on the fold, we will be computing the intrinsic derivative $D\psi_{[z]}$ of $\psi$ at a point $[z]\in Z/G$, which will give us a quadratic map $D\psi_{[z]}: \ker(d\psi_{[z]})\otimes \ker(d\psi_{[z]}) \to \operatorname{coker}(d\psi_{[z]})$.  An element $v\in \operatorname{Im}(d\psi)^o_{[z]}$ is then positively oriented if for any nonzero element $w\in \ker(d\psi_{[z]})$, we have $\langle D\psi_{[z]}(w\otimes w), v \rangle >0$, where we identify $\operatorname{coker}(d\psi_{[z]})$ with $\operatorname{Im}(d\psi_{[z]})^o\subset \frak{g}$.

Of course, we have an alternative approach for those who are less familiar with the intrinsic derivative.  The orientation is defined at a point $[z]\in Z/G$ as follows:

\begin{enumerate}
\item Choose a nonzero element $Y\in \ker(d\psi_{[z]})$ and extend it to a local vector field $\tilde{Y}$ near $[z]$.
\item Then, for any $X \in \operatorname{Im}(d\psi)^o_{[z]}$, we obtain a function $f(p)=\langle d\psi_p(\tilde{Y}), X \rangle$ defined in a neighborhood of $[z]$.
\item An element $X \in \operatorname{Im}(d\psi)^o_{[z]}$ will be positively oriented if $df_{[z]}(Y)>0$.
\end{enumerate}
We first show that $df_{[z]}(Y)$ must be nonzero.  In coordinates $(p,t)\in Z/G\times \R$ near $[z]$ with the fold identified with $\{t=0\}$, we may write $\psi$ as:

\begin{displaymath}
\psi(p,t)= \psi\big\vert_{Z/G}(p) + t^2F(p)
\end{displaymath}
where $F([z])$ is some nonvanishing smooth map $F:Z/G \to \fg^*$ and
\begin{displaymath}
\phi(p,t)=\psi\big\vert_{Z/G}(p) + tF(p)
\end{displaymath}
is a local embedding.  We have that $F(p)$ is transverse to the image of $d\psi_p$, hence any nonzero $X\in \operatorname{Im}(d\psi)^o_p$ has a nonzero pairing $\langle F(p), X \rangle$ ($F$ plays the role of the intrinsic derivative).  If $\langle F(p), X \rangle=0$, then $X$ annihilates $d\psi_p(T_p(Z/G)) + \R(F(p)) =\fg^*$, hence $X$ annihilates $\fg^*$ and must be $0$.  Now, at $([z],0)$, our choice of $Y$ must be $c\frac{\partial}{\partial t}$ for some $c\ne 0$.  If we extend it to a local vector field and compute, we obtain:
\begin{displaymath}
c\frac{\partial}{\partial t}\big\vert_{t=0}\langle d\psi_{([z],0)}(\tilde{Y}), X \rangle = c^2\langle F([z]), X \rangle \ne 0
\end{displaymath}
Now, the choice of $Y$ doesn't affect the sign of the answer: changing the sign of $c$ doesn't change the sign of $c^2$.  Thus, imposing the condition that $X$ is positively oriented if and only if $c^2\langle F([z]), X \rangle >0$ is independent of the choice of $Y$.  The orientation induced on $\operatorname{Im}(d\psi)^o$ is therefore intrinsic to the map $\psi$.
\end{proof}

\begin{lemma}\label{lem:nullfoliation}
Let $(M,\sigma,\mu:M \to \fg^*)$ be a toric, folded-symplectic manifold with co-orientable folding hypersurface.  Let $\ker(\sigma)\cap TZ$ be the null bundle over $Z$ with its orientation induced by $\sigma$.  Let $\pi:M \to M/G$ be the quotient map and let $\operatorname{Im}(d\psi)^o$ be the annihilator bundle of the orbital moment map $\psi:M/G \to \fg^*$ with its orientation induced by $\psi$.  Then the map:

\begin{displaymath}
(\cdot)_Z: \pi^*\operatorname{Im}(d\psi)^o \to (\ker(\sigma)\cap TZ)
\end{displaymath}
given pointwise by $(\cdot)_Z(z,X)=(X)_Z(z)$, is an orientation-preserving isomorphism of vector bundles over $Z$.
\end{lemma}

\begin{proof}
We first need to check that the image of the map $(\cdot)_Z$ actually lands inside $\ker(\sigma)\cap TZ$.  We will then show that it is an orientation-preserving map and, since the two bundles are line bundles, this will show that the map is an isomorphism.

Consider an element $X$ in the fiber of the annihilator bundle $\pi^*\operatorname{Im}(d\psi)^o_z \subset \fg$.  By definition of the annihilator bundle,

\begin{equation}\label{eq:image}
\begin{array}{ll}
\langle d\psi_{\pi(z)}, X \rangle =0 & \implies \\
\langle d(\psi\circ \pi)_z, X \rangle = 0 & \iff \\
\langle d\mu_z,X \rangle =0 & \iff \\
i_{X_Z(z)}\sigma_z =0 & \text{By definition of the moment map.}
\end{array}
\end{equation}
Here, we have used the fact that $X_Z(z)=X_M(z)$ since $G$ preserves the fold $Z$.  We need to show that the last line of \ref{eq:image} implies $X_Z\in \ker(\sigma)$.  If we can show that this implication is true on an open dense subset of $Z$, then the smoothness of $(\cdot)_Z$ implies that its image is inside $\ker(\sigma)\cap TZ$ over all of $Z$.  Let us find the requisite open dense subset.

The action of $G$ on $M$ is effective, hence the action is free on an open dense subset $M_{\{e\}}$ of $M$.  By corollary \ref{cor:eqfsnormal1}, $M_{\{e\}}$ is transverse to the fold $Z$, hence $Z_{\{e\}} = M_{\{e\}} \cap Z$ is nonempty.  Otherwise, every point in $Z$ would have a nontrivial stabilizer, hence every point in some invariant neighborhood of $Z$ would have a nontrivial stabilizer since the orbit-type strata are transverse to $Z$.  Since $M_{\{e\}}$ is open and dense, $Z_{\{e\}}$ is open and dense in $Z$.  Otherwise, there is a point $z\in Z$ and a neighborhood $U\subseteq$ of $z$ in $Z$ so that all points $p\in U$ have nontrivial stabilizer, meaning there is an invariant neighbourhood of $U$ in $M$ where every point $p$ has a nontrivial stabilizer.  Thus, $M_{\{e\}}$ is not dense, which is a contradiction.

Thus, on an open dense subset $Z_{\{e\}}$ of $Z$, the induced vector fields $X_Z$, $X\in \fg$, vanish if and only if $X=0$.  Let us restrict our attention to this subset.  If $z\in Z_{\{e\}}$, then $i_{X_Z(z)}\omega_z = 0$ if and only if $X_Z(z) \in \ker(\omega_z)$ and $X_Z(z)\ne 0$.  Thus, the last line of equation \ref{eq:image} shows that the $X_Z(z)$ is a nonzero element of $\ker(\sigma_z)$.  Since $X_Z(z)=(\cdot)_Z(z,X)$, we have shown that the image of $(\cdot)_Z$ is in $\ker(\sigma)\cap TZ$ on an open dense subset, hence it lies in $\ker(\sigma)\cap TZ$ everywhere.

The fact that the map $(\cdot)_Z$ is orientation preserving is most easily seen using the local model of proposition \ref{prop:eqfsnormal}.  We have a diagram:

\begin{displaymath}
\xymatrixcolsep{4pc}\xymatrix{
(Z \times \R, p^*i^*\sigma + d(t^2p^*\alpha)) \ar[d] \ar[dr]^-{\mu\vert_Z + t^2F} &  \\
Z/G \times \R \ar[r]^-{\bar{\mu}\vert_{Z/G} + t^2\bar{F}}           & \fg^*
}
\end{displaymath}
where $\alpha\in \Omega^1(Z)$ orients the null foliation and $F:Z \to \fg^*$ is defined by the pairing $\langle F(z), X \rangle = \alpha_z(X_M(z))$.  The orientation of $\operatorname{Im}(d\psi)^o$ is such that $X \in \operatorname{Im}(d\psi)^o_{[z]}$ is positively oriented if and only if $\langle \bar{F}([z]), X \rangle > 0$, but

\begin{displaymath}
\begin{array}{ll}
\langle \bar{F}([z]),X \rangle > 0 & \implies \\
\langle F(z),X \rangle > 0         & \implies \\
\alpha_z(X_M(z)) > 0               &
\end{array}
\end{displaymath}
which means that $X_M(z) = (\cdot)_Z(z,X)$ is positively oriented.
\end{proof}

Thus, the orbital moment map $\psi:M/G \to \fg^*$ allows us to completely reconstruct $\ker(\sigma)\cap TZ$ and, hence, the null foliation on the folding hypersurface $Z\subset M$.  On the other hand, there is an easy way to reconstruct the rest of the kernel bundle $\ker(\sigma)\to Z$ using a lifting recipe.

\begin{lemma}\label{lem:lifts}
Let $(M,\sigma,\mu:M\to \fg^*)$ be a toric folded-symplectic manifold with co-orientable folding hypersurface and orbital moment map $\psi:M/G \to \fg^*$.  If $[z]\in Z/G$ is a point in the folding hypersurface of $\psi$, then we may lift $\ker(d\psi_{[z]})$ stratum-by-stratum as follows.  If $v\in \ker(d\psi_{[z]})$ is a nonzero kernel element, then we may choose a lift of $v$ to a vector $\tilde{v}\in T_zM$ such that:

\begin{enumerate}
\item $d\pi_z(\tilde{v}) = v$,
\item $i_{\tilde{v}}\sigma=0$
\end{enumerate}
\end{lemma}

\begin{remark}
We are not defining a smooth lift in lemma \ref{lem:lifts}; we are simply saying that one may make a choice of a lift at each point in $Z/G$.  The two conditions guarantee that $\tilde{v}$ is transverse to $T_zZ$ and lies inside the kernel $\ker(\sigma_z)$.  Thus, if we couple the span of each lift with the image of the map $(\cdot)_Z: \pi^*\operatorname{Im}(d\psi)^o \to (\ker(\sigma)\cap TZ)$ from lemma \ref{lem:nullfoliation}, we recover the bundle $\ker(\sigma)$ in its entirety.
\end{remark}

\begin{proof}
Recall from the proof of proposition \ref{prop:umf} that we have a commutative diagram in equation \ref{eq:bigdiagram}, giving us a local factorization of the orbital moment map:

\begin{equation}
\xymatrixcolsep{4pc}\xymatrix{
(U,\sigma) \ar[r]^-{\phi} \ar[d]^{\pi} &  (Z\times \R, p^*i^*\sigma + d(t^2p^*\alpha)) \ar[r]^{\tilde{\gamma}(z,t)=(z,t^2)}\ar[d] & (Z\times \R, p^*i^*\sigma +d(tp^*\alpha)) \ar[dr]^{\mu_s} \ar[d] \\
U/G \ar[r]^{\bar{\phi}}              & Z/G \times \R \ar[r]^{\gamma([z],t)=([z],t^2)}                                       & Z/G\times \R \ar[r]^{\bar{\mu}_s} & \fg^*
}
\end{equation}
At a point $[z]\in Z/G$, the kernel is spanned by $\displaystyle\frac{\partial}{\partial t}$ and we may certainly lift it to $\displaystyle\frac{\partial}{\partial t}$ on $Z\times \R$, which satisfies the requisite conditions of the lemma.
\end{proof}

We have therefore (almost) proved the penultimate structure theorem regarding toric, folded-symplectic manifolds with co-orientable folding hypersurfaces.

\begin{theorem}\label{thm:structure}
Let $(M,\sigma,\mu:M\to \fg^*)$ be a toric, folded-symplectic manifold with co-orientable folding hypersurface.  Then,

\begin{enumerate}
\item The orbit type strata $M_H$ are transverse to the folding hypersurface and each $(M_H,i_{M_H}^*\sigma, \mu\vert_{M_H})$ is a toric, folded-symplectic manifold, hence $M$ is stratified by toric, folded-symplectic manifolds.
\item The orbit space $M/G$ is a manifold with corners and the boundary strata of $M/G$ are given by the images of the orbit-type strata $M_H/G$.
\item The moment map descends to $\psi:M/G\to \fg^*$, a unimodular map with folds.  Furthermore, since each $(M_H,i_{M_H}^*\sigma)$ is a toric, folded-symplectic manifold, the restriction of $\psi$ to $M_H/G$ is a map with fold singularities if we view it as a map into $\frak{h}^o$.  Hence $\psi:M/G \to \fg^*$ is a unimodular map with folds that restricts to maps with fold singularities on the boundary strata.
\item The null-foliation on $Z$ may be recovered from $\psi$, along with its orientation induced by $\sigma$ using the intrinsic derivative of $\psi$ and the map $(\cdot)_Z: \pi^*\operatorname{Im}(d\psi)^o \to \ker(\sigma)\cap TZ$ (q.v. lemma \ref{lem:nullfoliation}).
\item The remainder of the bundle $\ker(\sigma)$ can be constructed by choosing lifts of elements of $\ker(d\psi)$.
\item The representation of $H$ on the fibers of $(\widetilde{TM_H})^{\sigma}$ at $Z$ may be read from the orbital moment map.
\item The local structure of the folding hypersurface is determined by the image of $\psi(Z/G)$ (q.v. corollary \ref{cor:foldnorm}.
\end{enumerate}
Thus, the fold, the null foliation, the orientation, the kernel bundle, and the symplectic slice representation may all be recovered from the orbital moment map.  And, by proposition \ref{prop:torsympnorm}, one may recover all symplectic invariants away from the fold just by reading the weights of the symplectic slice representation from the orbital moment map.
\end{theorem}

\begin{proof}\mbox{ } \newline
The only detail we haven't proven is that, if $p\in M/G$, the representation of the stabilizer of a point in $\pi^{-1}(p)$ on the symplectic normal bundle to the orbit-type stratum containing $p$ is encoded in the moment map.  Let's summarize the proofs of all of the claims of theorem \ref{thm:structure} and prove this last detail along the way.

\begin{enumerate}
\item Lemma \ref{lem:stabtori} implies that $(M_H,i^*_{M_H}\sigma, \mu\big\vert_{M_H})$ is a toric, folded-symplectic manifold.
\item Proposition \ref{prop:corners} implies that $M/G$ is a manifold with corners.  Locally, the boundary strata of $M/G$ are determined by the images of the orbit type strata $M_H$, hence globally the boundary strata of $M/G$ are determined by $M_H/G$.
\item Proposition \ref{prop:umf} implies that $\psi:M/G \to \fg^*$ is a unimodular map with folds.  By part $1$, $M_H$ is a toric, folded-symplectic manifold, hence proposition \ref{prop:umf} implies that the restriction $\psi:M_H/G \to \frak{h}^o$ is a unimodular map with folds.  However, $M_H/G$ is a manifold without corners, hence unimodularity is a vacuous condition and we may remove it: $\psi:M_H/G \to \frak{h}^o$ is a map with fold singularities.
\item Lemmas \ref{lem:lifts} and \ref{lem:nullfoliation} imply parts $4$ and $5$ of the theorem.
\item The fact that the representation of $H$ can be read from $\psi$ follows from proposition \ref{prop:torsympnorm} and remark \ref{rem:slicerep}.  In particular, at a point $p$ away from $Z/G$, there is a neighborhood $U_p$ and a unimodular cone $C$ such that $\psi\big\vert_U : U \hookrightarrow C$ is an open embedding.  By proposition \ref{prop:torsympnorm}, the normals to the facets of $C$ span the integral lattice of a subtorus $H$ of $G$, hence their duals in $\frak{h}^*$ define the weights of the symplectic slice representation.  By remark \ref{rem:slicerep}, the representation on the symplectic slice at points away from $Z\subset M$ is canonically isomorphic to the representation on the fibers of $(\widetilde{TM_H})^{\sigma}$ and this representation doesn't change along connected components of $M_H$.  By lemma \ref{lem:attach1}, the assignment of a basis of the integral lattice of $H$ to points extends across the fold.  Thus, we may read the representation from $\psi$.
\item By corollary \ref{cor:foldnorm}, the folding hypersurface is equivariantly isomorphic to a hypersurface $\Sigma \subset T^*K \times \C^h$ for some subtori $K\le G$ and $H\le G$ of $G$, where $H$ acts on $\C^h$ via rotations.  The corollary states that this hypersurface is uniquely determined by the moment map image, hence $\psi$ locally determines $\Sigma \subset T^*K \times \C^h$ up to isomorphism.
\end{enumerate}
\end{proof}

We need one last lifting lemma before we leave the basic theory of toric, folded-symplectic manifolds behind.  To this end, it may be easiest to lead with an example.

\begin{example}\label{ex:invntvfs}
Consider the action of $S^1$ on $\C$ by rotations.  The orbit map is $q(z) = \vert z \vert^2 \in \R^+$.  We would like to know when one may lift a vector field on $\R^+$ to an invariant vector field on $\C$ via the quotient map $q$.  We claim that if a vector field $\displaystyle X=f\frac{\partial}{\partial t}$ is stratified, then it lifts to an invariant vector field on $\C$, namely the radial vector field.  Indeed, if $X$ is stratified on $\R^+$ then it must vanish at the origin.  Consequently, we may write it as $\displaystyle X=tg\frac{\partial}{\partial t}$ for some smooth function $g$.  We then define the lift of $X$ to be:

\begin{displaymath}
\tilde{X}(z)= \frac{1}{2}g\circ q R
\end{displaymath}
where $R(z)$ is the radial vector field.  In cartesian coordinates, we may write this lift as $\displaystyle\tilde{X}(x,y)=\frac{1}{2}(g\circ q)x\frac{\partial}{\partial x} + y\frac{\partial}{\partial y}$.  Since the quotient map is $q(x,y)=x^2+y^2$, we have:

\begin{displaymath}
dq_{(x,y)}(\tilde{X}) = \frac{1}{2}(g\circ q)(2x^2 + 2y^2)\frac{\partial}{\partial t}= (g\circ q)(x^2+y^2)\frac{\partial}{\partial t} = (X \circ q)(x,y)
\end{displaymath}
hence $dq(\tilde{X})=X\circ q$.  Now, this procedure applies more generally to the $\mathbb{T}^h$ action on $\C^h$ by rotations.  The orbit space is $(\R^+)^h$ with coordinates $(x_1,\dots,x_h)$.  A stratified vector field is a linear combination of the vector fields $\displaystyle \frac{\partial}{\partial x_i}$ and each of these lift to the radial vector fields on each factor of $\C$, hence any stratified vector field has a lift to $\C^h$.
\end{example}

\begin{lemma}\label{lem:liftbro}
Let $(M,\sigma,\mu:M\to \fg^*)$ be a toric, folded-symplectic manifold with orbit map $\pi:M\to M/G$ and orbital moment map $\psi:M/G \to \fg^*$.  Let $[z]\in Z/G$ be a point in the fold of $\psi$.  Suppose $X$ is a stratified vector field on $M/G$ so that $X_p \in \ker(d\psi_p)$ and $X_p\ne 0$ for all $p\in Z/G$.  Then there exists a neighborhood $U$ of $[z]$ and lift of $X$ to an invariant vector field $\tilde{X}$ on $\pi^{-1}(U)$ so that $\tilde{X}_z \in \ker(\sigma_z)$ for all $z\in \pi^{-1}(U)\cap Z$.  That is, we may lift stratified vector fields passing through the kernel of $d\psi$ to invariant vector fields passing through $\ker(\sigma)$.
\end{lemma}

\begin{proof}
Let $[z]\in Z/G$ and let $z\in \pi^{-1}([z])$ and let $H$ be the stabilizer of $p$.  Since the claim is local, we may assume that $M=Z\times \R$ where the kernel of $\sigma$ contains $\displaystyle \frac{\partial}{\partial t}$ and $Z=K\times\R^{g-h-1} \times \C^h$, where $g=\dim(G)$, $h=\dim(H)$, and $K$ is a subtorus of $G$ complementary to $H$.  Here, we are using the fact that the differential slice $T_zZ/T_z(G\cdot z)$ is isomorphic to $TZ_H/(T_z(G\cdot z)) \oplus (\widetilde{TZ_H})^{\sigma}$, where the second summand is a faithful, symplectic representation of $H$ with dimension $2\dim(H)$.  The orbit map is then:

\begin{displaymath}
q(k,x_1,\dots,x_{g-h-1},z_1,\dots,z_h,t) = (x_1,\dots,x_{g-h-1}, \vert z_1 \vert^2, \dots, \vert z_h \vert^2, t)
\end{displaymath}
If $t_1, \dots, t_h$ are the coordinates on $(\R)^+$, then any stratified vector field may be written as a linear combination of:

\begin{itemize}
\item the vector field $\displaystyle \frac{\partial }{\partial t}$,
\item the vector fields of the form $\displaystyle \frac{\partial}{\partial x_i}$, and
\item the vector fields of the form $\displaystyle t_i\frac{\partial}{\partial t_i}$.
\end{itemize}
Since each of these has a lift, there is no problem with producing a lift of any linear combination of them.  We want a specific lift, though: one which passes through the fold tangent to the kernel $\ker(\sigma)$.  A vector field that is tangent to $\ker(d\psi)$ at $t=0$ must have coefficients that vanish at $t=0$ for all terms except the $\displaystyle \frac{\partial}{\partial t}$ term.  When we lift, we pull back these coefficients via the quotient map, so the lifted vector field has the property that all terms except for the $\displaystyle \frac{\partial}{\partial t}$ term vanish at $t=0$, hence the lifted vector field takes values in $\ker(\sigma)$ at the fold.
\end{proof}

\subsection{The Categories $\mathcal{M}_{\psi}$ and $\mathcal{B}_{\psi}$}
We have made a strong case for the fact that the only two invariants of a toric, folded-symplectic manifold are the orbit space $M/G$, which is a manifold with corners, and the orbital moment map $\psi:M/G \to \fg^*$, which is a unimodular map with folds.  We therefore fix a manifold with corners, $W$, and a unimodular map with folds $\psi:W \to \fg^*$, and we ask: is it possible to classify all toric folded-symplectic manifolds whose orbit space is $W$ and whose orbital moment map is $\psi:W\to \fg^*$?  The answer will be \emph{yes}, but we will need a bit of machinery to prove it.  We first begin by collecting the data into a category.

\begin{definition}\label{def:empsi}
Let $W$ be a manifold with corners and let $\psi:W\to \fg^*$ be a unimodular map with folds, where $\fg$ is the Lie algebra of a torus $G$.  We define the category $\mathcal{M}_{\psi}(W)$ to be the category whose objects are triples:
\begin{displaymath}
(M,\sigma, \pi:M \to W)
\end{displaymath}
where $\pi$ is a quotient map and $(M,\sigma, \psi \circ \pi)$ is a toric, folded-symplectic manifold with co-orientable folding hypersurface, where the torus is $G$, with moment map $\psi \circ \pi$.  We refer to an object as a \emph{toric, folded-symplectic manifold over $\psi$}.  A morphism between two objects $(M_i,\sigma_i,\pi_i:M \to W)$, $i=1,2$, is an equivariant diffeomorphism $\phi:M_1 \to M_2$ that induces a commutative diagram:

\begin{displaymath}
\xymatrix{
M_1 \ar[rr]^{\phi} \ar[dr]^{\pi_1} & & M_2 \ar[dl]^{\pi_2} \\
                                   &W \ar[r]^{\psi}& \fg^*
}
\end{displaymath}
and satisfies $\phi^*\sigma_2=\sigma_1$, hence $\phi$ is an equivariant folded-symplectomorphism that preserves moment maps.  By definition, every morphism is invertible, hence $\mathcal{M}_{\psi}(W)$ is a groupoid.
\end{definition}

\begin{remark}\label{rem:empsi}
It turns out that $\mathcal{M}_{\psi}(W)$ is more than just a groupoid.  If $U\subset W$ is an open subset, then for any object $(M,\sigma, \pi:M \to W)$ we have the restricted object $(\pi^{-1}(U),\sigma, \pi\big\vert_{\pi^{-1}(U)}:\pi^{-1}(U) \to U)$, which we denote as $(M\big\vert_U, \sigma, \pi\big\vert_U)$, despite the fact that $U$ is an open subset of $W$.  Since $\psi\big\vert_U$ is a unimodular map with folds, we have that $(M\big\vert_U, \sigma \pi\big\vert_U)$ is an object of $\mathcal{M}_{\psi}(U)$.  Thus, for each pair of open sets $U, V\subseteq W$ of $M$, we have a restriction functor:

\begin{displaymath}
\big\vert_V^U: \mathcal{M}_{\psi}(U) \to \mathcal{M}_{\psi}(V)
\end{displaymath}
For any three open subsets $U,V,T$ with $T \subset V \subset U$, the restriction maps satisfy $\big\vert_T^V \circ \big\vert_V^U=\big\vert_T^U$, hence $\mathcal{M}_{\psi}: \operatorname{Open}(W)^{op} \to \mathcal{M}_{\psi}(\cdot)$ is a presheaf.  Each category $\mathcal{M}_{\psi}(U)$ is a groupoid by definition, hence $\mathcal{M}_{\psi}$ is a presheaf of groupoids.

\vspace{5mm}
If $(M_1,\sigma_1,\pi_1)$ and $(M_2,\sigma_2,\pi_2)$ are two toric, folded-symplectic manifolds over $\psi:W \to \fg^*$ and their restrictions to open subsets agree, then they must agree as toric-folded-symplectic manifolds.  If we have an open subset $U\subseteq W$ and an open cover of $\{U_i\}_{i\in I}$ with objects $(M_i, \sigma_i, \pi_i)$ in $\mathcal{M}_{\psi}(U_i)$ for each $i$ that satisfy:

\begin{displaymath}
(M_i\big\vert_{U_i\cap U_j}, \sigma_i ,\pi_i \big\vert_{U_i \cap U_j}) = (M_j\big\vert_{U_i\cap U_j}, \sigma_j, \pi_j\big\vert_{U_i\cap U_j})
\end{displaymath}
then we may form the space $(M,\sigma,\pi) = (\sqcup_i (M_i,\sigma_i \pi_i))/\sim$, where $\sim$ is the equivalence relation that states that two points are equivalent $p_i \sim p_i$ if they are both in $U_i\cap U_j$ and the form $\sigma$ at an equivalence class $[p_i]$, $\sigma_{[p_i]}$, is defined to be $\sigma_{p_i}$.  Similarly, the quotient map $\pi$ at $[p_i]$ is $\pi(p_i)$, which doesn't depend on $i$.  This space is a toric, folded-symplectic manifold.  Thus, $\mathcal{M}_{\psi}$ is a \emph{sheaf} of groupoids.
\end{remark}

It is not particularly efficient to study $\mathcal{M}_{\psi}$ directly.  For example, consider the question of whether or not $\mathcal{M}_{\psi}(W)$ is nonempty.  At this stage, we cannot answer such a question since we don't know how one might construct objects in $\mathcal{M}_{\psi}$.  On the other hand, we \emph{could} answer the question if we were able to assume that the action is free, which we discuss below in remark \ref{rem:nonempty}.  Thus, we construct a category where we assume all actions are principal.  Since we are keeping the orbit space $W$ fixed, this will mean that the total space of a principal bundle over $W$ will be a manifold with corners.  Since we are interested in Hamiltonian torus actions, we will need a definition of an Hamiltonian action on a manifold with corners.  To obtain such a definition, simply replace the words \emph{folded-symplectic manifold without corners} in definition \ref{def:fsham} with the words \emph{folded-symplectic manifold with corners}.  We now form the category of principal, toric, folded-symplectic bundles over a fixed unimodular map with folds, $\psi:W \to \fg^*$.

\begin{definition}\label{def:bpsi}
Let $\psi: W \to \fg^*$ be a fixed unimodular map with folds, where $\fg$ is the Lie algebra of a torus $G$.  We define $\mathcal{B}_{\psi}(W)$ to be the category whose objects are principal $G$-bundles $\pi:P \to W$ equipped with an invariant folded-symplectic form $\sigma$ with co-orientable folding hypersurface, denoted as a pair

\begin{displaymath}
(\pi:P \to W, \sigma)
\end{displaymath}
so that $\psi\circ \pi$ is a moment map for the principal action of the torus $G$ on $P$.  A morphism $\phi$ between two objects $(\pi_1:P_1\to W, \sigma_1)$ and $(\pi_2:P_2 \to W, \sigma_2)$ is a map of principal $G$ bundles:

\begin{displaymath}
\xymatrix{
P_1 \ar[rr]^{\phi} \ar[dr]^{\pi_1}& & P_2 \ar[dl]^{\pi_2} \\
                                  &W\ar[r]^{\psi} & \fg^*
}
\end{displaymath}
so that $\phi^*\sigma_2 = \sigma_1$, hence $\phi^*(\psi\circ \pi_2) =\psi \circ \pi_1$.  That is, $\phi$ preserves moment maps.
\end{definition}

\begin{remark}\label{rem:bpsi}
As in the case of $\mathcal{M}_{\psi}$, it is straightforward to show that $\mathcal{B}_{\psi}$ is a sheaf of groupoids on $W$.
\end{remark}

\begin{remark}\label{rem:nonempty}
Unlike the case of $\mathcal{M}_{\psi}(W)$, it is easy to show that $\mathcal{B}_{\psi}(W)$ is nonempty.  Consider the cotangent bundle $T^*G=G\times \fg^*$, where $G$ is a torus, with its canonical symplectic structure $\omega_{T^*G}$ and canonical moment map $\mu(\lambda, \eta)= -\eta$ given by projection.  We have a pullback diagram:

\begin{displaymath}
\xymatrixcolsep{3pc}\xymatrix{
\psi^*(T^*G) \ar[d]^{\pi} \ar[r]^{id_G \times \psi} & G\times \frak{g}^* \ar[d]^{\mu} \\
W \ar[r]^{\psi}                               & \fg^*
}
\end{displaymath}
Since $\psi$ is a map with fold singularities, $\psi\times id_G$ is a map with fold singularities: its determinant vanishes transversally in any coordinate chart and $\ker(d\psi)=\ker(d\psi \times d(id_G))$, hence the kernel is transverse to the folding hypersurface $\pi^{-1}(Z)$ (q.v. corollary \ref{cor:folds4-1}).  Thus, $\sigma=\psi^*\omega_{T^*G}$ is a folded-symplectic form and $\psi \circ \pi$ is a moment map for the action of $G$.  Since $\mu:T^*G \to \fg^*$ is a principal $G$-bundle, $\psi^*(T^*G)$ is a principal $G$ bundle.  Thus $(\pi:\psi^*(T^*G) \to W, \psi^*(\omega_{T^*G}))$ is an object in $\mathcal{B}_{\psi}(W)$.
\end{remark}

What is the purpose of defining $\mathcal{B}_{\psi}(W)$?  In a way, we will see that studying $\mathcal{B}_{\psi}(W)$ allows us to bypass the intricate global structures of objects in $\mathcal{M}_{\psi}(W)$.  After all, any object of $\mathcal{M}_{\psi}(W)$ is \emph{almost} a toric, folded-symplectic bundle over $\psi:W \to \fg^*$ since the action is free on an open dense subset.  So, we could simply replace points with stabilizers by corners, which is what $\mathcal{B}_{\psi}(W)$ does, and then remove the corners using a local cutting procedure.  This can be condensed into the following strategy:

\begin{enumerate}
\item Fix a unimodular map with folds, $\psi:W \to \fg^*$.
\item Classify objects in $\mathcal{B}_{\psi}(W)$ up isomorphism.  We call the isomorphism classes of objects in $\mathcal{B}_{\psi}(W)$ $\pi_0(\mathcal{B}_{\psi}(W))$.
\item Construct a functor $c:\mathcal{B}_{\psi}(W) \to \mathcal{M}_{\psi}(W)$.
\item Show the functor $c$ is an equivalence of categories, which amounts to showing it is an isomorphism of sheaves of groupoids.
\item Reap the rewards by noticing any equivalence of categories will induce a bijection on isomorphism classes of objects, hence $\pi_0(\mathcal{B}_{\psi}(W))= \pi_0(\mathcal{M}_{\psi}(W))$.
\end{enumerate}

Before we leave the basic theory of toric, folded-symplectic manifolds behind and endeavour to classify objects of $\mathcal{M}_{\psi}(W)$, let us record a local uniqueness statement for objects of $\mathcal{M}_{\psi}(W)$, which we will use to show that the functor $c$, which we have yet to construct or define, is an isomorphism of sheaves of groupoids.

\begin{lemma}\label{lem:locunique}
Let $\psi:W \to \fg^*$ be a unimodular map with folds with folding hypersurface $\hat{Z}$, where $\fg$ is the Lie algebra of a torus $G$.  Suppose we have two objects $(M_i,\sigma_i,\pi_i:M \to W)$, $i=1,2$, in $\mathcal{M}_{\psi}(W)$.  Then for any point $w\in W$, there exists a neighborhood $U$ of $w$ and an isomorphism
\begin{displaymath}
\phi:(M_1\big\vert_U ,\sigma_1, \pi_1\big\vert_U) \rightarrow (M_2\big\vert_U, \sigma_2, \pi_2\big\vert_U)
\end{displaymath}
of toric, folded-symplectic manifolds.  That is, there exists a morphism between $(M_1\big\vert_U ,\sigma_1, \pi_1\big\vert_U)$ and $(M_2\big\vert_U, \sigma_2, \pi_2\big\vert_U)$ in the category $\mathcal{M}_{\psi}(U)$.
\end{lemma}

\begin{proof} \mbox{ }\newline
We first show that one can obtain a local isomorphism of folding hypersurfaces $\gamma:U_1 \to U_2$, where $U_i\subseteq Z_i$ is an open subset, so that $\gamma^*(i_{Z_2}^*\sigma_2) = i_{Z_1}^*\sigma_1$ and $\gamma$ induces a commutative diagram:

\begin{displaymath}
\xymatrix{
(U_1,i_{Z_1}^*\sigma_1, \pi_1) \ar[dr]^{\pi_1} \ar[rr]^{\gamma} & & (U_2,i_{Z_2}^*\sigma_2,\pi_2)  \ar[dl]^{\pi_1} \\
                                                                &U_0\subseteq \hat{Z} \ar[r]^{\psi}& \fg^*
}
\end{displaymath}
We then show that one may extend this isomorphism to an isomorphism $\phi$ of neighborhoods $\tilde{U}_1\subseteq M_1$ and $\tilde{U}_2\subseteq M_2$ of $U_1$ and $U_2$, which makes the diagram commute:

\begin{displaymath}
\xymatrix{
(\tilde{U}_1,\sigma_1,\pi_1) \ar[dr]^{\pi_1} \ar[rr]^{\phi} & & (\tilde{U}_2,\sigma_2, \pi_2 \ar[dl]^{\pi_2}) \\
                                                              &U\subset W \ar[r]^{\psi} & \fg^*
}
\end{displaymath}
The local isomorphism of the folds is constructed as follows.  Pick a point $z\in \hat{Z}$ in the folding hypersurface of $\psi$.  By theorem \ref{thm:structure}, we may read the stabilizer and symplectic slice representation at points $p_i\in \pi_i^{-1}(\hat{z})$, $i=1,2$, from $\psi$ at $z$.  Thus, the stabilizer of $p_1$ is the same as the stabilizer of $p_2$ and we denote it as $H\le G$.  Let $\beta_1,\dots, \beta_h$ be the weights associated to $\hat{z}$ via $\psi$ (q.v. lemma \ref{lem:attach1}).  By theorem \ref{thm:structure}, these are the weights of the symplectic slice representations of $H$ at $p_1$ and $p_2$.  Since they have the same weights, they are isomorphic as representations.  In particular, they are both isomorphic to $\C^h$.  Choose a subtorus $K$ complementary to $H$.  Corollary \ref{cor:foldnorm} then gives us invariant neighborhoods $U_i$ of $\pi_i^{-1}(\hat{z})$ in $Z_i$ and a commutative diagram:

\begin{displaymath}
\xymatrixcolsep{3pc}\xymatrix{
U_1 \ar[d]^{\pi_1} \ar[r]^-{j_{Z_1}} &  T^*K \times \C^h \ar[d]^{\phi} & \ar[l]_-{j_{Z_2}} U_2 \ar[d]^{\pi_2}\\
U_0 \ar[r]^{\psi}                   &   \fg^*                         & \ar[l]^{\psi} U_0
}
\end{displaymath}
where the arrows on the top row are $K\times H$ equivariant open embeddings, or $G$-equivariant depending on ones perspective since we can precompose with an isomorphism.  The arrow $\phi$ is the moment map for the action of $G=K\times H$ on $T^*K \times \C^h$.  By corollary \ref{cor:foldnorm}, the image of $j_{Z_i}(U_i)$ is uniquely determined by the moment map image $\psi(U_0)$, hence $j_{Z_2}^{-1}\circ j_{Z_1}$ is an equivariant isomorphism of the folding hypersurfaces.  To obtain the requisite commuting diagram, we shrink $U_0$ so that $\psi$ is an embedding on $U_0$ and is thus invertible as a map onto its image.  The map $\gamma:=j_{Z_2}^{-1}\circ j_{Z_1}$ then covers $\psi\big\vert_{U_0}^{-1}\circ \psi\big\vert_{U_0}= id_{U_0}$.

\vspace{5mm}

To extend the isomorphism, we choose a stratified vector field $w$ in a neighborhood $U$ of $U_0$ so that $w$ takes nonzero values in $\ker(d\psi)$ at points of $U_0\subseteq \hat{Z}$.  By lemma \ref{lem:liftbro}, this vector field lifts to an invariant local vector field $w_1$ and $w_2$ defined in a neighborhood of $U_1$ and $U_2$ respectively.  Furthermore, at points of $U_i$, we may assume that $w_i$ takes values in the kernel of $\sigma_i$ by lemma \ref{lem:liftbro}.  The integral curves of $w_i$ are mapped to the integral curves of $w$ under the map $\pi_i$ by definition of a lift.  Let $\Phi_i$ be the local flow of $w_i$.  We define a diffeomorphism from a neighborhood $\tilde{U}_1$ onto a neighborhood $\tilde{U}_2$ as follows.  For each point $z\in U_1$, define

\begin{displaymath}
\phi(\Phi_1(z,t))=\Phi_2(\gamma(z),t)
\end{displaymath}
which is simply the map that takes the flow line of $w_1$ through $z$ to the flow line of $w_2$ through $\gamma(z)$.  By construction, this map covers the identity $id_U:U\to U$, where $U\subseteq W$ is the image of $\tilde{U}_i$ under $\pi_i$.  This is because $\gamma$ covers the identity on $U_0$ and the flowlines $\Phi_1(z,t)$, $\Phi_2(\gamma(z),t)$ project to the flow line of $w$ through $\pi_1(z)=\pi_2(\gamma(z))$.  Note that the map $\phi$ may not be a diffeomorphism, but it has maximal rank at points of $Z_1$ so there is a neighborhood on which it is a diffeomorphism.  It is equivariant:

\begin{displaymath}
\phi(g\cdot \Phi_1(z,t)) = \phi(\Phi_1(g\cdot z, t)) = \Phi_2(\gamma(g\cdot z),t)= \Phi_2(g\cdot \gamma(z),t) = g\cdot \Phi_2(\gamma(z),t) = g\cdot \phi(\Phi_1(z,t)).
\end{displaymath}
It also restricts to $\gamma$ on $Z_1$ since $\phi(\Phi_1(z,0))=\Phi_2(\gamma(z),0)= \gamma(z)$.  Lastly, $\phi$ doesn't necessarily satisfy $\phi^*\sigma_2= \sigma_1$.  However, $\phi^*\sigma_2$ and $\sigma_1$ agree at the folding hypersurface and they induce the same orientation on $U_1\subseteq Z_1$ since this orientation can be read from $\psi$, which is fixed (q.v. lemma \ref{lem:nullfoliation}).  These two conditions are enough to guarantee that the linear path $\sigma_s=(1-s)\sigma_1 + s\phi^*\sigma_2$ is folded-symplectic in a neighborhood of $U_1\subseteq Z_1$.  As we have seen (q.v. proof of proposition \ref{prop:eqfsnormal}), this path will generate an invariant time-dependent vector field $X_s$ and an isotopy $\phi_s$ such that $\phi_s^*\sigma_s=\sigma_0$.

This time dependent vector field must be tangent to orbits for all $s$.  Indeed, the map $\phi_s$ preserves the moment map $\pi_1\circ \psi$, hence $d\mu(X_s)=0$ for all $s$.  In particular, $X_s(p)\in T_p(G\cdot p)^{\sigma}$ for each $p$ near $U_1$.  Since the orbits are Lagrangian on an open dense subset, this implies that $X_s(p)$ is inside $T_p(G\cdot p)$ on an open dense subset, meaning it is everywhere tangent to orbits.

Now, since the orbits are compact, we may integrate $X_s$ to obtain a time-dependent flow for all $s$ that preserves orbits.  That is, $\phi_s$ isn't just a family of open embeddings: it's actually a family of equivariant diffeomorphisms.  Thus, $\phi \circ \phi_1$ is a local isomorphism of toric, folded-symplectic manifolds covering the identity map.
\end{proof}

\pagebreak

\section{Folded-Symplectic Reduction}
The goal of this section is to generalize symplectic reduction to folded-symplectic manifolds and everything here, save the symplectic reduction theorem, is original work.  The following theorem gives the recipe for constructing a reduced space, or symplectic quotient.

\begin{theorem}\label{thm:sred}
Let $(M,\omega,\mu:M\to \fg^*)$ be a symplectic manifold with a proper, Hamiltonian action of a Lie group $G$ and corresponding moment map $\mu:M\to \fg^*$.  Suppose $G$ acts on $\mu^{-1}(0)$ freely.  Then $\mu^{-1}(0)$ is a smooth manifold of codimension $\dim(G)$, $\pi:\mu^{-1}(0)\to \mu^{-1}(0)/G$ is a principal $G$ bundle, there exists $\omega_0 \in \Omega^2(\mu^{-1}(0)/G)$ such that $\pi^*\omega_0=\omega\big\vert_{\mu^{-1}(0)}$, and $\omega_0$ is symplectic.
\end{theorem}

\begin{proof}[Sketch of the proof] \mbox{ } \newline

\begin{itemize}
\item Let $p\in \mu^{-1}(0)$ and let $X\in \fg$ be a nonzero element.  If we identify $\fg^*$ with $T_0\fg^*$, then the differential $d\mu_p:T_pM \to \fg^*$ is a map into $\fg^*$.  We can show it is surjective by proving the annihilator of its image, $\operatorname{Im}(d\mu_p)^o$, is $\{0\}\subset \fg$.  We compute:

    \begin{displaymath}
    0=\langle d\mu_p,X \rangle \iff (d\langle \mu, X \rangle)_p=0 \iff -i_{X_M}\omega_p =0 \iff X_M(p)=0
    \end{displaymath}
    The action is free at $p$ by assumption, hence $X_M(p) \iff X=0$.  Thus, $X$ annihilates the image of $d\mu_p$ if and only if it is $0$.  We therefore have that $p$ is a regular point of $\mu$, hence $0$ is a regular value since $p$ was arbitrary.  We then have that $\mu^{-1}(0)$ is a smooth manifold of codimension $\dim(G)$.
\item The action of $G$ on $\mu^{-1}(0)$ is smooth, free, and proper by assumption, hence $\pi:\mu^{-1}(0)\to \mu^{-1}(0)/G$ inherits the structure of a principal $G$ bundle.  One could demonstrate this explicitly using the slice theorem: a neighborhood of an orbit in $\mu^{-1}(0)$ looks like a neighborhood of the zero section of $G\times V$, where $V$ is the differential slice.
\item The form $\omega\big\vert_{\mu^{-1}(0)}$ is basic since it is invariant under the action of $G$ and, for any $X\in \fg$ and $Y\in T_p\mu^{-1}(0)$, we have $(i_{X_M}\omega)(Y)= -d\langle \mu,X\rangle(Y)=\langle  -d\mu(Y),X\rangle =0$ since $Y$ is tangent to a the $0$ level set.  Thus, there exists $\omega_0$ such that $\pi^*\omega_0= \omega\big\vert_{\mu^{-1}(0)}$.
\item $\pi$ is a submersion, so $\pi^*$ is injective.  Since $0=d(\omega\big\vert_{\mu^{-1}(0)})=d\pi^*\omega_0 = \pi^*d\omega_0$ and $\pi^*$ is injective, we must have $d\omega_0=0$.  Proving that $\omega_0$ is non-degenerate requires one to show that, for each $p\in \mu^{-1}(0)$, $(T_p\mu^{-1}(0))^{\omega}=T_p(G\cdot p)$.
\end{itemize}
\end{proof}

\subsection{The Technique}
We will prove the following analog of theorem \ref{thm:sred}.

\begin{theorem}\label{thm:fsred}
Let $(M,\sigma,\mu:M\to \fg^*)$ be a folded-symplectic manifold without boundary with an Hamiltonian action of a compact, connected Lie group $G$ and moment map $\mu$.  If
\begin{enumerate}
\item $\mu^{-1}(0)$ is a manifold of codimension $\dim(G)$ and
\item $G$ acts on $\mu^{-1}(0)$ freely,
\end{enumerate}
then $\pi:\mu^{-1}(0) \to \mu^{-1}(0)/G$ is a principal $G$ bundle, there exists $\sigma_0 \in \Omega^2(\mu^{-1}(0)/G)$ such that $\pi^*\sigma_0=\sigma\big\vert_{\mu^{-1}(0)}$, and $\sigma_0$ is folded-symplectic.
\end{theorem}

\begin{definition}\label{def:fsred}
Let $(M,\sigma,\mu:M\to \fg^*)$ be a folded-symplectic manifold without boundary with an Hamiltonian action of a compact, connected Lie group $G$ and moment map $\mu$.  If
\begin{enumerate}
\item $\mu^{-1}(0)$ is a manifold of codimension $\dim(G)$ and
\item $G$ acts on $\mu^{-1}(0)$ freely,
\end{enumerate}
then theorem \ref{thm:fsred} implies that $\mu^{-1}(0)/G$ is a folded-symplectic manifold with fold form $\sigma_0$.  We define $M_{red}:=\mu^{-1}(0)/G$ to be the \emph{reduced space at 0}.  We define $\sigma_{red}:=\sigma_0$ to be the \emph{reduced form} on $M_{red}$.
\end{definition}

\begin{remark}
We will show via examples that one \emph{must} assume $\mu^{-1}(0)$ is a manifold, $\mu^{-1}(0)$ has codimension $\dim(G)$, and $G$ acts on $\mu^{-1}(0)$ freely in order to guarantee that $M_{red}$ is a folded-symplectic manifold.  Removing any one of these assumptions on $\mu^{-1}(0)$ allows one to construct examples where the reduced space either fails to be a manifold or fails to be a folded-symplectic space.
\end{remark}

Note that theorem \ref{thm:fsred} does not imply that $(Z\cap\mu^{-1}(0))/G$ is a folding hypersurface for $\sigma_0$.  It's possible that $\sigma_0$ could be symplectic, in which case the folding hypersurface is empty.  To rectify this deficit, we have a structural theorem that allows one to definitively state whether or not $\sigma_0$ has singularities by studying the intersection of $\mu^{-1}(0)$ with $Z$.

\begin{theorem}\label{thm:fsred1}
Let $(M,\sigma,\mu:M\to \fg^*)$ be a folded-symplectic manifold without boundary with an Hamiltonian action of a compact, connected Lie group $G$ and moment map $\mu$.  Let $Z$ be the folding hypersurface.  Suppose
\begin{enumerate}
\item $\mu^{-1}(0)$ is a manifold of codimension $\dim(G)$
\item $G$ acts freely on $\mu^{-1}(0)$.
\end{enumerate}
Then we may form the reduced space $(M_{red},\sigma_{red})$ by theorem \ref{thm:fsred}, where $M_{red}=\mu^{-1}(0)/G$ and $\sigma_{red}$ is folded-symplectic.  Let $\chi$ be a connected component of $\mu^{-1}(0)$.  Then either $\chi \subset Z$ or $\chi \pitchfork Z$.

\begin{enumerate}
\item If $\chi\subset Z$ then $(\chi/G, \sigma_{red})$ is a symplectic manifold.
\item If $\chi\pitchfork Z$ then $(\chi/G,\sigma_{red})$ is folded-symplectic with folding hypersurface $(\chi\cap Z)/G$.
\end{enumerate}

\end{theorem}

Our proof of theorem \ref{thm:fsred} requires a lemma.

\begin{lemma}\label{lem:pplfold}
Let $G$ be a Lie group.  Suppose $M_1$ and $M_2$ are manifolds with corners satisfying $\dim(M_1)=\dim(M_2)$.  Let $\pi_1:P_1\to M_1$ and $\pi_2:P_2\to M_2$ be two principal $G$ bundles and suppose $\psi:P_1 \to P_2$ is an equivariant map with fold singularities.  Then $\psi$ descends to a smooth map $\bar{\psi}:M_1 \to M_2$ with fold singularities.
\end{lemma}

\begin{proof}[Proof of theorem \ref{thm:fsred}]
First, let us note that the content of the theorem is that $\sigma_0$ is folded-symplectic: the proof that $\pi:\mu^{-1}(0)\to \mu^{-1}(0)/G$ is a principal $G$ bundle and that $\sigma\big\vert_{\mu^{-1}(0)}=\pi^*\sigma_0$ is the same as the proof given for theorem \ref{thm:sred}.  Thus, throughout our proof of theorem \ref{thm:sred}, we assume that $\sigma\big\vert_{\mu^{-1}(0)}=\pi^*\sigma_0$ and devote our attention to showing that $\sigma_0$ is folded-symplectic.

We first prove the theorem in the case where the fold, $Z\subset M$, is co-orientable and then use this result to study the non-coorientable case.  The normal form proposition \ref{prop:eqfsnormal} implies that for each point $p\in \mu^{-1}(0)$ there exists an invariant neighborhood $U$, an invariant symplectic form $\omega\in \Omega^2(U)$, and an equivariant fold map $\psi:U \to U$ so that $\psi^*\omega= \sigma$.  Furthermore, the action of $G$ is Hamiltonian for $\omega$ with symplectic moment map $\mu_s:U\to \fg^*$ and $\mu=\mu_s\circ\psi$.  Thus, $\mu^{-1}(0)\cap U = \psi^{-1}(\mu_s^{-1}(0))$ and we have a map:

\begin{displaymath}
\psi:\mu^{-1}(0) \to \mu_s^{-1}(0)
\end{displaymath}

By assumption, $G$ acts on $\mu^{-1}(0)$ freely, hence the stabilizer of $p$ is trivial and we may assume that $G$ acts on $U$ freely using the slice theorem to construct such a neighborhood.  Thus, $G$ acts on $\mu_s^{-1}(0)$ freely, meaning $\mu_s^{-1}(0)$ is a smooth manifold of codimension $\dim(G)$ since $\mu_s$ is a symplectic moment map.  We now have two principal bundles $\pi:\mu^{-1}(0) \to \mu^{-1}(0)/G$ and $\pi_1:\mu_s^{-1}(0) \to \mu_s^{-1}(0)/G$ and an equivariant map $\psi:\mu^{-1}(0) \to \mu_s^{-1}(0)$, giving us a commutative diagram:

\begin{displaymath}
\xymatrixcolsep{5pc}\xymatrix{
(\mu^{-1}(0),\sigma) \ar[r]^{\psi} \ar[d]^{\pi} & (\mu_s^{-1}(0),\omega) \ar[d]^{\pi_1} \\
(\mu^{-1}(0)/G,\sigma_0) \ar[r]^{\bar{\psi}} \ar[r]^{\bar{\psi}} & (\mu_s^{-1}(0)/G,\omega_0)
}
\end{displaymath}
where $\pi_1^*\omega_0=\omega\big\vert_{\mu^{-1}(0)}$.  By the symplectic reduction theorem (q.v. theorem \ref{thm:sred}), $\omega_0$ is symplectic.  Since the dimensions of the level sets agree, $\dim(\mu^{-1}(0))=\dim(U)-\dim(G)=\dim(\mu_s^{-1}(0))$, and $\psi:U\to U$ is a map with fold singularities, corollary \ref{cor:folds4-1} implies that $\psi\big\vert_{\mu^{-1}(0)}$ is a map with fold singularities and lemma implies $\bar{\psi}$ is a map with fold singularities.  Thus, $\bar{\psi}^*\omega_0$ is a folded-symplectic form on $\mu^{-1}(0)/G$.  It remains to show that $\bar{\psi}^*\omega_0 = \sigma_0$.  The commutativity of the diagram implies:

\begin{displaymath}
\sigma=\psi^*\omega = \psi^*\pi_1^*\omega_0 = \pi^*(\bar{\psi}^*\omega_0
\end{displaymath}
Since $\pi^*\omega_0=\omega$ and $\pi^*$ is injective, we have that $\pi^*(\bar{\psi}^*\omega_0)=\pi^*\sigma_0$ if and only if $\bar{\psi}^*\omega_0=\sigma_0$.  Thus, $\sigma_0$ is folded-symplectic.

Now, if $(M,\sigma)$ is not orientable, we may consider its orientable double cover $\tilde{M}$ given by its orientation covering.  Recall that this space is the set of pairs $(m,\eta)$, where $m\in M$ and $\eta$ is an orientation at $m$.  We have a covering map $p:\tilde{M} \to M$ whose fiber is $\mathbb{Z}/(2\mathbb{Z})$.  Since it is a local diffeomorphism, $p^*\sigma$ is a folded-symplectic form on $\tilde{M}$ with folding hypersurface $\tilde{Z}=p^{-1}(Z)$.  Since all folding hypersurfaces are orientable by proposition \ref{prop:orientation}, we have that $\tilde{Z}$ is an orientable hypersurface in an oriented manifold, hence it is co-orientable.

If we have an Hamiltonian action of $G$ on $(M,\sigma)$ with moment map $\mu:M \to \fg^*$, then there is a canonical lift of the action of $G$ to $\tilde{M}$: if $\tau_g:M \to M$ is the action of an element $g\in G$, $\tilde{\tau}_g:\tilde{M} \to \tilde{M}$ is given pointwise by $\tilde{\tau}_g(m,\eta)=(\tau_g(m),d\tau_g(\eta))$.  This action makes the projection map $p:\tilde{M}\to M$ equivariant: $\tau_g\circ p = p \circ \tilde{\tau}_g$.  Thus, $\tilde{\tau}_g^*p^*\sigma = p^*\tau_g^*\sigma=p^*\sigma$ and the action preserves the fold-form.  If $X_{\tilde{M}}$ is the vector field on $\tilde{M}$ induced by $X\in \fg$, then $i_{X_{\tilde{M}}}(p^*\sigma)=-d\langle p^*\mu,X\rangle$, hence the action of $G$ on $\tilde{M}$ is Hamiltonian with moment map $\tilde{\mu}:=\mu\circ p:\tilde{M}\to \fg^*$ and $\tilde{\mu}^{-1}(0)=p^{-1}(\mu^{-1}(0))$.  We therefore have a commutative diagram:

\begin{equation}\label{diagram:fsred1}
\xymatrixcolsep{5pc}\xymatrix{
(\tilde{\mu}^{-1}(0)),p^*\sigma) \ar[r]^{\tilde{\pi}} \ar[d]^p & (\tilde{\mu}^{-1}(0)/G,\tilde{\sigma}_0) \ar[d]^{\bar{p}} \\
(\mu^{-1}(0),\sigma) \ar[r]^{\pi}                           &  (\mu^{-1}(0)/G,\sigma_0)
}
\end{equation}

where $\tilde{\pi}^*\tilde{\sigma}_0=p^*\sigma$.  The action of $G$ on $\mu^{-1}(0)$ is smooth, free, and proper by assumption, hence the lifted action on $\tilde{\mu}^{-1}$ is smooth, free, and proper since $p:\tilde{\mu}^{-1}(0)\to \mu^{-1}(0)$ is a double cover.  The top row is therefore a principal $G$ bundle and the quotient $\tilde{\mu}^{-1}(0)/G$ is a smooth manifold.  By our study of the co-orientable case, $\tilde{\sigma}_0$ is folded-symplectic.  Commutativity of the diagram means that $\pi\circ p = \bar{p}\circ \tilde{\pi}$, meaning

\begin{displaymath}
\tilde{\pi}^*\bar{p}^*\sigma_0=p^*\pi^*\sigma_0 = p^*\sigma= \tilde{\pi}^*\tilde{\sigma}_0.
\end{displaymath}
Since $\tilde{\pi}$ is a submersion, $\tilde{\pi}^*$ is injective and we have $\tilde{\sigma}_0 = \bar{p}^*\sigma_0$.  Since $\pi\circ p= \bar{p}\circ\tilde{\pi}$ is a subermsion, $\bar{p}$ must be a submersion.  Since $\tilde{\mu}^{-1}(0)/G$ and $\mu^{-1}(0)$ have the same dimensions, $\bar{p}$ is a local diffeomorphism, meaning $\sigma_0$ is folded-symplectic if and only if $\bar{p}^*\sigma_0=\tilde{\sigma}_0$ is folded-symplectic. Since $\tilde{\sigma}_0$ is folded-symplectic, $\sigma_0$ must be folded-symplectic.
\end{proof}

\begin{proof}[Proof of theorem \ref{thm:fsred1}]
Again, we are assuming that the compact, connected group $G$ acts on $\mu^{-1}(0)$ freely and that $\mu^{-1}(0)$ is a smooth manifold of codimension $\dim(G)$.  We will also assume that the folding hypersurface $Z\subset M$ is co-orientable and study the non-coorientable case at the end.  We will first prove that if $\chi$ is a connected component of $\mu^{-1}(0)$, then $\chi$ is transverse to the fold, $Z$, or $\chi$ is contained in $Z$.  We will then study each case separately and argue that if $\chi \subset Z$, then $\chi/G$ is symplectic and if $\chi\pitchfork_s Z$, then $\chi/G$ is folded-symplectic with folding hypersurface $Z\cap \chi)/G$.

\begin{itemize}
\item Let us define two sets:

\begin{displaymath}
\begin{array}{ll}
\mathcal{A}_1:= & \{p\in \mu^{-1}(0) \vert \mbox{ } \text{$p$ is not in $Z$ or } T_p\mu^{-1}(0)\pitchfork T_pZ\} \\
\mathcal{A}_2:= & \{p\in \mu^{-1}(0) \cap Z \vert \mbox{ } T_p\mu^{-1}(0) \subset T_p Z\}
\end{array}
\end{displaymath}
where $\mathcal{A}_1$ is the set of all points where $\mu^{-1}(0)$ is transverse to the fold and $\mathcal{A}_2$ is the set of all points where $\mu^{-1}(0)$ intersects the fold tangentially.  Note that these are the only types of intersections $\mu^{-1}(0)$ can have with $Z$ since $Z$ is a hypersurface, hence $\mathcal{A}_1 \cup \mathcal{A}_2 = \mu^{-1}(0)$.  We also have $\mathcal{A}_1\cap \mathcal{A}_2=\emptyset$ by definition.  Now, $\mathcal{A}_1$ is an open set in $\mu^{-1}(0)$ since transversal intersections are preserved under perturbations.  Thus, if we can show $\mathcal{A}_2$ is also open, then $\mathcal{A}_1 \cup \mathcal{A}_2$ is an open cover of $\mu^{-1}(0)$ by disjoint open sets.  Hence, any connected component $\chi$ of $\mu^{-1}(0)$ must lie in $\mathcal{A}_1$ or $\mathcal{A}_2$.  If $\chi\subset \mathcal{A}_1$ then it intersects $Z$ transversally by definition.  If $\chi\subset \mathcal{A}_2$ then it is a subset of $Z$ by definition.

We begin by noting that if the fiber of the null bundle $\ker(\sigma)\cap TZ$ at $z\in \mu^{-1}(0)$ is not tangent to the orbit $G\cdot z$ at $z$, then $z$ is a regular point.  Indeed, if $X \in \fg$ and we assume $\langle d\mu_z, X\rangle=0$, then the induced vector field $X_M$ satisfies $(i_{X_M}\sigma)_z=0$, meaning $X_M \in (\ker(\sigma_z)\cap T_zZ) \cap T_z(G\cdot z) = \{0\}$.  Since the action is free at $z$, $X_M = 0 \iff X=0$.  Thus, the annihilator of the image of $d\mu_z$, $\operatorname{Im}(d\mu_z)^o$, is $\{0\}\subset \fg$ and $z$ is a regular point.  Thus, $T_z\mu^{-1}(0) =\ker(d\mu_z)$, meaning $\ker(\sigma_z)\subset \ker(d\mu_z)$ is tangent to $\mu^{-1}(0)$ at $z$.  Since $\ker(\sigma_z) \pitchfork T_zZ$, we have that $\mu^{-1}(0)$ intersects $Z$ transversally at $z$.  Thus,

\begin{property}\label{property}
If the fiber of $\ker(\sigma)\cap TZ$ at $z\in\mu^{-1}(0)$ is not tangent to the orbit $G\cdot z$ at $z$, then $\mu^{-1}(0)$ intersects $Z$ transversally at $z$, hence $z\in \mathcal{A}_1$.
\end{property}
Thus, if $z\in \mathcal{A}_2$ then $T_z\mu^{-1}(0)\subset T_zZ$ and the contrapositive of property \ref{property} shows that $\ker(\sigma_z)\cap T_zZ$ is a subspace of $T_z(G\cdot z)$.  Let us assume we have a point $z\in \mathcal{A}_2\subset \mu^{-1}(0)$.  By the equivariant normal form proposition (q.v. proposition \ref{prop:eqfsnormal}), we may assume that there is an invariant neighborhood $U\subset M$ of $z$, an equivariant fold map $\psi:U \to U$ with fold $Z\cap U$ and $\psi\vert_Z=id_Z$, and a symplectic structure $\omega\in \Omega^2(U)$ for which the action of $G$ is Hamiltonian satisfying $\psi^*\omega=\sigma$.  We let $\mu_s$ denote the symplectic moment map and note that $\mu_s\circ \psi = \mu$.  We may assume that $G$ acts on $U$ freely since the action of $G$ on $\mu^{-1}(0)$ is free by assumption.  We then have that $0$ is a regular value of the symplectic moment map, $\mu_s$, and $\mu_s^{-1}(0)$ is a submanifold of codimension $\dim(G)$ inside $U$.  Its tangent space is defined by $\ker(d\mu_s)$.  This construction gives us a commutative diagram:

\begin{equation}\label{diagram:fsred2}
\xymatrixcolsep{5pc}\xymatrix{
(\mu^{-1}(0),\sigma) \ar[r]^{\psi} \ar[d]^{\pi} & (\mu_s^{-1}(0),\omega) \ar[d]^{\pi_1} \\
(\mu^{-1}(0)/G,\sigma_0) \ar[r]^{\bar{\psi}} \ar[r]^{\bar{\psi}} & (\mu_s^{-1}(0)/G,\omega_0)
}
\end{equation}
where $\bar{\psi}$ is a map with fold singularities, $\omega_0$ is symplectic (by theorem \ref{thm:sred}), and $\sigma_0=\bar{\psi}^*\omega_0$ is folded-symplectic.  Note that proposition \ref{prop:eqfsnormal} also gives us an involution $i:U\to U$ satisfying $i\vert_Z=id_Z$ and $\psi\circ i = i$.  Since $\mu^{-1}(0)=\psi^{-1}(\mu_s^{-1}(0))= i^{-1}(\psi^{-1}(\mu_s^{-1}(0)))$, we see that the involution $i$ restricts to an involution:

\begin{equation}\label{eq:involution}
i:\mu^{-1}(0) \to \mu^{-1}(0), \mbox{ } \psi\circ i = \psi
\end{equation}
which we will use below.

Now, since $\ker(\sigma_z)\cap T_zZ\subset T_z(G\cdot z)$, there exists $X\in \fg$ such that $X_M(z)\in \ker(\sigma_z)$.  Since $z$ is a regular point of $\mu_s$, the annihilator of the image of $(d\mu_s)_z$ is $\{0\}$ and we must have $\langle (d\mu_s)_z, X \rangle \ne 0$.  The restriction of the fold map $\psi$ to $Z$ is the identity, hence $\mu_s\vert_Z=(\mu_s\circ \psi \vert_Z)=\mu\vert_Z$.  We then have:
\begin{equation}\label{eq:fsred1}
\langle (d\mu_s)_z, X \rangle \vert_{TZ} = \langle d\mu_z, X \rangle \vert_{TZ} =0
\end{equation}
where the rightmost equality follows since $X_M(z)\in \ker(\sigma)$, hence $0=(i_{X_M}\sigma)_z = -\langle d\mu_z, X \rangle$.  Since $\langle d\mu_s, X \rangle \ne 0$ and $\langle d\mu_s, X \rangle \big\vert_{TZ}=0$, we must have that $\langle (d\mu_s)_z(Y),X \rangle \ne 0$ for \emph{any} nonzero $Y\in T_zM$ not in $T_zZ$.  That is, $(d\mu_s)_z$ does not vanish on any direction transverse to $T_zZ$.  Thus, $\ker(d\mu_s)_z=T_z\mu_s^{-1}(0)$ must be contained in $T_zZ$ and so $\mu_s^{-1}(0)$ intersects $Z$ tangentially at $z$.

By assumption, $T_z\mu^{-1}(0)\subset T_zZ$ and we have just shown $T_z\mu_s^{-1}(0)\subset T_zZ$.  The map $\psi:U \to U$ restricts to the identity on the fold $Z$, hence $d\psi_z=id_{T_zZ}:T_z\mu^{-1}(0) \to T_z\mu_s^{-1}(0)$ is injective.  Since $\mu^{-1}(0)$ and $\mu_s^{-1}(0)$ have the same codimension in $U$, they have the same dimension and the differential $d\psi_z$ is an isomorphism.  Thus, there is a connected neighborhood $V_1\subset \mu^{-1}(0)$ of $z$ and a connected neighborhood $V_2\subset \mu_s^{-1}(0)$ of $\psi(z)=z$ so that $\psi:V_1 \to V_2$ is a diffeomorphism.  We claim that this is only possible if $V_1\subset Z$.  We use the involution $i:\mu^{-1}(0) \to \mu^{-1}(0)$ of equation \ref{eq:involution}, whose fixed point set is $\mu^{-1}(0)\cap Z$.  For each $p\in V_1\cap (Z\cap \mu^{-1}(0))$, there exists an involutive neighborhood $U_p\subset \mu^{-1}(0)$ of $p$.  That is, $i(U_p)=U_p$.  If there exists $x\in U_p\setminus Z$, then $i(x)\ne x$ since $x$ is not in the fixed-point set of of $i$.  By definition of $i$, we have $\psi\circ i =\psi$, meaning $\psi(i(x))=\psi(x)$.  Since $i(x)$ and $x$ are distinct points in $U_p\subset V_1$ and $\psi\vert_{V_1}$ is injective by definition of $V_1$, we have found a contradiction.  Thus, $U_p\subset (\mu^{-1}(0)\cap Z)$ and each $p\in V_1\cap (\mu^{-1}(0)\cap Z)$ has a neighborhood $U_p\subset V_1$ contained entirely in $\mu^{-1}(0)\cap Z$.  This means that the set of points in $V_1$ fixed by $i$ is closed \emph{and} open.  Since $V_1$ is connected, we have that all of $V_1$ is fixed and so $V_1\subset \mu^{-1}(0)\cap Z$.  Note that at each point $p\in V_1$ we have $T_p\mu^{-1}(0)=T_pV_1 \subset T_pZ$, hence $p\in \mathcal{A}_2$.  We have therefore shown that $V_1$ is a neighborhood of $z$ contained entirely within the set $\mathcal{A}_2$ which means $\mathcal{A}_2$ is open.

Since $\mathcal{A}_1$ and $\mathcal{A}_2$ are two disjoint open sets covering $\mu^{-1}(0)$, any connected component $\chi$ of $\mu^{-1}(0)$ must be contained in either $\mathcal{A}_1$ or $\mathcal{A}_2$.  If $\chi\subset \mathcal{A}_1$, then $\chi\pitchfork_s Z$ by definition of $\mathcal{A}_1$.  If $\chi\subset\mathcal{A}_2$, then $\chi\subset Z$, which proves the first claim of theorem \ref{thm:fsred1}.

\item Now, assume that $\chi$ is a connected component of $\mu^{-1}(0)$ and $\chi\subset Z$.  Diagram \ref{diagram:fsred2}

    \begin{displaymath}
    \xymatrixcolsep{5pc}\xymatrix{
    (\mu^{-1}(0),\sigma) \ar[r]^{\psi} \ar[d]^{\pi} & (\mu_s^{-1}(0),\omega) \ar[d]^{\pi_1} \\
    (\mu^{-1}(0)/G,\sigma_0) \ar[r]^{\bar{\psi}}    & (\mu_s^{-1}(0)/G,\omega_0)
    }
    \end{displaymath}
    becomes

    \begin{displaymath}
    \xymatrixcolsep{5pc}\xymatrix{
     (\chi,\sigma) \ar[r]^{id_Z} \ar[d]^{\pi} & (\chi,\omega) \ar[d]^{\pi_1} \\
    (\chi/G,\sigma_0) \ar[r]^{id}    & (\chi/G,\omega_0)
    }
    \end{displaymath}
    and $\sigma_0=\omega_0$ is symplectic.

\item Now suppose $\chi\subset \mu^{-1}(0)$ is a connected component transverse to $Z$.  We will show that at each point $z\in \chi\cap Z$, $\ker(\sigma_z)\subset T_z\chi$.  Assume for a moment that this is true.  If $\pi:\mu^{-1}(0) \to \mu^{-1}(0)/G$ is the quotient map, then $\ker(\sigma_z)\not\subset \ker(d\pi_z)$ since the kernel of $d\pi_z$ is the tangent space $T_z(G\cdot z)\subset T_zZ$ while $\ker(\sigma_z)$ contains a subspace transverse to $T_zZ$.  Therefore, there is nonzero element $v\in \ker(\sigma_z)\subset T_z\chi$ such that $d\pi_z(v)\ne 0$.  We then have $\sigma_0(d\pi_z(v),\cdot)= 0$ since $d\pi_z$ is surjective and $\sigma_0(d\pi_z(v),d\pi(c\dot))=i_v(\pi^*\sigma_0)_z=i_v\sigma_z=0$, which means $\sigma_0$ has a degeneracy at $\pi(z)$.  By theorem \ref{thm:fsred}, this degeneracy is a fold singularity.

    It remains to show that if $z\in\mu^{-1}(0)\cap Z$, then $\ker(\sigma_z)\subset T_z\mu^{-1}(0)$.  To this end, we use the setting of diagram \ref{diagram:fsred2} where we have a an invariant neighborhood $U$ of $z$ on which $G$ acts freely, a symplectic structure $\omega\in \Omega^2(U)$ for which the action of $G$ is Hamiltonian with moment map $\mu_s$, and an equivariant fold map $\psi:U \to U$ so that $\psi^*\omega=\sigma$, hence $\mu_s\circ \psi = \mu$.

    Recall from our derivation of property \ref{property} that if $\ker(\sigma_z)\cap T_zZ$ is \emph{not} a subspace of $T_z(G\cdot z)$, then $z$ is a regular point of $\mu$ and $T_z\mu^{-1}(0)=\ker(d\mu_z)$.  Since $\ker(\sigma_z)\subset \ker(d\mu_z)$, we have $\ker(\sigma_z)\subset T_z\mu^{-1}(0)$ as required.  We may therefore assume that the null space $\ker(\sigma_z)\cap T_zZ \subset T_z(G\cdot z)$ is tangent to the orbit through $z$.  As we have seen, this tangency implies that $T_z\mu_s^{-1}(0)\subset T_zZ$.  We also have that the map:

    \begin{equation}\label{eq:fsred3}
    d\psi_z: T_z\mu^{-1}(0)\cap T_zZ \rightarrow T_z\mu_s^{-1}(0)
    \end{equation}
    is injective.  We are assuming $\mu^{-1}(0)$ intersects $Z$ transversally, hence $\dim(\mu^{-1}(0)\cap Z)=\dim(U)-\dim(G)-1$.  Since $\dim(\mu_s^{-1}(0))=\dim(U)-\dim(G)$, we see that the image of $d\psi_z$ in equation \ref{eq:fsred3} has codimension $1$ by injectivity.  Choose a nonzero vector $v\in T_z\mu^{-1}(0)$ not in $T_zZ$, which is possible since we are assuming a transversal intersection.  If $d\psi(v)=0$ then $v\in \ker(d\psi_z)$ and we are done since, by diagram \ref{diagram:fsred2}, $\ker(\sigma_z) = \ker(d\psi_z) +(\ker(\sigma_z)\cap T_zZ)$, hence $\ker(\sigma_z)\subset T_z\mu^{-1}(0)$.  If $d\psi_z(v)\ne0$, then we must have that $d\psi(v) \in d\psi_z(T_z\mu^{-1}(0)\cap T_zZ)$.  Otherwise, $d\psi_z$ in equation \ref{eq:fsred3} would be injective and surjective, hence it would map $T_z\mu^{-1}(0)$ isomorphically onto $T_z\mu_s^{-1}(0)$.  This means $\psi:\mu^{-1}(0)\to \mu_s^{-1}(0)$ would be a diffeomorphism in a neighborhood of $z$.  We have seen that this implies a neighborhood of $z$ in $\mu^{-1}(0)$ is contained in $Z$, contradicting our assumption that $\mu^{-1}(0)$ intersects $Z$ transversally at $z$.  Thus, there is a nonzero vector $w\in T_z\mu^{-1}(0) \cap T_zZ$ such that $d\psi_z(v)=d\psi_z(w)$, which means $d\psi_z(v-w)=0$.  We then have that $v-w\in \ker(d\psi_z)\cap T_z\mu^{-1}(0)$.  It cannot be zero since $w\subset T_zZ$ and $v$ is not in $T_zZ$, implying that $\ker(\sigma_z)\subset \mu^{-1}(0)$.  As initially discussed, this fact is enough to guarantee that the reduced form $\sigma_0$ has a fold singularity.

\item Finally, if $Z$ is not co-orientable then $M$ cannot be orientable.  As in the proof of theorem \ref{thm:fsred}, we lift the action of $G$ to the orientation covering $p:\tilde{M} \to M$ and pull back the fold form $\sigma$ to the fold form $p^*\sigma$, which makes $\tilde{M}$ into an Hamiltonian $G$-manifold $(\tilde{M},p^*\sigma, \tilde{\mu}=\mu\circ p)$.  The folding hypersurface of $p^*\sigma$ is the inverse image of the folding hypersurface in $M$, $p^{-1}(Z)$.  Let $\chi \subset \mu^{-1}(0)$ be a connected component of $\mu^{-1}(0)$.  Since $p:\tilde{M} \to M$ is a covering map, it is a local diffeomorphism and:
    \begin{enumerate}
    \item $\chi \pitchfork Z$ if and only if $p^{-1}(\chi) \pitchfork p^{-1}(Z)$
    \item $\chi \subset Z$ if and only if $p^{-1}(\chi)\subset p^{-1}(Z)$.
    \end{enumerate}
    Since $p^{-1}(\chi)$ is a collection of connected components of $\tilde{\mu}^{-1}(0)$, our previous discussion implies that these connected components are either transverse to $p^{-1}(Z)$ or contained in $p^{-1}(Z)$.  Thus, $\chi\pitchfork_s Z$ or $\chi\subset Z$.  To finish, recall that in the proof of theorem \ref{thm:fsred} we constructed a commutative diagram:

    \begin{displaymath}
    \xymatrixcolsep{5pc}\xymatrix{
    (\tilde{\mu}^{-1}(0),p^*\sigma) \ar[r]^{\tilde{\pi}} \ar[d]^p & (\tilde{\mu}^{-1}(0)/G,\tilde{\sigma}_0) \ar[d]^{\bar{p}} \\
    (\mu^{-1}(0),\sigma) \ar[r]^{\pi}                           &  (\mu^{-1}(0)/G,\sigma_0)
    }
    \end{displaymath}
    where $(\tilde{\mu}^{-1}(0)/G,\tilde{\sigma}_0)$ is a folded-symplectic manifold, $\bar{p}$ is a local diffeomorphism, and $\bar{p}^*\sigma_0=\tilde{\sigma}_0$.  If $\chi\subset \mu^{-1}(0)$ is a connected component and $\chi\subset Z$, then $p^{-1}(\chi)\subset p^{-1}(Z)$ and $(p^{-1}(\chi)/G,\tilde{\sigma}_0)$ is symplectic.  Since $\bar{p}$ is a local diffeomorphism, $\sigma_0$ must be symplectic.  Now suppose $\chi\pitchfork Z$.  Then $p^{-1}(\chi)\pitchfork Z$ and $(p^{-1}(\chi)/G,\tilde{\sigma}_0)$ is folded-symplectic with folding hypersurface $p^{-1}(\chi\cap Z)/G$, which means $\sigma_0$ is folded-symplectic with folding hypersurface $\bar{p}(p^{-1}(\chi\cap Z)/G) = (\chi\cap Z)/G$.
\end{itemize}
\end{proof}

\subsection{Removing Assumptions on $\mu^{-1}(0)$}
We take a moment to discuss the differences between the symplectic and folded-symplectic reduction hypotheses.  In the symplectic setting, a free action of $G$ on $\mu^{-1}(0)$ is enough to guarantee that $0$ is a regular value, hence $\mu^{-1}(0)$ is a manifold of codimension $\dim(G)$.  In the folded-symplectic setting, a free action alone is not enough to guarantee that $\mu^{-1}(0)$ is a smooth manifold.

\begin{example}
Consider $M=(\R^2\setminus\{0\}) \times \R^2$ with coordinates $(r,\theta, x,y)$ and fold form $\sigma = (r-1)dy\wedge d\theta + ydr\wedge d\theta + dx\wedge dy + dx\wedge dr$.  To see that $\sigma$ is folded, we consider the top power $\omega^2 = (r-1-y)dy\wedge d\theta \wedge dx \wedge dr$, which is transverse to $0$.  The folding hypersurface is defined by $r-1-y=0$ and the kernel of $\sigma$ at the hypersurface $r-1=y$ contains $\frac{\partial}{\partial r}-\frac{\partial}{\partial y}$, which is not in $\ker(d(r-1-y))$, hence it is transverse to $\{r-1-y=0\}$.

Now, $S^1$ acts freely on $M$ by rotations in $\R^2\setminus \{0\}$.   A corresponding moment map is
\begin{displaymath}
\mu(r,\theta, x,y) = y(r-1)
\end{displaymath}
and $\mu^{-1}(0)$ is the union of the hyperplane $y=0$ with the product of the cylinder $r=1$ with $\R^2$.  These two submanifolds intersect transversally at the folding hypersurface when $r=y+1$, hence their union is not a submanifold.  Alternatively, $\mu^{-1}(0)$ is the product of $S^1$ with the union of the hyperplanes $r=1$ and $y=0$ in the $(x,y,r)$ coordinate system.  Since the these hyperplanes intersect transversally, we again see that $\mu^{-1}(0)$ is not a submanifold.
\end{example}
Thus, if we drop the assumption that $\mu^{-1}(0)$ is a manifold and simply require that the action of $G$ is free, $\mu^{-1}(0)$ may only be a topological space.  It therefore makes sense to require that $\mu^{-1}(0)$ is a manifold and that the action of $G$ is free.  However, in this scenario, the reduced space may only be a presymplectic manifold.

\begin{example}
Consider $M=\mathbb{T}^2 \times S^1 \times \R$ with coordinates $(\theta_1,\theta_2,\theta, h)$ and fold form

\begin{displaymath}
\sigma=\sin(\theta_1)d\theta_1 \wedge d\theta_2 + e^hdh\wedge d\theta_2 - dh\wedge d\theta
\end{displaymath}
where the folding hypersurface is $\{\theta_1=0\}$.  $S^1$ acts on $M$ via $\lambda \cdot (\lambda_1, \lambda_2, \lambda_3, h) = (\lambda_1, \lambda \lambda_2, \lambda \lambda_3, h)$.  A corresponding moment map is
\begin{displaymath}
\mu(e^{2\pi i \theta_1}, e^{2\pi i \theta_2}, \lambda_3, h) = \cos(\theta_1) + e^h - h.
\end{displaymath}
Then $\mu^{-1}(0) = \{h=0, \lambda_1 = -1\} \simeq \mathbb{T}^2$.  The quotient by $S^1$ is isomorphic to $\mathbb{T}$, which cannot be folded-symplectic since it is odd dimensional.

\end{example}
We therefore include the assumption that $\mu^{-1}(0)$ has codimension $\dim(G)$.  Lastly, if we assume that $\mu^{-1}(0)$ is a manifold of codimension $\dim(G)$ and drop the assumption that $G$ acts on $\mu^{-1}(0)$ freely, then the reduced space may not be a manifold.

\subsection{Application: Minimal Coupling}
We give a generalization of a construction due to Sternberg \cite{St} which will allow us to combine folded-symplectic structure on a manifold $M$ with a symplectic structure on a symplectic vector bundle $\pi:E \to M$ to produce a folded-symplectic structure in a neighborhood of the zero section of $E$.  We follow the approach found in \cite{LS}, hence we begin with the case of a principal bundle.

\begin{definition}\label{def:mincoupform}
Let $(M,\sigma)$ be a folded-symplectic manifold (without corners), let $G$ be a Lie group.  For the purposes of the discussion, we'll assume $G$ is compact and connected.  Let $\pi:P \to M$ be a principal $G$-bundle, let $\fg$ be the Lie algebra of $G$, and let $A\in \Omega^1(P,\fg)^G$ be a connection $1$-form on $P$.  Let $\langle \cdot, \cdot \rangle$ be the canonical pairing between $\fg$ and $\fg^*$.

Consider the space $P\times \fg^*$ with projection maps $pr_i$ onto the $i^{th}$ factor.  Let $G$ act diagonally via $g\cdot (p,\eta) = (pg^{-1},Ad^*(g)\eta)$.  We define:

\begin{displaymath}
\Omega_A:= \pi^*\sigma - d\langle pr_2, A\rangle
\end{displaymath}
to be the \emph{minimal coupling form} associated to the connection $A$.
\end{definition}

\begin{lemma}\label{lem:mincoupform}
The minimal coupling form $\Omega_A$ is nondegenerate in a neighborhood of the zero section $P\times \{0\}$ of $P\times \fg^*$.  Furthermore, the action is Hamiltonian with moment map given by $-pr_2:P\times \fg^* \to \fg^*$.
\end{lemma}

\begin{proof}\mbox{ } \newline
\begin{itemize}
\item Let $H\to P$ be the horizontal bundle of $TP$ afforded by the connection $A$, let $V\to P$ be the vertical bundle, and let $\dim(M)=2m$.
\item Since $\sigma$ is folded-symplectic, we have $\sigma^m \pitchfork 0$ by definition.  We then have that $((\pi^*\sigma)\big\vert_H)^m \pitchfork 0$, where we view $0$ as the zero section of $\Lambda^{2m}(H^*)$.
\item The $2$-form $d\langle pr_2,A\rangle$ is $\langle dpr_2, A\rangle$ at points of the zero section of $P\times \fg^*$, which is non-degenerate when restricted to $V\oplus T\fg^*$.
\item Thus, at points of the zero section, the top power of $\Omega_A$:

\begin{equation}\label{eq:mincoupform}
(\Omega_A)^{m + \dim(G)} = ((\pi^*\sigma)\big\vert_H)^m \wedge (\langle dpr_2, A \rangle\big\vert_{V\oplus T\fg^*})^{\dim(G)}
\end{equation}
vanishes transversally along the zero section.  Thus, it vanishes transversally in a neighborhood of the zero section.
\item From equation \ref{eq:mincoupform}, the intersection of the degenerate hypersurface of $\Omega_A$ with $P\times \{0\}$ is the degenerate hypersurface of $(\pi^*\sigma)\big\vert_H)^m$ intersected with $P\times \{0\}$, which is $\pi^{-1}(Z)\cap P\times \{0\}$.  The kernel of $\Omega_A$ at $\pi^{-1}(Z)\cap P \times \{0\}$ is the horizontal lift of the kernel of $\sigma$ at $Z$.  Since $\ker(\sigma)\pitchfork Z$, we must have that $\ker(\Omega_A)\pitchfork \pi^{-1}(Z)$ at points of $P\times \{0\}$.  Thus, $\ker(\Omega_A)$ is transverse to the degenerate hypersurface of $\Omega_A$ in a neighborhood of $P\times \{0\}$ and $\Omega_A$ is folded-symplectic in a neighborhood of $P\times \{0\}$.

\item The $G$-invariance of $\langle pr_2, A\rangle$ implies:

\begin{displaymath}
i_{X_P}\Omega_A = -i_{X_p}d\langle pr,A \rangle = di_{X_P}\langle pr_2,A \rangle = d\langle pr_2, X \rangle
\end{displaymath}
hence $\mu = -pr_2$ by the definition of a moment map.
\end{itemize}
\end{proof}

\begin{lemma}\label{lem:mincoupform1}
Let $G$ be a compact, connected Lie group and suppose $(F,\omega)$ is a symplectic manifold with an Hamiltonian action of $G$ and moment map $\mu:F \to \fg^*$.  Let $(M,\sigma)$ be a folded-symplectic manifold and let $\pi:P \to M$ be a principal $G$-bundle over a folded-symplectic manifold.  Then, for each choice of connection $1$-form $A$ on $P$, the associated bundle $P\times_G F$ carries a natural folded-symplectic structure in a neighborhood of $P\times_G \mu^{-1}(0)$.
\end{lemma}

\begin{proof}\mbox{ } \newline
\begin{itemize}
\item We exhibit $P\times_G F$ as a folded-symplectic reduced space, which will endow it with a natural folded-symplectic structure by theorem \ref{thm:fsred}.
\item Consider the product $(P\times \fg^*)\times F$ with the action of $G$ given by $g\cdot (p,\eta,f) = (pg^{-1},Ad^*(g)\eta, g\cdot f)$.  The folded-symplectic structure is given by $\Omega_A \oplus \omega$ where we assume, for the time being, that this structure is folded-symplectic on the whole space and not simply on a neighborhood of $P\times \{0\} \times F$.  Then the action is Hamiltonian with moment map:

    \begin{displaymath}
    J(p,\eta,f)= \mu(f) - \eta
    \end{displaymath}
\item $J$ is a submersion, hence $0$ is a regular value and $J^{-1}(0)$ is a codimension $\dim(G)$ submanifold.  It has a free action of $G$ since the action of $G$ on $P$ is free.  Thus, $(P\times \fg^* \times F)//_0G :=J^{-1}(0)/G$ is a folded-symplectic manifold.  Since $0$ is a regular value, $J^{-1}(0)$ intersects the folding hypersurface transversally, hence the reduced form on $J^{-1}(0)/G$ has a nonempty folding hypersurface.

\item The map $q: J^{-1}(0) \to P\times_G F$ given by $q(p,\eta,f) = [p,f]$ is a surjective submersion and the fibers are the $G$-orbits inside $J^{-1}(0)$.  Thus, $q$ descends to a diffeomorphism $\bar{q}:J^{-1}(0)/G \to P\times_G F$.  Since there is a natural folded-symplectic structure on $J^{-1}(0)/G$, $P\times_G F$ inherits a folded-symplectic structure.

\item Now, lemma \ref{lem:mincoupform} tells us that $\Omega_A$ is only symplectic in a neighborhood $U$ of $P\times \{0\}$.  Thus, the map $q$ restricted to $J^{-1}(0)\cap (U\times F)$ descends to an open embedding of $(J^{-1}(0) \cap (U\times F))/G$ into $P\times_G F$.  The image of this embedding contains a neighborhood of $P\times_G \mu^{-1}(0)$ since $J^{-1}(0)\cap (U\times F)$ contains $P\times \{0\} \times \mu^{-1}(0)$.  This neighborhood inherits a folded-symplectic structure from the reduced space $(J^{-1}(0)\cap (U\times F))/G$.

\end{itemize}
\end{proof}

Now, we apply lemma \ref{lem:mincoupform1} to the case where we have a symplectic vector bundle $E$ over a folded-symplectic base $(M,\sigma)$.

\begin{lemma}\label{lem:mincoupform2}
Let $(M,\sigma)$ be a folded-symplectic manifold and let $\pi:E \to M$ be a symplectic vector bundle over $M$.  Then for each choice of connection $1$-form $A$ on the frame bundle $Fr(E)$ of $E$, there exists a folded-symplectic form defined in a neighborhood $U$ of the zero section of $E$.
\end{lemma}

\begin{remark}
Note that we are not claiming anything about uniqueness of this folded-symplectic structure.  It is not clear that one may deform one minimal coupling form into another, hence it is not clear that we may deform the resultant folded-symplectic structures on $E$ into one another.
\end{remark}

\begin{proof} \mbox{ } \newline
\begin{itemize}
\item If we choose an almost complex structure $J$ on $E$, then we may identify the structure group of $E$ with $U(n)$, which is compact and connected.  The frame bundle $Fr(E)$ of $E$ is then a principal $U(n)$ bundle.
\item The typical fiber of $V$ is a symplectic vector space with a symplectic action of the structure group $U(n)$.  As we have discussed, such actions are Hamiltonian.  The moment map $\mu:V \to \operatorname{Lie(U(n))}^*$ satisfies $0\subset \mu^{-1}(0)$.
\item Now, $E$ is isomorphic to the associated bundle $Fr(E)\times_{U(n)} V$.  A choice of connection $A$ on $Fr(E)$ induces a folded-symplectic structure on a neighborhood of $Fr(E)\times_{U(n)} \mu^{-1}(0)$ by lemma \ref{lem:mincoupform1}, which is a neighborhood of $Fr(E)\times_{U(n)} \{0\}$.  Since this is the zero section of $E$, we have that a choice of connection $A$ on $Fr(E)$ endows a neighborhood of the zero section of $E$ with a folded-symplectic structure.
\end{itemize}
\end{proof}

\begin{remark}
The above construction is somewhat relevant in the case of toric, folded-symplectic manifolds.  Let $(M,\sigma,\mu:\to \fg^*)$ be a toric, folded-symplectic manifold.  We have shown that the orbit-type strata $M_H$ are folded-symplectic submanifolds.  We have also shown that there is a well-defined symplectic normal bundle $(\widetilde{TM_H})^{\sigma}$.  Lemma \ref{lem:mincoupform} provides us with a method of constructing a folded-symplectic structure in a neighborhood of the zero section of $(\widetilde{TM_H})^{\sigma}$.  If one could prove a uniqueness statement, then this construction would give a local model for a neighborhood of an orbit-type stratum.
\end{remark}

\pagebreak


\section{Classifying Toric, Folded-Symplectic Bundles: $\pi_0(\mathcal{B}_{\psi}(W))$}
Classifying toric folded-symplectic bundles up to isomorphism is a straightforward task and we will proceed as follows.  Let $G$ be a torus and let $\psi:W\to \fg^*$ be a unimodular map with folds, where $W$ is a manifold with corners.  We first observe that isomorphism classes of principal $G$ bundles over $W$ are parameterized by $H^2(W,\mathbb{Z}_G)$.  That is, for each principal $G$ bundle over $W$ there exists an element $c_1(P)\in H^2(W,\mathbb{Z}_G)$ called the \emph{first Chern class of} $P$, which specifies $P$ up to isomorphism (of principal $G$ bundles).  Next, we fix an element $c_1(P)\in H^2(W,\mathbb{Z}_G)$ and choose a representative bundle $\pi:P\to W$ from this diffeomorphism class.  In other words, we fix the structure of $P$ as a principal $G$ bundle.  We ask the question: how can we parameterize the folded-symplectic structures on $\pi:P\to W$ for which the action of $G$ is Hamiltonian with moment map $\psi\circ \pi$.  It turns out that the answer lies in the basic cohomology of $P$, $H^2(W,\R)$.  We will show there exists a characteristic class $c_{hor}(P) \in H^2(W,\R)$ that specifies the Hamiltonian, folded-symplectic structure of $P$ up to isomorphism.  Combining the two characteristic classes, we obtain a map:

\begin{displaymath}
c_1(\cdot) \times c_{hor}(\cdot):\pi_0(\mathcal{B}_{\psi}(W)) \to H^2(W,\mathbb{Z}_G\times \R)
\end{displaymath}
where $\mathcal{B}_{\psi}(W)$ is the category of toric, folded-symplectic bundles over $\psi:W\to \fg^*$ of definition \ref{def:bpsi} and $\pi_0(\mathcal{B}_{\psi}(W))$ is the set of isomorphism classes of objects.  We will argue that this map is a bijection, completing the classification.

\subsection{Classifying Principal Torus Bundles}
We begin by discussing the classification of principal $G$ bundles over $W$, where $G$ is a torus.  This task is perhaps not even deserving of its own section, but we would like to discuss it.  Our arguments are very similar to those found in \cite{Mi1} for the classification of vector bundles.  In particular, the example on p. 103 may be useful to the reader.  All we are doing is replacing the structure group with the torus $G$.  Given a torus, we have the following short exact sequence:

\begin{displaymath}
\xymatrix{
0 \ar[r] & \mathbb{Z}_G \ar[r]^{i} & \fg \ar[r]^{\exp} &  G \ar[r] & 0.
}
\end{displaymath}
This induces a short exact sequence of locally constant sheaves on $W$,
\begin{displaymath}
\xymatrix{
0 \ar[r] & \underline{\mathbb{Z}_G} \ar[r]^{i} & \underline{\fg} \ar[r]^{\exp}  &  \underline{G} \ar[r] & 0.
}
\end{displaymath}
We then have a long exact sequence in \v{C}ech cohomology with coefficients in sheaves of abelian groups:
\begin{displaymath}
\xymatrix{
\ldots \ar[r] & \check{\mathrm{H}}^k(W,\underline{\mathbb{Z}_G}) \ar[r] & \check{\mathrm{H}}^k(W,\underline{\fg}) \ar[r]  &  \check{\mathrm{H}}^k(W,\underline{G}) \ar[r] & \ldots.
}
\end{displaymath}
Now, principal $G$-bundles are parameterized by $\check{\mathrm{H}}^1(W,\underline{G})$ since a principal $G$ bundle $\pi:P \to W$ may be specified by a cover $\{U_i\}$ of $W$ and trivializations $\phi_i:P\big\vert_{U_i} \to U_i \times G$.  The trivializations give us transition maps $\phi_i\circ\phi_j^{-1}$, which are equivalent to maps $\phi_{ij}:(U_i\cap U_j) \to G$.  The $\phi_{ij}'s$ satisfy the cocycle condition, hence they give us a $1$-cocycle in \v{C}ech cohomology which specifies the bundle up to isomorphism.  The sheaf $\underline{\fg}$ is a fine sheaf since it admits partitions of unity, hence the \v{C}ech cohomology groups are $0$ in dimensions greater than $0$.  In particular, we have a short exact sequence of cohomology groups:

\begin{displaymath}
\xymatrix{
0 \ar[r] & \check{\mathrm{H}}^1(W,\underline{G}) \ar[r] & \check{\mathrm{H}}^2(W,\underline{\mathbb{Z}_G}) \ar[r] & 0
}
\end{displaymath}
meaning $\check{\mathrm{H}}^1(W,\underline{G}) \simeq \check{\mathrm{H}}^2(W,\underline{\mathbb{Z}_G})$.  On the other hand $\check{\mathrm{H}}^2(W,\underline{\mathbb{Z}_G})$ is isomorphic to the singular cohomology group $H^2(W,\mathbb{Z}_G)$, which finishes the classification of the structure as a principal $G$-bundle.

\begin{definition}\label{def:firstchern}
Let $P$ be a principal $G$ bundle over a manifold with corners $W$.  Let $c_1(P)$ be the unique element in $H^2(W,\mathbb{Z}_G)$ corresponding to the isomorphism class of $P$.  We call $c_1(P)$ the \emph{first Chern class of} $P$.
\end{definition}

\subsection{Classifying the Toric, Folded-Symplectic Structure}
Recall that we are assuming we have a unimodular map with folds $\psi:W \to \fg^*$.  Fix an isomorphism class $c_1(P)\in H^2(W,\mathbb{Z}_G)$ of principal $G$ bundles and fix a representative member $\pi:P \to W$.  Suppose we have two folded-symplectic structures $\sigma_1$,$\sigma_2$ on $P$ for which the action of $G$ is Hamiltonian with moment map $\psi\circ \pi$.  We are going to give a necessary and sufficient condition for them to be isomorphic.  In particular, we will show that their difference $\sigma_1-\sigma_2$ must be an exact, basic $1$-form.  First, we give a recipe for how one can construct such folded-symplectic structures.

\begin{lemma}\label{lem:construct}
Let $\psi:W\to \fg^*$ be a unimodular map with folds, where $\fg$ is the Lie algebra of a torus $G$.  Let $\pi:P \to W$ be a principal $G$-bundle and let $A\in \Omega^1(P,\fg)^G$ be a principal connection $1$-form on $P$.  Then,

\begin{enumerate}
\item The two form $\sigma =d\langle \psi\circ \pi, A \rangle$ is a folded-symplectic form on $P$, where $\langle \cdot , \cdot \rangle:\fg^*\otimes \fg \to \R$ is the standard pairing.
\item Furthermore, the action of $G$ on $(P,\sigma)$ is Hamiltonian with moment map $\mu=\psi \circ \pi$.
\end{enumerate}
\end{lemma}

\begin{proof} \mbox{ } \newline
Let $n=\dim(G)=\dim(W)$, where the dimensions must be the same since $\psi$ is a unimodular map with folds.
\begin{enumerate}
\item $\sigma$ is a closed $2$-form by definition, hence we need only show that $\sigma^n \pitchfork_s 0$ and $\sigma$ restricted to its degenerate hypersurface has maximal rank.  It may be written as:

    \begin{displaymath}
    \sigma = \langle (d\psi \circ d\pi) \wedge A \rangle + \langle \psi \circ \pi, dA \rangle
    \end{displaymath}
    Since $P$ is a principal torus bundle, the curvature form $dA$ is a basic $\fg$-valued $2$-form.  Thus, the top power satisfies:

    \begin{displaymath}
    \sigma^n = (\langle (d\psi \circ d\pi \wedge A \rangle ) ^n
    \end{displaymath}
    If we choose a basis $\{e_1,\dots,e_n\}$ of $\fg$, then we may write:

    \begin{displaymath}
    \begin{array}{l}
    A=\sum_{i=1}^n A_i\otimes e_i, \text{ where $A_i \in \Omega^1(P)^G$.} \\
    \psi=\sum_{i=1}^n \psi_i \otimes e_i^*, \text{ where $\psi_i\in C^{\infty}(W)$ is smooth.}
    \end{array}
    \end{displaymath}
    hence,
    \begin{displaymath}
    \sigma^n = C\pi^*(d\psi_1 \wedge\dots \wedge d\psi_n) \wedge(A_1\wedge \dots \wedge A_n)
    \end{displaymath}
    where $C\in \R$ and $C\ne 0$.  The form $A_1\wedge \dots \wedge A_n$ is non-degenerate on the fibers of $P$ and the form $(d\psi_1 \wedge \dots \wedge d\psi_n)$ vanishes transversally on $W$ by corollary \ref{cor:folds4-1}, hence $\sigma^n \pitchfork_s 0$, where we are using transversality in the sense of manifolds with corners.  From this calculation, we also have that the degenerate hypersurface of $\sigma$ is given by:

    \begin{displaymath}
    Z= \pi^{-1}(\hat{Z})
    \end{displaymath}
    where $\hat{Z}$ is the folding hypersurface of $\psi$.  We still need to show that the restriction of $\sigma$ to the degenerate hypersurface, $i_Z^*\sigma$, has maximal rank.  To this end, we will explicitly describe the kernel of $\sigma$ and see that it is $2$-dimensional and transverse to $Z$, hence its intersection with $TZ$ is $1$-dimensional.

    \begin{itemize}
    \item The first piece of the kernel will be constructed using the annihilator bundle of $\psi$, $\operatorname{Im}(d\psi)^o$ (q.v. definition \ref{def:annihilator}.  If $p\in Z$ is a point in the degenerate hypersurface of $\sigma$, then we take any nonzero element $V\in \operatorname{Im}(d\psi)^o_{\pi(p)}$ in the fiber of the annihilator bundle.  We then have that the induced vector field $V_P$ at $p$ is nonzero and lies in the kernel of $\sigma_p$:

        \begin{displaymath}
        (i_{V_P}\sigma)_p = \langle d\psi\circ d\pi_p \wedge A_p(V_P) \rangle = \langle d\psi \circ d\pi_p, V \rangle = 0
        \end{displaymath}
    \item The second portion of the kernel is slightly trickier to construct.  We first pick a nonzero element $X\in \ker(d\psi_{\pi(p)})$ and choose a horizontal lift $\tilde{X}$ to an element of $T_pP$.  Note that a horizontal lift satisfies $A(\tilde{X})=0$.  We then have that the contraction with $\sigma$ satisfies:

        \begin{displaymath}
        (i_{\tilde{X}}\sigma)_p = \langle d\psi\circ d\pi(\tilde{X}), A\rangle_p + \langle \psi \circ \pi, i_{\tilde{X}}dA\rangle_p = \langle \psi \circ \pi, i_{\tilde{X}}dA \rangle_p
        \end{displaymath}
        hence the contraction isn't necessarily $0$.  We would like to show that there is a vertical vector that we can add to $\tilde{X}$ to kill this extra term.  $dA$ is basic, hence $dA=\pi^*\beta$ for some $\beta\in \Omega^2(W)$ and the contraction at $p$ is $\beta_{\pi(p)}(X, d\pi(\cdot))= \pi^*(i_X\beta)$.  Consider the map:

        \begin{displaymath}
        \xymatrixrowsep{.5pc}\xymatrix{
        F:\fg \ar[r] & T_{\pi(p)}^*W \\
        \eta \ar[r]   & \langle d\psi_{\pi(p)}, \eta\rangle
        }
        \end{displaymath}
        The source and target space have the same dimension and the annihilator of the image of $d\psi$ is one dimensional, hence $F$ maps onto an $n-1$ dimensional subspace.  The image of $F$ is contained in the space of covectors which vanish on $\ker(d\psi_{\pi(p)})$.  Since this space also has dimension $n-1$, we have that $F$ surjects onto the space of covectors which vanish on $\ker(d\psi_{\pi(p)})$.
        Since $X\in \ker(d\psi_{\pi(p)})$, we have that $(i_X\beta)_{\pi(p)}$ is a covector vanishing on $\ker(d\psi_{\pi(p)})$.  Thus, there is an element $\eta \in \fg$ such that $\langle d\psi_{\pi(p)}, \eta \rangle = i_X\beta$.  We now add $\eta_P$ to our horizontal lift $\tilde{X}$ and compute the contraction:

        \begin{displaymath}
        (i_{\tilde{X} +\eta_P}\sigma)_p = \pi^*(i_X\beta)_p - \langle d\psi \circ d\pi ,A(\eta_P) \rangle _p = \pi^*(i_X\beta)_p -\pi^*(i_X\beta) =0
        \end{displaymath}
        hence $\tilde{X} + \eta_P$ is in the kernel of $\sigma$ which is a vector transverse to the hypersurface $Z$.
    \item Thus, the kernel of $\sigma$ at $p$ contains $\span\{\tilde{X} + \eta_P, V_P\}$.  We check that the dimension of $\ker(\sigma_p)$ is no larger than $2$.  The projection of $\ker(\sigma_p)$ via $d\pi_p$ is a surjection onto $\ker(d\psi_p)$ by proposition \ref{prop:umf}, hence any two vectors in $X_1, X_2$ in $\ker(\sigma_z)$ projecting onto $X\in \ker(d\psi_{\pi(p)})$ differ by a vector $V'_P$ tangent to the fibers of $P$.  We then have that:
        \begin{displaymath}
        (i_{V'_P}\sigma)_p = -\langle d\psi \circ d\pi ,A(V'_P)\rangle = \langle d\psi \circ d\pi, V'\rangle
        \end{displaymath}
        hence $V'$ annihilates the image of $d\psi$.  But the annihilator bundle is a line bundle, hence $V'_P$ is parallel to $V_P$ and so the kernel is spanned by $\tilde{X} + \eta_P$ and $V_P$, hence it is $2$-dimensional.
    \end{itemize}
    Finally, the action of $G$ is Hamiltonian with moment map $\psi\circ \pi$ since for all $X\in \fg$ we have:

    \begin{displaymath}
    i_{X_P}\sigma = i_{X_P}d\langle \psi \circ \pi, A\rangle = -d(i_{X_P}\langle \psi \circ \pi, A \rangle) = -d\langle \psi \circ \pi, X \rangle
    \end{displaymath}
    where the second equality follow from the $G$-invariance of the form $\langle \psi\circ \pi, A\rangle$.
\end{enumerate}
\end{proof}

\begin{remark}\label{rem:kernelconstruction}
Given a unimodular map with folds, $\psi:W\to \fg^*$, lemma \ref{lem:construct} gives us a recipe for constructing a folded-symplectic structure $\sigma$ on a principal $G$-bundle $\pi:M \to W$ so that the $G$-action is Hamiltonian, where the moment map is $\psi\circ \pi$.  Incidentally, we also found a recipe for constructing the kernel of $\sigma$ at each point $p$ in the folding hypersurface $Z=\pi^{-1}(\hat{Z})$, where $\hat{Z}$ is the folding hypesurface of $\psi$.  The steps are as follows:

\begin{itemize}
\item Fix a connection $A$ on $P$.  Indeed, lemma \ref{lem:construct} requires us to fix a connection on $P$.
\item At each $p\in Z$ in the folding hypersurface, we consider any nonzero element $V$ of the fiber of the annihilator bundle $\operatorname{Im}(d\psi)^{o}_{\pi(p)}$.  Then the induced vector field $V_P$ is in the kernel of $\sigma$ at $p$: $V_P(p)\in \ker(\sigma_p)$.
\item Choose a horizontal lift $\tilde{X}$ of any nonzero element of $\ker(d\psi_{\pi(p)})$.  Let $\eta \in \fg$ be a Lie algebra element satisfying $i_{\eta_P}\sigma = -\langle \psi\circ \pi, i_{\tilde{X}}dA\rangle$, which can be done since $i_{\tilde{X}}dA$ is the pullback of a covector vanishing on $\ker(d\psi_{\pi(p)})$.  Then $\tilde{X} + \eta_P(p)$ is in the kernel of $\sigma$ at $p$.  It is transverse to $Z$ since $X$ is transverse to the folding hypersurface $\hat{Z}$ of $\psi$ and $Z=\pi^{-1}(\hat{Z})$.
\item As we showed in the proof of \ref{lem:construct}, these two vectors span the kernel of $\sigma$ at $p$.
\end{itemize}
\end{remark}

We now begin the process of constructing a bijection between $H^2(W,\R)$ and isomorphic folded-symplectic structures on $P$ for which the action of $G$ is Hamiltonian with moment map $\psi\circ \pi$.  The first step is to produce a map from the space of closed two forms $\Omega^2(W)^c$ to the space of folded-symplectic structures on the principal bundle $P$.

\begin{lemma}\label{lem:map1}
Let $\psi:W \to \fg^*$ be a unimodular map with folds, where $\fg$ is the Lie algebra of a torus $G$.  Suppose $\pi:P \to W$ is a principal $G$ bundle and suppose $\sigma$ is a closed, invariant $2$-form for which the action of $G$ is Hamiltonian.  Let $A\in \Omega^1(P,\fg)^G$ be a connection $1$-form on $P$ and let $\Omega^2(W)^c$ be the space of closed $2$-forms.  Then

\begin{enumerate}
\item $\sigma = d(\langle \psi\circ \pi, A\rangle) + \pi^*\beta$ for some closed $2$-form $\beta\in\Omega^2(W)^c$, hence $\sigma$ is folded-symplectic.
\item Let $\mathcal{S}$ denote the space of folded-symplectic structures on $P$ for which the action of $G$ is Hamiltonian with moment map $\psi\circ \pi$.  Then the map $F_A:\Omega^2(W)^c \to \mathcal{S}$ given by $F_A(\beta) = d\langle \psi\circ \pi \rangle + \pi^*\beta$ is a bijection.
\end{enumerate}
\end{lemma}

\begin{proof}\mbox{ } \newline
\begin{enumerate}
\item We consider the difference $\Omega=\sigma -d\langle\psi\circ \pi, A \rangle$.  Since the action of $G$ is Hamiltonian for both forms with the same moment maps $\psi\circ \pi$, we have that for each $X\in \fg$ the contraction of $X_P$ with $\Omega$ satisfies:

    \begin{displaymath}
    i_{X_P}\Omega = i_{X_P}\sigma -i_{X_P}d\langle \psi\circ \pi, A \rangle = -d\langle \psi\circ \pi, X \rangle + d\langle \psi\circ\pi, X\rangle = 0
    \end{displaymath}
    Since the action of $G$ preserves both $\sigma$ and $d\langle \psi\circ \pi, A\rangle$, the action of $G$ also preserves $\Omega$.  Thus, $\Omega$ is a closed, invariant $2$-form that vanishes on vertical vectors, so it must be the pullback of a closed, basic $2$-form $\beta$.  That is, $\Omega=\pi^*\beta$ for some $\beta \in \Omega^2(W)^c$.  We therefore have:

    \begin{displaymath}
    \sigma = d\langle \psi\circ \pi, A \rangle + \pi^*\beta
    \end{displaymath}
    To see that $\sigma$ is folded-symplectic we compute its top power and we compute its kernel at degenerate points.
    \begin{itemize}
    \item We have $\sigma^n = (d\langle \psi\circ \pi, A \rangle + \pi^*\beta)^n = (d\langle \psi\circ \pi, A \rangle)^n$ since $\pi^*\beta$ is basic. Thus, $\sigma^n \pitchfork_s 0$ by our considerations in lemma \ref{lem:construct} and $\sigma$ degenerates on the hypersurface $Z=\pi^{-1}(\hat{Z})$.

    \item The kernel of $\sigma$ may be computed in the same way we computed the kernel of $d\langle \psi\circ \pi, A \rangle$ in remark \ref{rem:kernelconstruction}.  At each point $p\in Z$ in the degenerate hypersurface, we may choose a nonzero element $V\in \operatorname{Im}(d\psi)^o_{\pi(p)}$ and then $(i_{V_P}\sigma)_p=0$.  To obtain the piece of the kernel transverse to $Z$, we begin with $X\in \ker(d\psi_{\pi(p)})$ and take a horizontal lift to $\tilde{X} \in T_pP$.  Then,

        \begin{displaymath}
        (i_{\tilde{X}}\sigma)_p =\langle \psi\circ \pi, i_{\tilde{X}}dA \rangle_p + \pi^*(i_X\beta)_p
        \end{displaymath}
        which is the pullback of a covector that vanishes on $\ker(d\psi)_{\pi(p)}$.  Thus, there is an element $\eta \in \fg$ such that
        \begin{displaymath}
        \langle d\psi \circ d\pi_p, \eta\rangle = \langle \psi\circ \pi, i_{\tilde{X}}dA \rangle_p + \pi^*(i_X\beta)_p.
        \end{displaymath}
        The contraction of $\tilde{X} + \eta_P(p)$ with $\sigma$ at $p$ is then zero.
    \item The kernel cannot have a larger dimension since $\ker(d\psi_{\pi(p)})$ is $1$-dimensional and $\operatorname{Im}(d\psi)^o_{\pi(p)}$ is $1$-dimensional.  A larger kernel would imply that at least one of these spaces would have a larger dimension.
    \end{itemize}
\item We have therefore shown that the map $F_A(\beta) = d\langle \psi \circ \pi, A\rangle +\pi^*\beta$ has image in the space of folded-symplectic forms for which the action of $G$ is Hamiltonian with moment map $\psi\circ \pi$.  We have also shown that it is surjective.  It is injective because the map $\pi^*$ is injective, hence $F_A$ is a bijection.
\end{enumerate}
\end{proof}

\begin{definition}\label{def:connectionmap}
Let $\psi:W \to \fg^*$ be a unimodular map with folds, where $\fg$ is the Lie algebra of a torus, $G$.  Let $\pi:P \to W$ be a principal $G$-bundle with connection $A\in \Omega^1(P,\fg)^G$, let $\mathcal{S}$ be the space of folded-symplectic forms for which the action of $G$ is Hamiltonian with moment map $\psi\circ \pi$, let $\Omega^2(W)^c$ be the space of closed $2$ forms on $W$, and let $F_A:\Omega^2(W) \to \mathcal{S}$ be the bijection of lemma \ref{lem:map1}. We define $f_A: \mathcal{S} \to \Omega^2(W)$ to be the inverse of $F_A$.  We call $f_A$ the \emph{horizontal map} induced by the connection $A$.
\end{definition}

\begin{lemma}\label{lem:map2}
Let $\psi:W \to \fg^*$ be a unimodular map with folds, let $\pi:P \to W$ be a principal $G$ bundle, and let $\mathcal{S}$ be the space of folded-symplectic forms for which the action of $G$ is Hamiltonian with moment map $\psi\circ \pi$.  Then for any choice of connection $A\in \Omega^1(P,\fg)^G$, we obtain the horizontal map $f_A:\mathcal{S} \to \Omega^2(W)^c$.  If $p:\Omega^2(W)^c \to H^2(W,\R)$ is the projection map to homology, then the composite:

\begin{displaymath}
p\circ f_A: \mathcal{S} \to H^2(W,\R)
\end{displaymath}
is independent of the choice of connection $A$.
\end{lemma}

\begin{proof}\mbox{ } \newline
Suppose we choose a second connection $A'$ and perform the same constructions as in lemma \ref{lem:map1} to obtain $F_{A'}$ and, consequently, its inverse $f_{A'}$.  We let $\sigma \in \mathcal{S}$ and we let $\beta_A = f_A(\sigma)$, $\beta_{A'}=f_{A'}(\sigma)$ be the two different images in $\Omega^2(W)^c$. We compute:

\begin{displaymath}
\sigma = d\langle \psi\circ \pi, A \rangle + \pi^*\beta_A = d\langle \psi\circ \pi, A'\rangle + \pi^*\beta_{A'}
\end{displaymath}
by definition of $f_{A}$ and $f_{A'}$.  Consequently,

\begin{displaymath}
\pi^*(\beta_{A'}-\beta_A) = d(\langle \psi\circ \pi, A-A'\rangle )
\end{displaymath}
since $\langle\psi \circ \pi, A-A'\rangle$ is basic, i.e. it is $\pi^*\alpha$ for some $\alpha \in \Omega^1(W)$, we have:
\begin{displaymath}
\pi^*(\beta_A-\beta_{A'}) = d\pi^*\alpha = \pi^*d\alpha
\end{displaymath}
which means $\beta_A - \beta_{A'} = d\alpha$ since $\pi^*$ is injective.  Thus, $p\circ f_A(\sigma)=p\circ f_{A'}(\sigma)$ and so the map is independent of the choice of connection.
\end{proof}

\begin{definition}
Let $\psi:W \to \fg^*$ be a unimodular map with folds, where $\fg$ is the Lie algebra of a torus $G$.  Let $\pi:P \to W$ be a principal $G$ bundle
and let $\mathcal{S}$ be the space of folded-symplectic forms on $P$ for which the action of $G$ is Hamiltonian with moment map $\psi\circ \pi$.  For any connection $A\in \Omega^1(P,\fg)^G$, we obtain the horizontal map $f_A:\mathcal{S} \to \Omega^2(W)^c$.  By lemma \ref{lem:map2}, the composite $p\circ f_A: \mathcal{S} \to H^2(W,\R)$ is independent of the choice of connection $A$.  We define this map to be the \emph{horizontal chern class} $c_{hor}: \mathcal{S} \to H^2(W,\R)$.
\end{definition}

The horizontal chern class of a folded-symplectic structure will be a characteristic class: it behaves well under isomorphisms of principal $G$ bundles.
\begin{cor}\label{cor:map2}
Let $(\pi_i:P_i\to W, \sigma_i)$, $i=1,2$, be two toric, folded-symplectic bundles over $\psi:W\to \fg^*$  Suppose $\phi:P_1\to P_2$ is an isomorphism of principal $G$ bundles over $W$.  Then $c_{hor}(\phi^*\sigma_2)=c_{hor}(\sigma_2)$.
\end{cor}

\begin{proof}\mbox{ } \newline
Note that we actually don't need the folded structure $\sigma_1$, but we have inserted it since we are studying objects inside $\mathcal{B}_{\psi}(W)$ and all such objects come equipped with a folded-symplectic structure.  In any case, we fix a connection $A_2$ on $P_2$ and write:

\begin{displaymath}
\sigma_2 = d\langle \psi\circ \pi_2, A_2 \rangle \pi_2^*\beta_2
\end{displaymath}
and
\begin{displaymath}
\phi^*\sigma_2 = d\langle \psi\circ \pi_1, \phi^*A_2 \rangle + \pi_1^*\beta_2,
\end{displaymath}
where $\phi^*A_2$ is the pullback connection.  We then have that $c_{hor}(\phi^*\sigma_2)= [\beta_2] = c_{hor}(\sigma_2)$.
\end{proof}

We now show that if $\pi:P\to W$ is a fixed principal $G$ bundle over $\psi:W \to \fg^*$ and $\sigma_0$, $\sigma_1$ are two folded-symplectic structures on $P$ for which the action of $G$ is Hamiltonian with moment map $\psi\circ \pi$, then $(P,\sigma_0)$ and $(P,\sigma_1)$ are isomorphic if and only if their horizontal chern classes agree.

\begin{prop}\label{prop:map3}
Let $\psi:W \to \fg^*$ be a unimodular map with folds and let $\pi:P \to W$ be a principal $G$ bundle.  Let $\sigma_0$, $\sigma_1$ be two folded-symplectic forms on $P$ for which the action of $G$ is Hamiltonian with moment map $\psi\circ \pi$.  Then there exists a gauge transformation $\phi:P \to P$ satisfying $\phi^*\sigma_1=\sigma_0$ if and only if $c_{hor}(\sigma_0)=c_{hor}(\sigma_1)$.  That is, the cohomology class of the basic $2$-forms associated to $\sigma_0$ and $\sigma_1$ must be the same.
\end{prop}

\begin{proof}\mbox{ } \newline
\begin{enumerate}
\item We show the \emph{only if} portion first.  Suppose $\phi:P\to P$ is a gauge transformation satisfying $\phi^*\sigma_1=\sigma_0$.  Fix a connection $A$ on $P$.  We have

    \begin{equation}\label{eq:sigmai}
    \sigma_i = d\langle \psi\circ \pi, A\rangle + \pi^*\beta_i
    \end{equation}
    for $i=0,1$ by lemma \ref{lem:map1}.  Consequently,

    \begin{displaymath}
    \begin{array}{lcl}
    \phi^*\sigma_1 & = & d\langle \psi\circ \pi, \phi^*A \rangle + \pi^*\beta_1\\
                   & = & d\langle \psi\circ \pi, A \rangle + \pi^*\beta_0
    \end{array}
    \end{displaymath}
    since, by assumption, $\phi^*\sigma_1=\sigma_0$.  We then have:

    \begin{displaymath}
    \pi^*(\beta_0-\beta_1)=d(\psi\circ \pi,\phi^*A - A \rangle.
    \end{displaymath}
    Since $\langle \psi  \circ \pi, \phi^*A -A \rangle$ is a basic $1$-form, it is $\pi^*\gamma$ for some $\gamma\in\Omega^1(W)$.  Thus,
    \begin{displaymath}
    \pi^*(\beta_0-\beta_1)= d\pi^*\gamma = \pi^*d\gamma
    \end{displaymath}
    and since $\pi^*$ is injective, we have $\beta_0-\beta_1 =d\gamma$, hence $c_{hor}(\sigma_1)=[\beta_1]=[\beta_0]=c_{hor}(\sigma_0)$.

\item Now, suppose that $c_{hor}(\sigma_0)=c_{hor}(\sigma_1)$.  That is, by equation \ref{eq:sigmai}, $[\beta_1]=[\beta_0]$.  Then there is a one form $\gamma\in \Omega^1(W)$ such that $\beta_1=\beta_0 + d\gamma$.  We define the linear path of fold forms:

    \begin{equation}\label{eq:foldpath1}
    \sigma_s:=\sigma_0 + s\pi^*(d\gamma)
    \end{equation}
    which is folded for all $s$ since:
    \begin{itemize}
    \item $(\sigma_s)^n=\sigma_0^n$, hence $\sigma_s\pitchfork_s 0$ and $Z=(\sigma_s)^{-1}(0) = (\sigma_0)^{-1}(0)$ is an embedded submanifold with corners of $P$.
    \item Using the ideas of remark \ref{rem:kernelconstruction}, we have that
    \begin{equation}\label{eq:kernel}
    \ker(\sigma_s)_p=\span\{\tilde{X}+v_s, V\},
    \end{equation}
    where $\tilde{X}$ is some lift of $X\in \ker(d\psi)_{\pi(p)}$, $v_s$ is some vertical vector, and $V$ is a vertical vector generated by an element in the fiber of $\operatorname{Im}(d\psi)_{\pi(p)}^o$.  Because the kernel element $X\in \ker(d\psi)$ is transverse to the folding hypersurface $\hat{Z}$ of $\psi$ and $Z= \pi^{-1}(\hat{Z})$, we have that $\tilde{X} +v_s$ is transverse to $T_pZ$, hence $i_Z^*\sigma_s$ has maximal rank.
    \end{itemize}
    Using our usual Moser-type argument, we seek an isotopy satisfying $\phi_s^*\sigma_s=\sigma_0$, which amounts to solving:

    \begin{equation}\label{eq:pplmoser}
    di_{X_s}\sigma_s = -d\pi^*\gamma
    \end{equation}
    for a time-dependent vector field $X_s$, whose flow is $\phi_s$.  It is then sufficient to solve
    \begin{equation}\label{eq:pplmoser1}
    i_{X_s}\sigma_s = -\pi^*\gamma
    \end{equation}
    which has a smooth solution if and only if $\pi^*\gamma$ vanishes on $\ker(\sigma_s)$ at the points of $Z$ for all $s$ (q.v. proposition \ref{prop:Moser}).  By our description of the kernel of $\sigma_s$ in equation \ref{eq:kernel}, we will have a solution if and only if $\gamma \in \Omega^1(W)$ vanishes on $\ker(d\psi)$ at points of the folding hypersurface $\hat{Z}$ of $\psi$.  It is not clear that $\gamma$ satisfies this condition so we modify the right hand side of equation \ref{eq:pplmoser1} by adding the pullback of a \emph{closed} basic $1$ form $\pi^*\gamma_0$ so that
    \begin{equation}\label{eq:pplmoser2}
    i_{X_s}\sigma_s = -(\pi^*\gamma + \pi^*\gamma_0)=-\pi^*(\gamma - \gamma_0)
    \end{equation}
    has a smooth solution.  Note that this solution $X_s$ will satisfy:

    \begin{displaymath}
    di_{X_s}\sigma_s = -\pi^*d\gamma -\pi^*d\gamma_0 = -\pi^*d\gamma
    \end{displaymath}
    since $\gamma_0$ is closed, hence $X_s$ will solve equation \ref{eq:pplmoser} as required.  Since $\hat{Z}$ is co-orientable in $W$ and $\ker(d\psi)$ is stratified (by definition \ref{def:umf}), we may make the following choices:

    \begin{itemize}
    \item We choose an nonvanishing section $X \in \Gamma(\ker(d\psi))$ which we extend to a vector field $\tilde{X}$ on $W$.
    \item We choose a smooth function $g:W \to \R$ such that $g\big\vert_{\hat{Z}}=0$ and $dg(X)\big\vert_{\hat{Z}}=1$.
    \end{itemize}
    Set let $f$ be the product $f=(g)(\gamma(\tilde{X}))\in C^{\infty}(W)$.  We define:
    \begin{displaymath}
    \gamma_0 = -df =-)\gamma(\tilde{X})dg - (g)d(\gamma(\tilde{X})).
    \end{displaymath}
    For all $z\in \hat{Z}$, we have:
    \begin{displaymath}
    \begin{array}{lcl}
    i_{X(z)}(\gamma_z + (\gamma_0)_z) & = & \gamma(X) + i_X(\gamma_0) \\
                                      & = & \gamma(X) -\gamma(\tilde{X})dg(X) - (g)(d(\gamma(\tilde{X}))(X)) \\
                                      & = & \gamma(X) - \gamma(X)(1) - (0)(d(\gamma(\tilde{X}))(X)) \\
                                      & = & 0
    \end{array}
    \end{displaymath}
    where we have suppressed the subscript $z$ after the first step to avoid notational clutter.  Consequently equation \ref{eq:pplmoser2} has a smooth, $G$-invariant solution $X_s$.  The $G$-invariance of $X_s$ follows from the fact that $\sigma_s$ and $\pi^*(\gamma + \gamma_0)$ are both $G$-invariant.

    Now, we claim that $X_s$ is tangent to orbits for all $s$.  Indeed, for all $X\in \fg$ we have $(i_{X_s}\sigma_s)(X_P) = \pi^*(\gamma + \gamma_0)(X_P) =0$ since the right hand side is basic.  Thus, for all $p\in P$ we have that $X_s(p)\in T_p(G\cdot p)^{\sigma_s}$.  Since $\sigma_s$ is symplectic on an open dense subset and the orbits are Lagrangian, this implies that $X_s(p)\in T_p(G\cdot p)$ for $p$ in an open dense subset of $P$.  Thus, smoothness of $X_s$ implies $X_s$ is tangent to orbits everywhere.  In other words, $X_s$ is tangent to the fibers of $\pi:P \to W$.  Since the fibers are compact, we may integrate $X_s$ to obtain $\phi_s$, which maps fibers to fibers and is thus a gauge transformation: $\pi \circ \phi_s = \pi$.  We take $\phi_1$ as our requisite isomorphism.
\end{enumerate}
\end{proof}
We are now ready to classify objects in $\mathcal{B}_{\psi}(W)$.

\begin{theorem}\label{thm:bundleclassification}
Let $\psi:W \to \fg^*$ be a unimodular map with folds, where $\fg$ is the Lie algebra of a torus $G$.  Let $\mathcal{B}_{\psi}(W)$ be the category of toric, folded-symplectic bundles over $\psi$.  Then there is a bijection $b:\pi_0(B_{\psi}) \to H^2(W; \mathbb{Z}_G \times \R)$ given by:
\begin{displaymath}
b([(\pi:P \to W, \sigma)]) = c_1(P)\times c_{hor}(\sigma).
\end{displaymath}
That is, isomorphism classes of toric, folded-symplectic bundles over $\psi$ are parameterized by cohomology classes in $H^2(W,\mathbb{Z}_G\times \R)$.
\end{theorem}

\begin{proof}\mbox { } \newline
Let $(\pi_1:P \to W, \sigma_1)$ and $(\pi_2:P\to W, \sigma_2)$ be two toric, folded-symplectic bundles over $\psi:W \to \fg*$.  First, let us show that the map $b$ is defined.  Recall that isomorphism classes of principal $G$-bundles over $W$ are parameterized by their first Chern class $c_1 \in H^2(W,\mathbb{Z}_G)$.  Thus, there exists an isomorphism of principal $G$-bundles $\phi:P_1 \to P_2$ if and only if $c_1(P_1)=c_1(P_2)$.

Suppose $\phi:P_1\to P_2$ is an isomorphism of toric folded-symplectic bundles over $\psi:W \to \fg^*$.  That is $\phi$ is a map of principal bundles and $\phi^*\sigma_2= \sigma_1$.  Then corollary \ref{cor:map2} implies that $c_{hor}(\sigma_1)=c_{hor}(\phi^*\sigma_2) = c_{hor}(\sigma_2)$.  Since $\phi$ is an isomorphism of principle bundles, we have $c_1(P_1)=c_1(P_2)$ by the above remarks.  Thus, the map is well-defined.

Now, suppose $c_1(P_1)=c_1(P_2)$ and $c_{hor}(\sigma_1)=c_{hor}(\sigma_2)$.  Then there exists an isomorphism of principal bundles $\phi:P_1\to P_2$.  By corollary \ref{cor:map2}, $c_{hor}(\phi^*\sigma_2) = c_{hor}(\sigma_2)=c_{hor}(\sigma_1)$.  By proposition \ref{prop:map3} there exists a gauge transformation $\phi_1:P_1\to P_1$ such that $\phi_1^*(\phi^*\sigma_2)=\sigma_1$, hence $\phi\circ \phi_1:P_1 \to P_2$ is an isomorphism of toric, folded-symplectic bundles over $\psi:W \to \fg^*$.
\end{proof}

\pagebreak
\section{Construction of the Functor $c:\mathcal{B}_{\psi}(W) \to \mathcal{M}_{\psi}(W)$.}
The construction of the functor $c:\mathcal{B}_{\psi}(W) \to \mathcal{M}_{\psi}(W)$ is local in nature and is accomplished in three steps.  In the first step, we take an object $(\pi:P\to W, \sigma)$ of $\mathcal{B}_{\psi}(W)$ and we collapse its boundary strata in a natural way to obtain a topological space, which we call $c_{top}(P)$.  We then argue that there is a natural smooth structure on $c_{top}(P)$ which is constructed using a local cutting procedure which we explain below.  Furthermore, the charts on $c_{top}(P)$ are such that the transition maps are folded-symplectic maps, hence there is a global folded-symplectic structure on the space that we call $c(P)$, which is just $c_{top}(P)$ with its smooth structure constructed below.  Finally, we show that the assignment $P \to c(P)$ is functorial, hence $c$ is a well-defined functor.  This is the strategy used in \cite{KL} and we follow it very closely, the reason being that the ingredients for the construction are essentially the same: we have the notion of a symplectic slice representation (q.v. proposition \ref{prop:normbundle}), we have the ability to read the slice representation from the orbital moment map (q.v. theorem \ref{thm:structure}), and we have folded-symplectic reduction.  With these tools in hand, one need only add in a few extra remarks to show that the constructions in \cite{KL} extend across the folding hypersurface of $W$.  Before we begin, we recall the definitions of the presheaves (q.v. remark \ref{rem:empsi}) $\mathcal{M}_{\psi}$ and $\mathcal{B}_{\psi}$:

\begin{definition}
Let $W$ be a manifold with corners and let $\psi:W\to \fg^*$ be a unimodular map with folds, where $\fg$ is the Lie algebra of a torus $G$.  We define the category $\mathcal{M}_{\psi}(W)$ to be the category whose objects are triples:
\begin{displaymath}
(M,\sigma, \pi:M \to W)
\end{displaymath}
where $\pi$ is a quotient map and $(M,\sigma, \psi \circ \pi)$ is a toric, folded-symplectic manifold with co-orientable folding hypersurface, where the torus is $G$, with moment map $\psi \circ \pi$.  We refer to an object as a \emph{toric, folded-symplectic manifold over $\psi$}.  A morphism between two objects $(M_i,\sigma_i,\pi_i:M \to W)$, $i=1,2$, is an equivariant diffeomorphism $\phi:M_1 \to M_2$ that induces a commutative diagram:

\begin{displaymath}
\xymatrix{
M_1 \ar[rr]^{\phi} \ar[dr]^{\pi_1} & & M_2 \ar[dl]^{\pi_2} \\
                                   &W \ar[r]^{\psi}& \fg^*
}
\end{displaymath}
and satisfies $\phi^*\sigma_2=\sigma_1$, hence $\phi$ is an equivariant folded-symplectomorphism that preserves moment maps.  By definition, every morphism is invertible, hence $\mathcal{M}_{\psi}(W)$ is a groupoid.
\end{definition}

\begin{definition}
Let $\psi: W \to \fg^*$ be a fixed unimodular map with folds, where $\fg$ is the Lie algebra of a torus $G$.  We define $\mathcal{B}_{\psi}(W)$ to be the category whose objects are principal $G$-bundles $\pi:P \to W$ equipped with an invariant folded-symplectic form $\sigma$ with co-orientable folding hypersurface, denoted as a pair

\begin{displaymath}
(\pi:P \to W, \sigma)
\end{displaymath}
so that $\psi\circ \pi$ is a moment map for the principal action of the torus $G$ on $P$.  A morphism $\phi$ between two objects $(\pi_1:P_1\to W, \sigma_1)$ and $(\pi_2:P_2 \to W, \sigma_2)$ is a map of principal $G$ bundles:

\begin{displaymath}
\xymatrix{
P_1 \ar[rr]^{\phi} \ar[dr]^{\pi_1}& & P_2 \ar[dl]^{\pi_2} \\
                                  &W\ar[r]^{\psi} & \fg^*
}
\end{displaymath}
so that $\phi^*\sigma_2 = \sigma_1$, hence $\phi^*(\psi\circ \pi_2) =\psi \circ \pi_1$.  That is, $\phi$ preserves moment maps.
\end{definition}

\begin{remark}
If $\psi:W\to \fg^*$ is a unimodular map with folds and $W$ has no boundary, i.e. $W$ is a manifold without corners, then
\begin{displaymath}
\mathcal{B}_{\psi}(W) = \mathcal{M}_{\psi}(W).
\end{displaymath}
In general, if $W$ is a manifold with corners then its interior $\mathring{W}$ is a manifold, hence
\begin{displaymath}
\mathcal{B}_{\psi}(\mathring{W})=\mathcal{M}_{\psi}(\mathring{W})
\end{displaymath}
\end{remark}

Per our custom, we are assuming folding hypersurfaces are \emph{co-orientable}.
\subsection{Step 1 - Define $c_{top}(P)$}
Fix an object $(\pi:P\to W, \sigma)$ of $\mathcal{B}_{\psi}(W)$.  Recall that for each point $w\in W$, $\psi$ attaches to $w$ an integral basis $\{v_1,\dots, v_k\}$ of the integral lattice of a subtorus $K_w$ of $G$ (q.v. lemma \ref{lem:attach1}).  We let $\sim$ be the smallest equivalence relation on $P$ such that for all $p_1, p_2\in P$, $p_1 \sim p_2$ if and only if:

\begin{enumerate}
\item $\pi(p_1)=\pi(p_2)=:w$
\item There exists an element $k\in K_w$ satisfying $k\cdot p_1=p_2$.
\end{enumerate}

We let $c_{top}(P)$ be the topological space $P/\sim$.  Since the action of $G$ commutes with the action of each $K_w$, the action of $G$ on $P$ descends to a continuous action of $G$ on $c_{top}(P)$.  Furthermore, given a morphism $\phi\in Hom((P_1,\sigma_1),(P_2,\sigma_2))$ between two objects, the equivariance of $\phi$ implies that $\phi$ descends to an equivariant homeomorphism $c_{top}(\phi):c_{top}(A)\to c_{top}(B)$ and that $c_{top}:\mathcal{B}_{\psi} \to Top_G$ is a functor.

\begin{remark}\label{rem:restriction}
More specifically, $Top_G$ is the category whose objects are topological $G$-spaces and for any $X,Y \in \operatorname{Ob}(Top_G)$, $\hom(X,Y) = \{\phi:X \to Y \vert \space \text{ $\phi$ is an equivariant homeomorphism}\}$.  One can show that $Top_G$ is a presheaf of groupoids and that:
\begin{displaymath}
c_{top}:\mathcal{B}_{\psi} \to Top_G
\end{displaymath}
is a map of presheaves of groupoids.  Furthermore, note that the two spaces $c_{top}(P\vert_U)$ and $c_{top}(P)\vert_U$ are the same set and have the same topology, hence:
\begin{displaymath}
c_{top}(P\vert_U) = c_{top}(P)\vert_U
\end{displaymath}
Thus, for any morphism $\phi:P_1 \to P_2$ between objects in $\mathcal{B}_{\psi}$ and any open subset $U\subset W$, we will write $c_{top}(\phi)\vert_U$ to mean $c_{top}(\phi\vert_U)$.
\end{remark}

\subsection{Step 2 - Construction of Smooth and Folded-Symplectic Structures}
For each $w\in W$ lemma \ref{lem:attach1} attaches an integral basis $\{v_1,\dots, v_k\}$ of the integral lattice of a subtorus $K_w$ of $G$ and an effective symplectic representation $V$ of $K_w$ with weights $\{v_1^*,\dots, v_k^*\}$.  Given such a representation, we can form the global reduced space $cut(P):=(P \times V)//_0 K_w$.  Of course, this may not be a manifold, but we have the following existence lemma to remedy this problem locally.

\begin{lemma}\label{lem:cut}
Let $(\pi:P\to W, \sigma)$ be a toric folded-symplectic bundle over $\psi:W\to \fg^*$.  For each $w\in W$, $\exists$ a neighborhood $U_w$ of $w$ so that $cut(P\vert_{U_w}):= (P\vert_{U_w} \times V)//_0 K_w$ is a toric folded symplectic manifold over $\psi\vert_{U_w}: U_w \to \fg^*$.
\end{lemma}

\begin{proof}
Let $K_w$ be the subtorus of $G$ asociated to $w$ and $\frak{k}_w$ its Lie algebra.  Since faithful representations of tori are classified by their weights, we may assume $V= \C^k$.  Let $j:\frak{k}_w \to \fg$ be the inclusion of the Lie algebra of $K_w$, $j^*:\fg^*\to \frak{k}_w^*$ the corresponding projection, and $\xi_0 = j^*\psi(w)$.
\begin{itemize}
\item A moment map for the action of $K_w$ on $P\vert_{U_w} \times V$ is given by
\begin{equation}\label{eq:cutmoment}
\Phi(p,z)=(j^*\psi)\circ \pi(p) - \sum_{i=1}^k \vert z_i \vert^2 v_i^* -\xi_0.
\end{equation}
\item We note that $j^*\psi$ is a submersion at $w$.  This is because $\frak{k}_w$ is normal to the image of the stratum containing $w$.  The fold map induces fiberwise isomorphisms on normal bundles (q.v. corollary \ref{cor:folds4-1}), hence it sends directions normal to the stratum containing $w$ isomorphically onto directions transverse to the affine hyperplane corresponding to $w$ in the image of $\psi$.  Hence, there is a neighborhood $U$ of $w$ on which $j^*\psi$ is a submersion.
\item We may choose $U$ to be contractible so that $P\vert_U \simeq U\times G$.  More precisely, we have a commuting diagram:

\begin{displaymath}
\xymatrix{U\times G \ar[r]^{\simeq} \ar[dr]^{pr_1} & P\vert_U \ar[d]^{\pi} \ar[dr]^{\mu} & \\
        & U \ar[r]^{\psi} &\fg^*}
\end{displaymath}

\item By corollary \ref{cor:attach}, we may choose $U$ so that $\psi\vert_U=\bar{\mu}\circ\Psi$, where $\bar{\mu}:U \to \fg^*$ is a unimodular local embedding and $\Psi:U\to U$ is a map with fold singularities.  If $U$ is isomorphic to a subset of $R^n$, then both $\bar{\mu}$ and $\Psi$ extend to $\tilde{\mu}:\tilde{U}\to \fg^*$ and $\tilde{\Psi}:\tilde{U}\to \tilde{U}$, where $\tilde{U}$ is a manifold without corners containing $U$, $\tilde{\mu}$ is a local embedding, and $\tilde{\Psi}$ is a map with fold singularities.  This is because these maps are smooth, hence they extend by definition, and being an embedding or a fold map is an open property.

\item Because $\mu$ is a local embedding, we may assume that $\tilde{\mu}\vert_{\tilde{U}}$ is an embedding, $\mu(\tilde{U})$ is an open ball in $\fg^*$, and $\tilde{\mu}(\tilde{U}\setminus U)$ lies outside of the unimodular cone into which $U$ is mapped by $\psi$.  Recall that this cone is:

    \begin{displaymath}
    \{ \xi \in \fg^* \vert \langle \xi-\psi(w), v_i \rangle \ge 0 \}
    \end{displaymath}
    where the $v_i's$ are the primitive normals assigned to $w$ via lemma \ref{lem:attach1}.  Equivalently, we have:
    \begin{displaymath}
    \{\xi \in \fg^* \vert \langle j^*\xi - \xi_0, v_i \rangle \ge 0 \}
    \end{displaymath}

\item Consequently, if we consider $\tilde{\Phi}(u,g,z) = j^*\tilde{\mu}\circ\tilde{\Psi} \circ pr_1(u,g) - \sum_{i=1}^k \vert z_i \vert^2 v_i^* -\xi_0$, then $\tilde{\Phi}^{-1}(0) = \Phi^{-1}(0)$.

\item $0$ is a regular value of $\Phi$ since $j^*\psi$ is a submersion (shrink $U$ if necessary).  Thus, $\Phi^{-1}(0)$ is a manifold of codimension $\dim(K_w)$ on which $K_w$ acts freely.  We set $U_w=\tilde{U}\cap W$.  By theorem \ref{thm:fsred}, we have that $\Phi^{-1}(0)/K_w = (P\vert_{U_w} \times \C^k)//_0 K_w = cut(P\vert_{U_w})$ is a folded-symplectic manifold.

\item The action of $G$ on ($P\vert_{U_w} \times \C^k)$ commutes with the action of $K_w$ and the moment map $\mu$ and the projection $\pi$ descend to a moment map $\bar{\mu}: cut(P\vert_{U_w}) \to \fg^*$ and a quotient map $\bar{\pi}:cut(P\vert_{U_w})\to U_w$ satisfying $\bar{\mu} =\psi \circ \bar{\pi}$.

\item To see that the action of $G$ is toric, consider $\eta \in int(U_w)$.  This implies:
\begin{equation}\label{eq:toricaction}
\langle j^*\psi(\eta) -\xi_0, v_i \rangle > 0 \: \forall i, \: 0\le i \le k
\end{equation}

Consider $g_0\in G$ and $[\eta, g, z] \in cut(P\vert_{U_w})$.  We have:

\begin{displaymath}
g_0 \cdot [\eta,g,z] = [\eta, g_0g, z] = [\eta, g, z] \iff (\eta,g_0g,z) = (\eta, kg, kz)
\end{displaymath}
for some $k\in K_w$.  This implies that $k=g_0\in K_w$ and $g_0 z = z$.  But, $\psi(\eta) = \sum_{i=1}^n \vert z_i \vert^2v_i^* + \xi_0$ and no $z_i$ can be $0$ by \ref{eq:toricaction}, hence $z$ has no nontrivial stabilizer and $g_0=e$.  Thus, the action is effective.
\end{itemize}
\end{proof}

\begin{remark}
Suppose $(P_1,\sigma_1)$ and $(P_2,\sigma_2)$ are two objects of $\mathcal{B}_{\psi}$ and $\phi:A \to B$ is a morphism.  $\phi$ is $G$-equivariant, hence $\phi\times id: P_1\vert_{U_w} \times V \to P_2\vert_{U_w} \times V$ is $G\times K_w$-equivariant and induces an isomorphism of toric, folded-symplectic manifolds:

\begin{displaymath}
cut(\phi):cut(P_1\vert_{U_w}) \to cut(P_2\vert{U_w}),  \text{ $cut(\phi)[p,z]=[\phi(p),z]$}.
\end{displaymath}

Thus, for each $w\in W$ we have a functor $cut: B_{\psi\vert_{U_w}} \to M_{\psi\vert_{U_w}}$
\end{remark}

We will now construct natural $G$-equivariant homeomorphisms $\alpha^P_w: c_{top}(P\vert_{U_w}) \to cut(P\vert_{U_w})$.  As before, let $\{v_1, \dots, v_k\}$ be the integral basis attached to $w$ and let $K_w$ be the corresponding subtorus of $G$.  Let $\xi_0=j^*\psi(w)$, fix the representation of $K_w$ to be $\C^k$, and let $\frak{k}_w$ be $Lie(K_w)$.  Let $\mu_w:\C^k \to \frak{k}_w$ be the moment map for the action of $K_w$.  The construction of $\alpha^P_w$ depends on the fact that $(\C^k, \omega_w, \mu_w)$ is a toric symplectic $K_w$-manifold over the cone $\mu_w(\C^k) = \{\eta\in \frak{k}_w^* \vert \: \langle \eta, v_i \rangle \le 0 \text{ for } 1\le i \le k \}$.  Moreover,
\begin{enumerate}
\item the map $\mu_w:\C^k \to \mu_w(\C^k)$ has a continuous section
\begin{displaymath}
s:\mu_w(\C^k)\to \C^k, s(\eta) = (\sqrt{\langle -\eta, v_1 \rangle}, \dots , \sqrt{\langle -\eta, v_k \rangle})
\end{displaymath}
which is smooth over the interior of the cone $\mu_w(\C^k)$.
\item The stabilizer $K_z$ of $z\in \C^k$ deponds only on the face of the cone $\mu_w(\C^k)$ containing $\mu_w(z)$ in its interior:
    \begin{displaymath}
    K_z = \exp(\span_{\R} \{v_i \in \{v_1,\dots,v_k\} \vert \: \langle \mu_w(z),v_i \rangle = 0\})
    \end{displaymath}
\end{enumerate}

Let $U=U_w$ and $\nu=j^*\circ \mu: P\vert_U \to \frak{k}_w^*$, the $K_w$-moment map.  Then for any point $p\in P\vert_U$

\begin{displaymath}
\xi_0-\nu(p) \in \mu_w(\C^k)
\end{displaymath}

and
\begin{displaymath}
\begin{array}{lcl}
s(\xi_0 - \nu(p)) & = & (\sqrt{\langle \nu(p)-\xi_0, v_1 \rangle}, \dots \sqrt{\langle \nu(p)-\xi_0, v_k \rangle}) \\
                  & = & (\sqrt{\langle \mu(p)-\psi(w), v_1 \rangle}, \dots \sqrt{\langle \mu(p)-\psi(w), v_k \rangle})
\end{array}
\end{displaymath}
where $\mu=\psi\circ \pi:P\to \fg^*$ is the moment map for the action of $G$ on $P$.  This gives us a continuous map

\begin{displaymath}
\phi: P\vert_U \to \Phi^{-1}(0) \subset P\vert_U\times \C^k, \: \phi(p)=(p,s(\xi_0 - \nu(p)))
\end{displaymath}
where $\xi_0 = j^*\psi(w)$ and $\Phi$ is the moment map of equation \ref{eq:cutmoment}.  The image of $\phi$ intersects every $K_w$ orbit in $\Phi^{-1}(0)$.  Hence the composite
\begin{displaymath}
f=\tau\circ \phi: P\vert_U \to \Phi^{-1}(0)/K_w,
\end{displaymath}

where $\tau:\Phi^{-1}(0)\to \Phi^{-1}(0)/K_w$ is the orbit map, is surjective.  Next we show that the fibers of $f$ are precisely the equivalence classes of the relation $\sim$ defined in Step 1.  By definition, two points $p_1, p_2 \in P\vert_U$ are equivalent with respect to $\sim$ if and only if $\pi(p_1)=\pi(p_2)$ and there is an $a\in K_{\pi(p_1)}$ with $a\cdot p_2 = p_1$. On the other hand $f(p_1)=f(p_2)$ if and only if there is an $a\in K_w$ with:

\begin{displaymath}
(p_1, s(\xi_0-\nu(p_1))) = (a\cdot p_2, a\cdot s(\xi_0 - \nu(p_2)))
\end{displaymath}
For any point $\eta \in \mu_w(\C^k)$

\begin{displaymath}
a\cdot s(\eta) = s(\eta) \iff \text{ $a$ lies in the stabilizer $K_{s(\eta)}$ of $s(\eta)$}.
\end{displaymath}
For $\eta= \xi_0 - \nu(p_2) = \xi_0 -\nu(p_1) = j^*(w - \psi(\pi(p_1)))$,

\begin{displaymath}
\begin{array}{lll}
K_{s(\eta)} & = & \exp(\span_{\R}(\{v_i \in \{v_1,\dots,v_k\} \vert \: \langle \xi_0-\nu(p_1), v_i \rangle =0\})) \\
            & = & \exp(\span_{\R}(\{v_i \in \{v_1,\dots,v_k\} \vert \: \langle j^*\psi(w)-j^*\psi(\pi(p_1)),v_i \rangle =0\}))\\
            & = & \exp(\span_{\R}(\{v_i \in \{v_1,\dots,v_k\} \vert \: \langle \psi(w)-\psi(\pi(p_1),v_i\rangle =0\})) \\
            & = & K_{\pi(p_1)}.
\end{array}
\end{displaymath}

Therefore, $f(p_1)=f(p_2) \iff \: \pi(p_1)=\pi(p_2) \text{ and } \exists a\in K_{\pi(p)} \text{ satisfying } a\cdot p_1=p_2$. We conclude that the fibers of $f$ are precisely the equivalence classes of the relation $\sim$.  Therefore $f$ descends to a continuous bijection:

\begin{displaymath}
\alpha_w^P: c_{top}(P\vert_U) = (P\vert_U)/\sim \to \Phi^{-1}(0)/K = cut(P\vert_U), \: \alpha^P_w([p]) = [p,s(\xi_0-\nu(p))].
\end{displaymath}
The properness of $f$ implies that $f$ is closed, hence $\alpha^P_w$ is closed.  Thus, $\alpha^P_w$ is a homeomorphism.

\begin{remark}\label{rem:commute}
Notice that for any morphism $\phi:P_1\vert_U \to P_2\vert_U$ we have:

\begin{displaymath}
\begin{array}{lcl}
cut(\phi)(\alpha_w^{P_1}[p]) & =  & cut(\phi)([p,s(\xi_0 - \nu(p))]) \space = \space [\phi(p),s(\xi_0 - \nu(p))] \\
                             & =  & \alpha^{P_2}_w([\phi(p))] \space = \space \alpha^{P_2}_w(c_{top}([p]).
\end{array}
\end{displaymath}
\end{remark}

To finish the construction of the functor $c$, we must show that the transition maps $v= (\alpha_{w_2}^P) \circ (\alpha_{w_1}^P)^{-1}$ are isomorphisms of toric, folded-symplectic manifolds over $\psi:U_{w_1}\cap U_{w_2} \to \fg^*$.  We will need a lemma:

\begin{lemma}\label{lem:fsmaps}
Let $(M_1,\sigma_1)$ and $(M_2,\sigma_2)$ be two toric folded-symplectic manifolds of dimension $2n$.  Suppose $\phi:M_1\to M_2$ is a smooth bijection satisfying:
\begin{enumerate}
\item $\phi(Z_1)=Z_2$, where $Z_i\subset M_i$ is the fold for each $i$ and
\item $\phi^*\sigma_2=\sigma_1$.
\end{enumerate}
Then $\phi$ is a diffeomorphism.
\end{lemma}

\begin{proof} \mbox{ } \newline
We simply check that $d\phi$ has maximal rank everywhere.  Condition $1$ implies that $(\phi^*\sigma_2)_m=(\sigma_1)_m$ drops rank at $m$ if and only if $m\in Z_1 \: \iff \: \phi(m) \in Z_2$.  Since $\sigma_1$ is symplectic on $M_1\setminus Z_1$, we have that $m\in M_1\setminus Z_1$ implies $\operatorname{rank}(d\phi_m)=2n$.

At a point $z\in Z_1$, it suffices to perform a local computation to check that $d\phi_z$ has maximal rank.  Thus, we may assume that $M_1=M_2=\R^{2n}$ so that

\begin{displaymath}
(\sigma_1)^n=\phi^*(\sigma_2)^n  = \det(d\phi)(\sigma_2\circ \phi)^n = \det(d\phi)f(dx_1\wedge \dots \wedge dx_{2n})
\end{displaymath}

where $f:\R^{2n} \to \R$ is a function vanishing on $Z_1$.  Since $\sigma_1^n \pitchfork_s 0$ at $z$, we must have that $\det(d\phi)f \pitchfork_s 0$ at $z$.  This implies $\det(d\phi)_z \ne 0$ and so $\operatorname{rank}(d\phi_z)=2n$ and $\phi$ is a diffeomorphism.
\end{proof}

We first show that $(\alpha_{w_2}^P)\circ (\alpha_{w_1}^P)^{-1}$ satisfies condition $1$ of lemma \ref{lem:fsmaps} and then discuss why condition $2$ is satisfied.

\begin{lemma}\label{lem:foldtofold}
Let $(M_i,\sigma_i,\pi_i:M_i\to W)$, $i\in \{1,2\}$, be two toric folded-symplectic manifolds over a unimodular map with folds $\psi:W\to \fg^*$.  Suppose $\phi:M_1\to M_2$ is an equivariant map satisfying $\pi_2\circ \phi = \pi_1$.  Then $\phi(Z_1)=Z_2$.
\end{lemma}
\begin{proof}
Let $\hat{Z}$ be the folding hypersurface of $\psi$.  We have:

\begin{displaymath}
Z_1 = \pi_1^{-1}(\hat{Z}) = \phi^{-1}(\pi_2^{-1}(\hat{Z})) = \phi^{-1}(Z_2)
\end{displaymath}
Thus, $\phi(Z_1)\subset Z_2$.  To show equality, we show that $\phi\vert_{Z_1}$ surjects onto $Z_2$.  Fix $z_2\in Z_2$ and let $z=\pi_2(z_2)$.  There exists an element $z_1\in \pi_1^{-1}(z)$ and $\pi_2\circ\phi=\pi_1$ implies $\phi(z_1) \in \pi_2^{-1}(z)$, hence $h\cdot\phi(z_1)=z_2$ for some $h\in G$.  Thus, $\phi(h\cdot z_1) = h\cdot\phi(z_1) = z_2$ and $\phi(Z_1)=Z_2.$
\end{proof}

By construction of $cut(P\vert_U)$, we have a commuting diagram:
\begin{displaymath}
\xymatrix{
cut(P\vert_{U_{w_1}}) \ar[dr]_{\pi_1}  & c_{top}(P\vert_{(U_{w_1}\cap U_{w_2})})\ar[l]_-{\alpha^P_{w_1}} \ar[r]^-{\alpha^P_{w_2}}& cut(P\vert_{U_{w_2}})  \ar[dl]^{\pi_2} \\
                     &  U_{w_1}\cap U_{w_2}  \ar[r]^{\psi}     &    \fg^*}
\end{displaymath}

where $\pi_1, \pi_2$ are the quotient maps.  Technically, we are considering $cut(P\vert_{U_{w_i}})\vert_{(U_{w_1}\cap U_{w_2})}$, but we omit this extra notation.  Since the maps on the top line are equivariant and the diagram commutes, we see that $(\alpha_{w_2}^P)\circ (\alpha_{w_1}^P)^{-1}$ satisfies the conditions of lemma \ref{lem:foldtofold} and so $(\alpha_{w_2}^P)\circ (\alpha_{w_1}^P)^{-1}(Z_1)=Z_2$, where $Z_i \subset cut(P\vert_{U_{w_i}})$ is the folding hypersurface.

It remains to show that $v=(\alpha_{w_2}^P)\circ (\alpha_{w_1}^P)^{-1}$ is a folded-symplectic map.  It suffices to produce a smooth folded-symplectic map $\mathfrak{v}$ satisfying

\begin{equation}
\mathfrak{v} \circ \alpha_{w_1}^P = \alpha^P_{w_2}
\end{equation}

We then have that $\mathfrak{v} = v$ on the domain of $(\alpha_{w_1}^P)^{-1}$, meaning $v$ is a folded-symplectic map.  By our discussion and lemma \ref{lem:fsmaps}, $v$ will be an isomorphism of toric folded-symplectic manifolds.  It suffices to consider the case when $U_{w_1} \subset U_{w_2}$ since one may then compose with inclusions and restrictions to get a commutative diagram:

\begin{displaymath}
\xymatrix{
cut(P\vert_{U_{w_1}})\vert_{(U_{w_1}\cap U_{w_2})} \ar[d] \ar[r]^{\simeq} & cut(P\vert_{(U_{w_1}\cap U_{w_2})}) \ar[d]& cut(P\vert_{U_{w_2}})\vert_{(U_{w_1}\cap U_{w_2})} \ar[l]_{\simeq} \ar[d]\\
(U_{w_1}\cap U_{w_2}) \ar[r]^{id} & (U_{w_1}\cap U_{w_2}) & (U_{w_1}\cap U_{w_2}) \ar[l]_{id}}
\end{displaymath}

The composition of the top row, which is $v$ or $v^{-1}$ depending on the order of composition, will then be a diffeomorphism.

\begin{enumerate}
\item First, consider the special case when $K_{w_1}=K_{w_2}$.  Then the collections of the corresponding weights $\{v_j^{w_1}\}_{j=1}^k$, $\{(v_j^{w_2})\}_{j=1}^k$ are the same set.  Hence, there exists a linear, symplectic isomorphism $\tilde{\mathfrak{v}}:\C^k \to \C^k$ permuting the coordinates and intertwining the two representations and corresponding moment maps, which we denote $\mu_{w_1}^1$ and $\mu_{w_1}^2$.  Consequently, $id\times \tilde{\mathfrak{v}}: P\vert_{U_{w_1}} \times \C^k \to P\vert_{U_{w_1}} \times \C^k$ induces a folded-symplectomorphism of reduced spaces:

    \begin{displaymath}
    \mathfrak{v}: (P\vert_{U_{w_1}} \times \C^k)//_0 K_w \to (P\vert_{U_{w_1}} \times \C^k)//_0K_w, \: [p,z] \to [p,\tilde{\mathfrak{v}}(z)].
    \end{displaymath}

    Note that $\mu_{w_1}^1(\C^k)= \mu_{w_1}^2(\C^k)$ and denote the corresponding sections as $s_1:\mu_{w_1}^1(\C^k)\to \C^k$ and $s_2:\mu_{w_1}^1(\C^k) \to \C^k$.  By definition, $\tilde{\mathfrak{v}}(s_1)=s_2$.  We have:

    \begin{displaymath}
    \begin{array}{lcl}
    \mathfrak{v}(\alpha_{w_1}^P([p])) & = & \mathfrak{v}([p,s_1(\xi_0 - \nu(p))]) \\
                                           & = & [p, \tilde{\mathfrak{v}}(s_1(\xi_0-\nu(p)))] \\
                                           & = & [p, s_2(\xi_0 - \nu(p))] \\
                                           & = & \alpha_{w_2}^P([p])
    \end{array}
    \end{displaymath}
    hence $\bar{\mathfrak{v}}\circ \alpha_{w_1}^P = \alpha_{w_2}^P$ and $(\alpha_{w_2}^P) \circ (\alpha_{w_1}^P)^{-1}$ is a folded-symplectomorphism.

\item In general, we have a strict inclusion $\{v_j^{w_1}\}_{j=1}^{k_1} \subset \{v_j^{w_2}\}_{j=1}^{k_2}$ since $w_2$ lies in a boundary component of possibly larger codimension in the quadrant $U_{w_1}\cap U_{w_2}$.  By our study of case $1$, we may assume that $v_j^{w_1}=v_j^{w_2}$ for all $1\le j \le k_1$.  We may then reduce the clutter in the notation by dropping the superscripts $^{(w_1)}$ and $^{(w_2)}$ and setting $K_i= K_{w_i}$, $i=1,2$.

    By construction of the neighbourhoods $U_{w_i}$ (q.v. lemma \ref{lem:cut}), we have:
    \begin{itemize}
    \item $\langle \psi(w_1) - \psi(w_2), v_i \rangle =0$ for $i=1,\dots,k_1$ and
    \item for all $w\in U_{w_1}$,
    \begin{displaymath}
    \langle \psi(w) - \psi(w_2), v_i \rangle > 0 \text{ for } i=k_1+1,\dots, k_2.
    \end{displaymath}
    \end{itemize}
    Consequently, for any point $p\in P\vert_{U_{w_1}}$, the functions

    \begin{displaymath}
    p \to \sqrt{\langle \mu(p) - \psi(w_2), v_i\rangle }
    \end{displaymath}
    are smooth for $i=k_1 + 1, \dots, k_2$.  Also, for $p\in P\vert_{U_{w_1}}$

    \begin{displaymath}
    \langle \mu(p) - \psi(w_2), v_i \rangle = \langle \mu(p) - \psi(w_1), v_i \rangle
    \end{displaymath}
    for $i=1,\dots, k_1$.  Now consider the map:

    \begin{displaymath}
    \tilde{\mathfrak{v}}:P\vert_{U_{w_1}} \times \C^{k_1} \to P\vert_{U_{w_1}} \times \C^{k_2}
    \end{displaymath}

    given by
    \begin{equation}
    \tilde{\mathfrak{v}}(p,z_1,\dots,z_{k_1}) = (p,z_1,\dots, z_{k_1}, \sqrt{\langle \mu(p)-\psi(w_2), v_{k_1+1} \rangle}, \dots, \sqrt{\langle \mu(p)-\psi(w_2), v_{k_2} \rangle}).
    \end{equation}

    The map $\tilde{\mathfrak{v}}$ is smooth and $K_1$-equivariant.  Since $\tilde{\mathfrak{v}}^*(dz_j\wedge d\bar{z}_j)$=0 for $j>k_1$, it is folded-symplectic.  We have that:

    \begin{displaymath}
    \tilde{\mathfrak{v}}^{-1}(\Phi_2^{-1}(0)) \subset \Phi_1^{-1}(0)
    \end{displaymath}
    where the $\Phi_j: P\vert_{U_{w_1}} \times \C^{k_j} \to \frak{k}^*_j$, $j=1,2$ are the corresponding $K_j$ moment maps (q.v. \ref{eq:cutmoment}).  This is because:

    \begin{displaymath}
    (p,z)\in \Phi_j^{-1}(0) \iff \langle \psi(\pi(p)) - \psi(w_j), v_i \rangle = \vert z_i \vert^2 \text{ for all } i=1,\dots, k_j
    \end{displaymath}

    Consequently, $\tilde{\frak{v}}$ descends to a well-defined smooth folded-symplectic map

    \begin{equation}
    \frak{v}: \Phi_1^{-1}(0)/K_1 \to \Phi_2^{1}(0)/K_2
    \end{equation}

    given by

    \begin{displaymath}
    \frak{v}([p,z_1,\dots,z_{k_1}]) = [p,z_1,\dots, z_{k_1}, \sqrt{\langle \mu(p) - \psi(w_2), v_{k_1+1} \rangle}, \dots, \sqrt{\langle \mu(p) - \psi(w_2), v_{k_2} \rangle}].
    \end{displaymath}

    We have that $\frak{v}(\alpha_{w_1}^P([p]))= \alpha_{w_2}^P([p])$, which means the transition maps are isomorphisms of toric folded-symplectic manifolds over $U_{w_1}$.
\end{enumerate}

We define $c(P)$ to be $c_{top}(P)$ equipped with the structure of a toric folded-symplectic manifold endowed by the charts $\{c_{top}(P)\vert_{U_w}, \alpha_{w}^P\}_{w\in W}$.
\subsection{Step 3- Show $c$ is Functorial}
We finish by showing that any isomorphism $\phi:P_1\to P_2$ of toric folded-symplectic bundles induces an isomorphism $c(\phi)$. Recall that $\phi:P_1\to P_2$ induces a continuous map $c_{top}(\phi):c_{top}(P_1) \to c_{top}(P_2)$ given by $c_{top}(\phi)([p])=[\phi(p)]$.

By remark \ref{rem:commute}, we have that $cut(\phi)\circ \alpha_{w}^{P_1} = \alpha^{P_2}_w\circ c_{top}(\phi)$, hence $c_{top}(\phi)\vert_{U_w}= (\alpha^{P_2}_w)^{-1}\circ cut(\phi)\circ \alpha_w^{P_1}$ is smooth, equivariant, and folded-symplectic.  Because $\phi:P_1\to P_2$ covers $id:W\to W$, it follows that $c_{top}(\phi)$ covers $id:W\to W$, hence $c_{top}(\phi)$ is an isomorphism of toric folded-symplectic manifolds, which we denote $c(\phi)$.

Finally, because $c_{top}$ is a functor, it follows that $c$ satisfies the requirements to be a functor.  Hence, we have constructed a functor $c:\mathcal{B}_{\psi} \to \mathcal{M}_{\psi}$ as required.

\begin{remark}\label{rem:restriction1}
 We use remark \ref{rem:restriction} and restrict the structure maps $\alpha^P_w$ to see that for all $P\in \operatorname{Ob}(\mathcal{B}_{\psi})$ and for all open subsets $U\subset W$
\begin{displaymath}
c(P)\vert_U = c(P\vert_U)
\end{displaymath}
and for any morphism $\phi:P_1\to P_2$,
\begin{displaymath}
c(\phi)\vert_U = c(\phi\vert_U).
\end{displaymath}
Hence, we write $c(\phi)\vert_U$ when referring to $c(\phi\vert_U)$.
\end{remark}

We end the section with a lemma that will be used to prove that $c:\mathcal{B}_{\psi}(W) \to \mathcal{M}_{\psi}(W)$ is an equivalence of categories.
\begin{lemma}\label{lem:idfunctor}
The functor $c:\mathcal{B}_{\psi}\to \mathcal{M}_{\psi}$ is a map of presheaves of groupoids (q.v. remark \ref{rem:empsi}).  Moreover, if $\mathring{W}$ denotes the interior of $W$, then $c_{\mathring{W}}:\mathcal{B}_{\psi}(\mathring{W})\to \mathcal{M}_{\psi}(\mathring{W})$ is isomorphic to the identity functor.
\end{lemma}
\begin{proof}
The fact that $c$ is a map of presheaves of groupoids follows from the fact that $c_{top}$ is a map of presheaves of groupoids.  Over the interior $\mathring{W}$, the functor $c_{top}$ is isomorphic to the identity functor since the subtorus associated to any point $w\in \mathring{W}$ is $\{e\}$.  Because this subtorus is trivial we have
\begin{displaymath}
cut(P\vert_{U_w}) = (P\vert_{U_w} \times \{0\})//_0\{e\} \simeq P\vert_{U_w}
\end{displaymath}
as toric folded-symplectic manifolds over $U_w$.
\end{proof}

\subsection{$c:\mathcal{B}_{\psi}(W) \to \mathcal{M}_{\psi}(W)$ is an Equivalence of Categories}
We now prove the following theorem which states that $c$ is an equivalence of categories.  At the very end of the section, we will use this equivalence to provide the classification result for toric, folded-symplectic manifolds with co-orientable folding hypersurface.

\begin{theorem}\label{thm:equi-cats}
Let $\psi:W\to \fg^*$ be a unimodular map with folds.  The functor:
\begin{displaymath}
c:\mathcal{B}_{\psi}(W) \to \mathcal{M}_{\psi}(W)
\end{displaymath}
is an equivalence of categories.
\end{theorem}

\begin{remark}
The strategy for proving theorem \ref{thm:equi-cats} is borrowed from \cite{KL} and our proof is virtually identical: the main ingredients are the functor $c$, the classification of objects in $\mathcal{B}_{\psi}(W)$, and local equivalence (q.v. lemma \ref{lem:locunique}) of objects in $\mathcal{M}_{\psi}(W)$.  We list it here for the sake of completeness.  It works as follows:
\begin{enumerate}
\item We first recall that $c:\mathcal{B}_{\psi}\to \mathcal{M}_{\psi}$ is a map of presheaves of groupoids (q.v. lemma \ref{lem:idfunctor}).
\item Show that $c:\mathcal{B}_{\psi}(U)\to \mathcal{M}_{\psi}(U)$ is an equivalence of categories for every open subset $U\subseteq W$.  Hence, it is an isomorphism of presheaves of groupoids.
\end{enumerate}
\end{remark}

\begin{remark}
As a brief reminder, recall our notation for restrictions of objects and maps.  Suppose $U\subseteq W$ is open.
\begin{itemize}
\item If $\pi:M\to W$ is an object of $\mathcal{M}_{\psi}$ or $\pi:P\to W$ is an object of $\mathcal{B}_{\psi}$, we will use the notation $M\vert_U$, $P\vert_U$ to stand for the objects $\pi: \pi^{-1}(U)\to W$ in $\mathcal{M}_{\psi}(U), \mathcal{B}_{\psi}(U)$, respectively.
\item Given $\phi\in \hom(P_1,P_2)$, we write $\phi\vert_U$ to mean $\phi\vert_{P\vert_U}$.  Similarly, for $\varphi\in \hom(c(P_1),c(P_2))$ we write $\varphi\vert_U$ to mean $\varphi\vert_{c(P_1)\vert_{U}}$.
\end{itemize}
\end{remark}

The proof will be given in a series of lemmas.  We begin by showing that $c$ is fully faithful, which we will use to prove it is essentially surjective.

\begin{lemma}\label{lem:faithful}
For any open subset $U\subseteq W$ and for any two objects $P_1,P_2 \in \operatorname{Ob}(\mathcal{B}_{\psi}(U))$, the map:
\begin{displaymath}
c_U:\hom(P_1,P_2) \to \hom(c(P_1),c(P_2))
\end{displaymath}
is injective.
\end{lemma}

\begin{proof}
Let $\mathring{W}$ denote the interior of $W$ and let $\mathring{U}=U\cap \mathring{W}$.  By lemma \ref{lem:idfunctor}, $c_{\mathring{U}}:\hom(P_1\vert_{\mathring{U}}, P_2\vert_{\mathring{U}}) \to \hom (c(P_1)\vert_{\mathring{U}}, c(P_2)\vert_{\mathring{U}})$ is isomorphic to the identity.  There are isomorphisms $\delta_{P_1}, \delta_{P_2}$ so that for all $\phi\in \hom(P_1\vert_{\mathring{U}},P_2\vert_{\mathring{U}})$ the diagram

\begin{displaymath}
\xymatrix{
P_1\vert_{\mathring{U}} \ar[d]^{\phi} \ar[r]^{\delta_{P_1}} & c(P_1)\vert_{\mathring{U}} \ar[d]^{c(\phi)} \\
P_2\vert_{\mathring{U}} \ar[r]^{\delta_{P_2}}               & c(P_2)\vert_{\mathring{U}}}
\end{displaymath}
commutes.  Consequently, if $\phi_1,\phi_2 \in \hom(P_1,P_2)$ and $c(\phi_1)=c(\phi_2)$ then their restrictions $\mathring{\phi}_i:=\phi_i\vert_{\mathring{U}}$ satisfy:
\begin{displaymath}
\mathring{\phi}_i = (\delta_{P_2})^{-1}\circ c(\phi_i) \circ \delta_{P_1}
\end{displaymath}
Since $c(\phi_1)=c(\phi_2)$ by assumption, we have that $\mathring{\phi}_1=\mathring{\phi}_2$.  Since $\mathring{U}$ is dense in $U$, this implies that $\phi_1=\phi_2$.
\end{proof}

We will use the following theorem, which is theorem 3.1 in \cite{HaS}, to show that $c$ is also surjective as a map on $\hom$-sets.

\begin{theorem}\label{thm:has} Let $M$ be a manifold with an action of a torus $G$ and $h:M\to M$ a $G$-eequivariant diffeomorphism with $h(x)\in G\cdot x$ for all points $x\in M$.  Let $\pi:M\to M/G$ be the orbit map.  Then there exists a map $f:M/G \to G$ such that
\begin{displaymath}
h(x)= f(\pi(x))\cdot x
\end{displaymath}
for all $x\in M$ and such that $f\circ\pi$ is smooth.
\end{theorem}

\begin{lemma}\label{lem:surjective1}
For any open subset $U$ of $W$ and for any $P\in \operatorname{Ob}(\mathcal{B}_{\psi}(U))$ the map:
\begin{displaymath}
c:\hom(P,P) \to \hom(c(P),c(P))
\end{displaymath}
is onto.
\end{lemma}
\begin{proof}
By theorem \ref{thm:has}, given $\varphi\in \hom(c(P),c(P))$ there is a function $f:U\to G$ so that
\begin{displaymath}
\varphi(x)=f(\pi(x))\cdot x
\end{displaymath}
where $\pi:c(P)\to U$ is the quotient map and $f\circ \pi$ is smooth.  As in the proof of lemma \ref{lem:faithful}, we have
\begin{displaymath}
\varphi\vert_{\mathring{U}} = c(\mathring{\phi})
\end{displaymath}
where $\mathring{\phi}$ is given by:
\begin{displaymath}
\mathring{\phi}=(\delta_P)^{-1}\circ \varphi\vert_{\mathring{U}} \circ \delta_P.
\end{displaymath}
Hence for $p\in P\vert_{\mathring{U}}$,
\begin{displaymath}
\mathring{\phi}(p) = (\delta_P)^{-1}(f(\pi(\delta_P(p)))\cdot \delta_P(p))=(\delta_P)^{-1}(f(\pi(p))\cdot \delta_P(p)) = f(\pi(p))\cdot (\delta_P)^{-1}(\delta_P(p)) = f(\pi(p))\cdot p.
\end{displaymath}
Define the map $\phi:P \to P$ by
\begin{displaymath}
\phi(p) := f(\pi(p))\cdot p \text{ for all $p\in P$.}
\end{displaymath}
This map is $G$-equivariant and commutes with $\pi:P\to U$.  Since $f\circ \pi$ is smooth, the map $\phi$ is a diffeomorphism with inverse $\phi^{-1}(p)=f(\pi(p))^{-1}\cdot p$.  Moreover, since the restriction of $\phi$ to $P\vert_{\mathring{U}}$ is $\mathring{\phi}$, the map $\phi$ is folded-symplectic on $P\vert_{\mathring{U}}$.  Since this set is dense in $P$, we conclude that $\phi$ is folded-symplectic on all of $P$.  Since it maps the folding hypersurface to itself, lemma \ref{lem:foldtofold} implies that $\phi\in \hom(P,P)$.  It remains to check that $c(\phi)=\varphi$.  The functor $c$ commutes with restrictions to $P\vert_{\mathring{U}}$, hence
\begin{displaymath}
c(\phi)\vert_{c(P)\vert_{\mathring{U}}} = c(\mathring{\phi}) = \varphi\vert_{c(P)\vert_{\mathring{U}}}
\end{displaymath}
by construction.  Hence $c(\phi)$ and $\varphi$ are the same on an open dense subset.  Smoothness implies that they are the same on all of $c(P)$.
\end{proof}

\begin{lemma}\label{lem:fullyfaithful}
Suppose $U\subseteq W$ is an open subset with $H^2(U,\mathbb{Z})=0$.  Then for any $P_1,P_2 \in \operatorname{Ob}(\mathcal{B}_{\psi}(U))$ the map
\begin{displaymath}
c = c_U: \hom(P_1,P_2) \to \hom(c(P_1),c(P_2)),
\end{displaymath}
is a bijection.
\end{lemma}

\begin{proof}
Lemma \ref{lem:faithful} shows that $c_U$ is injective, hence we only need to show that it is surjective.  Let $\varphi\in \hom(c(P_1),c(P_2))$.  By theorem \ref{thm:bundleclassification}, there exists $\phi\in \hom(P_1,P_2)$.  Then $c(\phi)\in \hom(c(P_1),(P_2))$ and $c(\phi)^{-1}\circ\varphi \in \hom(c(P_1),c(P_1))$.  By lemma \ref{lem:surjective1}, there exists $\nu \in \hom(P_1,P_1)$ such that $c(\nu)=c(\phi)^{-1}\circ \varphi$.  Consequently,
\begin{displaymath}
\varphi = c(\phi) \circ c(\nu) = c(\phi\circ \nu)
\end{displaymath}

Since $\phi \circ \nu \in \hom(P_1,P_2)$, we are finished.
\end{proof}

To finish the proof that $c$ is fully faithful, it is enough to note that the functors $\underline{\hom}(P_1,P_2)$ and $\underline{\hom}(c(P_1),c(P_2))$ given by
\begin{displaymath}
\underline{\hom}(P_1,P_2)(U):= \hom(P_1\vert_U, P_2\vert_U)
\end{displaymath}
and
\begin{displaymath}
\underline{\hom}(c(P_1),c(P_2))(U):= \hom(c(P_1)\vert_U, c(P_2)\vert_U)
\end{displaymath}
are sheaves.  We have shown that $c$ commutes with restrictions, hence
\begin{displaymath}
c=c_u: \underline{\hom}(P_1,P_2)(U) \to \underline{\hom}(c(P_1),c(P)2))(U)
\end{displaymath}
is a map of sheaves.  By lemma \ref{lem:fullyfaithful} the map $c_U$ is a bijection for any contractible open set $U$.  It follows that $c:\underline{\hom}(P_1,P_2) \to \underline{\hom}(c(P_1),c(P_2))$ is an isomorphism of sheaves.  Consequently, it is a bijection on global sections.  That is,
\begin{displaymath}
c:\hom(P_1,P_2)\to \hom(c(P_1),c(P_2))
\end{displaymath}
is a bijection.  It remains to show that $c$ is essentially surjective.

\begin{lemma}\label{lem:descent}
Let $\{U_i\}_{i\in I}$ be an open cover of $W$, $U_{ij}:= U_i\cap U_j$, and $U_{ijk}:= U_i\cap U_j\cap U_k$ for all $i,j,k \in I$. Suppose we have a collection of objects $P_i\in \mathcal{B}_{\psi}(U_i)$ and isomorphisms $\Phi_{ij}:P_j\vert_{U_{ij}} \to P_i\vert_{U_{ij}}$ defining a cocycle: $\Phi_{ii}=id, \Phi_{ji}=\Phi_{ij}^{-1}$, and
\begin{displaymath}
\Phi_{ij}\vert_{U_{ijk}} \circ \Phi_{jk}\vert_{U_{ijk}} \circ \Phi_{ki}\vert_{U_{ijk}} = id
\end{displaymath}
for all triples $i,j,k\in I$.  Then there exists an object $P\in \mathcal{B}_{\psi}(W)$ and isomorphisms $\gamma_i:P\vert_{U_i} \to P_i$ so that
\begin{equation}\label{eq:descent}
\xymatrix{
P_j\vert_{U_{ij}} \ar[d]^{\Phi_{ij}} & \ar[l]^{\gamma_j} P\vert_{U_{ij}} \ar[d]^{\simeq} \\
P_i\vert_{U_{ij}} & \ar[l]^{\gamma_i} P\vert_{U_{ij}}}
\end{equation}
commutes.
\end{lemma}

\begin{proof}
We may take $P=(\sqcup_{i\in I} P_i)/{\sim}$ where $\sim$ is the equivalence relation defined by the $\Phi_{ij}'s$.  Then $P$ is a principal $G$-bundle over $W$ and the symplectic $G$-invariant folded-symplectic forms on the $P_i's$ define a $G$-invariant folded-symplectic form on $P$ (because the $\Phi_{ij}'s$ are folded-symplectic maps).  The maps $\gamma_i^{-1}:P_i \to P\vert_{U_i}$ are induced by the inclusions $P_i\hookrightarrow \sqcup_{j\in I}P_j$.
\end{proof}

\begin{lemma}\label{lem:essentially-surjective}
For any open subset $U\subseteq W$ the functor $c:\mathcal{B}_{\psi}(U) \to \mathcal{M}_{\psi}(U)$ is essentially surjective.
\end{lemma}

\begin{proof}
Given $M \in \mathcal{M}_{\psi}(U)$, we want to show that it is isomorphic to $c(P)$ for some $P\in \mathcal{B}_{\psi}(U)$. Since $\mathcal{B}_{\psi}(U)$ is nonempty, we may choose an object $P'\in \mathcal{B}_{\psi}(U)$.  By lemma \ref{lem:locunique} $c(P')$ and $M$ are locally isomorphic.  Therefore there is a cover $\{U_i\}_{i\in I}$ of $U$ and a family of isomorphisms $\{\phi_i:c(P')\vert_{U_i} \to M\vert_{U_i}\}.$  Set
\begin{displaymath}
P_i:= P'\vert_{U_i}.
\end{displaymath}
Consider the collection of isomorphisms
\begin{displaymath}
\phi_{ij}:= (\phi_i\vert_{U_{ij}})^{-1} \circ \phi_j\vert_{U_{ij}} : c(P_j)\vert_{U_{ij}} \to c(P_i)\vert_{U_{ij}}, \space i,j \in I
\end{displaymath}
Since $c$ is fully faithful, there are unique isomorphisms
\begin{displaymath}
\Phi_{ij}:P_j\vert_{U_{ij}} \to P_i\vert_{U_{ij}}
\end{displaymath}
with $c(\Phi_{ij}) = \phi_{ij}$.  Since $c$ commutes with restrictions to open subsets and since $\{\phi_{ij}\}_{i,j\in I}$ form a cocycle and the $\Phi_{ij}$ are unique, $\{\Phi_{ij}\}_{i,j\in I}$ form a cocycle as well.  By lemma \ref{lem:descent} there is $P\in \operatorname{Ob}(B_{\psi}(U))$ and a family of isomorphisms $\{\gamma_i:P\vert_{U_i} \to P_i\}$ so that \ref{eq:descent} commutes. Then
\begin{displaymath}
\xymatrix{
M\vert_{U_{ij}} \ar[d] & c(P_j)\vert_{U_{ij}} \ar[l]_{\phi_j} \ar[d]^{c(\Phi_{ij})} & c(P)\vert_{U_{ij}} \ar[l]_{c(\gamma_j)} \ar[d]^{Id} \\
M\vert_{U_{ij}} & \ar[l]^{\phi_i} c(P_i)\vert_{U_{ij}} & c(P)\vert_{U_{ij}} \ar[l]^{c(\gamma_i)}}
\end{displaymath}
commutes as well.  Consequently

\begin{displaymath}
\phi_i\circ c(\gamma_i)\vert_{U_{ij}} = \phi_j \circ c(\gamma_j)\vert_{U_{ij}}.
\end{displaymath}
Since $\hom(c(P),M)$ is a sheaf on $U$, the family $\{\phi_i\circ c(\gamma_i) : c(P)\vert_{U_i} \to M_{U_i}\}$ gives rise to a well defined isomorphism $\varphi:c(P)\to M$.
\end{proof}

Since $U$ was arbitrary, we can take $U=W$ and lemma \ref{lem:essentially-surjective} shows that $c:\mathcal{B}_{\psi}(W)\to \mathcal{M}_{\psi}(W)$ is essentially surjective.  This completes the proof of theorem \ref{thm:equi-cats} and gives us the following classification theorem for toric, folded-symplectic manifolds with co-orientable folding hypersurface.

\begin{theorem}\label{thm:Classification}
Let $\psi:W\to \fg^*$ be a unimodular map with folds, where $\fg$ is the Lie aglebra of a torus $G$.  Let $\mathcal{M}_{\psi}(W)$ be the category of toric, folded-symplectic manifolds over $\psi$ (necessarily with co-orientable folding hypersurface).  Then

\begin{displaymath}
\pi_0(\mathcal{M}_{\psi}(W)) = H^2(W,\mathbb{Z}_G \times \R)
\end{displaymath}
That is, isomorphism classes of toric, folded-symplectic manifolds are in bijection with cohomology classes in $H^2(W,\mathbb{Z}_G \times \R)$.  In particular, to every toric, folded-symplectic manifold $(M,\sigma,\pi:M\to W)$ over $\psi:W \to \fg^*$, we may associate a first Chern class $c_1(M)\in H^2(W,\mathbb{Z}_G)$ and a horizontal Chern class $c_{hor}(\sigma)$.
\end{theorem}

\begin{proof}\mbox{ } \newline
An equivalence of categories induces a bijection on isomorphism classes of objects.  Since $\pi_0(\mathcal{B}_{\psi}(W))= H^2(W,\mathbb{Z}_G \times \R)$ by theorem \ref{thm:bundleclassification}, theorem \ref{thm:equi-cats} implies that $H^2(W,\mathbb{Z}_G\times \R) = \pi_0(\mathcal{B}_{\psi}(W) = \pi_0(\mathcal{M}_{\psi}(W))$.
\end{proof}


\pagebreak

\begin{appendix}
\section{Manifolds with Corners}
\subsection{Definitions and Conventions}
We give basic definitions of manifolds with corners and do not go into much depth regarding how manifolds with corners should be treated.  The goal of the appendix is to define what it means for a map $f:M \to N$ of manifolds with corners to be transverse to a submanifold with corners $S\subset N$.  We then show that $f^{-1}(S)$ is a submanifold with corners of $M$ if $f$ is transverse to $S$.
\begin{definition}
A manifold with corners $W$ is a Hausdorff, second countable topological space with a collection of charts $(U_i,\phi_i)$, where $\phi_i:U_i \to \R^k \times (\R^+)^h$ is a homeomorphism onto an open subset of the quadrant $\R^k \times (\R^+)^h$.  The transition maps $\phi_i\circ \phi_j^{-1}$ are required to be diffeomorphisms in the sense that they are restrictions of diffeomorphisms defined on open subsets of $\R^n$ to the quadrants.  A \emph{submanifold with corners} $S\subset W$ is a topological subspace $S$ so that for each point $p\in S$ there is a chart $(U,\phi)$, $\phi:U \hookrightarrow \R^k \times (\R^+)^h$, where $\phi(S\cap U)$ is the zero set of a subset of the coordinates.
\end{definition}

\begin{definition}
Let $W$ be a manifold with corners.  In each coordinate chart $(U_i,\phi_i)$ one has the notion of the depth of a point, which is given by how many half-space coordinates are $0$.  Let $x_i$ be the coordinates on $\R^k$ and let $y_j$ be the coordinates on $(\R^+)^h$.  Then,
\begin{displaymath}
\operatorname{depth}_W(x_1,\dots,x_k,y_1,\dots, y_j) = \vert \{y_l \vert \mbox{ } y_l=0\} \vert
\end{displaymath}
That is, the depth is the number of $y_j's$ that are $0$.  Since the transition maps are diffeomorphisms of manifolds with corners, they preserve the depth function, hence $\operatorname{depth}_W$ is a well-defined map from $W$ to the integers.  The $k$-boundary $\partial^k(W)$ of $W$ consists of all points of depth $k$.  It is a smooth manifold, hence we have a decomposition of $W$ into smooth manifolds:

\begin{displaymath}
\sqcup_{i=1}^n \partial^k(W)
\end{displaymath}
We often simply refer to the $\partial^k(W)'s$ as the strata of $W$.
\end{definition}

\subsection{Transversality and Submanifolds with Corners}

Throughout this section, we will assume the following two statements from the differential geometry of manifolds (without corners) are true:

\begin{prop}\label{prop:A1:man1}
Let $M$,$N$ be two smooth manifolds (without corners) and let $S\subset N$ be a codimension $s$ submanifold (without corners).  If $f:M \to N$ is a smooth map satisfying $f\pitchfork S$, then $f^{-1}(S)$ is a smooth codimension $s$ submanifold (without corners) of $M$.
\end{prop}

\begin{prop}\label{prop:A1:man2}
Suppose $f:M \to N$ is a smooth map of $m$-dimensional manifolds (without corners).  If $p\in M$ and $\operatorname{rank}(df_p)=m$, then there exists a neighborhood $U\subset M$ of $p$ such that $f\vert_U$ is a diffeomorphism.
\end{prop}

\begin{cor}\label{cor:A1:man2}
Suppose $f:M \to N$ is a smooth map between manifolds with corners and suppose it is strata-preserving.  That is $f(\partial^k(M))\subseteq \partial^k(N)$ for all $k\ge 0$.  Suppose $p\in M$ is a point where $df_p$ is an isomorphism.  Then there exists a neighborhood $U$ of $p$ such that $f\big\vert_U$ is an open embedding of manifolds with corners.  In particular, it is a diffeomorphism onto its image.
\end{cor}

\begin{proof}
This is more of an observation than anything.  If $p\in M$ is a regular point, we may choose a neighborhood $U$ around $p$ that is isomorphic to a quadrant in $\R^m$.  The map $f$ extends to a smooth map $\tilde{f}$ in a neighborhood of $p$ in $\R^n$.  Since $df_p$ is an isomorphism at $p$, $d\tilde{f}_p$ is an isomorphism and there is a neighborhood $V$ of $p$ on which it is an open embedding.  The restriction of $f$ to $V\cap U$ is then an open embedding of manifolds with corners since, by assumption, $f$ is strata-preserving.
\end{proof}

The following definition is definition 4 of \cite{CD}.  Proposition \ref{prop:A1:trans} is theorem 6 of \cite{CD}.  Our proof of proposition \ref{prop:A1:trans} is similar, but we make a few modifications.

\begin{definition}\label{def:A1:trans}
Let $M$ be an $m$-dimensional manifold with corners, $N$ an $n$-dimensional manifold with corners, $S\subset N$ an $s$-dimensional submanifold with corners of $N$, and suppose $f:M\to N$ is smooth.  We say $f\pitchfork_s S$ if for all $k>0$ we have:
\begin{center}
\begin{displaymath}
f\vert_{\partial^k(M)} \pitchfork S,
\end{displaymath}
\end{center}
meaning $df_p(T_p\partial^k(M)) + T_{f(p)}S = T_{f(p)}N$ whenever $p\in f^{-1}(S)$.  In other words, we say $f\pitchfork_s S$ if its restriction to each stratum of $M$ is transverse to $S$ in the traditional sense of manifolds (without corners).
\end{definition}

\begin{prop}\label{prop:A1:trans}
Let $M$ be an $m$-dimensional manifold with corners, $N$ an $n$-dimensional manifold with corners, $S\subset N$ an codimension $s$ submanifold with corners of $N$, and suppose $f:M\to N$ is a smooth map such that $f\pitchfork_s S$ in the sense of definition \ref{def:A1:trans}.  Then $f^{-1}(S)$ is a smooth submanifold with corners of $M$ with codimension $s$.
\end{prop}

\begin{proof} \mbox{} \newline
Let $p\in f^{-1}(S)$.  We will prove the proposition in three stages.  First, we'll show that it is sufficient to study the case when $N=\R^l$ for some $l$ and $S=\{0\}$.  Next, we'll produce a preliminary change of coordinates near $p$ that will exhibit $S$ as a zero set on the stratum containing $p$, but may not do so away from the stratum.  We'll finish by applying a secondary diffeomorphism to fix the problem.

\begin{enumerate}
\item By definition of submanifold with corners, there exists a projection $P$ defined in a neighborhood $V$ of $f(p)$, $P:V \to \R^s$, so that $P^{-1}(0) = V\cap S$.  Then $F= P \circ f : f^{-1}(U) \to R^s$ is a smooth map satisfying
    \begin{enumerate}
    \item $F^{-1}(0)=S\cap f^{-1}(U)$ and
    \item $F$ restricted to a stratum is transverse to $0$, hence
    \item $0$ is a regular value of $F$,
    \end{enumerate}
    Thus, proving $f^{-1}(S)$ is a submanifold with corners near $p$ is equivalent to showing $F^{-1}(0)$ is a submanifold with corners near $p$

\item We may assume that a neighborhood of $p$ in $f^{-1}(U)$ is the product $W=[0,\epsilon)^k \times (-\epsilon, \epsilon)^{m-k}$ for some $\epsilon >0$ with coordinates $(x_1,\dots, x_k, y_1,\dots, y_{m-k})$ and $p=0$ is the origin.  Then we have a smooth map $F:W \to \R^s$, which extends to a smooth map $\tilde{F}:\tilde{W} \to \R^s$, where $\tilde{W}$ is a neighborhood of the origin.  This is simply the definition of what it means to be smooth on a subset of $\R^m$.
    \begin{itemize}
    \item Since $0$ is a regular value of $F$, we can assume $0$ is a regular value of $\tilde{F}$, shrinking $\tilde{W}$ if necessary.  We may also assume that $\tilde{W} = (-\epsilon, \epsilon)^m$ by shrinking $\epsilon$.
    \item The transversality assumption, $f\vert_{\partial^i(M)} \pitchfork_s S$, implies that $\tilde{F}\vert_{\{0\}\times (-\epsilon,\epsilon)^{m-k}} \pitchfork_s S$, hence $\tilde{F}^{-1}(0) \cap \{0\}\times (-\epsilon, \epsilon)^{m-k}$ is a codimension $s$ submanifold of $\{0\} \times (-\epsilon, \epsilon)^k$.
    \item Consequently, there is a diffeomorphism $\phi: \{0\} \times (-\epsilon, \epsilon)^{m-k}\to \{0\} \times (-\epsilon, \epsilon)$ of the stratum containing the $p=0$ so that:
        \begin{displaymath}
        \phi(\tilde{F}^{-1}(0)\cap \{0\}\times (-\epsilon,\epsilon)^{m-k})= \{(\vec{0}, y_1, \dots, y_{m-k})\vert \mbox{ } y_i=0 \text{ for } i>(m-k-s)\}
        \end{displaymath}
    \item Extend $\phi$ to the product neighborhood $\tilde{W}= (-\epsilon,\epsilon)^k \times (-\epsilon, \epsilon)^{m-k}$ using $\tilde{\phi}=id \times \phi$, where $id:(-\epsilon,\epsilon)^k \to (-\epsilon,\epsilon)^k$ is the identity.
    \end{itemize}

\item Now, note that our discussion in $2$ shows that $\tilde{F}^{-1}(0)\pitchfork_s \{0\} \times (-\epsilon,\epsilon)^{m-k}$.  That is, $\tilde{F}^{-1}(0)$ is transverse to the stratum containing $p=0$.  This is because $d\tilde{F}$ vanishes on $m-k-s$ directions in the stratum and \emph{doesn't} vanish on the other $s$ directions, hence it must vanish on some $k$ directions transverse to the stratum.  The same is then true for $\tilde{\phi}(\tilde{F}^{-1}(0))$ since $\tilde{\phi}$ is a diffeomorphism that preserves $\{0\} \times (-\epsilon,\epsilon)^{m-k}$.  Let $S'=\tilde{\phi}(\tilde{F}^{-1}(0))$.
    \begin{itemize}
    \item Consequently, the projection map $\gamma(x_1,\dots,x_k, y_1, \dots ,y_{m-k}) = (x_1,\dots,x_k,y_1,\dots, y_{m-k-s},0,\dots, 0)$ restricted to $S'$ is a diffeomorphism in a neighborhood of $0$.
    \item Let $pr:(-\epsilon, \epsilon)^{k} \times (-\epsilon,\epsilon)^{m-k} \to (-\epsilon, \epsilon)^{k}$ be the projection onto the first $k$ factors.  We define the map:

        \begin{displaymath}
        \Gamma(x_1,\dots,x_k,y_1,\dots,y_{m-k-s},\vec{z}) = (x_1,\dots,x_k,y_1,\dots,y_{m-k-s},\vec{z}-pr(\gamma\vert_{S'}^{-1}(\vec{x},\vec{y},0)))
        \end{displaymath}
        where $\vec{z}=(y_{m-k-s+1},\dots,y_{m-k})$, $\vec{x}=(x_1,\dots x_k)$, and $\vec{y}=(y_1,\dots,y_{m-k-s})$.
    \item Since $\gamma^{-1}$ doesn't depend on the coordinates $y_{m-k-s+1},\dots, y_{m-k}$ it is straightforward to show that $\Gamma$ is a submersion, hence it is a diffeomorphism in a neighborhood of $0$.
    \item Note that if $(x_1,\dots,x_k, y_1,\dots,y_{m-k-s},\vec{z})\in S'$, then $\vec{z}=pr((\gamma\vert_{S'})^{-1}(\vec{x},\vec{y},0))$, hence
        \begin{displaymath}
        \Gamma(x_1,\dots,x_k, y_1,\dots,y_{m-k-s},\vec{z}) = (x_1,\dots,x_k,y_1,\dots,y_{m-k-s},\vec{0})
        \end{displaymath}
        Conversely, $\Gamma(p_0)$ has vanishing $y_{m-k-s+1},\dots,y_{m-k}$ coordinates if and only if it has a preimage in $S'$.  Thus, $\Gamma$ maps $S'$ diffeomorphically onto the set $\{(x_1,\dots,x_k,y_1,\dots,y_{m-k-s},\vec{0})\vert \mbox{} x_i,y_i\in \R\}$.
    \end{itemize}
    To finish the proof, we simply compose the two diffeomorphisms $\Gamma$ and $id\times \phi$ to get a diffeomorphism exhibiting $\tilde{F}^{-1}(0)$ as the set $\{(x_1,\dots,x_k,y_1,\dots,y_{m-k-s},\vec{0})\vert \mbox{} x_i,y_i \in \R\}$.
\end{enumerate}
\end{proof}

\begin{remark}
In particular, the folding hypersurface is transverse to the faces of $M$.
\end{remark}

We have a version of Hadamard's lemma for vector bundles over manifolds with corners.
\begin{lemma}\label{lem:had}
Let $Z$ be a smooth manifold with corners and let $\pi:E\to Z\times \R$ be a rank $k$ vector bundle over $Z\times \R$.  Denote the coordinate on $\R$ by $t$.  Suppose $\beta:Z\times \R \hookrightarrow E$ is a smooth section satisfying $\beta(z,0)=0$ for all $z\in Z$.  Then there exists a unique smooth section $\mu:Z\times \R \hookrightarrow E$ satisfying $\beta= t\mu$.
\end{lemma}

\begin{proof}
We will show that if $\mu$ exists, then it is unique.  We will then show that one may always solve $\beta=t\mu$ for $\mu$ locally, after which we invoke uniqueness to patch together the local solutions.
\begin{enumerate}
\item We first address uniqueness.  If $\mu_1, \mu_2$ are two sections satisfying $t\mu_1=\beta=t\mu_2$, then $\mu_1=\frac{\beta}{t}=\mu_2$ on the open, dense subset $\{(z,t)\vert t\ne0\}$.  Since $\mu_1,\mu_2$ are smooth, $\mu_1=\mu_2$ everywhere.

\item We now show that $\mu$ exists.  Choose a trivialization $(U\times \R,\Phi)$ of $\pi$, where $U\subset Z$, so we have that $\Phi:E\vert_{U\times \R} \to U\times \R \times \R^k$ is an isomorphism of vector bundles.  Let $p:U\times \R\times \R^k\to \R^k$ be the standard projection.  Then $f:=p(\Phi(\beta\vert_U)):U\times \R \to \R^k$ is a vector-valued function on $U\times \R$ that satisfies $f(z,0)=0$ for each $z\in U$.  If $g:U\times \R \to \R$ is a function satisfying $g(z,0)=0$ for each $z\in U$, then we have:
    \begin{equation}
    g(z,t)=\int_0^1\frac{\partial}{\partial s} g(z,st)ds = \int_0^1 dg_{(z,st)}(t\pt)ds = t\int_0^1 dg_{(z,st)}(\pt)ds
    \end{equation}
    hence $g=th$ for some smooth function $h$.  The same reasoning then applies to vector-valued functions and so we have $f=tF$ for some smooth map $F:U\times \R \to \R^k$.  Therefore, $\Phi(\beta\vert_U)(z,t)=(z,t, tF(z,t))$ and if we define $\mu_U(z,t) = \Phi^{-1}(z, t, F(z,t))$ we obtain a local section $\mu_U$ of $\pi$ satisfying $t\mu_U = \beta\vert_U$.  Here, we are using that $\Phi$ is a linear map on the fibers of $E\vert_U$ and $U\times \R \times \R^k$.

    Cover $Z\times \R$ by a collection of neighborhoods $C=\{U\times \R\}$ so that $E\vert_{U\times \R}$ is trivializable for each $U\times \R \in C$.  Then we obtain a collection of sections $\{\mu_U\}_{U\times \R \in C}$ which glue together by uniqueness to give us $\mu$ satisfying $t\mu = \beta$.
\end{enumerate}
\end{proof}
\end{appendix}

\pagebreak

\end{document}